\documentclass[10pt]{amsart}

\usepackage{amsmath,amssymb,graphicx,amscd,pinlabel,yhmath}

\usepackage[all]{xy}

\usepackage[dvipsnames]{xcolor}
\SelectTips{cm}{}

\usepackage{hyperref}
\hypersetup{colorlinks=true,citecolor=brown,linkcolor=brown,urlcolor=black,filecolor=black}

\numberwithin{equation}{section}
\allowdisplaybreaks

\newtheorem{theorem}{Theorem}[section]
\newtheorem{proposition}[theorem]{Proposition}
\newtheorem{lemma}[theorem]{Lemma}
\newtheorem{cor}[theorem]{Corollary}
\newtheorem{conjecture}[theorem]{Conjecture}
\newtheorem*{claim*}{Claim}
\theoremstyle{definition}
\newtheorem{definition}[theorem]{Definition}
\newtheorem{example}[theorem]{Example}
\newtheorem{remark}[theorem]{Remark}

\author{Kazuo Habiro}
\address{Research Institute for Mathematical Sciences, Kyoto University, Kyoto 606-8502, Japan}
\curraddr{Department of Mathematics, Kyoto University, Kyoto 606-8502, Japan}
\email{habiro@math.kyoto-u.ac.jp}

\author{Gw\'ena\"el Massuyeau}
\address{Institut de Math\'ematiques de Bourgogne, UMR 5584, CNRS, 
Universit\'e Bourgogne Franche-Comt\'e, 21000 Dijon, France
\newline  \& \ {Institut de Recherche Math\'ematique Avanc\'ee, UMR 7501, CNRS, Universit\'e de Strasbourg,   67084 Strasbourg, France
 }}
\email{{gwenael.massuyeau@u-bourgogne.fr}}

\date{September 26, 2020 (First version: January 28, 2017)}

\newcommand\nc{\newcommand}
\newcommand\rnc{\renewcommand}

\nc{\bgm}{\color[rgb]{.7,0,.7}}
\nc{\bkh}{\color[rgb]{0,0,.7}}
\nc{\ekh}{\normalcolor{}}
\nc{\egm}{\normalcolor{}}
\nc\NOTEgm[1]{\marginpar{\bgm \tiny #1 \egm}}
\nc\NOTEkh[1]{\marginpar{\bkh \tiny #1 \ekh}}
\nc\NOTE{\NOTEkh}
\nc{\CUTgm}[1]{\marginpar{\bgm \textbf{\footnotesize CUT}: \Tiny #1  \egm}}
\nc{\CUTkh}[1]{\marginpar{\bkh \textbf{\footnotesize CUT}: \Tiny #1  \ekh}}
\nc\CUT{\CUTkh}
\nc\KH[1]{\bkh #1\ekh}
\nc\GM[1]{\bgm #1\egm}
\nc\blue[1]{{\color{blue}#1}}
\nc\red{\color[rgb]{1,0,0}}
\nc\fign[1]{{\color{red}{\mathrm{[[#1]]}}}}
\nc\figbn[2]{\color{red}{\mathrm{[[#1]]}}}
\nc\redd[1]{{\red #1}}
\nc\TEMP[1]{{\red #1}}

\nc\id{\operatorname{id}}
\nc\ad{{\operatorname{ad}}}
\nc\op{{\operatorname{op}}}
\nc\End{\operatorname{End}}
\nc\Span{\operatorname{Span}}
\nc\ev{{\operatorname{ev}}}
\nc\coev{{\operatorname{coev}}}
\nc\Mag{\operatorname{Mag}}
\nc\Mon{\operatorname{Mon}}
\nc\Gr{\operatorname{Gr}}
\nc\Hom{\operatorname{Hom}}
\nc\gr{\operatorname{gr}}
\nc\conv{\operatorname{conv}}
\nc\Ob{\operatorname{Ob}}
\nc\dbl{\mathsf{d}}
\nc\sfH{\mathsf{H}}
\nc\sfI{\mathsf{I}}
\nc\grp{{\mathrm{grp}}}
\nc\sfP{\mathsf{P}}
\nc\Aut{\operatorname{Aut}}
\nc\Homeo{\operatorname{Homeo}}

\nc\K {{\mathbb K}}
\nc\Q {{\mathbb Q}}
\nc\Z {{\mathbb Z}}
\nc\R {{\mathbb R}}
\nc\N {{\mathbb N}}
\nc\NZ{\N}
\nc\bfH{\mathbf{H}}
\nc\bfC{\mathbf{C}}
\nc\HH{\mathbf{H}}
\nc\Hrc{\hat{\bfH}}
\nc\Set{\mathbf{Set}}
\nc\Mod{\mathbf{Mod}}
\nc\HMod{\Mod_H}
\nc\bD{\boldsymbol{\Delta }}
\nc\AB{\mathbf{A}}
\nc\hAqp{\wh{\AB}_q^\varphi}
\nc\hA{\widehat{\AB}}
\nc\hAg{\hA^{\mathrm{grp}}}
\nc\hAq{\hA_q}
\nc\Aq{\AB_q}
\nc\bfr{\mathbf{r}}
\nc\W{\mathbf{W}}
\nc\bfA{{\mathbf{A}}}
\nc\w{\mathbf{w}}
\nc\bfP{\mathbf{P}}
\nc\CC{\mathbf{CC}}
\nc\bfF{\mathbf{F}}
\nc\cc{\mathbf{c}}
\nc\VectK{\mathbf{Vect}_\K}

\nc\EH {\mathcal{H}}
\nc\B {\mathcal{B}}
\nc\KB{\K\B}
\nc\KBq{\K\Bq}
\nc\Bq {\mathcal{B}_q}
\nc\hKBq{\widehat{\K\Bq}}
\nc\A {\mathcal{A}}
\nc\sA{{}^{s}\!\!\A}
\nc\T {{\mathcal{T}}}
\nc\Tq{\T_q}
\nc\F {{\mathcal{F}}}
\nc\AT{\A}
\nc\ATS{\A^{S}}
\nc\ATT{\A^{T}}
\nc\calC {{\mathcal C}}
\nc\calK {{\mathcal{K}}}
\nc\calF {{\mathcal{F}}}
\nc\calI {{\mathcal{I}}}
\nc\calJ {{\mathcal{J}}}
\nc\calV {{\mathcal{V}}}
\nc\calD {{\mathcal{D}}}
\nc\C{\calC}
\nc\Cob{\mathcal{C}ob}
\nc\LCob{\mathcal{LC}ob}
\nc\sLCob{{}^s\!\mathcal{LC}ob}
\nc\Vect{\mathbf{Vect}}
\nc\tsA{{}^{ts}\!\!\A }
\nc\tiT{\ti\T }
\nc\tiTq{\ti\T _q}
\nc\modA {\mathcal{A}}
\nc\modB {\mathcal{B}}
\nc\modD {\mathcal{D}}
\nc\D{\calD}
\nc\hD{\widehat{\D}}
\nc\modF {\mathcal{F}}
\nc\modT {\mathcal{T}}
\nc\J{\calJ}
\nc\I{\calI}
\nc\hI{\widehat{\I}}
\nc\wAq{\widehat{{\Aq}}}
\nc\wAqD{\wAq^{\calD}}
\nc\Jeq{\underset{\calJ}{\equiv}}
\nc\BT{\mathcal{BT}}
\nc\Mg{\mathcal{M}_{\g}}
\nc\V{\calV}

\nc\fS{\mathfrak{S}}
\nc\g{{\mathfrak g}}
\nc\Wg{W_\g }
\nc\Ug{U(\g)}
\nc\Ugh{\Ug[[h]]}
\nc\UgMod{\Mod_{\Ug}}
\nc\UgModh{\UgMod[[h]]}
\nc\UghMod{\Mod_{\Ugh}}

\nc{\centre}[1]{\begin{array}{c} #1 \end{array}}
\nc{\figtotext}[3]{\begin{array}{c}\includegraphics[width=#1pt,height=#2pt]{#3}\end{array}}
\nc\fig[1]{\raisebox{-3.0ex}{\includegraphics[height=9.5ex]{#1}}}
\nc\figb[2]{\raisebox{-3.6ex}{\includegraphics[height=#2ex]{#1}}}
\nc\nstrands{\downarrow^n}
\nc\capl{\!\figtotext{8}{8}{capleft}\!\!}
\nc\capx[1]{\capl_1 \cdots\!\capl_{#1}}
\nc\caplX{\!\!\figtotext{10}{9}{capleft}\!\!}
\nc\capn{\capx{n}}
\nc\CAP[1]{C_{#1}(\caplX)}
\nc\figz[1]{\raisebox{-2.7ex}{\includegraphics[height=8.8ex]{#1}}}
\nc\vhi[1]{\vcenter{\hbox{\includegraphics[scale=0.2]{#1}}}}

\nc\ott{\ot\cdots\ot}
\nc\lala{\la\negthinspace\la}
\nc\rara{\ra\negthinspace\ra}
\nc\ho{{\hat\otimes }}
\nc\half{\frac12}
\nc\xto[1]{\overset{#1}{\longrightarrow}}
\nc\yto[1]{\underset{#1}{\longrightarrow}}
\nc\np{\newpage}
\nc\ul{\underline}
\nc\simeqto{\overset{\simeq}{\rightarrow }}
\nc\ct{\overset{\cong}{\to}}
\nc\ti{\tilde}
\nc\bu{\bullet}
\nc\ra{\rangle}
\nc\la{\langle}
\nc\ot{\otimes}
\nc{\by}[1]{\stackrel{\eqref{#1}}{=}}
\nc\congto{\overset{\cong}{\longrightarrow}}
\nc\plim{\varprojlim}
\nc\tw{\tilde{\w}}
\nc\II{I}
\nc\tW{\tilde{W}}
\nc\KXY{\K\lala X,Y\rara}
\nc\ttau{\ti\tau}
\nc\ep{\epsilon}
\nc\al{\alpha}
\nc\be{\beta}
\nc\ga{\gamma}
\nc\et{\eta}
\nc\De{\Delta}
\nc\trr{\triangleright}
\nc\wt{\widetilde}
\nc\tZ{\wt{Z}}
\nc\wh{\widehat}
\nc\wtr{\widetilde{r}}
\nc\vp{\varphi}
\nc\lto{\longrightarrow}
\nc\tiZ{\widetilde{Z}}
\def\varpi{f}
\nc\tbe{\tilde{\beta}}
\nc\Xx[1]{{X_{#1}}}
\nc\Xn{\Xx{n}}
\nc\surface{F}
\nc\surf[1]{\Sigma_{#1,1}}
\newcommand\leftmapsto{\mathrel{\reflectbox{\ensuremath{\mapsto}}}}
\newcommand\rightmapsto{\mathrel{\reflectbox{\reflectbox{\ensuremath{\mapsto}}}}}

\newcommand{\free}[1]{F_{#1}}
\nc\vn{\varnothing}
\nc\bV{\bar{V}}
\nc\Zq{Z_q}
\nc\Zqphi{\Zq^{\varphi}}
\nc\hW{\widehat{W}}
\nc\GG{G}
\nc\bZ{\bar{Z}}
\nc\SKIP[1]{#1}
\nc\para[2]{\noindent [[#1]]\\#2}
\rnc\para[2]{\marginpar{\Small* #1}#2}
\nc\action{\rho}
\nc\bLCob{{}^b\!\LCob}
\nc\bA{{}^b\!\AB}

\title[The Kontsevich integral for bottom tangles]
{The Kontsevich integral\\for bottom tangles in handlebodies}

\begin{document}

\begin{abstract}
Using an extension of the Kontsevich integral to tangles in handlebodies similar to a construction given by Andersen, Mattes and Reshetikhin, we construct a functor $Z:\B\to \hA$, where $\B$ is the category of bottom tangles in handlebodies and $\hA$ is the degree-completion of the category $\AB$ of Jacobi diagrams in handlebodies.
{As a symmetric monoidal linear category, $\AB$ is the {linear} PROP governing ``Casimir Hopf algebras'', 
which are cocommutative Hopf algebras equipped with a primitive invariant symmetric $2$-tensor.}
The functor $Z$ induces a canonical isomorphism {$\gr\B\cong\AB$, where $\gr\B$ is
the associated graded of the Vassiliev--Goussarov filtration on $\B$.}
{To each Drinfeld associator $\varphi$ we associate a ribbon quasi-Hopf algebra $H_\varphi$ in $\hA$,
and we prove that the braided Hopf algebra {resulting from $H_\varphi$ by  ``transmutation''} 
is precisely the image by $Z$ of a canonical Hopf algebra in the braided category~$\B$.} 
Finally, we explain how~$Z$ refines the LMO functor, which is a TQFT-like functor extending the Le--Murakami--Ohtsuki invariant.
\end{abstract}

\maketitle
{ \setcounter{tocdepth}{1} \tableofcontents}

%
%
\section{Introduction}   \label{sec:intro}
\subsection{Background}\label{sec:background}

The {\emph{Kontsevich integral}} is a powerful knot
invariant, taking values {in the space of \emph{chord diagrams} or \emph{Jacobi diagrams}},
which are {unitrivalent} graphs encoding Lie-algebraic structures
\cite{Kontsevich,BN2}.  {It is universal among rational-valued
Vassiliev--Goussarov finite type invariants
\cite{Vassiliev,Gusarov}, and dominates various quantum link
  invariants such as the colored Jones polynomials}.
Le and Murakami~\cite{LM_combinatorial} and Bar-Natan~\cite{BN1} extended the Kontsevich integral to a functor
\begin{gather*}
  Z^\T:\T_q\longrightarrow\A
\end{gather*}
from the category $\T_q$ of framed oriented \emph{$q$-tangles} to the
category $\A$ of Jacobi diagrams.   {The} Kontsevich integral was
generalized to links and tangles in {thickened surfaces}
by Andersen, Mattes and Reshetikhin \cite{AMR} and by Lieberum~\cite{Lieberum}.

Le, Murakami and Ohtsuki \cite{LMO} {constructed a closed   $3$-manifold invariant by using the Kontsevich integral.}
After attempts of extending the {Le--Murakami--Ohtsuki (LMO)}
invariant to TQFTs {by Murakami and Ohtsuki \cite{MO} and by    Cheptea and Le \cite{CL},
Cheptea and the authors \cite{CHM} constructed a functor
\begin{gather*}
  \tiZ:\LCob_q\longrightarrow \tsA,
\end{gather*}
called the \emph{LMO functor}.
Here  $\tsA$ is the category of  \emph{top-substantial Jacobi diagrams}
and $\LCob_q$ is the {``non-strictification''} of the {braided}
strict monoidal category $\LCob$ of \emph{Lagrangian cobordisms}.
{The category $\LCob$} is a 
subcategory of the category $\Cob$ of cobordisms between
once-punctured surfaces, studied by Crane and Yetter~\cite{CY} and   Kerler~\cite{Kerler}.
The LMO functor gives representations of the monoids of
  homology cylinders and, in particular, the Torelli groups,  which were studied in~\cite{HM1,Massuyeau}. 
{(Other representations of the monoids of homology cylinders have also been derived from the LMO invariant
by Andersen, Bene, Meilhan and Penner \cite{ABMP}.)}

\subsection{The category $\B$ of bottom tangles in handlebodies}\label{sec:category-b-bottom-1}

We consider here the \emph{category $\B$ of bottom tangles in handlebodies}  \cite{Habiro}, 
{which we may regard} as a braided monoidal subcategory of $\LCob$ \cite{CHM}.
The objects of $\B$ are non-negative integers.  
For~$m\ge0$, let $V_m\subset \R^3$ denote the cube with $m$ handles:
\begin{gather*}
V_m:=\centre{
{\labellist \scriptsize \hair 2pt
\pinlabel{$1$} at 206 285
\pinlabel{$m$} at 403 284
\pinlabel{$\cdots$} at 305 380
 \pinlabel{$S$}  at 64 23
 \pinlabel{$\ell$} [l] at 536 50
\endlabellist
\includegraphics[scale=0.25]{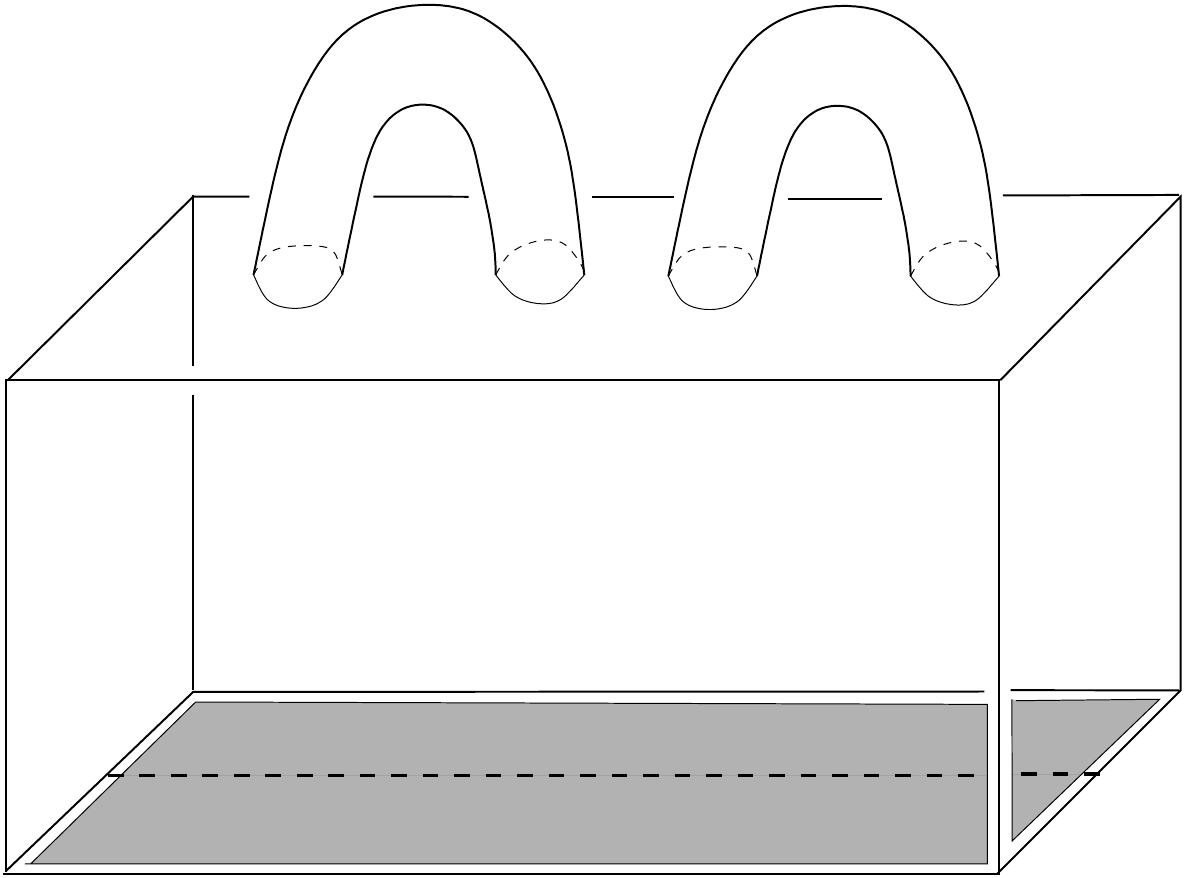}}}
\end{gather*}
The morphisms from $m$ to $n$
in $\B$ are the isotopy classes of $n$-component \emph{bottom tangles}
in~$V_m$, which are framed tangles, each consisting of $n$ arc
components whose endpoints are placed on a ``bottom line'' $\ell \subset \partial V_m$,
 in such a way that the two endpoints of each
component are adjacent on $\ell$. Here are
an example of a bottom tangle {and}  its projection diagram,
for $m=2$ and $n=3$:
\begin{equation}  \label{btt}
\centre{\labellist
\small\hair 2pt
\pinlabel{$\leadsto$}  at 115 27
\endlabellist
\centering
\includegraphics[scale=1.1]{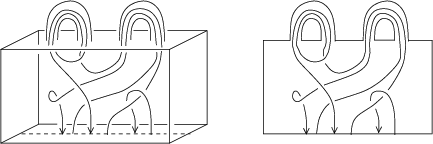}}
\end{equation}
 As another example, observe that $\B(0,1)$ is essentially the set of knots in $\R^3$. 
We associate to each $T:m\to n$ in $\B$ an 
embedding $i_T:V_n\hookrightarrow V_m$ {which fixes} {the} ``bottom square'' $S$
{and identifies $V_n$ with} a regular neighborhood in $V_m$ of the union  of $S$ with the $n$ components of $T$.
Then the composition  of $m\xto{T}n\xto{T'}p$ in $\B$ is
represented by the image $i_T(T')\subset V_m$. (See Section \ref{sec:categories} for further details.)

We can also define composition
in $\B$ using ``cube presentations'' of bottom tangles.  Each morphism
$T$ in $\B$ is represented by a bottom tangle which can be decomposed
into a tangle $U$, called a \emph{cube presentation} of $T$, and
parallel families of cores of the $1$-handles of the handlebody. For instance, 
the bottom tangle \eqref{btt}  has the following cube presentation:
$$
\centre{
\labellist
\small\hair 2pt
 \pinlabel{$\leadsto$}  at 110 22
 \pinlabel{$T=$} [r] at -2 22
 \pinlabel{$=U$} [l] at 219 22
\endlabellist
\centering
\includegraphics[scale=1.0]{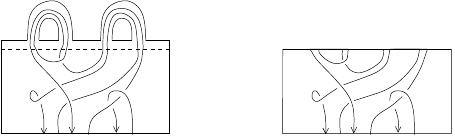}
}
$$
We can define the composition $T'\circ T$ of
two morphisms $T$ and $T'$ in $\B$ as the bottom tangle obtained by
putting a suitable cabling of $T$ on the top of a cube presentation of $T'$. For example,
$$
\centre{\labellist
\small \hair 2pt
 \pinlabel{$\circ$} at 55 19
 \pinlabel{$=$}  at 164 19
\endlabellist
\centering
\includegraphics[scale=1.0]{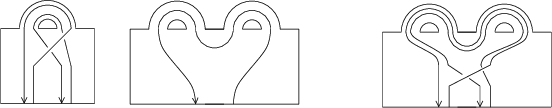}}
$$

We may identify $\B$ with the opposite $\mathcal{H}^\op$ of the
category $\mathcal{H}$ of isotopy classes of embeddings of
handlebodies rel $S$, via the above correspondence $T \mapsto i_T$.
The category~$\B$ is also isomorphic to the \emph{category $\sLCob$ of
special Lagrangian cobordisms} introduced in \cite{CHM}: each bottom
tangle $T:m\to n$ in $\B$ corresponds to the cobordism obtained as
the exterior of the embedding $i_T:V_n\hookrightarrow V_m$.
The category $\sLCob$, and hence $\B$, inherit
from $\LCob \subset \Cob$ a braided strict monoidal structure.
In $\B$, tensor product on objects is addition and tensor product on morphisms is juxtaposition;
the braiding $\psi=\psi_{1,1}:2\to2$, which determines all braidings in $\B$, is
$$
\psi  = \figtotext{45}{45}{psi_B} 
$$

The first author \cite{Habiro} (see also forthcoming \cite{Habiro2})
{introduced} the category $\B$ in order to study universal quantum
invariants of links and tangles
\cite{Hennings,Lawrence,Reshetikhin,Ohtsuki1,Kauffman} unifying
the Reshetikhin--Turaev quantum invariants associated with each ribbon Hopf algebra \cite{Drinfeld3,RT}.
Indeed, for each ribbon Hopf algebra $H$, there is a braided monoidal functor
\begin{gather}
  \label{e40}
  J^H: \B \lto \Mod_H
\end{gather}
extending the universal quantum link invariant to bottom tangles in handlebodies, 
where $\Mod_H$ denotes the category of left $H$-modules.

The  category $\B$ admits a Hopf algebra {object} $H^{\B}$,
whose counterpart in $\Cob$ was introduced by Crane and Yetter
\cite{CY} and Kerler \cite{Kerler}.  
{This Hopf algebra structure in $\B$ and $\Cob$ may
  be identified with the Hopf-algebraic structure for claspers
  observed in \cite{Habiro_claspers} (see \cite{Habiro,Habiro2}).}
The {braided monoidal} category
$\B$ is generated by the Hopf algebra $H^\B$ together with a few other
morphisms (see Section \ref{sec:generators_B}).  \emph{Transmutation}
introduced by Majid \cite{Majid:algebras-and-Hopf-algebras,Majid} is a
process of transforming each quasi-triangular Hopf algebra $H$ into a
braided Hopf algebra $\ul{H}$ in $\Mod_H$.  The functor $J^H$ maps the
Hopf algebra $H^\B$ in $\B$ to the transmutation $\ul{H}$ of $H$.

In {the present} {paper}, using the Kontsevich integral $Z^\T$, we construct
and study a functor 
$$\Zq^\varphi:\Bq \lto \hA_q^\varphi,$$
which is a refinement of the LMO functor $\tiZ$ on the category $\B \cong \sLCob \subset\LCob$,
{and which may be considered as a ``Kontsevich
integral version'' of the functor $J^H$ in \eqref{e40}.}
{The target category $\hA_q^\varphi$} of $\Zq^\varphi$ is constructed from the
\emph{category $\AB$ of Jacobi diagrams in handlebodies}, 
described below.    (See Section~\ref{sec:Jacobi} for further details.)

\subsection{The category $\AB$ of Jacobi diagrams in handlebodies}\label{sec:category-ab-jacobi-1}

We work over a fixed field $\K$ of characteristic $0$.
For $m\geq 0$,  let $\bV_m$ denote  the \emph{square with $m$ handles},
which is constructed by attaching $m$ $1$-handles on the top of a square
and can be regarded  as the image of the handlebody $V_m$ under the
projection $\R^3\twoheadrightarrow\R^2$.
Let $X_n := \capn$ be the $1$-manifold consisting of $n$ arc components.

The objects in $\AB$ are non-negative integers.
The morphisms from $m$ to $n$ in $\AB$ are linear
combinations of $(m,n)$-Jacobi diagrams, which are Jacobi diagrams on   $X_n$ mapped {into} $\bV_m$.
Specifically, an \emph{$(m,n)$-Jacobi diagram} $D$ consists of
\begin{itemize}
\item a unitrivalent graph $D$ such that each
trivalent vertex is oriented, and such that the set of univalent
vertices is embedded into the interior of $X_n$,
\item a map $X_n \cup D\to\bV_m$ that maps $\partial X_n$ into the
  ``bottom edge'' of $\bV_m$ in a way similar to how the endpoints of
  a bottom tangle are mapped into the bottom line of a handlebody.
\end{itemize}
Here is an example of a {$(2,3)$-Jacobi diagram}:
\begin{gather}
  \label{e4}
  D=\centre{\includegraphics[width=20ex]{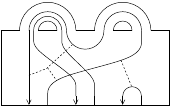}}: 2 \longrightarrow 3.
\end{gather}
As usual, the Jacobi diagrams obey the STU relations
$$
\labellist \small \hair 2pt
 \pinlabel{$=$}  at 127 39
 \pinlabel{$-$}  at 266 37
\endlabellist
\centering
\includegraphics[scale=0.35]{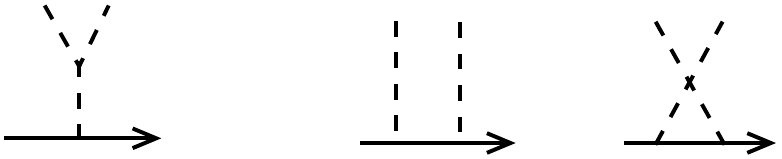}.
$$
Moreover, we identify Jacobi
diagrams that are homotopic in $\bV_m$ relative to the endpoints of $X_n$.
Since $V_m$ deformation retracts to $\bV_m$,
we could equivalently give the same definitions with $\bV_m$ replaced by $V_m$.
Thus the diagrams of the above kind are also referred to as \emph{Jacobi diagrams in handlebodies}.

A \emph{square presentation} of {an $(m,n)$-Jacobi diagram $D$}
is a  usual Jacobi diagram~$U$ ({i.e.,} a morphism in the target
category $\A$ of the Kontsevich integral $Z^\T$) which
yields $D$ by attaching parallel copies of cores of the $1$-handles in
$\bV_m$.  For example, here is a square presentation of $D$ in \eqref{e4}:
\begin{gather}
  \label{e4U}
  U=\centre{\includegraphics[width=20ex]{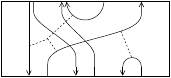}}.
\end{gather}
Although \emph{not} every {$(m,n)$-Jacobi diagram} admits a square presentation,
the STU relation implies that every morphism $m\to n$ in $\AB$ is a linear combination of such
diagrams admitting square presentations.

Composition in $\AB$ is defined by using  square presentations,
similarly to how composition in $\B$ is defined by using cube presentations.  
For $l\xto{D'}m\xto{D}n$ in $\AB$ and  a square presentation $U$ of $D$,
the composition $D\circ D':l\to n$ is the stacking of a suitable cabling
$C_U(D')$ on the top of $U$.  Here the \emph{cabling} $C_U(D')$ is obtained
from $D'$ by replacing each component of $X_m$ with its parallel
copies so that the target of $C_U(D')$ matches the source of $U$;
we also replace each univalent vertex attached to a component of $X_m$ 
with the sum of all ways  of attaching it (with signs) to the parallel copies of this component.
For example, if $D:2\to 3$ and $U$ are as in \eqref{e4}
and~\eqref{e4U}, respectively,  and~if 
\begin{gather*}
   D'=\centre{\includegraphics[height=13ex]{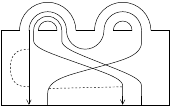}}:2 \longrightarrow 2,
\end{gather*}
then we have\\[-0.7cm]
\begin{gather*}
   D\circ D'=\centre{\includegraphics[height=22ex]{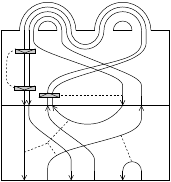}} :2 \longrightarrow 3.
\end{gather*}
(Here we  use  ``boxes'' to denote the above-mentioned {operation on univalent vertices};
this notation is explained in  Example \ref{ex:box}.)

The category $\AB$ has a structure of a linear symmetric strict
monoidal category.  Tensor product on objects is addition, and tensor
product on morphisms is juxtaposition. The symmetry in $\AB$ is
  determined by
$$
P=P_{1,1}= \figtotext{40}{40}{P_AB} :2\longrightarrow 2.
$$
{Moreover,} the morphism spaces $\AB(m,n)$ are graded with the usual degree of
Jacobi diagrams (i.e., half the number of vertices), and their degree
completions $\hA(m,n)$ form a linear category $\hA$, called the \emph{degree-completion} of $\AB$.

We remark that Jacobi diagrams in surfaces, such as
squares with handles, were considered earlier in the
above-mentioned works \cite{AMR,Lieberum}.
In Section~\ref{sec:Jacobi}, we will define $\AB(m,n)$ in a rather different way as a
space of \emph{colored} Jacobi diagrams.  The latter are essentially the same
as {$(m,n)$-Jacobi diagrams, i.e., Jacobi diagrams on $X_n$ mapped {into}
$\bV_m\simeq V_m$}, but the maps in $\bV_m \simeq V_m$ are
specified by decorating the components of $X_n$   and the dashed part of the diagram
with some  \emph{beads}.  These beads are labeled by  elements~of
$$
\pi_1(\bV_m) \cong \pi_1(V_m)  =F(x_1,\dots,x_m) =:F_m,
$$
the free group {on the elements} 
$x_1,\dots,x_m$ corresponding to the $1$-handles {of $V_m$}.
Colored Jacobi diagrams  appeared in \cite{GL,GR} for instance.

\subsection{Construction of a functor $Z^\B$}\label{sec:constr-funct-zb}

The \emph{non-strictification} $\C_q$ of a strict mo\-noidal category
$\C$ ({whose object monoid is free}) is the non-strict monoidal
category obtained from $\C$ by forcing the tensor product
to be not strictly associative but associative up to canonical
isomorphisms; see Section \ref{sec:tangles} for the definition.
For example, the category $\T_q$ of $q$-tangles, which is the source category of the Kontsevich integral $Z^\T$,
is the non-strictification of the
strict monoidal category $\T$ of tangles.  The object set
$\Ob(\T)$ of $\T$ is the free monoid $\Mon(\pm)$ on two letters $+,-$
corresponding to downward and upward strings; correspondingly, the set
$\Ob(\T_q)$ is the free unital magma $\Mag(\pm)$ on $+,-$, consisting of
fully-parenthesized words in $+,-$
such as $+$, $-$, $(-+)$, $(-(++))$, including the empty word $\vn$.

Non-strictification is applied to strict monoidal categories such as
$\B$ and $\hA$ to produce non-strict monoidal categories $\Bq$ and $\hAq$.
{Since $\Ob(\B)= \Ob(\hA) = \{0,1,\dots\}$ can be identified with $\Mon(\bu)$, the free monoid on one
letter $\bu$, we may set $\Ob(\Bq)= \Ob(\hA_q) = \Mag(\bu)$, the free unital magma on~$\bu$.
The latter consists of parenthesized words in $\bu$ such as
$\vn,\bu,(\bu\bu),((\bu\bu)\bu)$.  The length of $w\in\Mag(\bu)$ is
denoted by $|w|$.
The morphisms {in} $\Bq$ are called \emph{bottom $q$-tangles in handlebodies}.}

Recall that a Drinfeld associator $\varphi = \varphi(X,Y)$ is a group-like
element of $\KXY$ satisfying the so-called \emph{pentagon} and \emph{hexagon} equations \cite{Drinfeld2};
see Section \ref{sec:drinfeld-associators}.  
{Here} is the main construction of {the present} paper.

\begin{theorem}[see Theorem \ref{r16}]
  \label{r31}
  For each Drinfeld associator $\varphi$, there is a braided monoidal
  functor
  \begin{gather}
  \label{e57}
  \Zq^\varphi:\Bq \longrightarrow \hA_q^\varphi
  \end{gather}
  from $\Bq$, the non-strictification of the category $\B$, to
  $\hAqp$, a ``deformation'' of the non-strictification of $\hA$
  which is determined by $\varphi$.
\end{theorem}

To prove Theorem \ref{r31},
we will construct a tensor-preserving functor
\begin{gather}
  \label{e58}
  Z^\B: \Bq\longrightarrow \hA.
\end{gather}
If we ignore the monoidal structures, the categories $\hAqp$ and $\hA$
are equivalent in a natural way and, under this equivalence,
the functors $\Zq^\varphi$ and $Z^\B$ are essentially the same for each $\varphi$.
We construct the functor
$Z^\B$ by using
the usual Kontsevich integral
$Z^\T:\T_q\to\A$ as follows.
Here  $Z^\T$ is defined from the Drinfeld associator~$\varphi$, using the normalization 
\begin{gather*}
\begin{array}{rcll}
{Z^\T}\Big(\!\!\! \begin{array}{c}   \figtotext{18}{9}{capleft} \\ {}^{(+-)}  \end{array}\!\!\!\Big)
& = &\!\!\!
\begin{array}{c}{
\labellist  \scriptsize \hair 2pt
\pinlabel{$1$} at 122 115
\endlabellist
\includegraphics[scale=0.15]{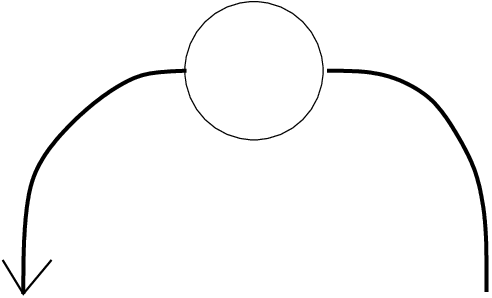} 
} \end{array}  &
:\ \vn\lto {(+-)}\quad \text{in $\AT$},\\
{Z^\T}\Big(\!\!\! \begin{array}{c} {}_{(+-)} \\ \figtotext{18}{9}{cupright} \end{array} \!\!\!\Big) & = &\!\!\!
\begin{array}{c}
\labellist \scriptsize  \hair 2pt
\pinlabel{$\nu$} at 115 35
\endlabellist
\includegraphics[scale=0.15]{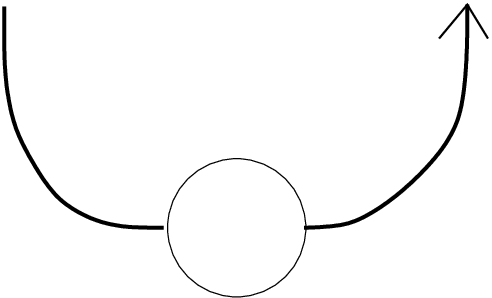} 
\end{array}
 &
:\ {(+-)} \lto\vn\quad \text{in $\AT$},
\end{array}
\end{gather*}
where $\nu$ is the usual normalization factor. 
(In the literature, one often uses the normalization {with}
both $1$ and $\nu$ in the above identities {being} 
replaced with  $\nu^{1/2}$, so that the invariant behaves well under $\pi$-rotation of
tangles.  In our case, like in~\cite{CHM},  it is more important  to have a simple value on $\capl$.)

Consider $T:v\to w$ in $\Bq$ with $|v|=m$ and $|w|=n$.
In order to define $Z^\B(T)$, we choose a projection diagram of  $T$
\begin{gather}
  \label{e100}
  \def\tmpA{$T_0$}
  \def\tmpB{$T_1$}
  \def\tmpC{$T_m$}
  T=\centre{\begin{picture}(0,0)%
\includegraphics{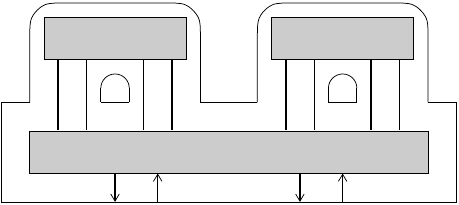}%
\end{picture}%
\setlength{\unitlength}{3158sp}%
\begingroup\makeatletter\ifx\SetFigFont\undefined%
\gdef\SetFigFont#1#2#3#4#5{%
  \reset@font\fontsize{#1}{#2pt}%
  \fontfamily{#3}\fontseries{#4}\fontshape{#5}%
  \selectfont}%
\fi\endgroup%
\begin{picture}(4576,2025)(889,-8334)
\put(2324,-7316){\makebox(0,0)[lb]{\smash{{\SetFigFont{7}{8.4}{\rmdefault}{\mddefault}{\updefault}{\color[rgb]{0,0,0}$\;\cdots$}%
}}}}
\put(4571,-7475){\makebox(0,0)[lb]{\smash{{\SetFigFont{7}{8.4}{\rmdefault}{\mddefault}{\updefault}{\color[rgb]{0,0,0}$\;u_m'$}%
}}}}
\put(3715,-7322){\makebox(0,0)[lb]{\smash{{\SetFigFont{7}{8.4}{\rmdefault}{\mddefault}{\updefault}{\color[rgb]{0,0,0}$\;\cdots$}%
}}}}
\put(3716,-7468){\makebox(0,0)[lb]{\smash{{\SetFigFont{7}{8.4}{\rmdefault}{\mddefault}{\updefault}{\color[rgb]{0,0,0}$\;u_m$}%
}}}}
\put(1470,-7316){\makebox(0,0)[lb]{\smash{{\SetFigFont{7}{8.4}{\rmdefault}{\mddefault}{\updefault}{\color[rgb]{0,0,0}$\;\cdots$}%
}}}}
\put(1470,-7459){\makebox(0,0)[lb]{\smash{{\SetFigFont{7}{8.4}{\rmdefault}{\mddefault}{\updefault}{\color[rgb]{0,0,0}$\;u_1$}%
}}}}
\put(2324,-7459){\makebox(0,0)[lb]{\smash{{\SetFigFont{7}{8.4}{\rmdefault}{\mddefault}{\updefault}{\color[rgb]{0,0,0}$\;u_1'$}%
}}}}
\put(3078,-6748){\makebox(0,0)[lb]{\smash{{\SetFigFont{7}{8.4}{\rmdefault}{\mddefault}{\updefault}{\color[rgb]{0,0,0}$\cdots$}%
}}}}
\put(3097,-8235){\makebox(0,0)[lb]{\smash{{\SetFigFont{7}{8.4}{\rmdefault}{\mddefault}{\updefault}{\color[rgb]{0,0,0}$\cdots$}%
}}}}
\put(4572,-7328){\makebox(0,0)[lb]{\smash{{\SetFigFont{7}{8.4}{\rmdefault}{\mddefault}{\updefault}{\color[rgb]{0,0,0}$\;\cdots$}%
}}}}
\put(3035,-7885){\makebox(0,0)[lb]{\smash{{\SetFigFont{7}{8.4}{\rmdefault}{\mddefault}{\updefault}{\color[rgb]{0,0,0}\tmpA}%
}}}}
\put(1871,-6721){\makebox(0,0)[lb]{\smash{{\SetFigFont{7}{8.4}{\rmdefault}{\mddefault}{\updefault}{\color[rgb]{0,0,0}\tmpB}%
}}}}
\put(4144,-6725){\makebox(0,0)[lb]{\smash{{\SetFigFont{7}{8.4}{\rmdefault}{\mddefault}{\updefault}{\color[rgb]{0,0,0}\tmpC}%
}}}}
\end{picture}%
},
\end{gather}
composed of $q$-tangles
\begin{gather*}
  T_0:\ti v\lto w(+-),\quad T_i:\vn\lto u_iu_i'\quad (i=1,\dots,m) ,
\end{gather*}
where
\begin{itemize}
\item $u_1,u_1',\dots,u_m,u_m'\in\Mag(\pm)$,
\item $\ti{v} := v(u_1u_1',\dots,u_mu_m')$ is obtained from the
non-associative word $v$ in $\bu$ by substituting
$u_1u_1',\dots,u_mu_m'$ into the $m$ $\bu$'s, and
$w(+-):=w(+-,\dots,+-)$ is defined similarly.
\end{itemize}
Then we define $Z^\B(T):m\to n$ in $\hA$ by
\begin{gather}
  \label{e25}
  \def\tmpA{${\!\!\!Z^\T}(T_0)$}
  \def\tmpB{${\!\!\!Z^\T}(T_1)$}
  \def\tmpC{${\!\!\!Z^\T}(T_m)$}
  Z^\B(T)= \centre{}.
\end{gather}

We remark that the above definition of $Z^\B(T)$, {simply} as an invariant of
tangles in handlebodies, is  similar to the definition
of the Kontsevich integral of links in {thickened surfaces}
given by Andersen, Mattes and Reshetikhin~\cite{AMR}; see also Lieberum~\cite{Lieberum}.

\begin{theorem}[see Theorem \ref{th:extended_Kontsevich}]
  \label{r46}
  {There is a functor
  $Z^\B: \Bq \to \hA$ {such that}
\begin{itemize}
\item on objects $w\in\Mag(\bu)$, we have ${Z^\B}(w)=|w|$,
\item on morphisms $T:v\to w$ in $\Bq$ decomposed as \eqref{e100}, we have \eqref{e25}.
\end{itemize}
Furthermore, the functor $Z^\B$ is tensor-preserving, i.e.,
${Z^\B}(T \otimes T') = {{Z^\B}(T) \otimes {Z^\B}(T')}$
  for  morphisms $T$ and $T'$ in $\Bq$.}
\end{theorem}

The functor $Z^\B$ is \emph{not} monoidal since
{it does not preserve the associativity isomorphisms.}
The braided monoidal category $\hA_q^\varphi$ mentioned in Theorem
\ref{r31} is constructed from the non-strictification $\hA_q$ of $\hA$
by redefining the associativity isomorphisms and braidings to be the
images by $Z^\B$ of those of $\Bq$.  Then Theorem~\ref{r31} follows
from Theorem \ref{r46}.

\subsection{Basic properties of $Z^\B$}		\label{sec:basic-properties-zb-1}

{Here are some basic properties of $Z^\B$.}

The functor $Z^\B$ extends the usual Kontsevich integral for
bottom $q$-tangles in a cube, i.e., for each $T:\vn\to w$ in $\Bq$, regarded also as $T:\vn\to w(+-)$ in $\T_q$,
we have $Z^\B(T)=Z^\T(T)$. 	

{We can enrich $\AB$ {and $\hA$} over cocommutative coalgebras}, {i.e.,
the morphism spaces in $\AB$ and $\hA$ {have} cocommutative coalgebra structures,
and the compositions and tensor products on them are coalgebra maps}
(see Proposition \ref{r8}).  It follows that $Z^\B$ takes values
in the group-like part of {$\hA$} (see Proposition~\ref{prop:group-like}).

Let $\bfF$ denote the category of finitely generated free groups.
Consider the functor $$h:\B\cong\mathcal{H}^\op \lto \bfF^{\op}$$ that maps
each bottom tangle $T:m\to n$ to the homomorphism $(i_T)_*:F_n\to F_m$
between free groups.  This functor gives an $\bfF^\op$-grading of the
category~$\B$.  Similarly, we have an $\bfF^\op$-grading of the linear
category $\AB$ and its completion $\hA$, where the $\bfF^\op$-degree
of each $(m,n)$-Jacobi diagram $D$ is the homotopy class of the {underlying}
map $X_n\to \bV_m$. It follows that $Z^\B$ preserves
$\bfF^\op$-grading (see Proposition \ref{prop:homotopy_types}).

The degree $0$ part of $Z^\B(T)$, which {belongs to} $\AB_0\cong\K \bfF^{\op}$, 
is given by the homotopy class {$h(T)$} of  {the components of $T$ in the handlebody}. 
The degree $1$ part of $Z^\B(T)$, which we do not study
in the present paper, is given by equivariant linking numbers of {the} components
of $T$ in the handlebody.  We give the values of $Z^\B$ up to degree $2$ {on} the generators of the monoidal
category $\Bq$ (see Proposition~\ref{r29}).

\subsection{$Z^\B$ as {a} universal finite type invariant}  \label{sec:zb-as-universal} 

{The} main property of the invariant~$Z^\B$ is the universality
among Vassiliev--Goussarov {finite type} invariants, {formulated functorially}.
{Similarly to the case of usual tangles in a cube \cite{KT}},
we define the \emph{Vassiliev--Goussarov filtration}
$$
\KBq = \V^0 \supset \V^1 \supset \V^2 \supset \cdots
$$
on the linearization $\K \B_q$ of the non-strict monoidal category  $\B_q$,
and  we consider the same filtration  on the  linear strict monoidal category $\K\B$.
The braiding {in}~$\B$  induces a {symmetry in} 
the associated graded $\Gr \K\B$ of  $\K\B$.
{We give $\hA^\varphi_q$ the degree filtration.}

\begin{theorem}[see   Theorem \ref{r20} and Theorem \ref{th:universality}]
  \label{r45}
  The functor $\Zq^\varphi:\hKBq\to\hAqp$ is an isomorphism of filtered
  linear braided monoidal categories. Consequently, $Z^\B$ induces an
  isomorphism $\Gr\KB\cong \AB$ of graded linear symmetric monoidal categories.
\end{theorem}

It is hoped that Vassiliev--Goussarov invariants distinguish knots in
  $S^3$ \cite{Vassiliev}; since the usual Kontsevich integral  is universal among such
invariants, the hope is that  the functor $Z^\T$ is faithful.  More generally, we expect  the following.

\begin{conjecture}
  \label{r35}
  The functor $Z^\B$ (resp.~$\Zq^\vp$) is faithful.  In other
  words, $Z^\B$ (resp.~$\Zq^\vp$)  is a complete
  invariant of bottom tangles in handlebodies.
\end{conjecture}

\subsection{The functor $Z^\B$ as a refinement of the LMO functor}	\label{sec:functor-zb-as}

{The functor $Z^\B$ refines} the LMO functor $\tiZ$ in the following way.

\begin{theorem}[see
Theorem \ref{th:Kontsevich_to_LMO} and Remark~\ref{r38}]
  \label{r32}
  We have a commutative diagram of functors:
  \begin{gather}
    \label{e17}
    \centre{\xymatrix{
      \Bq\ar[r]^{Z^\B}\ar[d]_{E}&
      \hA\ar[d]^\kappa\\
      \LCob_q\ar[r]_{\tiZ}&
      \tsA.
  }}
  \end{gather}
\end{theorem}
Here the functor $E:\Bq\to\LCob_q$, with the image being $\sLCob_q$, is the faithful functor that maps
each bottom tangle in a handlebody to its exterior viewed as a Lagrangian cobordism.
The linear functor $\kappa:\hA\to\tsA$ is a
variant of the ``hair map'' defined in \cite{GK,GR},
and {we may also regard it}  as a diagrammatic enhancement of the ``Magnus expansion'':
$$
F_m \hookrightarrow
\K \langle\!\langle X_1,\dots, X_m \rangle\!\rangle, \quad
x_i \longmapsto \exp(X_i) = 1+ X_i + \frac{X_i^2}{2!} + \cdots.
$$

\begin{theorem}[see Theorem \ref{prop:non-injectivity}]
  \label{r34}
  The ``hair functor'' $\kappa:\hA\to \tsA$ is not faithful.  In
 fact, if $m,n\ge1$, then the map $\kappa:\hA(m,n)\to \tsA(m,n)$ is
 not injective.
\end{theorem}

Thus the functor $Z^\B$ {\emph{properly}} refines
the restriction of the LMO functor~$\tiZ$ to~$\sLCob$.  We prove the above
theorem by adapting 
Patureau-Mirand's proof \cite{Patureau-Mirand} of
the non-injectivity of the ``hair map'', {which itself uses Vogel's results \cite{Vogel}.}
The authors do not know
whether $Z^\B$ is strictly stronger than $\tiZ$
as an invariant of bottom tangles in handlebodies.  In fact, we
conjecture that the LMO functor $\tiZ:\LCob_q\to \tsA$ itself is faithful.

{Recall that} the construction of the LMO functor $\wt Z$ involves surgery
presentations of Lagrangian cobordisms. 
{Here surgery} translates into {the \emph{Aarhus integral},}
which Bar-Natan, Garoufalidis, Rozansky and Thurston \cite{BGRT} introduced
in their reconstruction of the LMO invariant.
The construction of the functor $Z^\B$ in the present paper is simpler than that of $\wt Z$ 
since it does not involve these surgery techniques.

\subsection{Presentation of $\AB$}\label{sec:presentation-ab-1}

The category $\bfF$ of finitely generated free groups is a symmetric
monoidal category, and it is well known that it is  freely generated
as such by a commutative Hopf algebra~\cite{Pirashvili}.
By generalizing {another combinatorial} proof of this {fact}
given in \cite{Habiro3}, we obtain the following presentation of $\AB$. 

\begin{theorem}[see Theorem \ref{AB-pres}]
  \label{r36}
  The linear symmetric strict monoidal category~$\AB$ is freely
  generated by a ``Casimir Hopf algebra''.
\end{theorem}

{In other words, $\AB$ is the linear PROP (see \cite{MacLane,Markl}) governing Casimir Hopf  algebras.}
Here a \emph{Casimir Hopf algebra} in a linear symmetric monoidal
category $\C=(\C,\ot,I)$ is a cocommutative Hopf algebra $H$ in $\C$
equipped with a \emph{Casimir $2$-tensor}, i.e., a morphism $c:I\to
H\ot H$ which is \emph{primitive}, \emph{symmetric} and \emph{ad-invariant}.
(See Definition \ref{r11}.)
The Casimir Hopf algebra~$(H,c)$ in $\AB$ alluded to in Theorem~\ref{r36} is defined in \eqref{e51}.

{To illustrate this {kind of structure}, 
consider a Lie algebra  $\g$  with an ad-invariant, symmetric element $t\in\g^{\ot2}$.}
Then the universal enveloping algebra $U(\g)$ together
with $t\in\g^{\ot2}\subset U(\g)^{\ot2}$ is a Casimir Hopf algebra in
the category $\VectK$ of $\K$-vector spaces.  Thus, by Theorem \ref{r36},
there is a unique linear symmetric monoidal functor
$$
W_{(\g,t)}:\AB \lto  \VectK
$$ 
which maps the Casimir Hopf algebra $(H,c)$ in $\AB$ to the Casimir Hopf algebra $(U(\g),t)$.
Following the usual terminology, we call $W_{(\g,t)}$ the \emph{weight system}
associated to the pair~$(\g,t)$.

\subsection{Ribbon quasi-Hopf algebras in $\hA$}\label{sec:ribbon-quasi-hopf-2}

Recall that a \emph{quasi-Hopf algebra} $H$ \cite{Drinfeld1} (see also
\cite{Kassel}) is a variant of a Hopf algebra, where coassociativity
does not hold {strictly}, but is controlled by {a} $3$-tensor
$\varphi\in H^{\ot3}$; see Section \ref{sec:ribbon-quasi-hopf-1} for
the definition.  
The notions of quasi-triangular {and ribbon} Hopf algebras,  used in the construction of quantum 
link invariants \cite{RT}, {admit}  quasi-Hopf versions, {using which one can}
construct link invariants {as well}~\cite{AC}.
{One can also consider quasi-Hopf algebras in symmetric monoidal categories.}

{As is well known, if} $t$ is an ad-invariant, symmetric $2$-tensor for a Lie algebra $\g$ as above,
then each Drinfeld associator $\vp\in\KXY$ induces a
ribbon quasi-Hopf algebra structure on $U(\g)[[h]]$. {Here is a universal version of this fact.}

\begin{theorem}[see Theorem \ref{prop:r3}]
  \label{r37}
{For each Drinfeld associator $\varphi$, {the Casimir Hopf algebra $(H,c)$ in $\AB$ induces} a canonical ribbon quasi-Hopf
  algebra $H_\varphi$ in $\hA$.}
\end{theorem}

Specifically, the weight system $W_{(\g,t)}$ associated to the above
pair $(\g,t)$ maps~$H_\varphi$ to the quasi-triangular quasi-Hopf
structure on $U(\g)[[h]]$ considered in \cite{Drinfeld2}.

Klim \cite{Klim} generalized Majid's transmutation to  quasi-Hopf algebras.
We can perform transmutation in {arbitrary} symmetric monoidal categories.  In
particular, by transmutation, {the quasi-triangular quasi-Hopf algebra}
$H_\varphi$ yields a Hopf algebra $\ul{H_\varphi}$ in the braided
monoidal category $\Mod_{H_\vp}$ of left $H_\vp$-modules in $\hA$.  
On the other hand, by Theorem \ref{r31}, the Hopf algebra $H^{\Bq}$ in $\Bq$
(corresponding to the Hopf algebra $H^\B$ in $\B$) is mapped by the
braided monoidal functor $\Zq^\varphi:\Bq \to \hA_q^\varphi$ into a Hopf algebra
$\Zq^\varphi(H^{\Bq})$ in $\hA_q^\varphi$.

\begin{theorem}[see  Theorem \ref{r19}]
  \label{r43}
  The Hopf algebra $\Zq^\varphi(H^{\Bq})$ in $\hAq^\vp$ and 
  {the transmutation} $\ul{H_\varphi}$ in $\Mod_{H_\vp}$ coincide, through
  a canonical embedding {$\hA_q^\varphi\to\Mod_{H_\vp}$.}
\end{theorem}

To prove {Theorem \ref{r43}}, we compute {the values of} $\Zq^\varphi$ on a generating system of~$\B_q$
{including} the {structure} morphisms of $H^{\Bq}$; see Proposition \ref{prop:Z_generators}.

\subsection{Organization of the paper}\label{sec:organization-paper}

{We organize the rest of the paper} as follows.
In Section~\ref{sec:categories}, we define the categories $\B$, $\mathcal{H}$ and $\sLCob$.
In Section~\ref{sec:Kontsevich}, we recall the definition of the usual Kontsevich integral ${Z:=Z^\T}$.
In Section~\ref{sec:Jacobi}, we define the category $\bfA$ of Jacobi diagrams in
handlebodies and we start studying its algebraic structure.
In Section~\ref{sec:presentation}, we go further in this study
by giving a presentation of $\bfA$ as a linear  symmetric monoidal category.
In Section~\ref{sec:quasi-Hopf}, we show that {each} Drinfeld associator $\varphi=\varphi(X,Y)$ {yields} a ribbon quasi-Hopf algebra $H_{\varphi}$ in the
degree-completion $\hA$ of $\bfA$
and, in Section~\ref{sec:weight_systems}, we consider the weight system functors on $\AB$
 associated to Lie algebras with symmetric ad-invariant $2$-tensors.
The construction of the  functor ${Z:=}Z^\B:\B_q\to\hA$
is done in Section~\ref{sec:extended_Kontsevich}, where we also give
{some of its basic} properties.
In Section~\ref{sec:construction_Zq}, we define the braided monoidal functor $\Zq^\varphi:{\B_q\to\hAq^\varphi}$:
thanks to this variant of $Z^\B$, we  interpret the values of~$Z^\B$ on a generating system of $\B_q$
as the result of applying Majid's transmutation to the ribbon quasi-Hopf algebra~$H_{\varphi}$.
In Section~\ref{sec:universality}, we show that $\Zq^\varphi$ induces an isomorphism
of braided monoidal categories between the completion of $\K\B_q$
with respect to the Vassiliev--Goussarov filtration and~$\hAq^\varphi$.
In Section~\ref{sec:LMO_functor}, we explain how the functor $Z^\B:\B_q\to\hA$ refines the LMO
functor $\widetilde{Z}:\LCob_q\to\tsA$.
Finally, in Section~\ref{sec:perspectives}, we explain some applications that we expect from our results.

\subsection{Conventions}\label{sec:conventions}

In what follows, we fix a field $\K$ of characteristic $0$.
By a ``vector space'' (resp.\ a ``linear map''), we always mean  a
``$\K$-vector space'' (resp.\ a ``$\K$-linear map'').

Let $\NZ=\{0,1,2,\ldots\}$ be the set of non-negative integers.  The
unit interval is denoted by $\II:=[-1,1]\subset\R$, and we denote by
$(\vec{x},\vec{y},\vec{z})$ the usual frame of $\R^3$ given by
$\vec{x}=(1,0,0)$, $\vec{y}=(0,1,0)$, $\vec{z}=(0,0,1)$.  

By a ``monoidal functor'' between {(strict or non-strict)} monoidal categories, we always mean
a \emph{strict} monoidal functor.

\subsection{Acknowledgements}\label{sec:acknowledgements}

The work of K.H. is partly supported by JSPS KAKENHI Grant Number 15K04873;
{the work of G.M. is  supported in part by the project ITIQ-3D, funded by the ``R\'egion Bourgogne Franche-Comt\'e.''}
{The authors are grateful to {Mai Katada and} Jean-Baptiste
Meilhan for helpful comments on the {previous versions} of the manuscript.}

%
%
\section{The category $\B$ of bottom tangles in handlebodies} \label{sec:categories}

In this section, we define three strict monoidal categories
\begin{itemize}
\item $\B$ of bottom tangles in handlebodies \cite{Habiro},
\item $\EH$ of embeddings of handlebodies \cite{Habiro4},
\item $\sLCob$ of special Lagrangian cobordisms \cite{CHM},
\end{itemize}
with the same object monoid $\Ob(\B)=\Ob(\EH)=\Ob(\sLCob)=\NZ$.  They
are essentially the same structures since we have isomorphisms of
strict monoidal categories
\begin{gather*}
  \B \cong \EH^\op\cong \sLCob.
\end{gather*}
The categories $\B$ and $\EH$ will be studied in more detail in \cite{Habiro2}.

Let $m,n,p$ be non-negative integers throughout this section.

\subsection{Bottom tangles in handlebodies}   \label{sec:bott-tangl-handl}

Let $V_m \subset \R^3$ denote the handlebody of genus $m$ that is
obtained from the cube $\II^3 \subset \R^3$ by attaching $m$ handles
on the top square $\II^2 \times \{1\}$:
\begin{gather}
  \label{e3}
V_m:=\centre{
{\labellist \scriptsize \hair 2pt
\pinlabel{$\vec x$} [l] at 656 90
\pinlabel{$\vec y$} [bl] at 636 127
\pinlabel{$\vec z$} [b] at 599 148
\pinlabel{$1$} at 206 285
\pinlabel{$m$} at 403 284
\pinlabel{$\cdots$} at 305 380
 \pinlabel{$S$}  at 64 23
 \pinlabel{$\ell$} [l] at 536 50
  \pinlabel{$A_1$} [l] at 144 166
    \pinlabel{$\cdots$} [l] at 285 166
 \pinlabel{$A_m$} [l] at 340 167
\endlabellist
\includegraphics[scale=0.28]{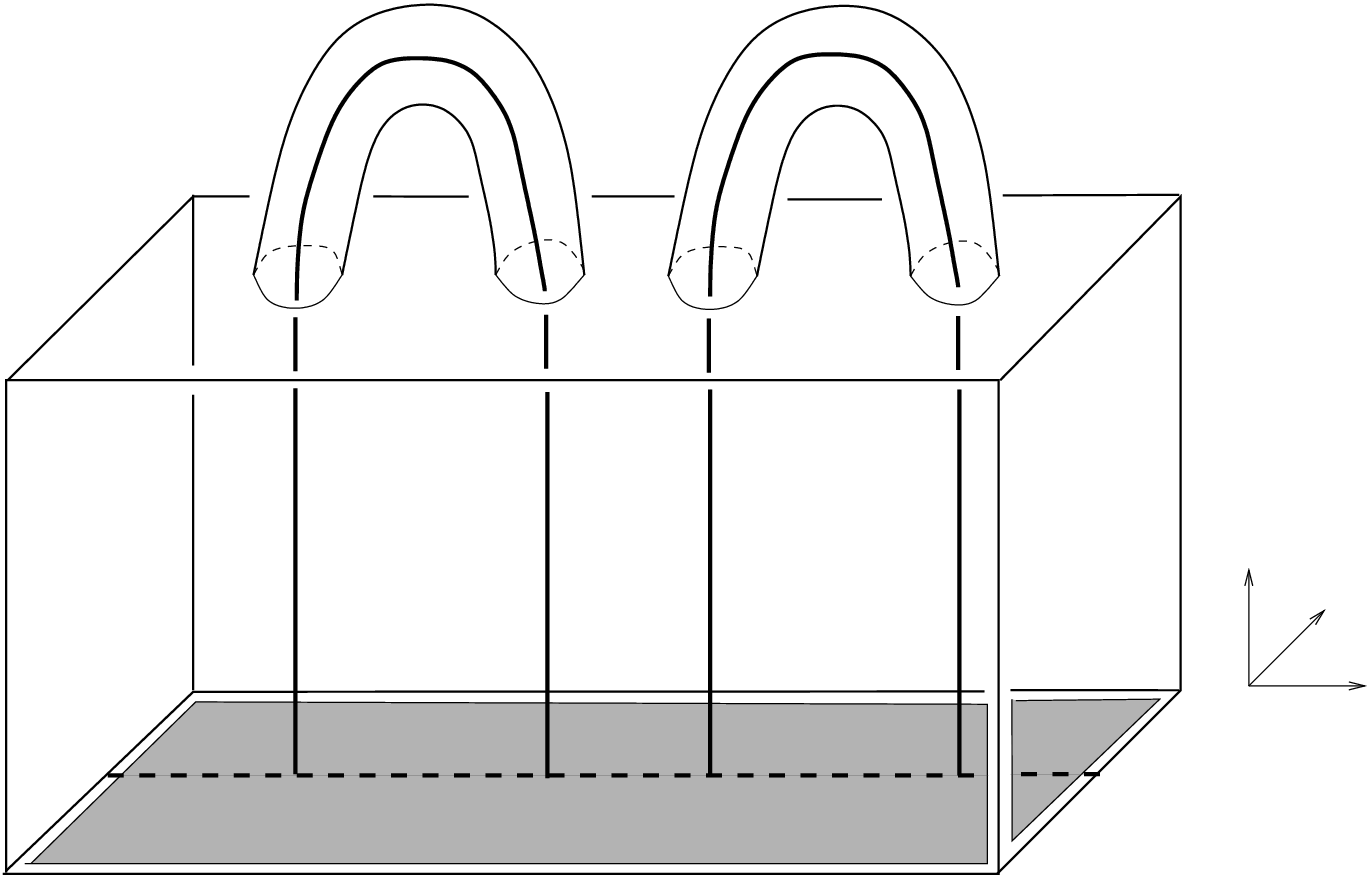}}}
\end{gather}
We call $S:= \II^2 \times \{-1\}$ the \emph{bottom square} of
$V_m$ and $\ell:=\II \times \{0\} \times \{-1\}$ the \emph{bottom
line} of $V_m$.  Let $A_1,\dots,A_m$ denote the arcs obtained from the
cores of the handles by ``stretching'' the ends down to $\ell$.

An $n$-component \emph{bottom tangle} $T=T_1\cup\dots\cup T_n$ in
$V_m$ is a framed, oriented tangle consisting of $n$ arc components
$T_1,\dots,T_n$ such that
\begin{enumerate}
\item the endpoints of $T$ are uniformly distributed along $\ell$,
\item for $i=1,\ldots,n$, the $i$-th component $T_i$ runs
from the $2i$-th endpoint to the $(2i-1)$-st endpoint, where we  count the endpoints of $T$  from the left.
\end{enumerate}

We usually depict bottom tangles by drawing their orthogonal
projections onto the plane $\R\times \{1\} \times \R$ and assuming the
\emph{blackboard framing convention}; i.e., the framing is given by
the vector field $\vec{y}$.  For example, here is a $3$-component
bottom tangle in $V_2$ together with a projection diagram:
\begin{gather}
\label{e13}
\centre{\labellist
\small\hair 2pt
\pinlabel{$\leadsto$}  at 115 27
\endlabellist
\centering
\includegraphics[scale=1.3]{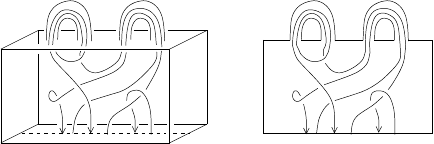}}
\end{gather}

\subsection{The category $\B$ of bottom tangles in handlebodies}  \label{sec:category-b-bottom}

{Morphisms from $m$ to $n$ in $\B$ are}
isotopy classes of $n$-component bottom
tangles in $V_m$.  {Define} the composition of two bottom
tangles $m\xto{T}n\xto{T'}p$
by 
$$
  T' \circ  T = i_T(T') :m \lto p,
$$
where
$$
i_T: V_n \hookrightarrow V_m
$$ is an embedding which maps $S\subset V_n$ identically onto
$S\subset V_m$ and maps $A_i$ {onto}~$T_i$ {in a framing-preserving way} for all $i=1,\dots,n$.  Here is
an example of the composition of morphisms {$2\to1\to2$}:
\begin{gather}
  \label{e36}
\centre{\labellist
\small \hair 2pt
 \pinlabel{$\circ$} at 55 19
 \pinlabel{$=$}  at 164 19
\endlabellist
\centering
\includegraphics[scale=1.0]{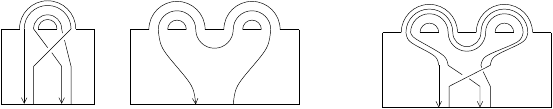}}
\end{gather}
The identity morphism $\id_m: m\to m$ in $\B$ is the union $A:=A_1 \cup \cdots \cup A_m$ of the ``stretched'' cores of the handles of $V_m$:
\begin{gather}
  \label{e53}
  \id_m =
\centre{\labellist
\scriptsize \hair 2pt
 \pinlabel{$1$} [rb] at 26 93
 \pinlabel{$m$} [rb] at 100 94
  \pinlabel{$\cdots$}  at 83 87
 \pinlabel{$\cdots$}  at 85 33
\endlabellist
\centering
\includegraphics[scale=0.6]{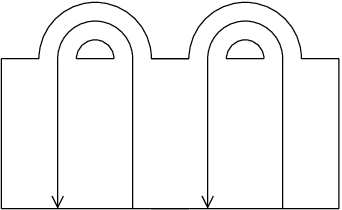}}
\end{gather}
The tensor product in $\B$ is juxtaposition.

\subsection{The category $\EH$ of embeddings of handlebodies}  \label{sec:categ-eh-embedd}

{Morphisms from $n$ to~$m$ in $\EH$ are}
isotopy classes rel $S$ of embeddings
$V_n \hookrightarrow V_m$ restricting to $\id_S$. {Define} the composition and
the identity in $\EH$ in the obvious way.

We have an isomorphism $\B \cong \EH^{\op}$ of categories given by
$$
\begin{array}{rcl}
\EH(n,m) &  \stackrel{\cong }{\longrightarrow} & \B(m,n) \\
\hbox{  $(V_n \stackrel{i}{\hookrightarrow} V_m)$ }&\longmapsto & \hbox{  $ \big(i(A) \subset V_m\big)$} \\
\hbox{  $(V_n \stackrel{i_T}{\hookrightarrow} V_m) $} &\mathrel{\reflectbox{\ensuremath{\longmapsto}}} & \hbox{  $(T \subset V_m)$},
\end{array}
$$
transporting the strict monoidal structure of $\B$ to $\EH$.

\subsection{The category $\sLCob$ of special Lagrangian cobordisms} \label{sec:sLCob}

Here we will define the category $\sLCob$ of special Lagrangian
cobordisms.  We will not need it until Section \ref{sec:LMO_functor};
we define it here for comparison with $\B$.

Let $\surf m$ be the compact, connected, oriented
surface of genus $m$ with one boundary component, located at the top
of $V_m \subset \R^3$:\\[0.2cm]
$$
 \surf m := \centre{\labellist \scriptsize  \hair 2pt
 \pinlabel{$1$} [b] at 229 219
 \pinlabel{$m$} [b] at 487 218
\pinlabel{$\cdots$} at 360 120
\pinlabel{$\circlearrowleft$} at 80 43
\endlabellist
\includegraphics[scale=0.25]{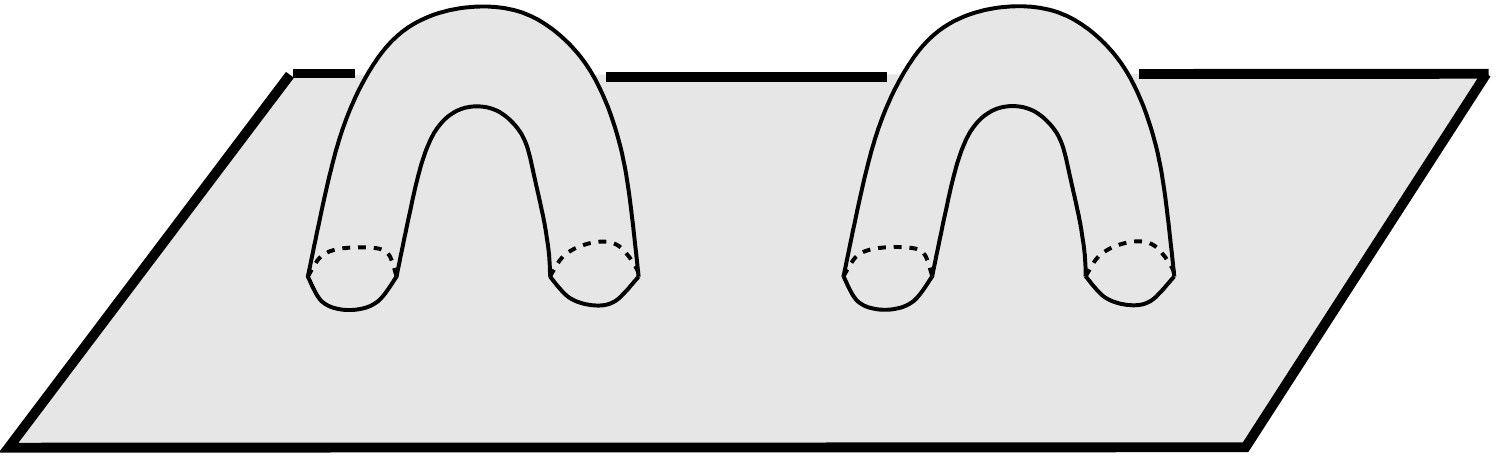}}
$$
We identify $\partial \surf m$ with $\partial\II^2$.

A \emph{cobordism} from $\surf m$ to $\surf n$ is an equivalence
class of pairs $(C, c)$  {of} a compact, connected, oriented
3-manifold {$C$} and  an orientation-preserving homeomorphism
\begin{gather*}
  c: \left((- \surf n )\cup_{\partial \II^2 \times
  \lbrace -1 \rbrace} \left(\partial \II^2 \times \II \right)
\cup_{\partial \II^2 \times \lbrace 1 \rbrace} \surf m\right)
 \longrightarrow \partial {C}.
\end{gather*}
Here, two cobordisms $(C, c)$ and $(C', c')$ are 
\emph{equivalent} if there is a homeomorphism $f : C \rightarrow
C'$ such that $c' = f|_{\partial C} \circ c$.  For instance, the
handlebody $V_m$ (with the obvious boundary parametrization) defines a
cobordism from $\surf m$ to $\surf 0$.  More generally, every
$n$-component bottom tangle $T \subset V_m$ defines a cobordism
$$
E_T := (E_T,e_T)
$$ from $\surf m$ to $\surf n$ by considering the exterior $E_T$ of $T$ in
$V_m$, together with the boundary parametrization $e_T$ induced by the framing of $T$.

Define the category $\Cob$ of $3$-dimensional cobordisms introduced by
Crane \& Yetter \cite{CY} and Kerler \cite{Kerler} as follows.  Set
$\Ob(\Cob) = \NZ$.  {Morphisms from $m$ to $n$ in $\Cob$ are} 
equivalence
classes of cobordisms from $\surf m$ to $\surf n$.
We  obtain the
composition $C' \circ C :m\to p $ of $C' = (C',c') :n\to p $ and $C = (C,c) :m\to n$
from $C'$ and $C$
by identifying the target surface of $C$ with the source surface of $C'$ using the
boundary parametrizations.  The identity morphism $\id_m :m\to m $
is the cylinder $\surf m \times \II$ with the boundary
parametrization defined by the identity maps.

We equip $\Cob$ with a strict monoidal structure such that $m\otimes
m'=m+m'$, and {we obtain} the tensor product $C \ot C'$ of $C = (C,c)$ and $C' =
(C',c')$  from $C$ and $C'$ by identifying the right square
$c(\{1\} \times \II \times \II)$ of $\partial C$ with the left square
$c'(\{-1\} \times \II \times \II)$ of $\partial C'$.

A cobordism $C$ from $\surf m$ to $\surf n$ is said to be
\emph{special Lagrangian} if we have
$$
V_n \circ C = V_m \ :\ m \longrightarrow 0.
$$
The special
Lagrangian cobordisms form a monoidal subcategory $\sLCob$ of $\Cob$.
We have an isomorphism $\B \cong \sLCob$ of strict monoidal
categories given by
$$
\begin{array}{rcl}
\B(m,n) &  \stackrel{\cong }{\longrightarrow} & \sLCob(m,n) \\
\hbox{ $(T \subset V_m)$} &\longmapsto &  \hbox{ $E_T$} \\
\hbox{ $\big(A \subset (V_n \circ C) \big)$} &\mathrel{\reflectbox{\ensuremath{\longmapsto}}} & \hbox{ $C$}.
\end{array}
$$

%
%
\section{Review of the Kontsevich integral}		\label{sec:Kontsevich}

In this section, we briefly review the combinatorial construction of
the Kontsevich integral of tangles in the cube.
See \cite{BN1,LM_combinatorial,KT,Ohtsuki} for further details.

\subsection{Free monoids and magmas} \label{sec:free-monoids-magmas}

For a finite set $\{s_1,\dots,s_r\}$, let $ \Mon(s_1,\dots,s_r) $
denote the free monoid on $s_1,\dots,s_r$, consisting of words in
the letters $s_1,\dots,s_r$.  For $w \in\Mon(s_1,\dots,s_r)$, let
$\vert w\vert$ denote the length of $w$, and $w_1,\dots, w_{\vert w
\vert}$ the consecutive letters in $w$.

Let also $\Mag(s_1,\dots,s_r)$ denote the free unital magma on
$s_1,\dots,s_r$, consisting of non-associative words in
$s_1,\dots,s_r$. Let
\begin{gather*}
  U: \Mag(s_1,\dots,s_r) \longrightarrow \Mon(s_1,\dots,s_r)
\end{gather*}
be the (surjective) map forgetting parentheses.  Sometimes the word
$U(w)$ for $w\in\Mag(s_1,\dots,s_r)$ will be simply denoted by $w$.

\subsection{The category $\T$ of tangles in the cube}	\label{sec:cat-tangles}

By a \emph{tangle} in the cube $\II^3$ we mean a framed, oriented
tangle $\gamma$ in $\II^3$, whose boundary points are on the
intervals $\II \times \{0\} \times \{-1,1\}$.  We assume that the
framing at each endpoint is the vector $\vec y$.  In figures we use
the blackboard framing convention as before.

The \emph{source} $s(\gamma)\in\Mon(\pm):= \Mon(+,-)$ of a tangle
$\gamma$ is the word in $+$ and~$-$ that are read along the oriented
interval $\II \times \{0\} \times \{+1 \}$, where each boundary point
of $\gamma$ is given the sign $+$ (resp.\ $-$) when the orientation of
$\gamma$ at that point is downwards (resp.\ upwards).  The
\emph{target} $t(\gamma) \in\Mon(\pm)$ of $\gamma$ is defined
similarly.  The tangle $\gamma$ is said to be \emph{from} $s(\gamma)$
\emph{to} $t(\gamma)$.

We define the strict monoidal \emph{category $\T$ of tangles} (in the cube)
 as follows.  Set $\Ob(\T) = \Mon(\pm)$.
{Morphisms from $w$ to $w'$ in $\T$ are}
the isotopy classes of
tangles from $w$ to $w'$.  {We obtain the} composition $\gamma \circ \gamma'$
of two tangles $\gamma$ and $\gamma'$ such that $t(\gamma')=s(\gamma)$
 by gluing $\gamma'$ on the top of $\gamma$.  The identity
$\id_w : w\to w$ of $w\in\Ob(\T)$ is the trivial tangle
with straight vertical components.  The tensor product in the
strict monoidal category $\T$ is juxtaposition.

\subsection{The category $\T_q$ of $q$-tangles in the cube} \label{sec:tangles}

Here we define the category $\T_q$ of $q$-tangles in the cube as the
``non-strictification''
of the strict monoidal category~$\T$.  Since we use this construction also for
other categories, we first give a general definition.

Let $\C$ be a strict monoidal category such that the object monoid
$\Ob(\C)$ is a free monoid $\Mon(S)$ on a set $S$.  Then the
\emph{non-strictification} of $\C$ is the (non-strict) monoidal
category $\C_q$ defined as follows.  Set $\Ob(\C_q)=\Mag(S)$, the free
unital magma on $S$.  Let $U:\Mag(S)\to\Mon(S)$ be the canonical map,
forgetting parentheses.  Set $\C_q(x,y)=\C(U(x),U(y))$ for
$x,y\in\Ob(\C_q)=\Mag(S)$.  The compositions, identities and tensor
products in $\C_q$ are given by those of $\C$. 
We define the associativity
isomorphism by
\begin{gather*}
  \alpha_{x,y,z}=\id_{x\ot y\ot z}
  \in\C_q((x\ot y)\ot z, x\ot(y\ot z))=\C(x\ot y\ot z,x\ot y\ot z).
\end{gather*}
Note that the tensor product in $\C_q$ is strictly left and right
unital, i.e., $\varnothing\ot x=x=x\ot\varnothing$ for
$x\in\Ob(\C_q)$, where $\varnothing\in\Mag(S)$ is the unit.
Then $\C_q$ is a monoidal category, which is not strict if $S$ is not
empty.  The map $U:\Ob(\C_q)\to\Ob(\C)$ extends to an equivalence of
categories
\begin{gather*}
  U:\C_q   \overset{\simeq}{\longrightarrow}  \C
\end{gather*}
such that $U(f)=f$ for all $f\in\C_q(x,y)=\C(U(x),U(y))$.
If $\C$ is a braided (resp.\ symmetric) strict monoidal category, then the
non-strictification $\C_q$ naturally has the structure of a braided
(resp.\ symmetric) non-strict monoidal category.

Now, define the non-strict braided monoidal {\emph{category $\T_q$ of $q$-tangles} 
(in the cube)} to be the non-strictification of $\T$.  Since
$\Ob(\T)=\Mon(\pm)$, we have $\Ob(\T_q)=  \Mag(\pm):= \Mag(+,-)$.

\subsection{Cabling} \label{sec:cabling}

Here we review the definition of the ``cabling'' operations for
$q$-tangles in the cube.

Define the duality involution $w \mapsto w^*$ on $\Mag(\pm)$
inductively by $\varnothing^*=\varnothing$, $\pm^*=\mp$ and
$(ww')^*=(w')^* w^*$.
For $w\in\Mag(\pm) $ and  $\varpi: \{1, \dots,
\vert w\vert \} \to \Mag(\pm)$, we obtain $C_\varpi(w)\in\Mag(\pm)$
from $w$ by replacing each of its consecutive letters $w_i$ with the
subword $\varpi(i)$ (resp.\ $\varpi(i)^*$) if $w_i=+$ (resp.\ $w_i=-$).
For {every} $q$-tangle $\gamma:w\to w'$ and
{every} map $\varpi: \pi_0(\gamma) \to \Mag(\pm)$,\footnote{{Here
the reader is warned that $\gamma$ should not be thought of as a
\emph{morphism} in $\T_q$, especially if $\gamma$ has (more than one) closed components.  Note
that if $\gamma$ denotes a morphism in $\T_q$, i.e., an isotopy class of
$q$-tangles, then ``$\pi_0(\gamma)$'' is not well-defined.}} let $C_\varpi(\gamma)$
be the {$q$-}tangle obtained from $\gamma$ by replacing each connected
component $c \subset \gamma$ with the $\varpi(c)$-cabling of $c$.
(For instance, if $\varpi(c)=-$, then the $\varpi(c)$-cabling of $c$
is obtained by reversing the orientation of $c$ and, if
$\varpi(c)=(++)$, then the $\varpi(c)$-cabling of $c$ is obtained by
doubling $c$ using the given framing.)  We call $C_\varpi(\gamma)$ the
\emph{$\varpi$-cabling} of $\gamma$, and we regard it as a morphism
$$
C_\varpi(\gamma):
C_{\varpi_s}(w)\longrightarrow C_{\varpi_t}(w')
$$ in $\T_q$.  Here $\varpi_s: \{1, \dots, \vert w\vert \} \to
\Mag(\pm)$ denotes the composition of $\varpi$ and the map $\{1,\dots,
\vert w \vert\} \to \pi_0(\gamma)$ relating the top boundary points of
$\gamma$ to its connected components, and $\varpi_t: \{1, \dots, \vert
w'\vert \} \to \Mag(\pm)$ is defined similarly.

One can easily verify the following lemma explaining the behavior of
the cabling operation on compositions.

\begin{lemma} \label{rem:functoriality_cabling}
For $q$-tangles $\gamma$ and $\gamma'$ with $s(\gamma)=t(\gamma')$ and
maps $\varpi: \pi_0(\gamma) \to \Mag(\pm)$ and $\varpi':
\pi_0(\gamma') \to \Mag(\pm)$ with $\varpi_s = \varpi_t'$, we have
\begin{equation}
C_{\varpi \cup \varpi'}(\gamma \circ \gamma') = C_\varpi(\gamma) \circ C_{\varpi'}(\gamma'),
\end{equation}
where $\varpi \cup \varpi'$ denotes the unique map $\pi_0(\gamma \circ
\gamma') \to\Mag(\pm)$ compatible with $\varpi$ and $\varpi'$ through
the canonical maps $\pi_0(\gamma) \to \pi_0(\gamma \circ \gamma')$ and
$\pi_0(\gamma') \to \pi_0(\gamma \circ \gamma')$.
\end{lemma}

\subsection{Spaces of Jacobi diagrams}  \label{sec:spac-jacobi-diagr}

Let $X$ be a compact, oriented $1$-manifold. A \emph{chord diagram}
$D$ on $X$ is a disjoint union of unoriented arcs, called
\emph{chords}, and whose set of endpoints is embedded in the interior
of $X$.  We identify two chord diagrams $D$ and $D'$ on $X$ if there
is a homeomorphism $(X \cup D,X) \to (X \cup D',X)$ preserving the
orientations and connected components of $X$.  Let $\A(X)$ be the
vector space generated by chord diagrams on $X$ modulo the 4T relation:
\begin{equation}   \label{eq:4T}
\labellist
\small\hair 2pt
 \pinlabel{$+$}  at 185 55
 \pinlabel{$=$}  at 423 55
 \pinlabel{$+$}  at 653 57
 \pinlabel{ 4T} [t] at 415 3
\endlabellist
\centering
\includegraphics[scale=0.2]{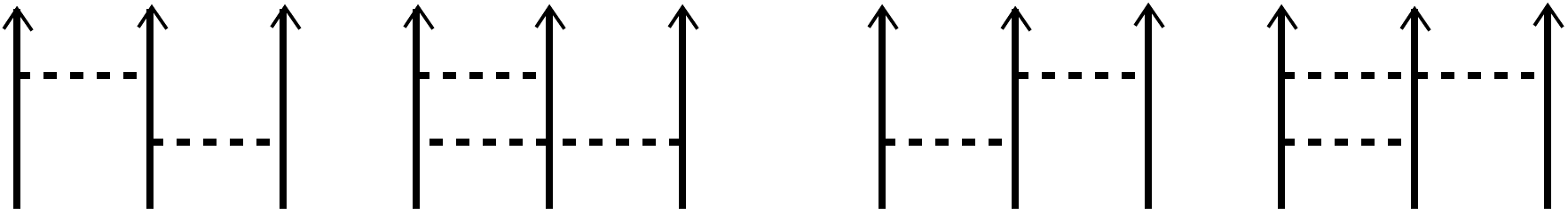}
\end{equation}

\vspace{0.2cm}\noindent Here the dashed lines represent chords, and
the solid lines are intervals in $X$ with the orientation inherited
from $X$.

Bar-Natan \cite{BN2} gave an alternative definition of $\A(X)$ as
follows.  A \emph{Jacobi diagram} $D$ on $X$ is a unitrivalent graph
such that each trivalent vertex is oriented (i.e., equipped with a
cyclic ordering of the incident half-edges), the set of univalent
vertices is embedded in the interior of $X$, and such that each
connected component of $D$ contains at least one univalent vertex.
We identify two Jacobi diagrams $D$ and $D'$ on $X$ if there is
a homeomorphism $(X \cup D,X) \to (X \cup D',X)$ preserving the
orientations and connected components of $X$ and respecting the
vertex-orientations.  In pictures, we draw the $1$-manifold part $X$
with solid lines, and the graph part $D$ with dashed lines, and the
vertex-orientations are counterclockwise.  For instance, we can view
chord diagrams as Jacobi diagrams without trivalent vertices.  The
vector space $\A(X)$ is isomorphic to, hence identified
with, the vector space generated by Jacobi diagrams on $X$ modulo the STU relation:
\begin{equation}  \label{eq:STU} 
\labellist \small \hair 2pt
\pinlabel{STU} [t] at 150 0
 \pinlabel{$=$}  at 127 39
 \pinlabel{$-$}  at 266 37
\endlabellist
\centering
\includegraphics[scale=0.3]{STU}
\end{equation}

\vspace{0.2cm}\noindent As proved in \cite[Theorem 6]{BN2}, the STU
relation implies the AS and IHX relations:
\begin{equation} \label{eq:AS_IHX}
\labellist \small \hair 2pt
\pinlabel{AS} [t] at 102 0
\pinlabel{IHX} [t] at 552 0
\pinlabel{$= \ -$}  at 102 46
\pinlabel{$-$} at 444 46
\pinlabel{$+$} at 566 46
\pinlabel{$=0$} at 679 46
\endlabellist
\centering
\includegraphics[scale=0.3]{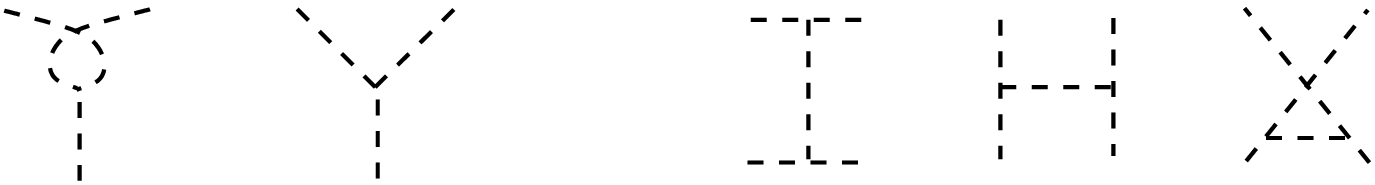}
\end{equation}
\vspace{0.cm}

Note that $\A(X)$ is a graded vector space, where we define the
\emph{degree} of a Jacobi diagram to be half the
total number of vertices.  Let $\A(X)$ also denote its degree-completion.

\begin{example} \label{ex:box}
The \emph{box notation} is a useful way to represent
certain linear combinations of Jacobi diagrams:
$$
\labellist
\small\hair 2pt
 \pinlabel{$:=$} at 193 68
 \pinlabel{$-$}  at 402 70
  \pinlabel{\scriptsize $\cdots$}  at 103 101
 \pinlabel{\scriptsize $\cdots$}  at 325 101
 \pinlabel{\scriptsize $\cdots$}  at 515 101
  \pinlabel{\scriptsize $\cdots$}  at 715 101
  \pinlabel{\scriptsize $\cdots$}  at 992 101
  \pinlabel{$\pm \cdots +$} at  830 69
 \pinlabel{$+$}  at 605 69
\endlabellist
\centering
\includegraphics[scale=0.33]{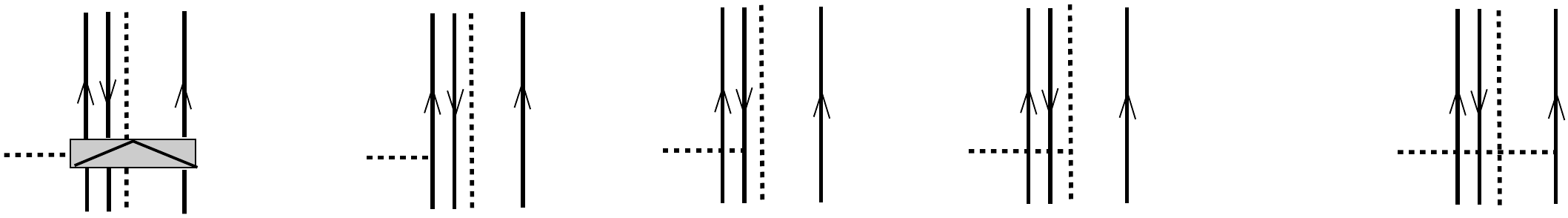}
$$
Here, dashed edges and solid arcs are allowed to go through the box,
and each of them contributes to one summand in the box notation.
A solid arc contributes with a plus or minus sign,
depending on the compatibility of its orientation with the direction
of the box.  A dashed edge always contributes with a plus
sign, the orientation of the new trivalent vertex being determined by
the direction of the box.  We also define\\[0.cm]
$$
\labellist
\small\hair 2pt
 \pinlabel{\scriptsize $\cdots$}  at 103 111
 \pinlabel{$:=$}  at 182 44
 \pinlabel{\scriptsize $\cdots$}   at 310 115
 \pinlabel{so that }   at 430 46
 \pinlabel{$=-$}  at 684 42
 \pinlabel{\scriptsize $\cdots$}  at 622 119
 \pinlabel{\scriptsize $\cdots$}  at 822 114
 \pinlabel{.}  at 884 42
\endlabellist
\centering
\includegraphics[scale=0.31]{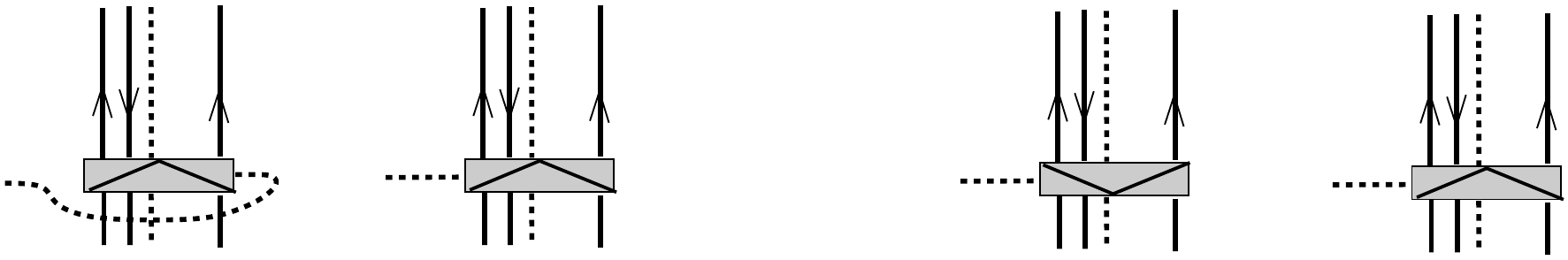}
$$
\end{example}

\subsection{The category $\AT$ of Jacobi diagrams} \label{sec:cat-Jacobi}

A compact, oriented $1$-manifold $X$ is said to be \emph{polarized} if
$\partial X$ is decomposed into a \emph{top part} $\partial_+ X$ and a
\emph{bottom part} $\partial_-X$ with each of them totally ordered.
The \emph{target} $t(X)\in\Mon(\pm)$ of $X$ is the word obtained from
$\partial_- X $ by replacing each positive (resp.\ negative) point
with $+$ (resp.~$-$).  The \emph{source} $s(X)\in\Mon(\pm)$ of $X$ is
defined similarly using $\partial_+X$, but the rule for the signs
$+,-$ is reversed.

\begin{example} \label{ex:polarized_1}
  Every $q$-tangle  is naturally regarded as a polarized $1$-manifold.
\end{example}

\begin{example} \label{ex:polarized_2}
  For $w\in\Mon(\pm)$, let $\downarrow  \stackrel{w}{\cdots}  \downarrow$
  denote the identity tangle $\id_w$ as a polarized $1$-manifold.
\end{example}

We define the \emph{category $\AT$ of Jacobi diagrams} as follows.
Set $\Ob(\AT)=\Mon(\pm)$, and for $w, w' \in\Mon(\pm)$ {set} 
\begin{gather}
  \label{e27}
 \A(w,w') = \coprod_X \A(X)_{\fS_{c(X)}}
\end{gather}
where $X$ runs over  homeomorphism classes of polarized $1$-manifolds
with ${s(X)=w}$ and $t(X)=w'$, $c(X)$ is the number of closed components of $X$,
the symmetric group $\fS_{c(X)}$ acts on $\A(X)$ by permutation of closed components
and $ \A(X)_{\fS_{c(X)}}$ denotes the {space of coinvariants}.
The composition $D \circ D'$ of a Jacobi diagram $D$ on a polarized
$1$-manifold $X$ with a Jacobi diagram $D'$ on a polarized
$1$-manifold $X'$ with $s(X)=t(X')$ is the Jacobi diagram $D
\sqcup D'$ on  $X \cup_{s(X)=t(X')} X'$.
The identity $\id_w$ of $w\in\Ob(\AT)$ is the empty Jacobi
diagram on ~$\downarrow\stackrel{w}{\cdots}\downarrow$.

The category $\AT$ admits a strict monoidal structure such that
the tensor product on objects is concatenation of words, and the
tensor product on morphisms is juxtaposition of Jacobi diagrams.

\begin{remark}
  \label{r30}
  Note that the category $\A$ is \emph{not} linear, since we can not add up
  two Jacobi diagrams with the same source and target {but {with different underlying polarized $1$-manifolds}.}
  {However, by setting}
  \begin{gather*}
    \A(w,w')  =  \bigoplus_X {\A(X)_{\fS_{c(X)} }}
  \end{gather*}
  instead of \eqref{e27}, we obtain a \emph{linear} strict  monoidal
  category $\A$.  We  sometimes need this {linear} version of $\A$.
\end{remark}

Finally, we have the following  analogs of the cabling
operations for $q$-tangles recalled in Section \ref{sec:tangles}.
We define the \emph{duality} $w\mapsto w^*$ on $\Mon(\pm)$
similarly to that on $\Mag(\pm)$.  For $w\in\Mon(\pm)$ and $\varpi: \{1,\dots, \vert  w \vert\} \to \Mon(\pm) $, we define
$C_\varpi(w)$ as in the non-associative case.  For {every}
{$D\in\A(X)$ representing a morphism in $\A(w,w')$}
with $w,w' \in\Mon(\pm) $ and {every} map $\varpi: \pi_0(X) \to \Mon(\pm) $, we define the \emph{$\varpi$-cabling} of $D$
{as an element $C_f(D)\in\A(C_f(X))$ representing a morphism}
$$
C_\varpi(D): C_{\varpi_s}(w) \longrightarrow
C_{\varpi_t}(w') \quad \text{in $\AT$}
$$ as follows. Let $\varpi_s$ be the obvious map $\{1, \dots, \vert w
\vert\}\to \pi_0(X)$ composed with $\varpi$, and define $\varpi_t$
similarly.  Then we obtain $C_\varpi(D)$ from $D$ by applying, to
every connected component $c\subset X$, the usual ``deleting
operation'' $\epsilon$ if $\vert \varpi(c) \vert=0$, or the usual
``doubling operation'' $\Delta$ repeatedly to get
$\vert \varpi(c) \vert$ new solid components if $\vert \varpi(c) \vert>0$,
and then the usual ``orientation-reversal operation'' $S$ to every new solid
component corresponding to a letter $-$ in the word~$\varpi(c)$.  (The
definitions of the operations $\epsilon$, $\Delta$ and $S$
{appear}  in \cite[\S 6.1]{Ohtsuki} for instance.)

We can easily verify the following analog of Lemma
  \ref{rem:functoriality_cabling}.

\begin{lemma} \label{rem:functoriality_cabling_bis}
  Let $D$ and $D'$ be Jacobi diagrams on polarized $1$-manifolds
  $X$ and $X'$, respectively, with $s(X)=t(X')$.  Let $\varpi:  \pi_0(X) \to \Mon(\pm)$, $\varpi': \pi_0(X') \to
  \Mon(\pm)$ be maps
  with $\varpi_s = \varpi_t'$.  Then we have
\begin{equation}
C_{\varpi \cup \varpi'}(D \circ D') = C_\varpi(D) \circ C_{\varpi'}(D'),
\end{equation}
where $\varpi \cup \varpi'$ denotes the unique map
$\pi_0(X\cup_{s(X)=t(X')}X')\to\Mon(\pm)$ compatible with $\varpi$ and
$\varpi'$.
\end{lemma}

\subsection{The Kontsevich integral $Z$} \label{sec:usual_Kontsevich}

Let $\Phi \in\A(\downarrow \downarrow \downarrow)$ be an
\emph{associator}.  In other words, $\Phi$ is the exponential of a
series of connected Jacobi diagrams on $\downarrow \downarrow
\downarrow$ which trivializes if any of the three strings is deleted,
and $\Phi$ is solution of one ``pentagon'' equation and two
``hexagon'' equations; see \cite[(6.11)--(6.13)]{Ohtsuki}. Define
\begin{gather}
  \label{e19}
\nu  =  \left(
\begin{array}{c}
\labellist
\small\hair 2pt
 \pinlabel{$S_2(\Phi)$}   at 147 227
\endlabellist
\centering
\includegraphics[scale=0.1]{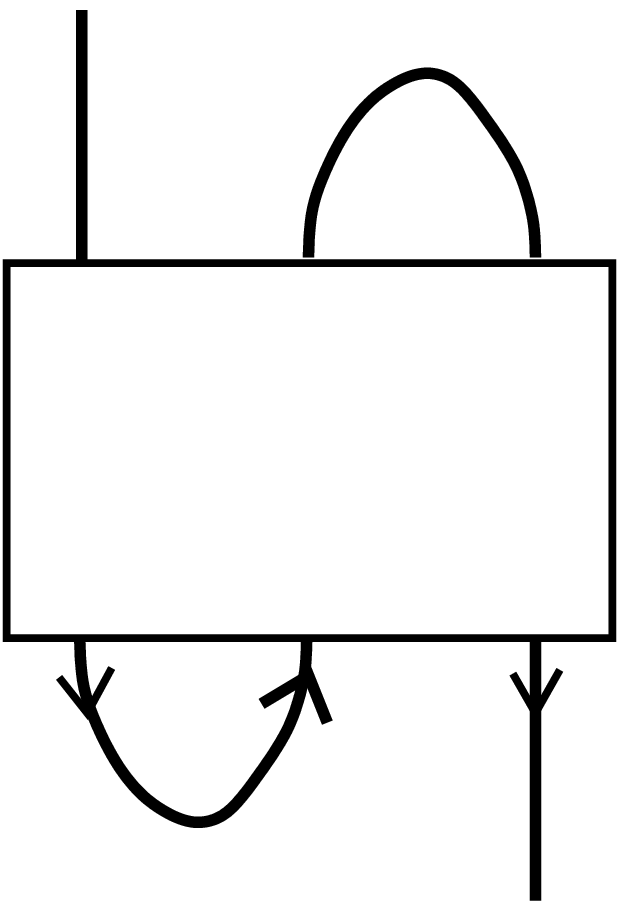} \end{array}\right)^{-1}\!\! = \figtotext{5}{40}{one_strand} +\frac{1}{48} \figtotext{20}{40}{phi_on_strand} +(\deg >2)
\ \in\A(\downarrow),
\end{gather}
where $S_2: \A(\downarrow \downarrow \downarrow) \to  \A(\downarrow \uparrow \downarrow)$
is the diagrammatic ``orientation-reversal operation'' applied to the second string.

\begin{theorem}[See \cite{BN1,Cartier,LM_combinatorial,Piunikhin,KT}] \label{th:Kontsevich}
Fix $a,u \in\Q$ with ${a+u=1}$.  There is a unique
  tensor-preserving functor $Z: \T_q \rightarrow \AT$ such that
\begin{itemize}
\item[(i)] $Z$ is the canonical map $U:\Mag(\pm) \to \Mon(\pm)$ on objects,
\item[{(ii)}] for {$\gamma:w\to w'$ in $\Tq$}, we have $Z(\gamma)\in {\A(\gamma)_{\fS_{c(\gamma)} }}  \subset \AT( w, w')$,
\item[{(iii)}] for {$\gamma:w\to w'$ in $\Tq$} and $\ell \in\pi_0(\gamma)$,
the value of $Z$ on the $q$-tangle obtained from $\gamma$ by reversing
the orientation of $\ell$ is $S_\ell(Z(\gamma))$,
\item[{(iv)}] $Z$ takes the following values on elementary $q$-tangles:\\[-0.5cm]
\end{itemize}
$$
\ \begin{array}{rcll}
Z \bigg( \!\! \begin{array}{c} {}_{(++)} \\ \figtotext{12}{12}{pcrossing} \\ {}^{(++)} \end{array} \!\! \bigg)
&= & \qquad \begin{array}{c}
\labellist \scriptsize \hair 2pt
\pinlabel{$\exp\big(\frac{1}{2}$} [r] at  -5 111
\pinlabel{$\big)$} [l] at  75 111
\endlabellist
\includegraphics[scale=0.2]{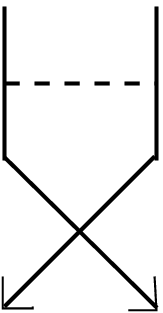}
\end{array}
&  \in\A\Big(\!\!\! \begin{array}{c} {}_{++} \\ \figtotext{12}{12}{crossing} \\ {}^{++} \end{array} \!\!\!\Big)  \subset \AT(++,++),\\
\ Z\bigg(\!\! \begin{array}{c} {}_{(w (w'w''))} \\ \figtotext{20}{15}{associator} \\ {}^{((ww') w'')}\end{array}
\!\! \bigg) & = &  C_{w,w',w''}(\Phi)  &  \in\A\Big(\!\! \downarrow \stackrel{\, ww'w''}{\cdots \cdots } \downarrow \!\! \Big)\subset \AT(ww'w'',ww'w'')\\
&&& \hbox{for} \ w,w',w'' \in\Mag(\pm),\\
Z\Big(\!\!\! \begin{array}{c}   \figtotext{20}{10}{capleft} \\ {}^{(+-)}  \end{array}\!\!\!\Big) & = &\!\!\!
\begin{array}{c}{
\labellist  \scriptsize \hair 2pt
\pinlabel{$\nu^a$} at 122 115
\endlabellist
\includegraphics[scale=0.2]{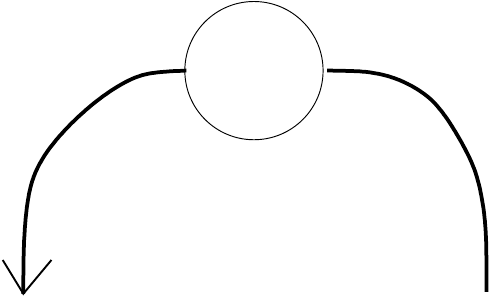}
} \end{array}  &  \in\A\Big(\!\!\!\! \begin{array}{c}   \figtotext{15}{10}{capleft} \\ {}^{+-}  \end{array} \!\!\!\! \Big) \subset \AT(\varnothing,+-),\\
Z\Big(\!\!\! \begin{array}{c} {}_{(+-)} \\ \figtotext{20}{10}{cupright} \end{array} \!\!\!\Big) & = &\!\!\!
\begin{array}{c}
\labellist \scriptsize  \hair 2pt
\pinlabel{$\nu^u$} at 115 35
\endlabellist
\includegraphics[scale=0.2]{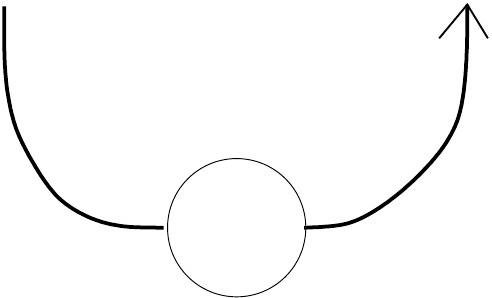} \end{array}
 &  \in\A\Big(\!\!\!\!\begin{array}{c} {}_{+-} \\ \figtotext{15}{10}{cupright} \end{array}\!\!\!\!\Big) \subset \AT(+-,\varnothing).
\end{array}
$$
\end{theorem}

The proof of the case $a=u=1/2$ \cite[Theorem 6.7, Proposition 6.8(2)]{Ohtsuki}
apply to the general case.  (The reader should,
however, be aware that the composition laws for the categories
$\T_q$ and $\A$ adopted in \cite{Ohtsuki} are opposite to ours.)

Now we review the behavior of the Kontsevich integral under cabling.
For
${w\in\Mag(\pm)}$, define
$a_w, a'_w,u_w, u'_w \in\A(\downarrow \stackrel{w}{\cdots} \downarrow) \subset \AT(w,w)$ by
$$
ZC_w\Big(\!\!\!\begin{array}{c}   \figtotext{20}{10}{capleft} \\ {}^{(+-)}  \end{array}\!\!\!\Big)  =
\centre{\labellist
\small\hair 2pt
 \pinlabel{$\stackrel{w^*}{\cdots}$}   at 363 65
 \pinlabel{$\stackrel{w}{\cdots}$}  at 104 67
 \pinlabel{$a_w$}   at 107 184
 \pinlabel{$C_wZ\Big($} [r] at 225 467
 \pinlabel{$\Big)$} [l] at 380 469
 \pinlabel{\scriptsize $(+-)$} [t] at 302 450
\endlabellist
\centering
\includegraphics[scale=0.15]{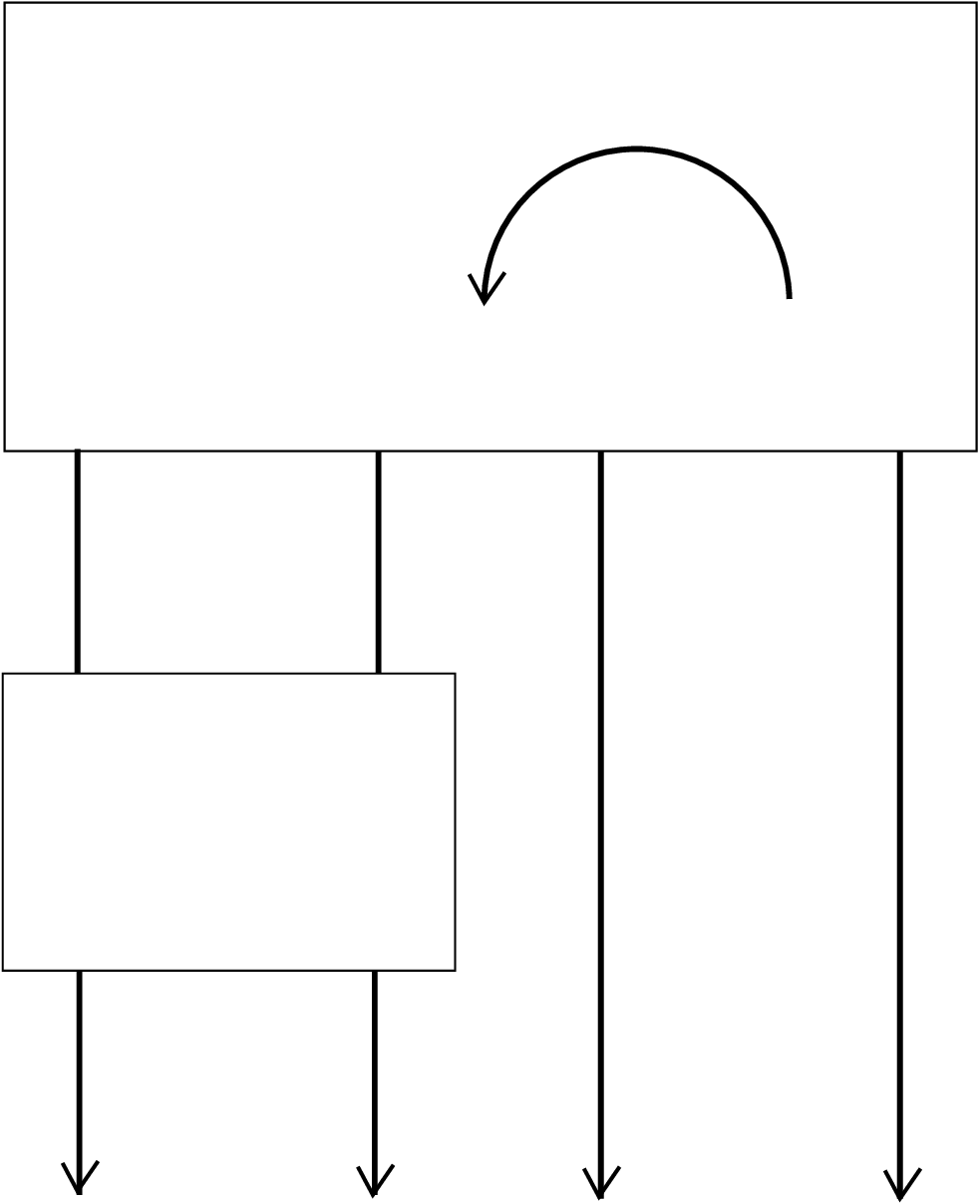}} , \quad
ZC_w\Big(\!\!\!\begin{array}{c}   \figtotext{20}{10}{capright} \\ {}^{(-+)}  \end{array}\!\!\!\Big)  =
\centre{\labellist
\small\hair 2pt
 \pinlabel{$\stackrel{w}{\cdots}$}   at 363 65
 \pinlabel{$\stackrel{w^*}{\cdots}$}  at 104 67
 \pinlabel{$a'_w$}   at 367 184
 \pinlabel{$C_wZ\Big($} [r] at 225 467
 \pinlabel{$\Big)$} [l] at 375 469
 \pinlabel{\scriptsize  $(-+)$} [t] at 302 450
\endlabellist
\centering
\includegraphics[scale=0.15]{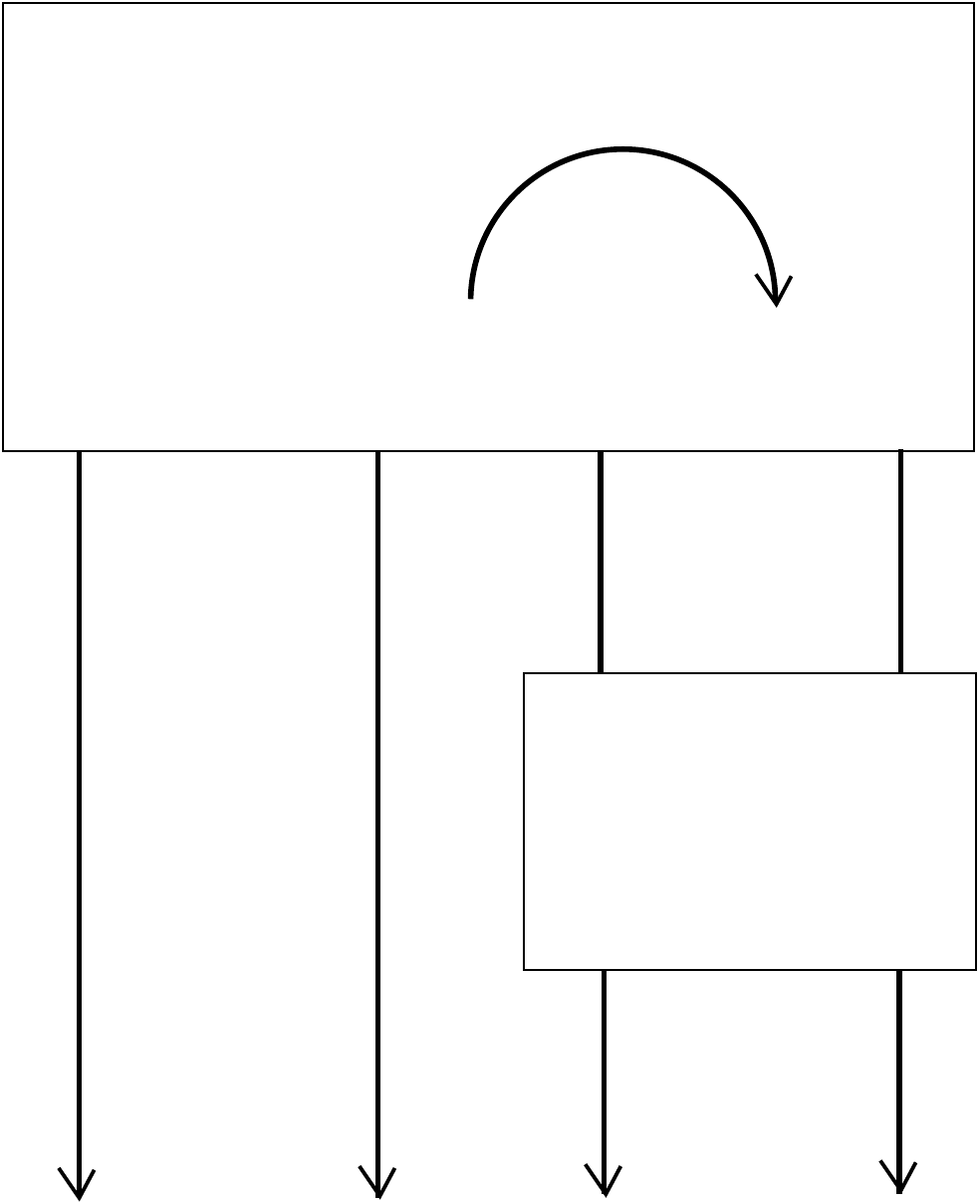}} , \
$$
$$
ZC_w\Big(\!\!\! \begin{array}{c} {}_{(+-)} \\ \figtotext{20}{10}{cupright} \end{array}\!\!\! \Big) =
\centre{\labellist
\small\hair 2pt
 \pinlabel{$\stackrel{w}{\cdots}$} [b] at 114 240
 \pinlabel{$\stackrel{w^*}{\cdots}$} [b] at 361 240
 \pinlabel{$u_w$}  at 104 397
 \pinlabel{$C_wZ\Big($} [r] at 211 98
 \pinlabel{$\Big)$} [l] at 384 94
 \pinlabel{\scriptsize  $(+-)$} [b] at 298 120
\endlabellist
\centering
\includegraphics[scale=0.15]{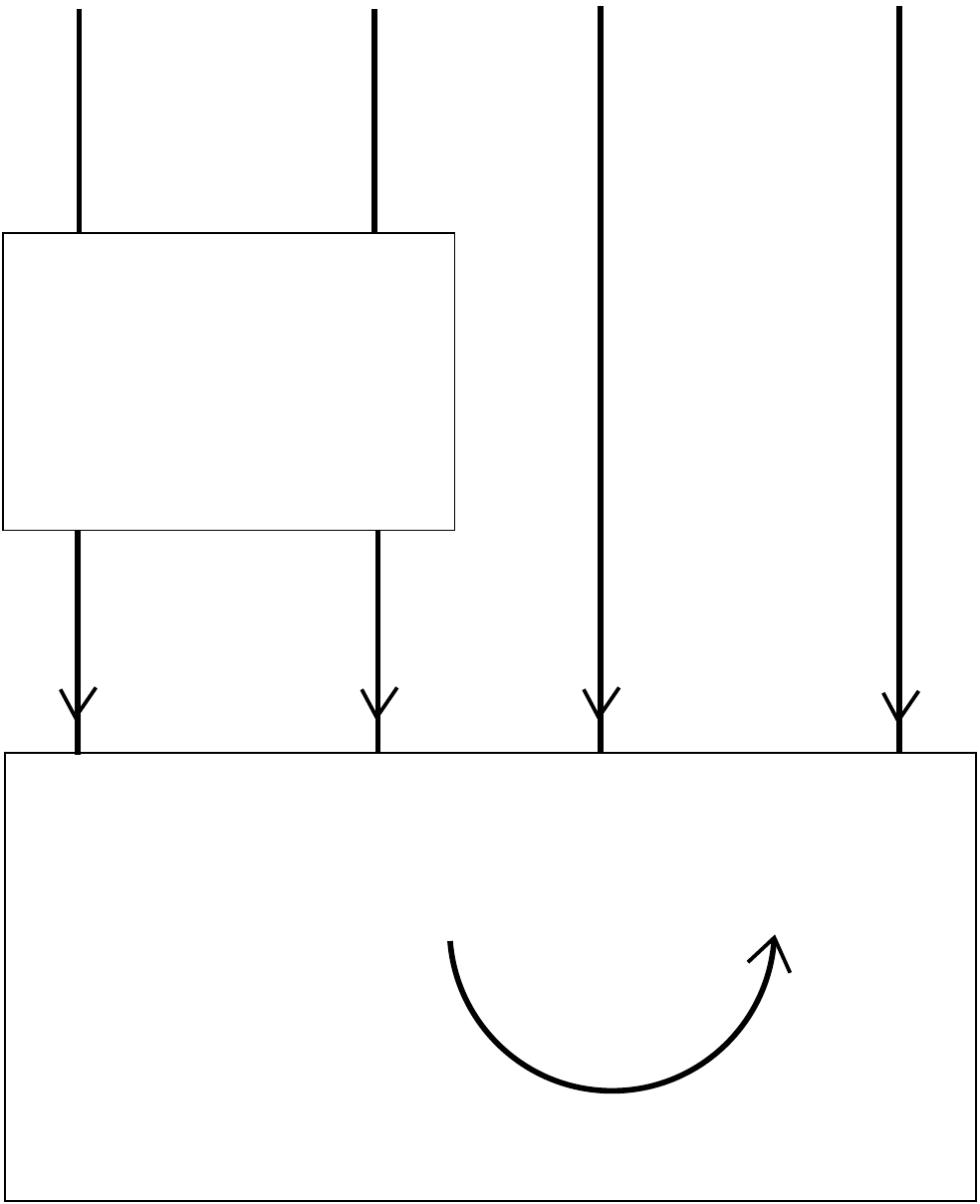}}
, \quad
ZC_w\Big(\!\!\! \begin{array}{c} {}_{\ (-+)} \\ \figtotext{20}{10}{cupleft} \end{array}\!\!\! \Big) =
\centre{\labellist
\small\hair 2pt
 \pinlabel{$\stackrel{w^*}{\cdots}$} [b] at 114 247
 \pinlabel{$\stackrel{w}{\cdots}$} [b] at 361 246
 \pinlabel{$u'_w$}  at 364 397
 \pinlabel{$C_wZ\Big($} [r] at 211 98
 \pinlabel{$\Big)$} [l] at 384 94
 \pinlabel{\scriptsize  $(-+)$} [b] at 298 120
\endlabellist
\centering
\includegraphics[scale=0.15]{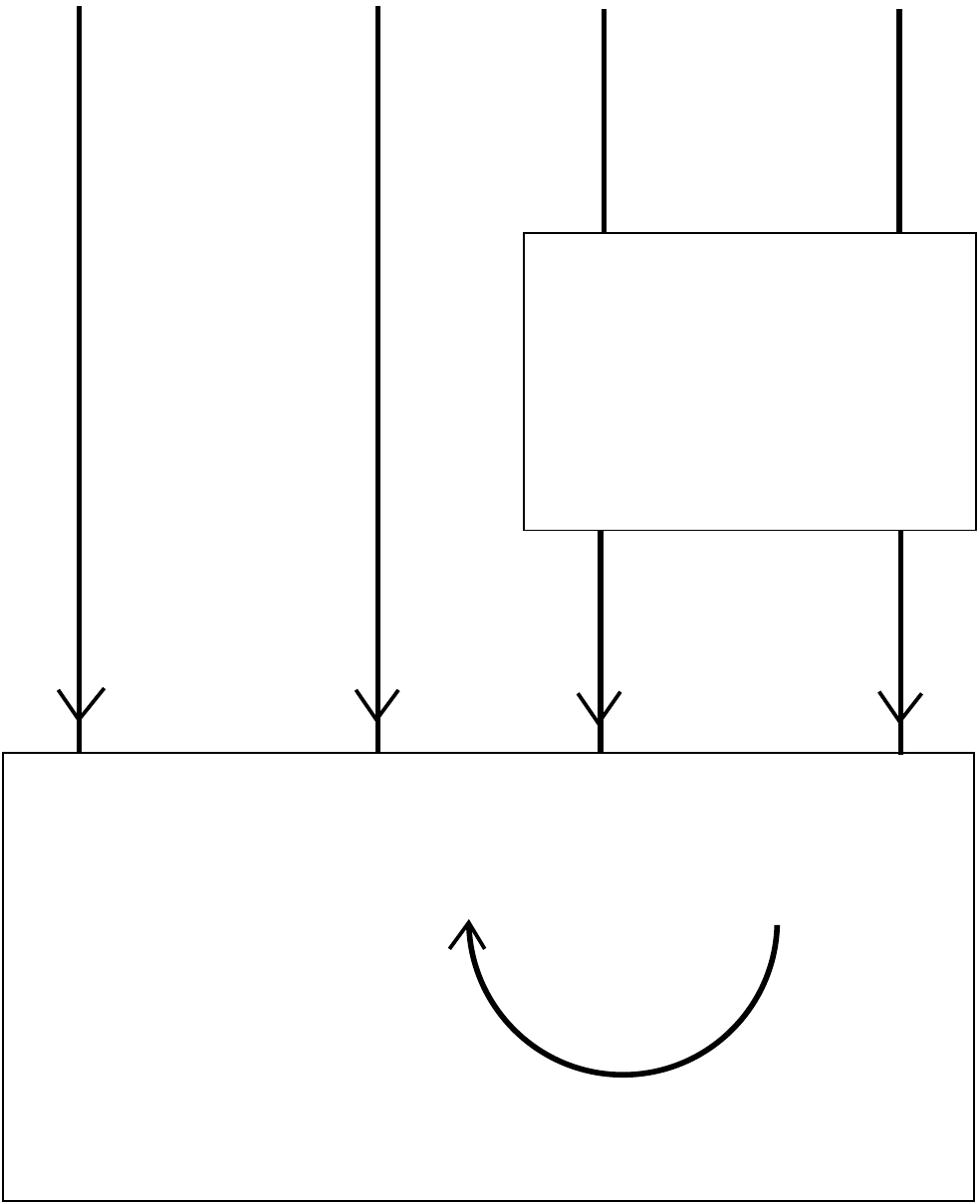}}.
$$

\begin{lemma} \label{lem:a_u_a'_u'}
For $w\in\Mag(\pm)$, we have $a_w=a'_w$, $u_w=u'_w$  and $a_w= (u_w)^{-1}$.
\end{lemma}

\begin{proof}
Using Lemmas \ref{rem:functoriality_cabling} and \ref{rem:functoriality_cabling_bis},
we can deduce $a_w=a'_w$, $u_w=u'_w$ and $a'_w  u_w=1$ from
$$
\labellist
\small\hair 2pt
 \pinlabel{,}  at 509 110
 \pinlabel{$=$}  at 169 110
 \pinlabel{$=$}  at 783 110
  \pinlabel{,}  at 1110 110
 \pinlabel{$=$}  at 1530 110
\endlabellist
\centering
\includegraphics[scale=0.2]{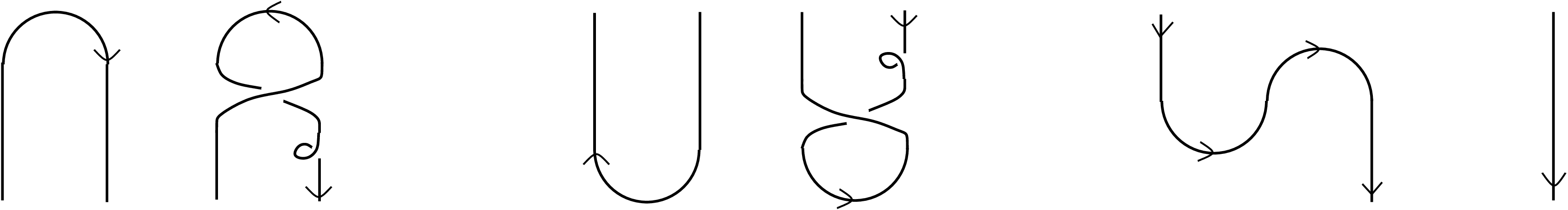}
$$ respectively.  See the proof of \cite[Proposition 6.8(1)]{Ohtsuki}
for the case $a=u=1/2$ and $w=(++)$.
We can easily adapt the arguments given there  to the general case.
\end{proof}

For  $w\in\Mon(\pm)$ and  $\varpi: \{1, \dots, \vert w \vert\} \to \Mag(\pm) $, we obtain
$$
c(w,\varpi):C_\varpi(w)\lto C_\varpi(w)\quad \text{in $\AT$}
$$
from $\id_w:w\to w$ by replacing, for each $i\in\{1, \dots, \vert w \vert\} $,
the $i$-th string with $a_{\varpi(i)}$ if $w_i=+$ and with
$\id_{\varpi(i)^*}$ if $w_i=-$.

\begin{lemma} \label{lem:cabling}
{For a $q$-tangle {$\gamma:w\to w'$} and $\varpi:\pi_0(\gamma) \to  \Mag(\pm)$, we have}
\begin{equation*}
Z C_\varpi(\gamma) = c(w', \varpi_t) \circ C_\varpi Z(\gamma) \circ c(w,\varpi_s)^{-1}.
\end{equation*}
Here, $\varpi_s$ is the composition of $\varpi$ with the map
$\{1,\dots, \vert w\vert \} \to \pi_0(\gamma)$ relating the top
boundary points of $\gamma$ to its connected components, and
$\varpi_t$ is defined similarly.
\end{lemma}

\begin{proof}
This {lemma} is proved by adapting the arguments of \cite[Lemma~4.1]{LM_parallel}, and by using Lemma \ref{lem:a_u_a'_u'}.
\end{proof}

To conclude this section, we emphasize that there are several ``good''
choices of $a$ and $u$ in Theorem \ref{th:Kontsevich}.  The most
common choice is to take $a =u=1/2$. However, for technical
convenience, we set
$$
a= 0, \quad u=1.
$$ Thus, in what follows, the ``cabling anomaly'' $a_w \in\A(\downarrow \stackrel{w}{\cdots} \downarrow) \subset \AT(w,w)$
 assigned to $w\in\Mag(\pm)$ satisfies
\begin{gather}
  \label{e61}
ZC_w\Big(\!\!\!\begin{array}{c}   \figtotext{20}{10}{capleft} \\ {}^{(+-)}  \end{array}\!\!\!\Big)  =
\centre{\labellist
\small\hair 2pt
 \pinlabel{$\stackrel{w^*}{\cdots}$}   at 363 65
 \pinlabel{$\stackrel{w}{\cdots}$}  at 104 60
 \pinlabel{$a_w$}   at 107 184
\endlabellist
\centering
\includegraphics[scale=0.1]{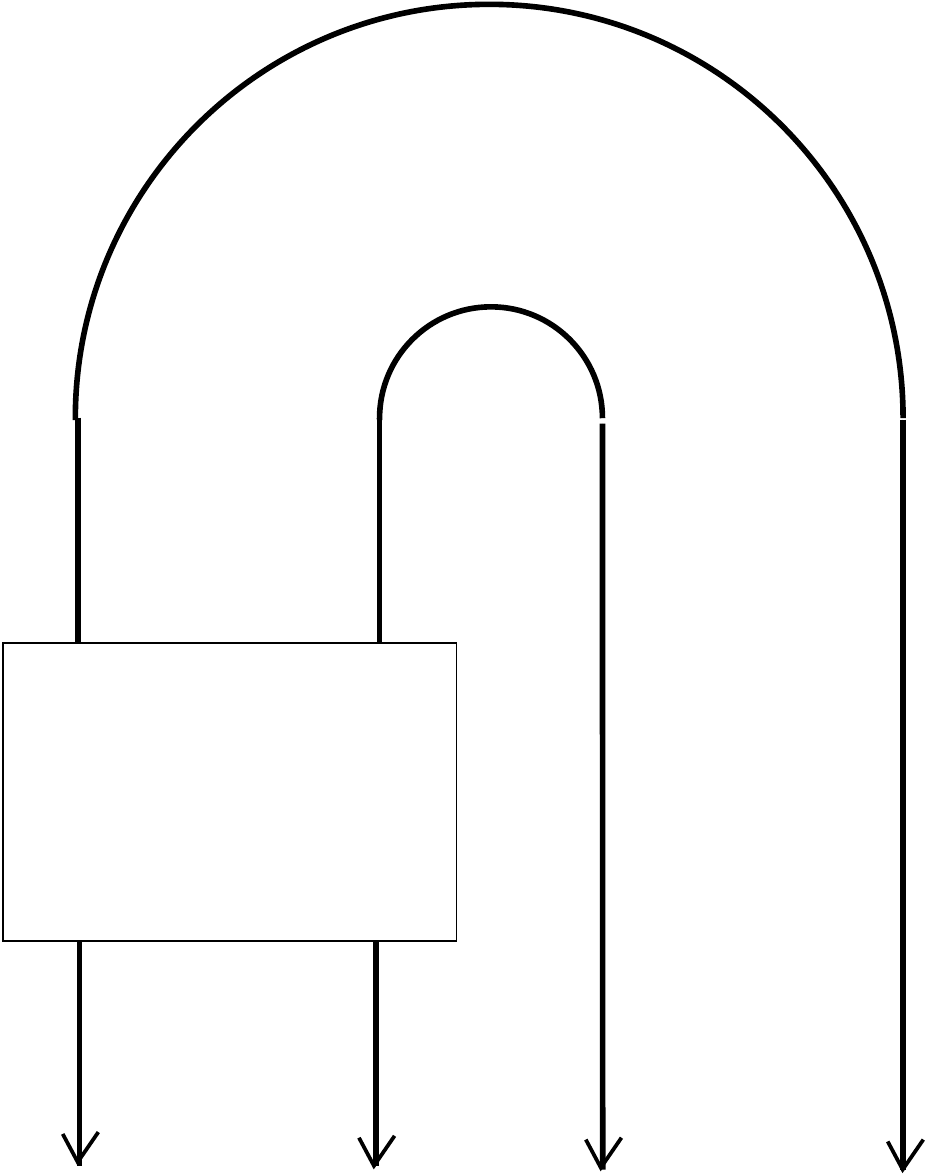}} .
\end{gather}

%
%
\section{The category $\AB$ of Jacobi diagrams in handlebodies}  \label{sec:Jacobi}

In this section, we introduce the linear symmetric strict
monoidal category $\AB$ of Jacobi diagrams in handlebodies.

\subsection{Spaces of colored Jacobi diagrams} \label{sec:colored_Jd}

Here we define the notion of Jacobi diagrams \emph{colored by
elements of a group} \cite{GL,Lieberum}, and define the space
$\mathcal{A}(X,\pi)$ of $\pi$-colored Jacobi diagrams on a
$1$-manifold $X$, where $\pi$ is a group.

Let $S$ be a set, and $D$ a Jacobi diagram on a compact, oriented
$1$-manifold $X$.  An \emph{$S$-coloring} of $D$ consists of an
orientation of each edge of $D$ and an $S$-valued function on a
(possibly empty) finite subset of $(\operatorname{int}X\cup D) \setminus
\operatorname{Vert}(D)$.  In figures, the $S$-valued function is
encoded by ``beads'' colored with elements of $S$. We identify two
$S$-colored Jacobi diagrams $D$ and $D'$ on $X$ if there is a
homeomorphism $(X \cup D,X) \cong (X \cup D',X)$ preserving the
orientations and the connected components of $X$, respecting the
vertex-orientations and compatible with the $S$-colorings.  These
definitions for Jacobi diagrams restrict to chord diagrams.

Now, let $S=\pi$ be a group.  Two $\pi$-colorings of a chord
diagram $D$ on $X$ are said to be \emph{equivalent} if they are
related by a sequence of the following local moves:
\begin{equation} \label{eq:moves_ccd}
\labellist
\scriptsize \hair 2pt
 \pinlabel{$x$} [t] at 26 219
 \pinlabel{$y$} [t] at 68 219
 \pinlabel{$xy$} [t] at 196 219
 \pinlabel{$1$} [t] at 414 220
 \pinlabel{$\leftrightarrow$}   at 126 226
 \pinlabel{$,$}   at 307 226
 \pinlabel{$,$} [l] at 614 226
 \pinlabel{$\leftrightarrow$}   at 127 118
 \pinlabel{$,$}   at 311 118
 \pinlabel{$\leftrightarrow$}   at 128 7
 \pinlabel{$,$}   at 312 13
 \pinlabel{$,$} [l] at 614 117
 \pinlabel{}   at 613 10
 \pinlabel{$\leftrightarrow$}   at 487 225
 \pinlabel{$x$} [t] at 26 113
 \pinlabel{$y$} [t] at 69 113
 \pinlabel{$xy$} [t] at 192 112
 \pinlabel{$1$} [t] at 411 112
 \pinlabel{$\leftrightarrow$}   at 486 118
 \pinlabel{$x$} [t] at 49 5
 \pinlabel{$\overline{x}$} [t] at 192 5
 \pinlabel{$x$} [t] at 378 4
 \pinlabel{$x$} [b] at 570 43
 \pinlabel{$x$} [t] at 572 5
 \pinlabel{$\leftrightarrow$} at 486 10
  \pinlabel{$.$} [l] at 614 10
  \pinlabel{\small $\forall x,y\in\pi$,\quad}  at -80 119
\endlabellist
\centering
\includegraphics[scale=0.39]{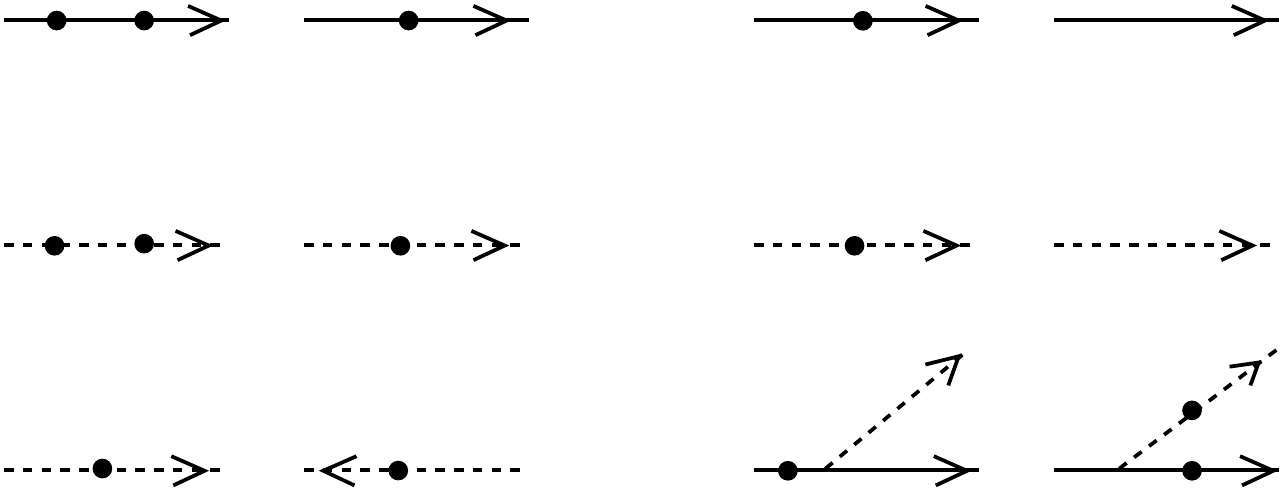}
\end{equation}
Here and in what follows, we use the notation $\overline{x}=x^{-1}$.
(In the fifth relation above, {it is understood that,} if there are several beads on the
reversed edge, then the colors at all the beads on it should be inverted.)

\begin{example}
Here are several equivalent $\pi$-colored chord
diagram on $\uparrow\, \uparrow$, where $x,y\in\pi$.
$$
\labellist
\footnotesize\hair 2pt
 \pinlabel{$=$}   at 141 70
 \pinlabel{$=$}   at 321 70
 \pinlabel{$=$}   at 498 71
 \pinlabel{$x$} [r] at 1 68
 \pinlabel{$\overline{x}$} [b] at 50 109
 \pinlabel{$xy$} [t] at 40 27
 \pinlabel{$y$} [l] at 85 73
 \pinlabel{$x$} [r] at 188 99
 \pinlabel{$x$} [b] at 221 99
 \pinlabel{$\overline{x}$} [b] at 239 108
 \pinlabel{$xy$} [t] at 224 25
 \pinlabel{$y$} [l] at 269 74
 \pinlabel{$x$} [r] at 366 95
 \pinlabel{$1$} [b] at 405 101
 \pinlabel{$xy$} [t] at 400 25
 \pinlabel{$y$} [l] at 448 69
 \pinlabel{$y$} [b] at 583 102
 \pinlabel{$xy$} [t] at 581 25
 \pinlabel{$1$} [r] at 627 23
 \pinlabel{$y$} [l] at 627 126
 \pinlabel{$x$} [r] at 542 100
\endlabellist
\centering
\includegraphics[scale=0.35]{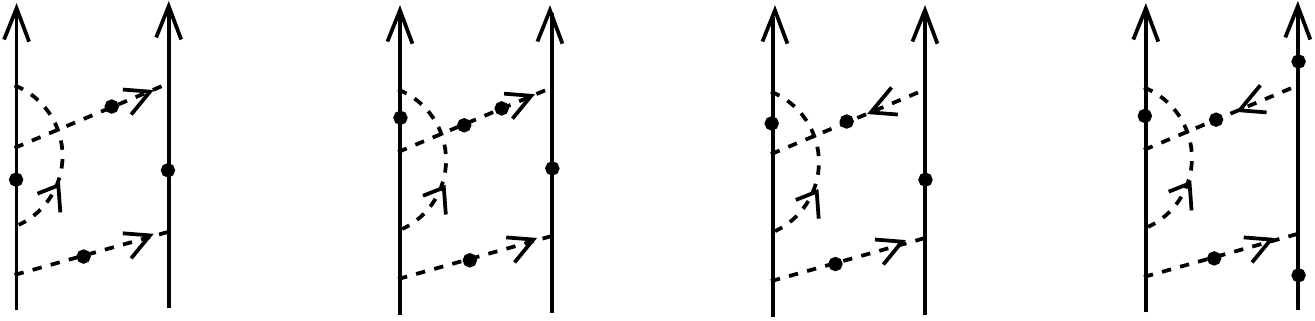}
$$
\end{example}

Similarly, two $\pi$-colorings of a Jacobi diagram $D$ on $X$ are
said to be \emph{equivalent} if they are related by a  sequence of the local moves in \eqref{eq:moves_ccd} and
$$
\qquad \qquad
\labellist
\normalsize \hair 2pt
 \pinlabel{$\leftrightarrow$}   at 191 55
 \pinlabel{\small $x$} [b] at 49 62
 \pinlabel{\small $x$} [b] at 349 84
 \pinlabel{\small $x$} [t] at 350 32
 \pinlabel{ $\forall x\in\pi$,} [r] at -40 56
  \pinlabel{$.$} [l] at 399 58
\endlabellist
\centering
\includegraphics[scale=0.3]{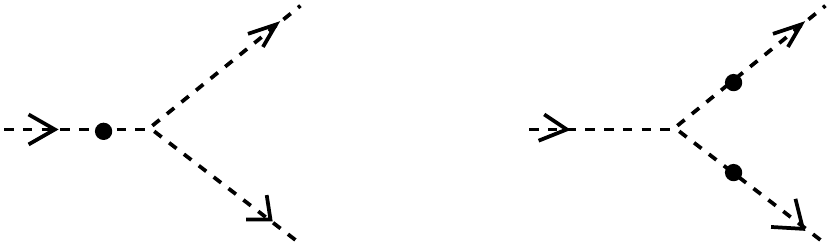}
$$ Thus, $\pi$-colored Jacobi diagrams generalize $\pi$-colored chord
diagrams.  Here is  a topological interpretation of $\pi$-colorings.

\begin{lemma} \label{lem:colorings}
Let $D$ be a Jacobi diagram on $X$ with no closed component.
Then there is a bijection between the set of equivalence classes
of $\pi$-colorings of $D$ and $\Hom\big(\pi_1((X\cup D)/\partial X,\{\partial X\}), \pi\big)$.
\end{lemma}

\begin{cor}
  \label{r15}
  Let $D$ and $X$ be as in Lemma \ref{lem:colorings}.  Let
  $\pi = \pi_1(M,\star)$ for a {pointed} space $(M,\star)$, and  assume
  that $\partial X$ is embedded in a contractible neighborhood of
  $\star$ in $M$.  Then the equivalence classes of
  $\pi$-colorings of $D$ correspond bijectively to the homotopy
  classes (rel $\partial X$) of continuous maps $(X \cup D, \partial
  X) \to (M,\partial X)$.
\end{cor}

\begin{proof}[Proof of Lemma \ref{lem:colorings}]
Let $c$ be a $\pi$-coloring of $D$, and let $\alpha$ be a loop in
$(X\cup D)/\partial X$ based at $\{\partial X\}$.  Let
$\varphi_c(\alpha) \in\pi$ be the product of the contributions of all
the consecutive beads along $\alpha$, where each bead contributes
either by its color or its inverse depending on compatibility of
$\alpha$ with the orientation at the bead.  This clearly defines a
homomorphism $\varphi_c: \pi_1\big((X\cup D)/\partial X,\{\partial
X\}\big) \to \pi$, depending only on the equivalence class of $c$.
Thus, we obtain a map $\{c\} \mapsto \varphi_c$, from the set of
equivalence classes of $\pi$-colorings of $D$ to
$\operatorname{Hom}\big(\pi_1\big((X\cup D)/\partial X,\{\partial
X\}\big), \pi\big)$.  One can construct the inverse map by using a
maximal tree of the graph $(X\cup D)/\partial X$; see
\cite[Lemma~4.3]{GL} for a very similar result.
\end{proof}

A $\pi$-colored Jacobi diagram $D$ on $X$ is said to be
\emph{restricted} if $D$ has no bead (but there may be beads on $X$).
Two restricted $\pi$-colorings of a Jacobi diagram $D$ on $X$ are said
to be \emph{equivalent} if they are related by a sequence of the first
two moves in \eqref{eq:moves_ccd} and
$$
\qquad
\labellist
\small\hair 0pt
 \pinlabel{ \quad $\forall x\in\pi$,} [r] at -40 100
 \pinlabel{a  component of $D$}  at 210 172
 \pinlabel{$\cdots$} at 300 23
 \pinlabel{$\leftrightarrow$}  at 540 127
 \pinlabel{a  component of $D$}  at 859 179
 \pinlabel{$x$} [br] at 668 48
 \pinlabel{$\overline{x}$} [bl] at 740 48
 \pinlabel{$x$} [br] at 809 48
 \pinlabel{$\overline{x}$} [bl] at  885 48
 \pinlabel{$x$} [br] at 990 45
 \pinlabel{$\overline{x}$} [bl] at 1075 48
 \pinlabel{$\cdots$}  at 941 27
\endlabellist
\centering
\includegraphics[scale=0.22]{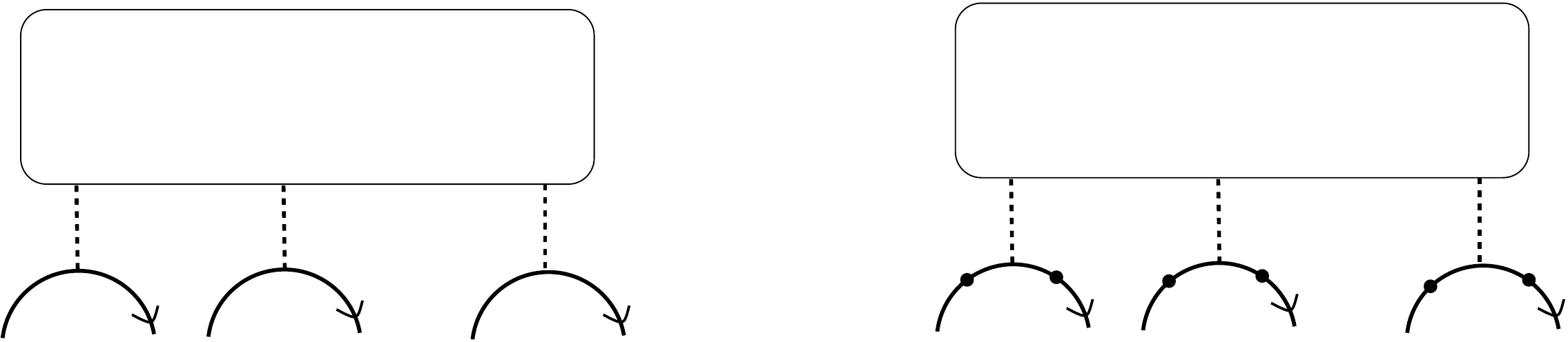}
$$ The above figure shows all the univalent vertices (and their
neighborhoods in $X$) of the same connected component of $D$.
For instance, if $D$ is a chord diagram, then there are exactly two such vertices.

Let $\A^{\operatorname{ch}}(X,\pi)$
(resp.\ $\A^{\operatorname{Jac}}(X,\pi)$) denote the vector space
generated by equivalence classes of $\pi$-colored chord (resp.\ Jacobi)
diagrams on $X$, modulo the 4T (resp.\ STU) relation.  There are also
``restricted'' versions $\A^{\operatorname{ch,r}}(X,\pi)$ and
$\A^{\operatorname{Jac,r}}(X,\pi)$ of $\A^{\operatorname{ch}}(X,\pi)$
and $\A^{\operatorname{Jac}}(X,\pi)$, respectively.  We have a
commutative diagram of canonical maps:
\begin{gather}
  \label{e26}
  \vcenter{\xymatrix{
\A^{\operatorname{ch,r}}(X,\pi)  \ar[r]^-{u^{\operatorname{ch}}} \ar[d]_-{\phi^{\operatorname{r}}} & \A^{\operatorname{ch}}(X,\pi)  \ar[d]^-{\phi} \\
\A^{\operatorname{Jac,r}}(X,\pi) \ar[r]_-{u^{\operatorname{Jac}}}  & \A^{\operatorname{Jac}}(X,\pi).
}}
\end{gather}

The special case $\pi=\{1\}$ of Theorem \ref{th:chord_Jac} below is
due to Bar-Natan \cite[Theorem~6]{BN2}.  The general case seems to be
new.

\begin{theorem}  \label{th:chord_Jac}
All the maps in \eqref{e26} are isomorphisms. Furthermore, the AS
and IHX relations hold in $\A^{\operatorname{Jac,r}}(X,\pi)$ and $\A^{\operatorname{Jac}}(X,\pi)$.
\end{theorem}

\begin{proof}
By the STU relation, $\phi^{\operatorname{r}}$ is
surjective.  The map $\phi$ is also surjective by the same reason and
the following observation: {each} bead in the neighborhood of a univalent
vertex of a $\pi$-colored Jacobi diagram on $X$ can be displaced from
$D$ using the last move of \eqref{eq:moves_ccd} (without changing the
equivalence class of the $\pi$-coloring).  Therefore, it suffices to
prove that
\begin{itemize}
\item[(i)] $\phi$ is injective,
\item[(ii)] $u^{\operatorname{ch}}$ is an isomorphism,
\item[(iii)] the AS and IHX relations hold in $\A^{\operatorname{Jac}}(X,\pi)$.
\end{itemize}

The AS and IHX relations in $\A^{\operatorname{Jac}}(X,\pi)$ reduce
to the STU relation by using the above observation and the arguments of the last two paragraphs in the proof of
\cite[Theorem 6]{BN2}. This proves (iii).

To prove (ii) we construct an inverse to $u^{\operatorname{ch}}$. Applying the operation\\[0.2cm]
$$
\labellist
\scriptsize \hair 2pt
 \pinlabel{$x_1$} [b] at 98 120
 \pinlabel{$x_2$} [b] at 119 120
 \pinlabel{$\cdots$} [b] at 145 120
 \pinlabel{$x_r$} [b] at 169 120
 \pinlabel{$\overline{x}$} [br] at 580 35
 \pinlabel{$x:= \prod_{i=1}^r x_i$} [bl] at 652 35
 \pinlabel{\small $\longmapsto$}  at 346 65
\endlabellist
\centering
\includegraphics[scale=0.4]{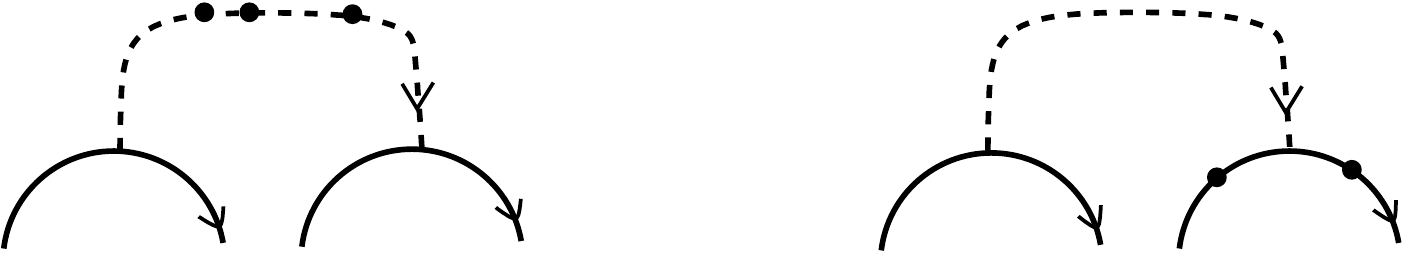}
$$
to all the chords transforms each $\pi$-colored chord diagram on $X$ into a restricted one.
It is easy to check that this operation maps equivalent $\pi$-colorings to equivalent $\pi$-colorings and
defines an inverse to $u^{\operatorname{ch}}$.

To prove (i), we partly follow the proof of \cite[Theorem~6]{BN2}.
 Let $Y$ be a compact, oriented $1$-manifold.  Let
 $\mathcal{D}^{\operatorname{Jac}}(Y,\pi)$ denote the set of
 equivalence classes of $\pi$-colored Jacobi diagrams on $Y$.  For
 $k\ge0$, let
 $\mathcal{D}^{\operatorname{Jac}}_k(Y,\pi)\subset\mathcal{D}^{\operatorname{Jac}}(Y,\pi)$
 consist of diagrams with exactly $k$ trivalent vertices.  Let
 $\psi_0: \mathcal{D}^{\operatorname{Jac}}_0(Y,\pi) \to
 \A^{\operatorname{ch}}(Y,\pi)$ be the canonical map.

\begin{claim*}
There are maps $\psi_k:\mathcal{D}^{\operatorname{Jac}}_k(Y,\pi) \to
\A^{\operatorname{ch}}(Y,\pi)$ for $k\ge1$ such that we have
$$
\psi_k(D^S) =  \psi_{k-1}(D^T_i) - \psi_{k-1}(D^U_i)
$$ for $k\ge1$ and $D^S \in\mathcal{D}^{\operatorname{Jac}}_k(Y,\pi)$, where $i$ denotes a
univalent vertex of $D^S$ adjacent to a trivalent vertex $v_i$ and
where $D^T_i,D^U_i \in\mathcal{D}^{\operatorname{Jac}}_{k-1}(Y,\pi)$
differ from $D^S$ around $i$ as shown in the STU relation
\eqref{eq:STU}.
\end{claim*}

Applying this claim to $Y=X$, we obtain a left inverse
$\psi:\A^{\operatorname{Jac}}(X,\pi) \to
\A^{\operatorname{ch}}(X,\pi)$ to $\phi$.  This proves (i) and
concludes the proof of Theorem \ref{th:chord_Jac}.
\end{proof}

\begin{proof}[Proof of Claim]
By the 4T relation, $\psi_1$ is well defined from $\psi_0$.  Let
$k>1$ and suppose $\psi_1,\dots,\psi_{k-1}$ have been defined for
all compact, oriented $1$-manifolds $Y$.  To have $\psi_k$
well defined, we need to check
\begin{equation} \label{eq:i_and_j}
\psi_{k-1}(D^T_i) - \psi_{k-1}(D^U_i) = \psi_{k-1}(D^T_j) - \psi_{k-1}(D^U_j)
\end{equation}
for {all} $D^S \in\mathcal{D}^{\operatorname{Jac}}_k(Y,\pi)$ and {all}
univalent vertices $i$ and $j$ of $D^S$ adjacent to some trivalent
vertices $v_i$ and $v_j$, respectively.  If $v_i \neq v_j$, then we
can apply the argument in the second paragraph of the proof of
\cite[Theorem~6]{BN2}.  If $v_i=v_j$, then the arguments provided in
\cite{BN2} for this situation do not fully apply when $\pi\not=\{1\}$,
because of the ``exceptional case'' alluded to in the third paragraph
of the proof of \cite[Theorem~6]{BN2}.  Thus, we need a different
proof.

First, observe that the maps $\psi_0,\dots,\psi_{k-1}$ defined so far
have the following properties: for every oriented $1$-manifold
$Y'\!\uparrow$ with a distinguished component $\uparrow$, the diagrams
\begin{equation} \label{eq:2_properties}
\centre{\xymatrix{
 \K
 \mathcal{D}^{\operatorname{Jac}}_i(Y'\!\uparrow,\pi) \ar[r]^-{\psi_i} \ar[d]_-\Delta & \A^{\operatorname{ch}}(Y'\!\uparrow ,\pi) \ar[d]_-\Delta \\
 \K
 \mathcal{D}^{\operatorname{Jac}}_i(Y'\!\uparrow \uparrow,\pi) \ar[r]_{\psi_i} & \A^{\operatorname{ch}}(Y'\!\uparrow\uparrow,\pi)
 }}
 \ \hbox{and} \ \centre{\xymatrix{
 \K
 \mathcal{D}^{\operatorname{Jac}}_i(Y'\!\uparrow,\pi) \ar[r]^-{\psi_i} \ar[d]_-S & \A^{\operatorname{ch}}(Y'\!\uparrow ,\pi) \ar[d]_-S \\
 \K
 \mathcal{D}^{\operatorname{Jac}}_i(Y'\downarrow,\pi) \ar[r]_{\psi_i} & \A^{\operatorname{ch}}(Y'\downarrow,\pi)
 }}
\end{equation}
commute for  $i\in\{0,\dots,k-1\}$.  (Here the doubling
operations $\Delta$ and the orientation-reversal operations $S$ for
colored chord/Jacobi diagrams are defined in the same way as for
uncolored chord/Jacobi diagrams, except that beads of a duplicated
component should be repeated on each new component, and beads of a
reversed component should be transformed into their inverses.)  Next, we
draw $D^S$ as follows:\\[0.2cm]
$$
D^S=\centre{\labellist
\scriptsize\hair 2pt
 \pinlabel{$\cdots$}  at 218 273
 \pinlabel{$i$} [bl] at 107 95
 \pinlabel{$j$} [br] at 227 95
 \pinlabel{$\overbrace{\hphantom{aaaaaaaaaaaaaaaaaaaaaa}}^{X'}$} [b] at 169 284
  \pinlabel{$X_i$} [r] at 4 20
 \pinlabel{$X_j$} [l] at 329 2
\endlabellist
\centering
\includegraphics[scale=0.3]{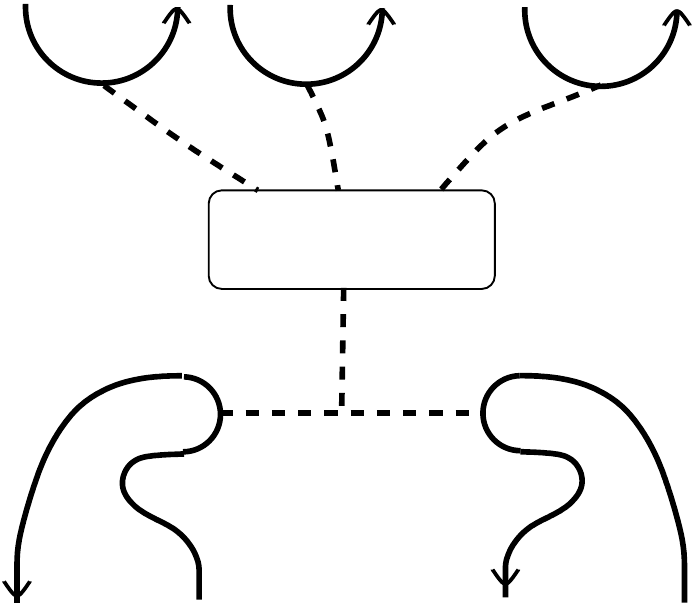}}
$$ Here the arcs $X_i,X_j$ are neighborhoods in $X$ of the vertices
$i,j$, and $X' \subset X$ is a neighborhood of the remaining univalent
vertices of $D^S$.  From this local picture of $D^S$, we define
$$
R=\centre{\labellist
\scriptsize\hair 2pt
 \pinlabel{$\cdots$}  at 210 190
\endlabellist
\centering
\includegraphics[scale=0.25]{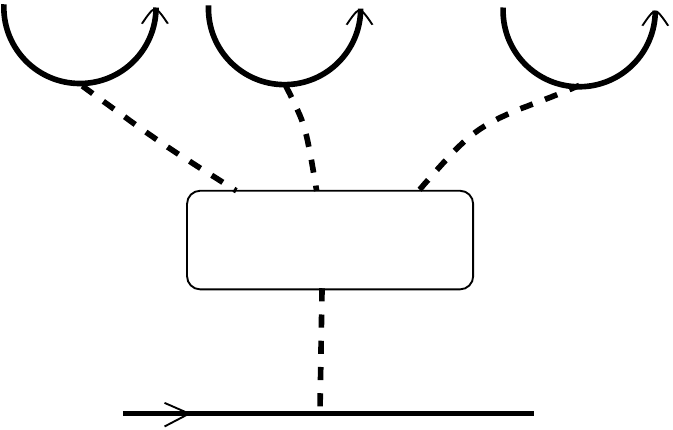}} \in\mathcal{D}_{k-1}^{\operatorname{Jac}}(X'\!\rightarrow,\pi)
$$
and expand
$$
\psi_{k-1}(R)  = \sum_l \varepsilon_l \cdot \centre{\labellist
\small \hair 2pt
 \pinlabel{$\cdots$}  at 210 190
  \pinlabel{$l$}  at 160 90
\endlabellist
\centering
\includegraphics[scale=0.25]{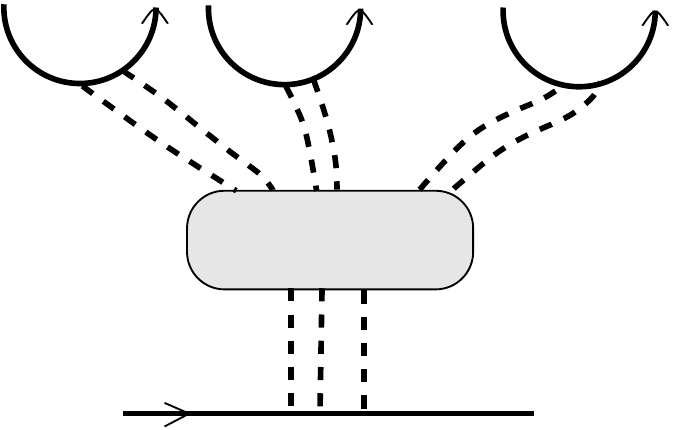}} \in\A^{\operatorname{ch}}(X'\!\rightarrow,\pi).
$$
Since
$$
D^T_i - D^U_i =
\centre{\labellist
\scriptsize\hair 2pt
 \pinlabel{$\cdots$} at 226 291
\endlabellist
\centering
\includegraphics[scale=0.22]{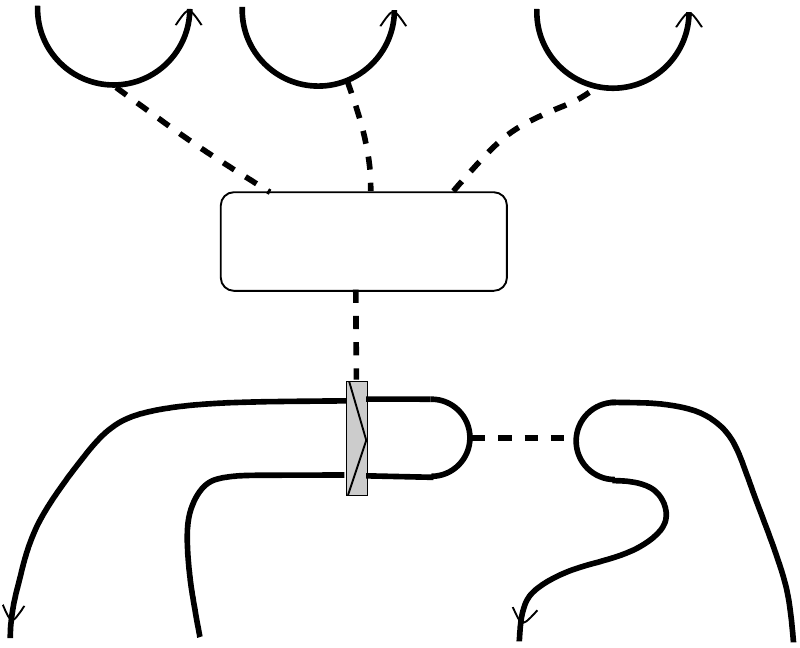}}
\quad \hbox{and} \qquad
D^T_j - D^U_j = \centre{\labellist
\scriptsize \hair 2pt
 \pinlabel{$\cdots$} at 271 291
\endlabellist
\centering
\includegraphics[scale=0.22]{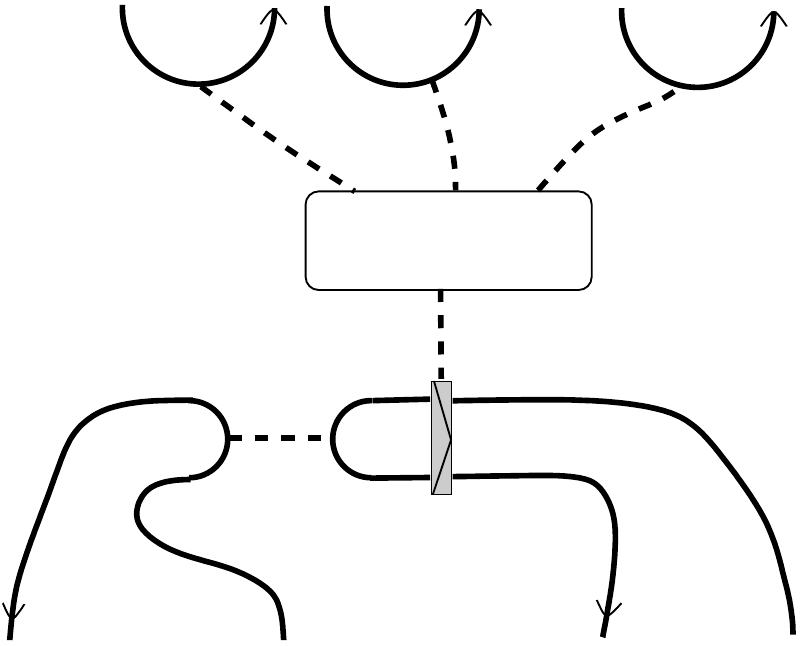}}
$$
we deduce from \eqref{eq:2_properties} that
$$
\psi_{k-1}(D^T_i)-\psi_{k-1}(D^U_i) =  \sum_l \varepsilon_l \cdot \!\!\!\!\!
\centre{\labellist
\small\hair 2pt
 \pinlabel{$l$}  at 171 194
 \pinlabel{$\cdots$} at 226 291
\endlabellist
\centering
\includegraphics[scale=0.25]{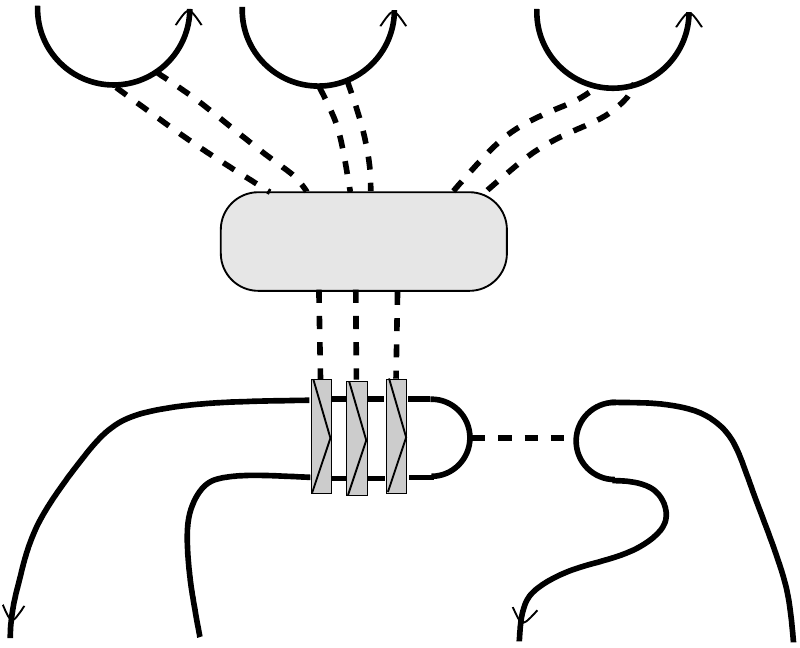}}
\in\A^{\operatorname{ch}}(X,\pi)
$$
and
$$
\psi_{k-1}(D^T_j)-\psi_{k-1}(D^U_j) =  \sum_l \varepsilon_l \cdot \!\!\!\!\!
\centre{\labellist
\small\hair 2pt
 \pinlabel{$l$}  at 210 194
 \pinlabel{$\cdots$} at 271 291
\endlabellist
\centering
\includegraphics[scale=0.25]{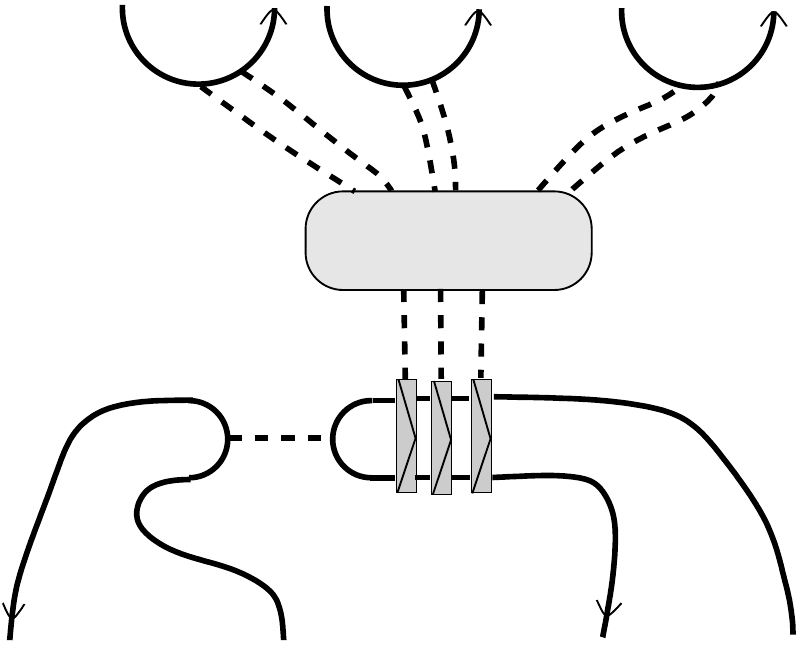}}
\in\A^{\operatorname{ch}}(X,\pi).
$$
Thus, the identity \eqref{eq:i_and_j} follows from the local relation
$$
\forall x \in\pi, \qquad
\centre{\labellist
\small\hair 2pt
 \pinlabel{$x$} [r] at 74 106
 \pinlabel{$x$}  [l] at 413 106
 \pinlabel{$=$}  at 242 80
\endlabellist
\centering
\includegraphics[scale=0.3]{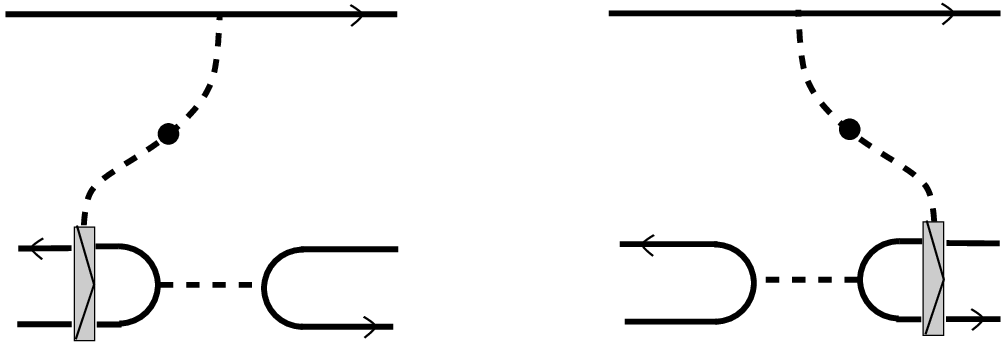}}
$$ in spaces of $\pi$-colored chord diagrams, which is equivalent to
the 4T relation.
\end{proof}

In what follows, let $\A(X,\pi)$ denote the isomorphic spaces
$$\A^{\operatorname{ch,r}}(X,\pi)\cong\A^{\operatorname{ch}}(X,\pi)\cong\A^{\operatorname{Jac,r}}(X,\pi)\cong\A^{\operatorname{Jac}}(X,\pi).$$
  For instance, if $\pi=\{1\}$, then we have $\A(X)\cong\A(X,\{1\})$.
In general, $\A(X)$ embeds into $\A(X,\pi)$ by the following lemma.

\begin{lemma}  \label{lem:usual_to_colored}
If $X$ has no closed component,
then the canonical map from $\A(X)$ to $\A(X,\pi)$ is injective.
\end{lemma}

\begin{proof}
For every Jacobi diagram $D$ on $X$ with $\pi$-coloring $c$, define $p(D,c)\in\A(X)$ by
\begin{equation} \label{eq:p}
p(D,c) = \left\{\begin{array}{ll}
D \hbox{ with $c$ deleted} &  \hbox{if $ \varphi_c$ is trivial,} \\
0 & \hbox{otherwise,}
\end{array}\right.
\end{equation}
where $\varphi_c:  \pi_1\big((X\cup D)/\partial X,\{\partial X\}\big) \to \pi$ is the  homomorphism corresponding to $c$ by
Lemma~\ref{lem:colorings}.  Observe that, for {all} $\pi$-colored
Jacobi diagrams $(D,c)$ and $(D',c')$ on $X$ involved in an STU relation, there is a homotopy equivalence $h: X\cup D {\overset{\simeq}{\lto}}
X\cup D'$ rel $\partial X$ such that $\varphi_{c'} \circ h_* =
\varphi_c$.  Hence \eqref{eq:p} induces a linear map $p: \A(X,\pi)
\to \A(X)$. Clearly, $p \circ i =\id_{\A(X)}$, where
$i: \A(X) \to \A(X,\pi)$ is the canonical map.
\end{proof}

\begin{remark}
  There are analogs of Lemmas~\ref{lem:colorings} and
    \ref{lem:usual_to_colored} for compact, oriented $1$-manifolds
    with closed components.  Moreover, we can extend Lemma
    \ref{lem:usual_to_colored} as follows: if $X$ has no closed
    components, then the map $\A(X,\pi')\to\A(X,\pi)$ induced by an
    injective group homomorphism $\pi'\to\pi$ is injective.  We will
    not need these generalizations in what follows.
\end{remark}

\subsection{The category $\AB$ of Jacobi diagrams in handlebodies} \label{sec:category-ab-jacobi}

Now we introduce the linear \emph{category $\AB$ of Jacobi
diagrams in handlebodies}.  Set $\Ob(\AB)=\NZ$.

For $m\ge0$, let $\free{m}=F(x_1,\dots,x_m)$ be the free group on
$\{x_1,\dots,x_m\}$.  We identify $\free{m}$ with $\pi_1(V_m,\ell)$
(see Section \ref{sec:categories}).  Here $x_1,\dots,x_m$ are
represented by the ``stretched cores'' $A_1,\dots,A_m$ of the handles
of $V_m$.  For $n\ge0$, let
\begin{gather*}
  \Xn = \capn
\end{gather*}
be an oriented $1$-manifold consisting of $n$ arc components.

For $m,n\ge0$, set
$$
\AB(m,n) =\A(\Xn,\free{m}),
$$ which is generated by $\free{m}$-colored Jacobi diagrams on $\Xn$.  We
will call them \emph{$(m,n)$-Jacobi diagrams} for brevity.
Using Corollary \ref{r15},
we may regard {an} $(m,n)$-Jacobi diagram as a homotopy class rel $\partial \Xn$ of maps
\begin{equation} \label{eq:cJd_as_hcm}
\Xn\cup D \longrightarrow V_m.
\end{equation}
Here we assume that the $2n$ boundary points of $\Xn$ are uniformly
distributed along the line $\ell$.  Since $V_m$ deformation-retracts
onto a square with $m$ handles, we can present an $(m,n)$-Jacobi
diagram $D$ by a projection diagram of the corresponding homotopy
class of maps \eqref{eq:cJd_as_hcm}.

\begin{example}
  \label{r23}
Here are a $(2,3)$-Jacobi diagram and its projection diagram in
the square with handles:
\begin{gather}
  \label{e59}
\centre{\labellist
\scriptsize\hair 2pt
 \pinlabel{\small $\leadsto$}  at 106 24
 \pinlabel{$x_2$}  at 55 16
 \pinlabel{$x_1^2$}  at 42 8
 \pinlabel{$x_2$}  at 18 10
 \pinlabel{$\overline{x_1}$}  at 30 51
 \pinlabel{$1$} [t] at 27 2
 \pinlabel{$2$} [t] at 59 1
 \pinlabel{$3$} [t] at 92 1
\endlabellist
\centering
\includegraphics[scale=1.3]{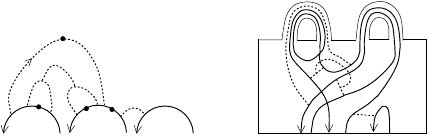}}
\end{gather}
\end{example}

We will use the following convention for presenting the morphisms in
$\AB$.  A~\emph{square presentation} of a restricted $(m,n)$-Jacobi
diagram $D$ is a projection diagram of $D$ in the square with $m$
handles, such that the dashed part of $D$ does not appear in the
handles.  Thus a square presentation of $D$ consists of  words
$w_1,\dots,w_n \in\Mon(\pm)$ and a Jacobi diagram
$$
S :w_1 w_1^*\cdots  w_mw_m^*\lto (+-)^n\quad \text{in $\AT$}
$$
 such that
$$
\labellist
\scriptsize \hair 1pt
 \pinlabel{$\cdots$}  at 363 26
 \pinlabel{$\cdots$}  at 360 306
 \pinlabel{$x_1$} [b] at 181 296
 \pinlabel{$x_1$} [b] at 181 367
 \pinlabel{$x_m$} [b] at 541 294
 \pinlabel{$x_m$} [b] at 542 366
 \pinlabel{$\stackrel{w_1}{\cdots}$} [b] at 110 260
 \pinlabel{$\stackrel{w_m}{\cdots}$} [b] at 473 260
 \pinlabel{\normalsize $S$}  at 362 126
  \pinlabel{\normalsize $\cdot$}  at 762 126
 \pinlabel{\normalsize $D=$} [r] at -40 126
 \pinlabel{$+$} [t] at 144 3
 \pinlabel{$-$} [t] at 253 2
 \pinlabel{$+$} [t] at 469 2
 \pinlabel{$-$} [t] at 576 2
\endlabellist
\centering
\includegraphics[scale=0.23]{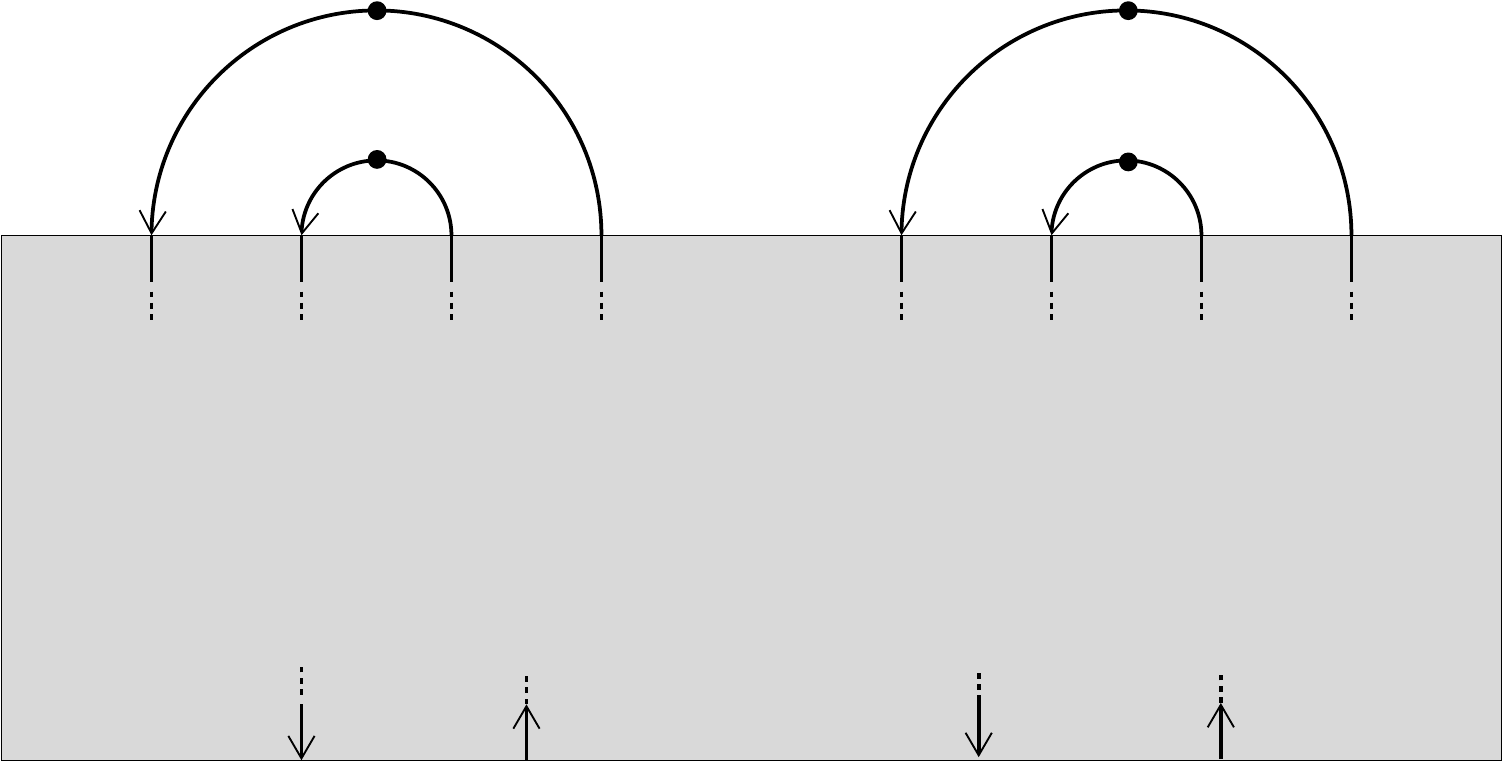}
$$
\vspace{0.2cm}

\noindent
In what follows, we write $\dbl(w) =  w w^*\in\Mon(\pm)$ for
 $w\in\Mon(\pm)$.  Since $\AB(m,n)$ is spanned by restricted
$(m,n)$-Jacobi diagrams, we can regard its elements as
linear combinations of square presentations.

\begin{example}
Here is a restricted $(2,3)$-Jacobi diagram $D$, together with a
square presentation $S$ such that $w_1=w_2=++$:
$$
\centre{
\labellist
\scriptsize\hair 2pt
 \pinlabel{\small $\leadsto$}  at 106 24
 \pinlabel{\small {$D$}}  at 5 26
  \pinlabel{\small {$S$}}  at 135 26
 \pinlabel{$x_2$}  at 55 16
 \pinlabel{$x_1^2$}  at 42 8
 \pinlabel{$x_2$}  at 18 10
 \pinlabel{$1$} [t] at 27 2
 \pinlabel{$2$} [t] at 59 1
 \pinlabel{$3$} [t] at 92 1
\endlabellist
\centering
\includegraphics[scale=1.1]{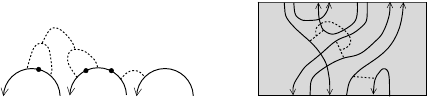}}
\hbox{$:\dbl(w_1)\dbl(w_2)\lto (+-)^3$.}
$$
\end{example}

Now we define the composition in $\AB$.  We compose an
$(n,p)$-Jacobi diagram $D'$ with an $(m,n)$-Jacobi diagram $D$ as
follows.  First, we may assume that each bead of $D'$ is colored by
$x_i^{\pm1}$ for some $i$, by using the moves in
\eqref{eq:moves_ccd}.  For each $j\in\{1,\dots,n\}$, let $k_j$ be
the number of beads colored by $x_j^{\pm 1}$ in $D'$ and number them
from $1$ to $k_j$ in an arbitrary way.  This defines a word $\kappa(j)
\in\Mon(\pm)$ of length $k_j$ by assigning a letter $+$ to each
$x_j$-colored bead and a letter $-$ to each $x_j^{-1}$-colored
bead. Let
$$
C_\kappa (D) \in\A(\Xx{k_1+\dots+k_n},\free{m})
$$ be the linear combination of $(m,k_1+\cdots+k_n)$-Jacobi diagrams obtained from
$D$ by $\kappa$-cabling, i.e., by repeated applications of the
deleting operation $\epsilon$, the doubling operation $\Delta$
and the orientation-reversal operation $S$.  By using the
correspondence between the beads of $D'$ and the solid components of
$C_\kappa (D)$ induced by their numberings, we can identify some
local neighborhoods of the former with the latter in an
orientation-preserving way.  Thus, by ``gluing'' $C_\kappa (D)$ to
$D'$ accordingly, we obtain a linear combination of $(m,p)$-Jacobi diagrams
$$
D' \mathop{\tilde\circ}  D \in\A{(\Xx{p},\free{m})} = \AB(m,p).
$$
Clearly, $D' \mathop{\tilde\circ} D$
depends only on the equivalence
class of $D$, but not on the numbering of the beads of $D'$.  By the
STU relation, $D' \mathop{\tilde\circ} D$ depends only on the equivalence class of $D'$.

\begin{example}
  \label{r26}
We can describe the operation $\tilde{\circ}$ in terms of projection diagrams in squares with handles,
using the box notation recalled in Example \ref{ex:box}.  For
instance, let $m=n=p=2$ and
$$
\labellist
\scriptsize \hair 2pt
 \pinlabel{\small $D':=$} [r] at 0 41
 \pinlabel{$x_1x_2$} [b] at 65 72
 \pinlabel{$\overline{x_1}$} [b] at 228 69
 \pinlabel{$1$} [l] at 120 13
 \pinlabel{$2$} [l] at 283 11
 \pinlabel{,}  at 360 35
 \pinlabel{\small $D:=$} [r] at 498 42
 \pinlabel{$1$} [l] at 626 7
 \pinlabel{$2$} [l] at 794 9
 \pinlabel{$x_1$} [b] at 572 68
 \pinlabel{$x_1$} [b] at 227 132
\endlabellist
\centering
\includegraphics[scale=0.22]{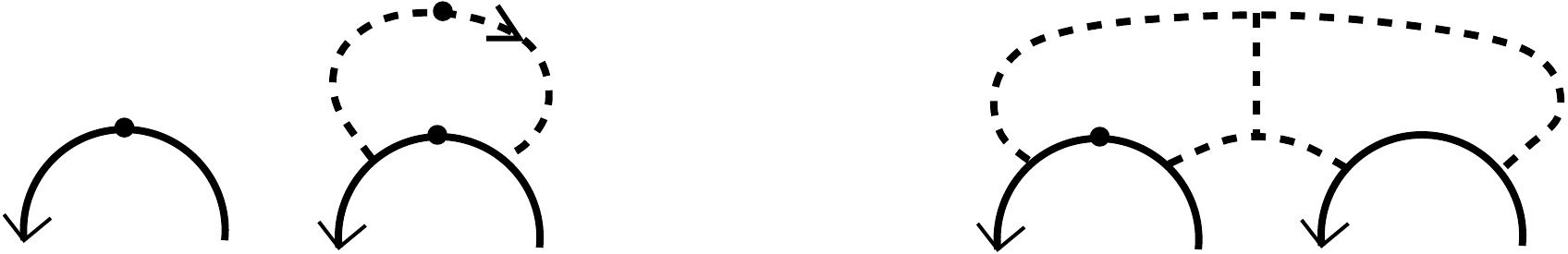}
$$
with the projection diagrams
\begin{gather*}
\labellist
\small\hair 2pt
 \pinlabel{$D'=$} [r] at 0 28
 \pinlabel{,} at 101 25
 \pinlabel{$D=$} [r] at 143 28
 \pinlabel{{.}} at 245 25
\endlabellist
\centering
\includegraphics[scale=0.9]{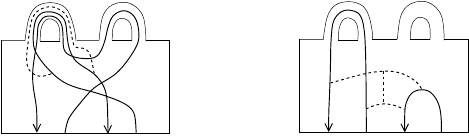}
\end{gather*}
Then, setting $x=x_1$, we have
$$
\labellist
\tiny\hair 1pt
 \pinlabel{$-$}   at 198 332
 \pinlabel{$+$}   at 407 336
 \pinlabel{$-$} [r] at 8 196
  \pinlabel{\small $D' \mathop{\tilde\circ} D= $} [r] at -20 332
 \pinlabel{$+$}   at 218 199
 \pinlabel{$-$}   at 424 199
 \pinlabel{$+$} [r] at 4 37
 \pinlabel{$-$}   at 211 36
 \pinlabel{$+$}   at 427 35
 \pinlabel{$x\,$} [b] at 24 326
 \pinlabel{$x$} [b] at 138 365
 \pinlabel{$\overline{x}$} [b] at 142 333
 \pinlabel{$x$} [b]   at 143 221
 \pinlabel{$\overline{x}$} [b] at 151 194
 \pinlabel{$x\,$} [b]  at 33 186
 \pinlabel{$x\,$} [b]   at 30 31
 \pinlabel{$x$} [b]   at 143 68
 \pinlabel{$\overline{x}$} [b] at 147 40
 \pinlabel{$x\,$} [b]  at 234 326
 \pinlabel{$x$} [b]   at 348 365
 \pinlabel{$\overline{x}$} [b] at 352 334
 \pinlabel{$x\,$} [b]  at 241 187
 \pinlabel{$x$} [b]   at 356 223
 \pinlabel{$\overline{x}$} [b] at 359 194
 \pinlabel{$x\,$} [b]   at 241 33
 \pinlabel{$x$} [b] at 356 70
 \pinlabel{$\overline{x}$} [b] at 359 41
 \pinlabel{$x\,$} [b]  at 457 35
 \pinlabel{$\overline{x}$} [b] at 575 44
 \pinlabel{$x$} [b]  at 571 73
 \pinlabel{$x\,$} [b]   at 454 188
 \pinlabel{$\overline{x}$} [b] at 573 196
 \pinlabel{$x$} [b]  at 568 226
 \pinlabel{$x\,$} [b]   at 448 328
 \pinlabel{$\overline{x}$} [b] at 564 336
 \pinlabel{$x$} [b]  at 559 366
\endlabellist
\centering
\includegraphics[scale=0.3]{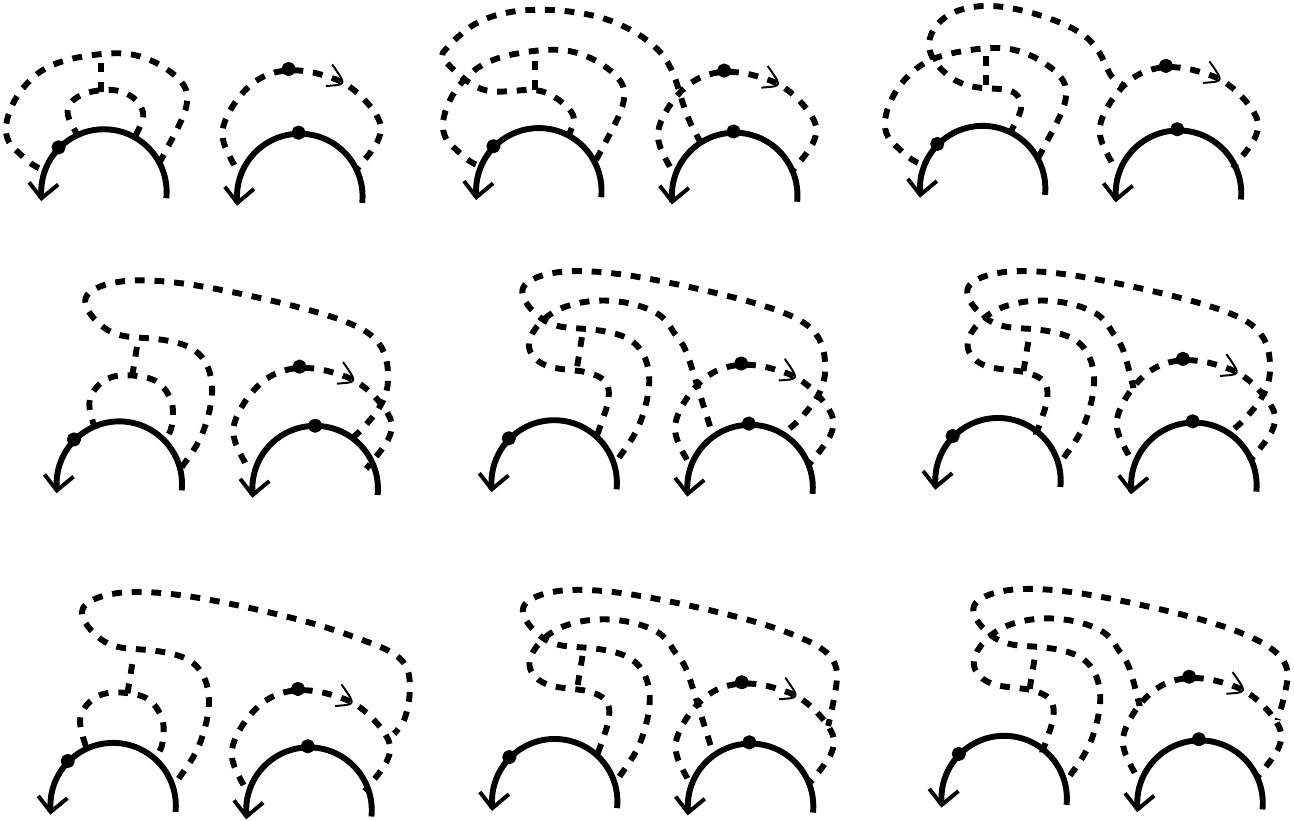}
$$
with the projection diagram
\begin{gather*}
\labellist
\small\hair 2pt
 \pinlabel{$D' \mathop{\tilde\circ} D=$} [r] at -5 28
\pinlabel{{.}} at 91 28
\endlabellist
\centering
\includegraphics[scale=1.1]{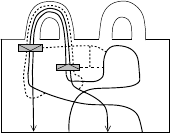}
\end{gather*}
\end{example}

One can easily verify the following lemma.

\begin{lemma} \label{ex:restricted}
Let $D$ be a restricted $(m,n)$-Jacobi diagram and let $D'$ be a
restricted $(n,p)$-Jacobi diagram, with square presentations
$$
S :  \dbl(w_1) \,\cdots\, \dbl(w_m) \lto (+-)^n \quad  \hbox{and} \quad
S' :  \dbl(w'_1)  \,\cdots\, \dbl(w'_n)  \lto (+-)^p,
$$
respectively. Then
$$
 S' \circ C_\varpi(S) :C_{\varpi_t}(  \dbl(w_1)
 \,\cdots\, \dbl(w_m) ) \lto (+-)^p
$$
is a square presentation of $D'\mathop{\tilde\circ} D$,
where $\circ$ denotes the composition in  $\AT$,
$$
\varpi:  \pi_0(\hbox{{$1$}-manifold underlying $S$}) \longrightarrow \Mon(\pm)
$$
is defined in the obvious way from $w'_1,\dots,w'_n$ and the polarized oriented $1$-manifold underlying $S$,
and $\varpi_t:  \{1,\dots, 2\vert w_1\vert+\cdots + 2\vert w_m\vert\} \to \Mon(\pm)$ is induced by $\varpi$.
\end{lemma}

Using Lemma \ref{ex:restricted}, we can easily prove the following.

\begin{lemma}
For $m,n,p\ge0$, there is a unique bilinear map
\begin{gather*}
\circ:\AB(n,p) \times \AB(m,n)  \longrightarrow \AB(m,p)
\end{gather*}
such that $D'\circ D = D' \mathop{\tilde\circ}  D$ for {each}
$(m,n)$-Jacobi diagram $D$ and {each} $(n,p)$-Jacobi diagram  $D'$.
\end{lemma}

Finally, the following lemma shows that we have a well-defined linear
  category $\bfA$ with the above composition $\circ$ and the identity
\begin{gather*}
  \id_n:=\labellist \scriptsize\hair 2pt
 \pinlabel{$1$} [l] at 115 9
 \pinlabel{$n$} [l] at 355 7
 \pinlabel{\ $\cdots$\ }  at 200 31
 \pinlabel{$x_1$} [b] at 66 58
 \pinlabel{$x_n$} [b] at 304 58
\endlabellist
\centering
\includegraphics[scale=0.15]{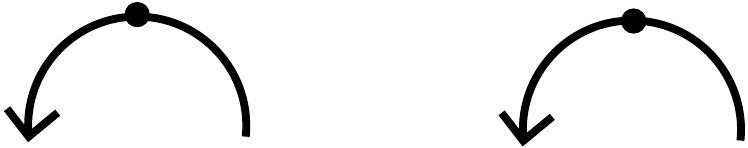}
\quad =
\centre{\labellist
\scriptsize \hair 2pt
 \pinlabel{$1$} [rb] at 26 93
 \pinlabel{$n$} [rb] at 100 94
  \pinlabel{$\cdots$}  at 83 87
 \pinlabel{$\cdots$}  at 85 33
\endlabellist
\centering
\includegraphics[scale=0.6]{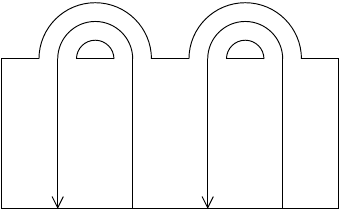}}.
\end{gather*}

\begin{lemma}
If $m\xto{D}n\xto{D'}p\xto{D''}q$ in $\AB$, then we have
\begin{eqnarray}
\label{eq:id_on_left}
 \big(\labellist \scriptsize\hair 2pt
 \pinlabel{$1$} [l] at 115 9
 \pinlabel{$n$} [l] at 355 7
 \pinlabel{\ $\cdots$\ }  at 200 31
 \pinlabel{$x_1$} [b] at 66 58
 \pinlabel{$x_n$} [b] at 304 58
\endlabellist
\centering
\includegraphics[scale=0.15]{id_in_AB} \ \  \ \big)  \circ D &=& D, \\[0.1cm]
\label{eq:id_on_right} D'  \circ \big(\labellist \scriptsize\hair 2pt
 \pinlabel{$1$} [l] at 115 9
 \pinlabel{$n$} [l] at 355 7
 \pinlabel{\ $\cdots$\ }  at 200 31
 \pinlabel{$x_1$} [b] at 66 58
 \pinlabel{$x_n$} [b] at 304 58
\endlabellist
\centering
\includegraphics[scale=0.15]{id_in_AB} \ \ \ \big) &=& D', \\
\label{eq:associativity}
D'' \circ (D' \circ D) & =&   (D'' \circ D') \circ D.
\end{eqnarray}
\end{lemma}

\begin{proof}
We may assume that $D$, $D'$ and $D''$ are restricted.
Let
\begin{gather*}
S : \dbl(w_1) \,\cdots\, \dbl(w_m) \lto (+-)^n,\\
S' : \dbl(w'_1) \,\cdots\, \dbl(w'_n) \lto (+-)^p,\\
S'' :    \dbl(w''_1) \,\cdots\, \dbl(w''_p) \lto (+-)^q.
\end{gather*}
be square presentations of $D$, $D'$ and $D''$, respectively.

\medskip
First, by Lemma \ref{ex:restricted}, a square presentation of
$\big(\labellist \scriptsize\hair 2pt
 \pinlabel{$1$} [l] at 115 9
 \pinlabel{$n$} [l] at 355 7
 \pinlabel{\ $\cdots$\ }  at 200 45
 \pinlabel{$x_1$} [b] at 66 61
 \pinlabel{$x_n$} [b] at 304 61
\endlabellist
\centering
\includegraphics[scale=0.12]{id_in_AB} \ \  \ \big)  \circ D $ is
$$
(\downarrow \uparrow \cdots \downarrow \uparrow ) \circ S = S.
$$
This proves \eqref{eq:id_on_left}.

Next, a square presentation of $D'  \circ \big(\labellist \scriptsize\hair 2pt
 \pinlabel{$1$} [l] at 115 9
 \pinlabel{$n$} [l] at 355 7
 \pinlabel{\ $\cdots$\ }  at 200 45
 \pinlabel{$x_1$} [b] at 66 61
 \pinlabel{$x_n$} [b] at 304 61
\endlabellist
\centering
\includegraphics[scale=0.12]{id_in_AB} \ \ \ \big)$ is
$$
S' \circ C_\varpi(\downarrow \uparrow \cdots \downarrow \uparrow )= S',
$$
where $\varpi: \pi_0(\downarrow \uparrow \cdots \downarrow \uparrow)\to \Mon(\pm)$
is the unique map such that $C_{\varpi_t}(+- \cdots +-) =
\dbl(w'_1) \,\cdots\, \dbl(w'_n) $.  This proves
\eqref{eq:id_on_right}.

Finally, a square presentation of $D'' \circ (D' \circ D)$ is
$$
 S'' \circ C_{\varpi'} \big(S' \circ C_{\varpi}(S)\big)
 = S'' \circ \big(C_{\varpi'_0}(S') \circ  C_{\varpi_0}(S) \big)
 $$ for some maps $\varpi, \varpi',\varpi_0$ and $\varpi'_0$.  By
 associativity of the composition in $\A$, the latter is a square
 presentation of $(D'' \circ D') \circ D$.  This proves
 \eqref{eq:associativity}.
\end{proof}

\subsection{A symmetric monoidal structure on $\AB$} \label{sec:symmetry_A}

We define a symmetric monoidal structure on the linear category $\AB$ as follows.
The tensor product on objects is addition.  The monoidal unit is $0$.
The tensor product on morphisms is juxtaposition followed by
relabelling the solid arcs and the beads.  More precisely, we
  obtain the tensor product $D\ot D'$ of an
$(m,n)$-Jacobi diagram $D$ and an $(m',n')$-Jacobi diagram $D'$
from the juxtaposition of $D$ and
$D'$ by renaming $\capl_j$ in $D'$ with $\capl_{n+j}$ for
$j=1,\dots,{n'}$, and replacing $x_i$ with $x_{m+i}$ for
$i=1,\dots,{m'}$.

\begin{lemma}
\label{r27}
The strict monoidal category $\AB$ admits a symmetry defined by\\[0.cm]
\begin{gather}
  \label{e6}
  P_{m,n} =
 \centre{ \labellist
\scriptsize \hair 2pt
 \pinlabel{$\cdots$}  at 29 3
 \pinlabel{$\cdots$}  at 84 4
  \pinlabel{$\underbrace{\hphantom{aaaaaaaa}}_n$} [t] at 29 3
 \pinlabel{$\underbrace{\hphantom{aaaaaaaa}}_m$} [t] at 84 4
 \pinlabel{$\cdots$}  at 30 39
 \pinlabel{$\cdots$}  at 85 39
   \pinlabel{$\overbrace{\hphantom{aaaaaaaa}}^n$} [b] at 85 41
 \pinlabel{$\overbrace{\hphantom{aaaaaaaa}}^m$} [b] at  30 41
\endlabellist
\centering
\includegraphics[scale=1.0]{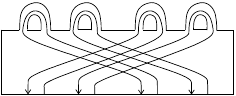}} :m+n\lto n+m.
\end{gather}
\end{lemma}

\vspace{0.1cm}
\begin{proof}
We show that the $P_{m,n}$ are natural in $m$ and $n$.
To post-compose a Jacobi diagram in $\AB(n + m,k)$ with
$P_{m,n}$, one transforms the labels of the beads by
$$
x_1 \mapsto x_{n+1}, \dots, x_m \mapsto x_{n+m}, \quad x_{m+1} \mapsto x_1, \dots, x_{m+n} \mapsto x_n.
$$
To pre-compose a Jacobi diagram in $\AB(k,m + n)$ with
$P_{m,n}$, one transforms the labels of the arcs by
$$
1 \mapsto n+1, \dots, m \mapsto n+m, \quad  m+1\mapsto 1, \dots, m+n \mapsto n.
$$
It follows that, for {$U:m\to m'$ and $V:n\to n'$}, we have
$$(V \otimes U) \circ P_{m,n} = P_{m',n'}  \circ (U \otimes V),$$
i.e., $P_{m,n}$ is natural.

One can easily check the other axioms of symmetric monoidal category.
\end{proof}

Before giving a presentation of the category $\bfA$ in
the next section, we describe some additional structures in $\AB$.

\subsection{Two gradings on $\AB$}  \label{sec:two_gradings}

{We first define} an $\NZ$-grading {on} $\AB$.  We have
\begin{gather*}
  \AB(m,n)=\bigoplus_{{k  \in \N}} \AB_k(m,n)
\end{gather*}
for $m,n\ge0$, where $\AB_k(m,n)$ is spanned by Jacobi diagrams
of degree $k$.  (Recall that the \emph{degree} of a Jacobi diagram is
half the total number of its vertices.)
It is easy to check that $\AB$ has the structure of an $\NZ$-graded linear strict monoidal category.
In what follows, $\NZ$-gradings are simply referred to as ``gradings''.

Let $\hA$ denote the degree-completion of $\bfA$ with respect to
  the above-defined grading on $\bfA$.  Thus, we set
  $\Ob(\hA)=\Ob(\AB)=\N$, and $\hA(m,n)$ is the degree-completion of $\AB(m,n)$.

{Before defining the second grading on $\AB$, we define the
  notion of a linear strict monoidal category graded over a strict
  monoidal category.  This generalizes the notion of a linear category
  graded over a category considered in \cite{Lowen,Tamaki}.}
Let~$\calD$ be a strict monoidal category.  A \emph{$\calD$-grading}
on a linear, strict monoidal category~$\calC$ consists of a monoid
homomorphism $i:\Ob(\C) \to \Ob(\D)$ and a direct sum decomposition
$$
\C(m,n) = \bigoplus_{d\, :\, i(m) \to i(n)} \C(m,n)_d 
$$
for {each pair of objects $m,n$ in $\C$}, such that
\begin{itemize}
\item $\id_m \in \C(m,m)_{\id_{i(m)}}$ for each $m \in \Ob(\C)$,
\item $ \C(n,p)_e  \circ  \C(m,n)_d \subset \C(m,p)_{e \circ d}  $ for all $m,n,p \in \Ob(\C)$
and all morphisms $i(m)\xto{d}i(n)\xto{e}i(p)$  in $\D$,
\item $\C(m,n)_d \otimes \C(m',n')_{d'} \subset \C(m\otimes m',n\otimes n')_{d\otimes d'}$
for all $m,n,m',n' \in \Ob(\C)$
and all morphisms  $d : i(m) \to i(n),d' : i(m') \to i(n')$ in $\calD$.
\end{itemize}
Then we say that the linear strict monoidal category $\C$ is \emph{$\D$-graded}, or that $\C$ is \emph{graded over $\D$}.

{For instance, we may regard the $\N$-grading of $\AB$ defined above as {a}
  grading over the commutative monoid $\N$, viewed as a strict
  monoidal category with one object.}

Now we define another grading of $\bfA$.  Let $\bfF$ be the full
subcategory of the category of groups with $\Ob(\bfF)
:=\{\free{n}\;|\;n \ge0\}${, and identify $\Ob(\bfF)$ with $\NZ$ in
the natural way}.  The category $\bfF$ has a symmetric strict monoidal
structure given by free product.  We define an $\bfF^\op$-grading on
$\AB$ as follows.  The \emph{homotopy {class}} of an $(m,n)$-Jacobi
diagram {$D$} is the homomorphism $h(D): \free{n} \to \free{m}$ that
maps each generator $x_j$ to the product of the beads along the
oriented component $\!\figtotext{8}{8}{capleft}\!\!_j$; {we emphasize
that $h(D)$ is independent of the dashed part of $D$.}  Then we have
\begin{gather*}
  \AB(m,n) = \bigoplus_{d\in\bfF^\op(m,n)}\AB(m,n)_d,
\end{gather*}
where $\AB(m,n)_d$ is   spanned by
  Jacobi diagrams of homotopy {class} $d$.   {It is easy to check that} $\AB$ has the structure
  of an $\bfF^\op$-graded linear strict monoidal category.

Let $\AB_0$ denote the degree $0$ part of $\AB$, which is a
linear, symmetric strict monoidal subcategory of~$\AB$.  The
morphisms in $\AB_0$ are linear combinations of Jacobi diagrams in
handlebodies without dashed part, which are fully determined by their
homotopy {classes}.  Thus, there is an isomorphism of linear
symmetric strict monoidal categories
\begin{gather*}
  h: \AB_0 \congto \K\bfF^\op,
\end{gather*}
where $\K\bfF^\op$ denotes the linearization of $\bfF^\op$.  The
isomorphism $h$ extends to a full linear functor $h: \AB \to
\K\bfF^\op$ vanishing on morphisms of positive degree.

\begin{remark}
The $\bfF^\op$-grading of $\AB$ induces a (completed)
  $\bfF^\op$-grading on the degree-completion $\hA$ in the obvious  way.  We have
\begin{gather*}
  \hA(m,n) = \mathop{\widehat{\bigoplus}}_{d\in\bfF^\op{(m,n)}}\hA(m,n)_d
\end{gather*}
where $\hA(m,n)_d$ is the degree-completion of $\AB(m,n)_d$, and
$\widehat{\bigoplus}$ denotes the completed direct sum.
\end{remark}

\subsection{Coalgebra enrichment of $\AB$}   \label{sec:coalgebras}

Here we define coalgebra structures on the spaces $\AB(m,n)$
($m,n\ge0$) by generalizing the usual coalgebra structures of the
spaces of Jacobi diagrams \cite{BN2}.  Moreover, we show that the
category $\AB$ is enriched over cocommutative coalgebras.
{(See \cite{Kelly} for the definitions in enriched category theory.)}

Define a linear map $$\Delta :  \AB(m,n) \longrightarrow \AB(m,n)\otimes \AB(m,n)$$ by
\begin{gather*}
  \Delta (D) = \sum_{D= D' \sqcup D''} D' \otimes D''
\end{gather*}
for every $(m,n)$-Jacobi diagram $D$, where the sum is over all
splittings of $D$ as the disjoint union of two parts $D'$
and $D''$.  Define also a linear map
\begin{gather*}
  \epsilon :  \AB(m,n)\longrightarrow \K
\end{gather*}
by $\ep(D)=1$ if $D$ is the empty diagram, and $\ep(D)=0$ otherwise.
It is easy to see that $(\AB(m,n),\Delta,\ep)$ is a cocommutative
coalgebra.

\begin{proposition}
\label{r8}
The symmetric monoidal category $\AB$ is enriched over the
{symmetric} monoidal category of cocommutative coalgebras.  In
other words, the linear maps
\begin{gather*}
  \circ=\circ_{m,n,p}:\AB(n,p)\ot\AB(m,n)  \longrightarrow   \AB(m,p)\quad (m,n,p\ge0),\\
\K \longrightarrow  \AB(m,m),\quad 1 \longmapsto   \id_m\quad
  (m\ge0),\\
  \otimes:\AB(m,n)\otimes\AB(m',n')\longrightarrow\AB(m+m',n+n')\quad (m,n,m',n'\ge0),\\
  \K\longrightarrow \AB(m+n,n+m),\quad
1\longmapsto P_{m,n}\quad (m,n\ge0)
\end{gather*}
are coalgebra maps.
\end{proposition}

To prove this {proposition}, we need the {lemma below}, which
one can easily verify.

\begin{lemma} \label{ex:coalgebra_structures}
Let $S$
be a square presentation of a restricted  $(m,n)$-Jacobi diagram~$D$.
Then $\Delta(S)$ (resp.\ $\epsilon(S)$), the usual comultiplication
(resp.\ counit) of Jacobi diagrams applied to $S$,
is a square presentation of $\Delta(D)$ (resp.\ $\ep(D)$).
\end{lemma}

\begin{proof}[Proof of Proposition \ref{r8}]
 We will check that $\circ_{m,n,p}$ is a coalgebra map; clearly,
{so are the other maps listed in the proposition.
   Consider restricted Jacobi diagrams $D:m\to n$ and $D':n\to p$ in $\AB$}
   with square presentations $S$ and $S'$, respectively.  By Lemma
 \ref{ex:restricted}, $D' \circ D$ admits a square presentation of the
 form $S' \circ C_\varpi(S)$ for a map $\varpi$ determined by $S'$.
 The connected components of the dashed part of $C_\varpi(S)$ are in
 one-to-one correspondence with those of $S$.  Hence we have
\begin{gather*}
  \begin{split}
     &\Delta(S' \circ C_\varpi(S))\\
    =& \sum_{S = S_* \sqcup S_{**}}   \sum_{S' = S'_* \sqcup S'_{**}}
    \big(S'_* \circ C_{\varpi}(S_*)\big) \otimes   \big(S'_{**} \circ C_{\varpi}(S_{**})\big)\\
     =& (\circ\ot\circ)\Big(\sum_{S = S_* \sqcup S_{**}}   \sum_{S' = S'_* \sqcup S'_{**}}
     \big(S'_* \ot C_{\varpi}(S_*)\big) \otimes   \big(S'_{**} \ot C_{\varpi}(S_{**})\big)\Big)\\
    =& (\circ\ot\circ)(\id\ot P\ot\id)\Big(
    \sum_{S' = S'_* \sqcup S'_{**}}\big(S'_* \ot S'_{**}\big) \otimes
    \sum_{S = S_* \sqcup S_{**}}  \big( C_{\varpi}(S_*) \ot C_{\varpi}(S_{**})\big)\Big)\\
    =& (\circ\ot\circ)(\id\ot P\ot\id)\big(
    \Delta(S') \otimes (C_f\ot C_f)\De(S)\big)\\
    =& \big((-\circ C_f(-))\ot(-\circ C_f(-))\big)(\id\ot P\ot\id)
    (\De\ot\De)(S'\ot S),
  \end{split}
\end{gather*}
where $P$ is the linear map $x\ot y\mapsto y\ot x$.
We deduce from Lemma \ref{ex:coalgebra_structures} that
\begin{gather*}
  \De(D'\circ D)=(\circ\ot\circ)(\id\ot P\ot\id)(\De\ot\De)(D'\ot D),
\end{gather*}
i.e., $\circ=\circ_{m,n,p}$ preserves comultiplication.  Clearly,
$\circ_{m,n,p}$  preserves counit, i.e.,
we have $\ep(D'\circ D)=\ep(D')\, \ep(D)$.
Hence $\circ_{m,n,p}$ is a coalgebra map.
\end{proof}

\begin{cor}
  For $m\ge0$ the coalgebra structure of $\AB(m,m)$ and the
  endomorphism algebra structure of $\AB(m,m)$ makes $\AB(m,m)$ a
  cocommutative bialgebra.
\end{cor}

The coalgebra structure on $\AB(m,n)$ induces a coalgebra structure on
$\hA(m,n)$.  By Proposition \ref{r8}, $\hA$ also is enriched over
cocommutative coalgebras.
Let
\begin{gather*}
\hAg(m,n)= \big \{f\in\hA(m,n)  \;|\; \Delta(f)=f\ot f,\;\ep(f)=1 \big\}
\end{gather*}
be the group-like part of $\hA(m,n)$.  Then the sets $\hAg(m,n)$
for $m,n\ge0$ form a symmetric monoidal subcategory of $\hA$, which we
call the \emph{group-like part of~$\hA$}.

%
%
\section{Presentation of the category $\AB$}  \label{sec:presentation}

In this section, we give a presentation of the category $\AB$ of
Jacobi diagrams in handlebodies.

\subsection{Hopf algebras in symmetric monoidal categories} \label{sec:hopf-algebr-symm}

Let $\calC$ be a symmetric strict monoidal category, with monoidal
unit $I$ and symmetry $P_{X,Y}: X\ot Y \to Y\ot X$.

Let $H$ be a Hopf algebra in $\calC$ with the multiplication, unit,
comultiplication, counit and antipode
\begin{gather*}
  \mu: H\ot H\to H,\quad \eta: I\to H,\quad \Delta: H\to H\ot
  H,\quad \epsilon: H\to I,\quad
  S: H\to H.
\end{gather*}
 The axioms for a \emph{Hopf algebra} in $\calC$ are
\begin{gather}
  \label{h1}
  \mu (\mu \ot \id)=\mu (\id\ot \mu ),\quad
  \mu (\eta \ot \id)=\id=\mu (\id\ot \eta ),\\
  \label{h2}
  (\Delta \ot \id)\Delta =(\id\ot \Delta )\Delta ,\quad
  (\epsilon \ot \id)\Delta=\id =(\id\ot \epsilon )\Delta ,\\
  \label{h3}
    \epsilon \eta =\id_I,\quad
    \epsilon \mu =\epsilon \ot \epsilon ,\quad
    \Delta \eta =\eta \ot \eta ,\quad
    \Delta \mu =(\mu \ot \mu )(\id\ot P\ot \id)(\Delta \ot \Delta ),\\
    \label{h4}
    \mu (\id\ot S)\Delta =\mu (S\ot \id)\Delta =\eta \epsilon.
\end{gather}
Here and in what follows, we write $\id=\id_H$ and $P=P_{H,H}$ for simplicity.
In the following, we assume that $H$ is \emph{cocommutative}, i.e., we have
\begin{gather}
  \label{h5}
  P\Delta=\Delta.
\end{gather}

We will also use the notions of algebras and coalgebras in symmetric
monoidal categories, defined by axioms \eqref{h1} and \eqref{h2},
respectively.  For $m\ge0$, define $\mu^{[m]}: H^{\ot m}\to H$ and
$\Delta^{[m]}: H\to H^{\ot m}$ inductively by
\begin{gather*}
  \mu^{[0]}=\eta,\quad \mu^{[1]}=\id,\quad
  \mu^{[m]}=\mu(\mu^{[m-1]}\ot\id)\quad (m\ge2),\\
  \Delta^{[0]}=\epsilon,\quad \Delta^{[1]}=\id,\quad
  \Delta^{[m]}=(\Delta^{[m-1]}\ot\id)\Delta\quad (m\ge2).
\end{gather*}

A (left) \emph{$H$-module} in $\calC$ is an object $M$ with a morphism
$\action: H \otimes M \to M$, called a \emph{(left) action}, such that 
\begin{gather*}
  \action (\mu \otimes \id_M) = \action (\id_H \otimes
  \action),\quad
  \action (\eta \otimes \id_M) = \id_M.
\end{gather*}
For $H$-modules $(M,\action)$ and $(M',\action')$, a
morphism $f:M\to M'$ is a \emph{morphism of $H$-modules} if
\begin{gather*}
f \action = \action' (\id_H \otimes f).
\end{gather*}
Since $H$ is a cocommutative Hopf algebra, the category $\HMod$ of
$H$-modules inherits from $\calC$ a symmetric strict monoidal
structure. Specifically, the tensor product of two $H$-modules
$(M,\action)$ and $(M',\action')$ is $M \otimes M'$ with the
action
$$
(\action \otimes \action') ( \id_H \otimes P_{H,M} \otimes \id_{M'} )
(\Delta \otimes \id_M \otimes \id_{M'}) :H \otimes M \otimes M'\to
M \otimes M'.
$$
The monoidal unit in $\HMod$ is the \emph{trivial} $H$-module $(I,\epsilon)$.

Define the (left) \emph{adjoint action} $\ad: H\ot H\to H$ by
\begin{gather*}
  \ad  =   \mu^{[3]}(\id\ot \id\ot S)(\id\ot P)(\Delta\ot\id).
\end{gather*}
Since $H$ is cocommutative, all the structure morphisms
$\mu,\eta,\Delta,\epsilon,S$ of $H$ as well as the symmetry $P_{H,H}$
are $H$-module morphisms with respect to the adjoint action.  Thus,
the $H$-module $(H,\ad)$ is a cocommutative Hopf algebra in $\HMod$.

\subsection{Convolutions}    \label{sec:convolutions}

Let $\calC$ be a symmetric strict monoidal category.  Let
$(A,\mu_A,\eta_A)$ be an algebra and $(C,\Delta_C,\epsilon_C)$ a
coalgebra in $\C$.  We define the \emph{convolution product} on $\C(C,A)$
\begin{gather}
  *: \C(C,A)\times \C(C,A) \longrightarrow \C(C,A)
\end{gather}
 by
\begin{gather*}
  f * g = \mu_A (f\ot g)\Delta_C
\end{gather*}
for {$f,g:C\to A$}.  This operation is associative with unit $\eta_A\epsilon_C$.

A morphism {$f:C\to A$} is \emph{convolution-invertible} if there is
{$g:C\to A$} such that $f*g=g*f=\eta_A\epsilon_C$.  In this case,
we call $g$  the \emph{convolution-inverse} to $f$,
and it is denoted by $f^{-1}$ if there is no fear of confusing it with
the inverse of $f$.

In what follows, we mainly use convolutions when $A=H^{\ot n}$
and $C=H^{\ot m}$ ($m,n\ge0$) for a Hopf algebra $H$ in $\C$.
For example, the convolution on $\C(H,H^{\ot 2})$ is given by $$f*g =
\mu_2 (f\ot g)\Delta,$$ where $\mu_2:=(\mu\ot\mu)(\id\ot P\ot \id)$, and
the convolution on $\C(I,H^{\ot n})$ is given by
\begin{gather}
  \label{e12}
  f*g = \mu_n (f\ot g),
\end{gather}
where we define $\mu_n: H^{\ot n}\ot H^{\ot n}\to H^{\ot n}$
inductively by $\mu_0  =\id_I$, $\mu_1 = \mu$ and
\begin{gather*}
  \mu_n = (\mu_{n-1}\ot \mu)(\id^{\ot n-1}\ot P_{H,H^{\ot (n-1) }}\ot\id)\quad (n\ge2).
\end{gather*}
This convolution product is defined whenever $(H,\mu,\eta)$ is an algebra in $\C$.

\subsection{Casimir Hopf algebras}\label{sec:casim-hopf-algebr}

Let $H$ be a cocommutative Hopf algebra in a linear
  symmetric strict monoidal category $\C$.

\begin{definition}
  \label{r11}
  A \emph{Casimir $2$-tensor} for $H$ is a morphism $ c: I\to H^{\ot2}$
  which is \emph{primitive}, \emph{symmetric} and \emph{invariant}:
\begin{gather}
  \label{e:add}(\Delta\ot \id)c=c_{13}+c_{23},\\
  \label{e:sym}Pc=c,\\
  \label{e:inv}(\ad\ot\ad)(\id\ot P\ot \id)(\Delta\ot c)=c\epsilon,
\end{gather}
where $c_{13} := (\id\ot\eta\ot\id)c$ and $c_{23} := \eta\ot c$.

 By a \emph{Casimir Hopf algebra} in $\calC$, we mean a cocommutative
Hopf algebra in $\calC$ equipped with a Casimir $2$-tensor.
\end{definition}

The condition \eqref{e:inv} means that $c:I\to H^{\ot2}$ is a
morphism of $H$-modules.  Thus a Casimir Hopf algebra $(H,c)$ in
$\calC$ is also a Casimir Hopf algebra in $\HMod$.

Here are elementary properties of Casimir $2$-tensors:
\begin{eqnarray}
  \label{e13'} & &    (\id\ot\Delta)c=c_{12}+c_{13}, \\
  \label{e13''} && (\epsilon\ot\id)c=(\id\ot\epsilon)c=0, \\
  \label{e13'''} &&  (S\ot\id)c=(\id\ot S)c=-c.
\end{eqnarray}

\begin{lemma}
  \label{r4}
  For $c:I\to H^{\ot2}$, the identity \eqref{e:inv} is equivalent to
  \begin{gather}
    \label{e:inv2}
    \Delta*  c\ep  =  c\ep  *\Delta.
  \end{gather}
\end{lemma}

\begin{proof}
  It is easy to see that \eqref{e:inv} is equivalent to $\Delta*  c\epsilon *\Delta^-=c\epsilon$,
  where $\Delta^-:=(S\ot S)\Delta$ is the convolution-inverse to
  $\Delta$.  Thus \eqref{e:inv} is equivalent to
  \eqref{e:inv2}$*\Delta^-$.  Since $\Delta$ is convolution-invertible, \eqref{e:inv} and \eqref{e:inv2} are equivalent.
\end{proof}

\begin{proposition}
  \label{r6}
  Let $(H,c)$ be a Casimir Hopf algebra.  Then we have a
  version of the 4T relation in $\C(I,H^{\ot 3})$:
  \begin{gather}
    \label{e:4T}
    (c_{12}+c_{13})*c_{23} = c_{23}*(c_{12}+c_{13}).
  \end{gather}
\end{proposition}

\begin{proof}
  Using \eqref{e13'} and \eqref{e:inv2}, we have
\begin{eqnarray*}
     (c_{12}+c_{13})*c_{23}
     \ = \ (\id\ot\Delta)c*c_{23}
     &=& (\id\ot(\Delta*c\ep))c \\
     & = & (\id\ot (c\ep*\Delta))c \\
     &= & c_{23}*(\id\ot\Delta)c
     \ = \ c_{23}*(c_{12}+c_{13}).
  \end{eqnarray*}
\end{proof}

\begin{example}
  \label{r3}
  (1) In Section \ref{sec:weight_systems}, we consider \emph{Casimir Lie algebras} (including semi-simple Lie algebras)
  and observe that their universal enveloping algebras are instances of Casimir Hopf algebras.

  (2) Every linear combination of Casimir $2$-tensors is a Casimir
  $2$-tensor.  In particular, $0: I\to H\ot H$ is a Casimir
  $2$-tensor.
\end{example}

\subsection{Casimir Hopf algebras and infinitesimal braidings}  \label{sec:casimir-hopf-algebra}

The {above} notion of Casimir Hopf algebra 
{is a Hopf-algebraic} version of the notion of infinitesimal braiding
for {symmetric monoidal categories, 
introduced by Cartier \cite{Cartier} (see also \cite{Kassel})}.  

Recall that an \emph{infinitesimal braiding}  {in} a linear symmetric strict monoidal
category $\C$ is a natural transformation
\begin{gather*}
  t_{x,y}:x\ot y\lto x\ot y
\end{gather*}
such that
\begin{gather}
  \label{e67}P_{x,y}t_{x,y}=t_{y,x}P_{x,y},\\
  \label{e68}t_{x,y\ot z }=t_{x,y}\ot \id_z +(\id_x\ot P_{z,y})(t_{x,z}\ot \id_y)( \id_x \ot P_{y,z})
\end{gather}
for $x,y,z\in\Ob(\C)$.  {For instance,} {the linear version of} the
category $\A$ of Jacobi diagrams {(see Remark \ref{r30})} admits an
infinitesimal braiding; see \cite[Section~XX.5]{Kassel}.  Note that
\eqref{e67} and \eqref{e68} imply
\begin{gather}
  \label{e8}
  t_{x\ot y,z}= ( P_{y,x} \ot  \id_z) (\id_y\ot t_{x,z})( P_{x,y} \ot \id_z )+ \id_x \ot t_{y,z}.
\end{gather}

Let $H$ be a cocommutative Hopf algebra in $\C$.  A  $H$-module
$(x,\action)$ is said to be \emph{trivial} if $\action=\ep\ot\id_x$.
An infinitesimal braiding $t_{x,y}$ in  {the symmetric monoidal category $\HMod$ of   $H$-modules} is called \emph{strong}
if it vanishes whenever $x$ or $y$ is a trivial $H$-module.
{The following shows that {strength of infinitesimal braiding in
    module categories} is automatic {for} some {underlying} symmetric monoidal categories~$\C$, such as the category $\Vect_\K$ of vector spaces.

\begin{proposition}
Let $\C$ be a linear symmetric strict monoidal category. We assume that the
functor $\C(I,-): \C \to \Vect_\K$ is faithful, and the tensor
product map $\C(I,x) \otimes  \C(I,y) \to \C(I,x\ot y)$ is surjective for {each} $x,y \in \Ob(\C)$.
Then, for {every} cocommutative Hopf algebra $H$, {every} infinitesimal braiding $t$ in $\Mod_H$ is strong.
\end{proposition}

\begin{proof}
The assumptions on $\C$ imply that the map
\begin{gather*}
  \tau_{x,y}: \C(x \ot y, x \ot y) \lto \Hom_\K\big( \C(I,x) \otimes
 \C(I,y) , \C(I,x\ot y)\big)
\end{gather*}
defined by $ \tau_{x,y}(a) := (b\ot c \mapsto a  (b\ot c))$ is injective for  $x,y\in \Ob(\C)$.
Let $x,y \in \Ob(\Mod_H)$ {with $y$ being}
a trivial $H$-module. Then, for {each} $b:I\to x$, $c:I\to y$ in $\C$, we have
  \begin{gather*} 
    \tau_{x,y}(t_{x,y})(b\ot c) = t_{x,y}(b\ot c) = t_{x,y} (\id_x \ot c) b     =(\id_x\ot c)t_{x,I} b =0,
  \end{gather*}
  since $c$ is {an} $H$-module morphism and \eqref{e68} implies that $t_{x,I}=0$ for {every} infinitesimal braiding.
  Since $\tau_{x,y}$ is injective, we have $t_{x,y}=0$.
  Thus $t$ is strong.
\end{proof}
}

{
We now prove that,  given a cocommutative Hopf algebra~$H$ in $\C$,
there is a one-to-one correspondence between Casimir $2$-tensors for $H$ and strong infinitesimal braidings {in} $\Mod_H$.
This {result} generalizes \cite[Proposition XX.{4.2}]{Kassel}, where $\C=\Vect_\K$.
{Let $H^l:=(H,\mu)\in\Mod_H$, the regular representation of $H$.

\begin{proposition}
  \label{r49}
  Let $H$ be a cocommutative Hopf algebra in a linear symmetric {strict} monoidal  category $\C$.
  \begin{enumerate}
  \item[(a)]
    {Every} Casimir $2$-tensor for $H$ induces a strong infinitesimal braiding in  $\HMod$  defined by 
  \begin{equation}  \label{t_from_c}
    t_{x,y}= (\action_x\ot\action_y)( \id_H \ot P_{H,x}\ot \id_y)(c\ot \id_x \ot \id_y ):x\ot
    y\lto x\ot y
  \end{equation}
  for  $H$-modules $x=(x,\action_x)$ and  $y=(y,\action_y)$.
  \item[(b)] {Every} strong infinitesimal braiding $t$ in~$\Mod_H$ induces a Casimir $2$-tensor 
    \begin{gather}
      \label{e277}
    c:= t_{H^l,H^l}(\et\ot\et): I\lto H\ot H 
  \end{gather}
  for $H$ in $\C$ and $t_{x,y}$ is of the form \eqref{t_from_c} for {each $x,y\in\Mod_H$}.
  \end{enumerate}
\end{proposition}
}

\begin{proof}
{We} {only} sketch the proof {of (a),} leaving the details to the
  reader.  It is easy to check \eqref{e67} and \eqref{e68}.
  Naturality of {$t$}, i.e., $t_{x',y'}(f\ot g)=(f\ot g)t_{x,y}$
  for $f:x\to x'$ and $g:y\to y'$ in $\Mod_H$, follows from
  the definition of $H$-module morphisms.  We can check that $t_{x,y}$
  is an $H$-module morphism by using \eqref{e:inv2} and the definition
  of $H$-modules.  Using $(\ep\ot\id)c=0=(\id\ot\ep)c$, we see
  that $t_{x,y}$ is a strong infinitesimal braiding.
  
{We now prove (b).} 
  We first verify \eqref{t_from_c}.  Note that for each $H$-module
  $x=(x,\action_x)$, the action $\action_x$ gives a 
  morphism $\action_x:H^l\ot x^\ep\to x$ {in $\Mod_H$}, where $x^\ep:=(x,\ep\ot\id_x)$ is the trivial $H$-module.
  {Therefore, the naturality of $t$ implies that}
  \begin{eqnarray}
    \notag t_{x,y}     &=&t_{x,y}\, (\action_x\ot\action_y)\, (\et\ot\id_x\ot\et\ot\id_y)\\
    \label{ttt} &=&(\action_x\ot\action_y)\, t_{H^l\ot x^\ep,H^l\ot y^\ep}\, (\et\ot\id_x\ot\et\ot\id_y).
  \end{eqnarray}
  Using \eqref{e68} and \eqref{e8},
  we can express $t_{H^l\ot x^\ep,H^l\ot y^\ep}$ 
  as a sum of four {morphisms} involving  $t_{H^l,H^l}$, $t_{H^l,y^\ep}$, $t_{x^\ep,H^l}$ and
  $t_{x^\ep,y^\ep}$, with the last three being $0$ since $t$ is
  strong.  Hence,
  \begin{gather*}
    t_{H^l\ot x^\ep,H^l\ot y^\ep}= (\id_{{H}} \ot P_{H,x}\ot\id_y)\, (t_{H^l,H^l}\ot\id_x\ot\id_y)\, (\id_{{H}} \ot P_{x,H}\ot\id_{{y}}).
  \end{gather*}
  This and \eqref{ttt} imply \eqref{t_from_c}.
  
  { We now check the axioms of a Casimir $2$-tensor for $c$. We easily
    obtain \eqref{e:sym} from \eqref{e67}.  The identity \eqref{e:add}
    follows from \eqref{e8} since $\Delta: H^l \to H^l \otimes H^l$ is
    a morphism of $H$-modules.  It remains to} verify \eqref{e:inv}
    or, equivalently, \eqref{e:inv2}.  Since $t_{H^l,H^l}$ is a
    morphism in $\Mod_H$, we have
  \begin{eqnarray*}
    &&(\mu\ot\mu)(\id\ot P\ot\id)(\De\ot\id\ot\id)(\id\ot t_{H^l,H^l})\\
    & =&t_{H^l,H^l}(\mu\ot\mu)(\id\ot P\ot\id)(\De\ot\id\ot\id)
  \end{eqnarray*}
  {and, by pre-composing with} $\id\ot\et\ot\et$, we obtain $\De* c\ep
  =t_{H^l,H^l}\De$.  {Moreover, \eqref{t_from_c} with $x=y=H^l$
  implies} $t_{H^l,H^l}\De=c\ep*\De$.  Hence \eqref{e:inv2}.
\end{proof}

\begin{remark}
  \label{r50}
  It is \emph{not} possible to generalize Proposition \ref{r49}(b) to
  infinitesimal braidings that are not strong.  Here is a
  counterexample.  Let $\C$ be a linear symmetric strict monoidal
  category  equipped with a non-zero infinitesimal braiding $t$.
  Consider the {\emph{trivial} Hopf algebra  in $\C$, defined by} 
  $H=I$ with $\mu=\et=\De=\ep=S=\id_I$.  Then $t$ is an infinitesimal braiding in
  $\Mod_{I}$ via the canonical isomorphism ${\Mod_{I}\cong\C}$.  
  Since every {$I$-module is trivial,}   $t$ is not strong in $\Mod_{I}$. 
  {However}, the  Casimir $2$-tensor $c$ for {$I$} given in \eqref{e277} is zero since
  $I^l=(I,\id_I)$ {is the monoidal unit  of  ${\Mod_{I}\cong\C}$}.  
  Therefore, $c$ and $t$ are not related by \eqref{t_from_c}.
\end{remark}  

\subsection{Casimir elements} \label{sec:casimir-elements}

Now we give an alternative viewpoint on Casimir $2$-tensors.
Let $H$ be a cocommutative Hopf algebra in a linear symmetric strict monoidal category $\calC$.

\begin{definition}
  A \emph{Casimir element} for $H$ is a morphism $r: I\to H$
    which is \emph{central} and \emph{quadratic}:
    \begin{gather}
      \label{e1}
      \mu(\id\ot r)=\mu(r\ot \id),\\
      \label{e9}
      r_{123}-r_{12}-r_{13}-r_{23}+r_1+r_2+r_3=0,\\
      \label{e0}
      S r = r,
    \end{gather}
    where
    \begin{gather*}
      r_{123}:=\Delta^{[3]}r,\quad  r_{12} :=\Delta r\ot \eta,\quad
    r_{13}:=(\id\ot\eta\ot\id)\Delta r,\quad  r_{23}:=\eta \ot \Delta r,\\
    r_1:=r\ot\eta\ot\eta,\quad  r_2:=\eta\ot r\ot\eta,\quad   r_3:=\eta\ot\eta\ot r.
    \end{gather*}
\end{definition}

The notion of a Casimir Hopf
algebra is equivalent to that of a cocommutative Hopf algebra with a Casimir element, as follows.

\begin{proposition}
  \label{r1}
 There is a one-to-one correspondence
  \begin{equation} \label{correspondence}
\xymatrix{
\left\{\hbox{Casimir $2$-tensors for $H$} \right\}
\ar@<.7ex>[rr]^-{c\;\rightmapsto\;r_c}
&& \ar@<.7ex>[ll]^-{c_r\;\leftmapsto\;r}
\left\{\hbox{Casimir elements for $H$} \right\},
}
\end{equation}
associating to a Casimir $2$-tensor $c$ a Casimir element
    \begin{gather}
      \label{e22}
      r_c:=\frac12\mu c \ :I\lto H,
    \end{gather}
and  to a Casimir element $r$ a Casimir $2$-tensor
    \begin{gather}
      \label{e14}
      c_r  := \Delta r -r\ot\eta-\eta\ot r \ :I\lto H\ot H.
    \end{gather}
\end{proposition}

\begin{proof}
  Let $c$ be a Casimir $2$-tensor.  Let $r= r_c$.  We have
    \eqref{e0} by \eqref{e:sym} and \eqref{e13'''}.
    Post-composing $\mu(\id\ot S)$ to \eqref{e:inv2} gives
    $r\epsilon=\id* r\ep *S$;  taking $(-)*\id$, we obtain
    $r\ep*\id=\id* r\ep$, equivalent to \eqref{e1}.  Finally, \eqref{e9} follows from
    \begin{gather*}
      r_{123}=r_1+r_2+r_3+c_{12}+c_{13}+c_{23},\\
      r_{ij}=r_i+r_j+c_{ij}\quad (1\le i<j\le3).
    \end{gather*}
    Therefore $r=r_c$ is a Casimir element.

  Now let $r$ be a Casimir element.  Set $c=c_r$.  Then
  \eqref{e:add} follows from \eqref{e9}, and \eqref{e:sym} follows
  from the cocommutativity of $H$.  Moreover, \eqref{e:inv2} follows
  from \eqref{e1}.  Hence $c=c_r$ is a Casimir $2$-tensor.

  If $c$ is a Casimir $2$-tensor, then  $c_{(r_c)}=c$
  follows from \eqref{e:add} and \eqref{e:sym}.  If $r$ is a
  Casimir element, then
  \begin{eqnarray*}
    r_{(c_r)} &\by{e13'''}& r_{(-(\id \otimes S) c_r)} \\
    & \by{e22}& - \frac{1}{2}\mu (\id \otimes S) c_r \\
    &\by{e14}& - \frac{1}{2}\Bigl(\mu (\id \otimes S) \Delta r
    -\mu (\id \otimes S)(r\ot\eta)  -\mu (\id \otimes S)(\eta\ot r)\Bigr) \\
    &=& - \frac{1}{2}( \eta \epsilon r - r - Sr)
    \ \by{e7} \ \frac{1}{2} (r+Sr) \ \by{e0} \ r .
  \end{eqnarray*}
 \end{proof}

\subsection{Presentation of $\AB$}   \label{sec:presentation-ab}

 Recall from Section \ref{sec:symmetry_A} that $\AB$ is a
linear symmetric monoidal category.  Define morphisms in $\AB$
\begin{gather}
  \label{e51}
  \begin{aligned}
  \eta=\figz{eta}:0\to1,\quad
  \mu=\figz{mu}:2\to1,\quad
  \epsilon=\figz{epsilon}:1\to0,\\
  \Delta=\figz{Delta}:1\to2,\quad
  S=\figz{S}:1\to1,\quad
  c=\figz{r}:0\to2.
  \end{aligned}
\end{gather}

\begin{proposition}
  \label{r14}
  We have a Casimir Hopf algebra $(1,\eta,\mu,\epsilon,\Delta,S,c)$ in $\AB$.
\end{proposition}

\begin{proof}
  One can easily verify the axioms of a Hopf algebra.  (In fact,
  this can also be checked by reducing up to homotopy the topological
  arguments given in \cite{Habiro2} for the category~$\B$.  See also
  \cite{Habiro3} for related algebraic arguments in the symmetric
  monoidal category $\bfF$ of finitely generated free groups.)
  The cocommutativity  follows from
  \begin{gather*}
    P\Delta=\figb{PDelta}{10}=\Delta
  \end{gather*}
  where we write $P=P_{1,1}$.

  Now we check for $c$ the relations \eqref{e:add},
  \eqref{e:sym} and \eqref{e:inv2} of a Casimir  $2$-tensor. We have~\eqref{e:add}:
  \begin{gather*}
    \begin{split}
      &(\Delta\ot\id)c=\fig{Delta-r}
      =\fig{r13}+\fig{r23}
      =c_{13}+c_{23}.
    \end{split}
  \end{gather*}
  We have \eqref{e:sym}:
  \begin{gather*}
    Pc=\figb{Pr}{10}=c.
  \end{gather*}
  We have \eqref{e:inv2}:
  \begin{eqnarray*}
    \Delta*  c\ep  &=& (\mu\ot \mu)(\id\ot P\ot \id)(\Delta\ot c) \\
      &= & \figb{conv-Delta-r}{12}    \  = \ \figb{conv-r-Delta}{12}\\
      &= & (\mu\ot \mu)(\id\ot P\ot \id)(c\ot \Delta)  \ = \  c\ep  *\Delta.
  \end{eqnarray*}
\end{proof}

Let $H$ denote the cocommutative Hopf algebra
  $(1,\mu,\eta,\Delta,\epsilon,S)$ in $\AB$.
We prove the following theorem in the rest of this section.

\begin{theorem}
  \label{AB-pres}
  As a linear symmetric strict monoidal category, $\AB$ is
  free on the Casimir Hopf algebra $(H,c)$.
\end{theorem}

{
\begin{remark}
Hinich and Vaintrob \cite{HV} proved that the  algebra $\A(\circlearrowleft)$ of chord diagrams on a circle 
(i.e., the target of the usual Kontsevich integral of knots)
is in some sense the ``universal enveloping algebra'' of the generating object in the linear PROP governing ``Casimir Lie algebras''. 
This gives a universal property for the space $\AB(0,1)\cong \A(\circlearrowleft)$,
whereas Theorem \ref{r36} gives a universal property for the entire category $\AB$.
\end{remark}
}

\subsection{The category $\bfP$ generated by a Casimir Hopf algebra}   \label{sec:categ-bfp-gener}

Let $\bfP$ be the free linear symmetric strict monoidal
category on a Casimir Hopf algebra
$(\sfP,c)=(\sfP,\mu,\eta,\Delta,\epsilon,S,c)$.  Thus, as a
linear symmetric strict monoidal category, $\bfP$ is generated by
the object $\sfP$ and the morphisms $\mu$, $\eta$, $\Delta$,
$\epsilon$, $S$ and $c$, and
all the relations in $\bfP$ are derived from the axioms of a
linear symmetric monoidal category and the relations
\eqref{h1}--\eqref{h5} and \eqref{e:add}--\eqref{e:inv}.
In other words, $\bfP$ is the linear PROP (see \cite{Markl}) governing Casimir Hopf algebras.
Define a grading of $\bfP$ by
\begin{gather*}
  \deg(\mu)=\deg(\eta)=\deg(\Delta)=\deg(\epsilon)=\deg(S)=0,\quad \deg(c)=1.
\end{gather*}
For $m\ge0$, we identify the object $\sfP^{\ot m}$ with $m$.

The category $\bfP$ has the following universal property.  If $\calC$
is a linear symmetric strict monoidal category and $(H,c)$ is a
Casimir Hopf algebra in $\calC$, then there is a unique linear
symmetric monoidal functor $F=F_{(H,c)}: \bfP\to \calC$ which
maps the Casimir Hopf algebra $(\sfP,c)$ in $\bfP$ to the Casimir Hopf
algebra $(H,c)$ in $\calC$.

Consequently, since $\AB$ has a Casimir Hopf algebra $(H,c)$ by
Proposition \ref{r14}, there is a unique ({graded})
linear symmetric  monoidal functor
\begin{gather*}
  F=F_{(H,c)}:\bfP \longrightarrow \AB
\end{gather*}
mapping $(\sfP,c)$ to $(H,c)$.  To prove Theorem~\ref{AB-pres},
we need to show that $F$ is an isomorphism.

\subsection{The space $W(m,n)$ of tensor words}  \label{sec:space-wm-n}

{Let $m,n\ge0$ be integers.} 
A \emph{tensor word} from $m$ to $n$ of
\emph{degree} $k$ is an expression of the form
\begin{gather*}
  \w= w_1\ott w_n,
\end{gather*}
where
\begin{itemize}
\item for each $i\in\{1,\ldots,n\}$, $w_i$ is a word in the symbols
  \begin{gather*}
    \{x_j,x_j^{-1}\;|\;1\le j\le m\}\cup\{c'_p,c''_p\;|\;1\le p\le  k\},
  \end{gather*}
\item each of $c'_p,c''_p$ ($1\le p\le k$) appears in the concatenated
  word $w_1\cdots w_n$ exactly once.
\end{itemize}
In this case we write $\w: m\to n$.
For example,
\begin{gather}
  \label{e5}
  \w = x_1c'_1c'_2 \ot c'_3c''_1c''_3x_2 \ot x_1^{-1}c''_2x_2x_1:  2 \longrightarrow 3
\end{gather}
is a tensor word of degree $3$.
As we will see below, the symbols $x_j^{\pm1}$ may be considered as
elements of the free  group $\free{n}$ on $x_1,\dots,x_n$.

Two tensor words $\w,\w': m\to n$ are \emph{equivalent} if  they have the same degree $k$ and
they are related by a permutation of $\{1,\ldots,k\}$.  For example,
the above $\w$ is equivalent to the tensor word
\begin{gather*}
  (12)\w :=  x_1c'_2c'_1 \ot c'_3c''_2c''_3x_2 \ot x_1^{-1}c''_1x_2x_1:  2 \longrightarrow  3
\end{gather*}
obtained from $\w$ by exchanging $(c'_1,c''_1)$ and $(c'_2,c''_2)$.
Let $[\w]$ denote the equivalence class of $\w$.  Let $\tW(m,n)$
denote the vector space with basis consisting of equivalence classes of
tensor words from $m$ to $n$.

Let $W(m,n)$ denote the quotient space of $\tW(m,n)$ by the subspace
generated by the following elements:
\begin{itemize}
\item (\emph{chord orientation}) \ $[\w]-[\w']$, where $\w$ and $\w'$
  differ by interchanging $c'_p$ and $c''_p$ for some $p$,
\item (\emph{cancellation}) \ $[\w]-[\w']$, where $\w$ and $\w'$ differ  locally as
  \begin{align*}
    \w&=(\cdots x_i^{\pm1}x_i^{\mp1} \cdots),\\
    \w'&=(\cdots \phantom{x_i^{\pm1}x_i^{\mp1}} \cdots),
  \end{align*}
  for some $i$,
\item (\emph{bead slide}) \ $[\w]-[\w']$, where $\w$ and $\w'$ differ locally as
  \begin{align*}
    \w&=(\cdots x_i c'_p \cdots x_i c''_p \cdots),\\
    \w'&=(\cdots c'_p x_i\cdots c''_p x_i \cdots),
  \end{align*}
      for some $i$ and $p$,
\item (\emph{4T}) \ $[\w_1]-[\w_2]-[\w_3]+[\w_4]$ and
  $[\w_1]-[\w_2]-[\w_5]+[\w_6]$, where $\w_1,\ldots,\w_6$
  differ locally as
  \begin{align*}
    \w_1&= (\cdots c'_i\cdots c''_ic'_{i'}\cdots c''_{i'}\cdots),\\
    \w_2&= (\cdots c'_i\cdots c'_{i'}c''_i\cdots c''_{i'}\cdots),\\
    \w_3&= (\cdots c'_i\cdots c'_{i'}\cdots c''_{i'}c''_i\cdots),\\
    \w_4&= (\cdots c'_i\cdots c'_{i'}\cdots c''_ic''_{i'}\cdots),\\
    \w_5&= (\cdots c'_ic'_{i'}\cdots c''_{i'}\cdots c''_i\cdots),\\
    \w_6&= (\cdots c'_{i'}c'_i\cdots c''_{i'}\cdots c''_i\cdots).
  \end{align*}
  for some $i,i'$ with $i\neq i'$.
\end{itemize}
In the above expressions, each  $\cdots$ means a subexpression of a tensor word possibly
containing the tensor signs.

\subsection{The isomorphism $\tau: W(m,n)\to\AB(m,n)$} \label{sec:admiss-chord-diagr}

By an \emph{admissible chord diagram} from $m$ to $n$ of degree $k$
we mean a restricted $\free{m}$-colored chord diagram $D$ on $\Xn = \capn$ with $k$ chords,
with each bead in $D$  labelled by one of $x_1^{\pm1},\ldots,x_m^{\pm1}$.

We define a linear map $$ \ttau: \tW(m,n) \longrightarrow \AB(m,n)$$ as follows.
Given a tensor word $\w=w_1\ott w_n: m\to n$ of degree $k$, put the
symbols appearing in each $w_i$ on the $i$th strand $\capl_i$ (in the order inverse to the orientation) and, for
each $j=1,\ldots,k$, connect the two points labelled by $c'_j$ and
$c''_j$ with a chord.  Then we obtain an admissible chord diagram
$\ttau(\w)$, regarded as an element of $\AB_k(m,n)$.
For example, for $\w$ in \eqref{e5} we have
\begin{gather}
  \label{e2}
  \def\zza{\small $x_1$}\def\zzb{\small $x_2$}\def\zzc{\small $x_1^{-1}$}\def\zzd{\small $x_2$}\def\zze{\small $x_1$}
  \ttau(\w)=
\scalebox{0.85}{\raisebox{-8.5ex}{\begin{picture}(0,0)%
\includegraphics{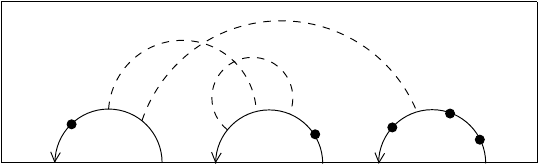}%
\end{picture}%
\setlength{\unitlength}{3947sp}%
\begingroup\makeatletter\ifx\SetFigFont\undefined%
\gdef\SetFigFont#1#2#3#4#5{%
  \reset@font\fontsize{#1}{#2pt}%
  \fontfamily{#3}\fontseries{#4}\fontshape{#5}%
  \selectfont}%
\fi\endgroup%
\begin{picture}(4309,1313)(-1061,-9044)
\put(1819,-8667){\makebox(0,0)[lb]{\smash{{\SetFigFont{10}{12.0}{\familydefault}{\mddefault}{\updefault}{\color[rgb]{0,0,0}\zzc}%
}}}}
\put(-674,-8611){\makebox(0,0)[lb]{\smash{{\SetFigFont{10}{12.0}{\familydefault}{\mddefault}{\updefault}{\color[rgb]{0,0,0}\zza}%
}}}}
\put(751,-8911){\makebox(0,0)[lb]{\smash{{\SetFigFont{10}{12.0}{\familydefault}{\mddefault}{\updefault}{\color[rgb]{0,0,0}$c'_3$}%
}}}}
\put(-74,-8836){\makebox(0,0)[lb]{\smash{{\SetFigFont{10}{12.0}{\familydefault}{\mddefault}{\updefault}{\color[rgb]{0,0,0}$c'_2$}%
}}}}
\put(-299,-8761){\makebox(0,0)[lb]{\smash{{\SetFigFont{10}{12.0}{\familydefault}{\mddefault}{\updefault}{\color[rgb]{0,0,0}$c'_1$}%
}}}}
\put(1135,-8822){\makebox(0,0)[lb]{\smash{{\SetFigFont{10}{12.0}{\familydefault}{\mddefault}{\updefault}{\color[rgb]{0,0,0}$c''_3$}%
}}}}
\put(2845,-8819){\makebox(0,0)[lb]{\smash{{\SetFigFont{10}{12.0}{\familydefault}{\mddefault}{\updefault}{\color[rgb]{0,0,0}\zze}%
}}}}
\put(2593,-8558){\makebox(0,0)[lb]{\smash{{\SetFigFont{10}{12.0}{\familydefault}{\mddefault}{\updefault}{\color[rgb]{0,0,0}\zzd}%
}}}}
\put(906,-8787){\makebox(0,0)[lb]{\smash{{\SetFigFont{10}{12.0}{\familydefault}{\mddefault}{\updefault}{\color[rgb]{0,0,0}$c''_1$}%
}}}}
\put(1450,-8690){\makebox(0,0)[lb]{\smash{{\SetFigFont{10}{12.0}{\familydefault}{\mddefault}{\updefault}{\color[rgb]{0,0,0}\zzb}%
}}}}
\put(2202,-8799){\makebox(0,0)[lb]{\smash{{\SetFigFont{10}{12.0}{\familydefault}{\mddefault}{\updefault}{\color[rgb]{0,0,0}$c''_2$}%
}}}}
\end{picture}%
}}.
\end{gather}

\begin{lemma}
  \label{r7}
  The map $\ttau$ is surjective, and induces an isomorphism
  \begin{gather*}
    \tau: W(m,n) \longrightarrow \AB(m,n).
  \end{gather*}
\end{lemma}

\begin{proof}
  The lemma follows directly from the isomorphism
  $u^{\mathrm{ch}}:\mathcal{A}^{\mathrm{ch,r}}(X_n,\free{m})\to\mathcal{A}^{\mathrm{ch}}(X_n,\free{m})$
  obtained in Theorem~\ref{th:chord_Jac}.
\end{proof}

\subsection{The map $\alpha: W(m,n)\to\bfP(m,n)$}    \label{sec:map-alpha:-wm}

We will assign to every tensor word $\w: m\to n$ of degree $k$
a morphism $\tilde\alpha(\w): m\to n$ in $\bfP$ of degree $k$.
For example, for the tensor word $\w:2\to3$ of degree $3$ in \eqref{e5},
corresponding to the admissible chord diagram $\tilde\tau(\w)$ in
\eqref{e2}, we have graphically
\begin{gather*}
\tilde\alpha(\w)=
\raisebox{-12.5ex}{\begin{picture}(0,0)%
\includegraphics{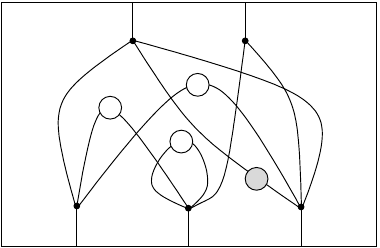}%
\end{picture}%
\setlength{\unitlength}{3947sp}%
\begingroup\makeatletter\ifx\SetFigFont\undefined%
\gdef\SetFigFont#1#2#3#4#5{%
  \reset@font\fontsize{#1}{#2pt}%
  \fontfamily{#3}\fontseries{#4}\fontshape{#5}%
  \selectfont}%
\fi\endgroup%
\begin{picture}(3024,1974)(-161,-11323)
\put(935,-9586){\makebox(0,0)[lb]{\smash{{\SetFigFont{10}{12.0}{\familydefault}{\mddefault}{\updefault}{\color[rgb]{0,0,0}\small $x_1$}%
}}}}
\put(1832,-9598){\makebox(0,0)[lb]{\smash{{\SetFigFont{10}{12.0}{\familydefault}{\mddefault}{\updefault}{\color[rgb]{0,0,0}\small $x_2$}%
}}}}
\put(674,-10238){\makebox(0,0)[lb]{\smash{{\SetFigFont{10}{12.0}{\familydefault}{\mddefault}{\updefault}{\color[rgb]{0,0,0}\small $c$}%
}}}}
\put(1386,-10057){\makebox(0,0)[lb]{\smash{{\SetFigFont{10}{12.0}{\familydefault}{\mddefault}{\updefault}{\color[rgb]{0,0,0}\small $c$}%
}}}}
\put(1245,-10506){\makebox(0,0)[lb]{\smash{{\SetFigFont{10}{12.0}{\familydefault}{\mddefault}{\updefault}{\color[rgb]{0,0,0}\small $c$}%
}}}}
\put(1844,-10818){\makebox(0,0)[lb]{\smash{{\SetFigFont{10}{12.0}{\familydefault}{\mddefault}{\updefault}{\color[rgb]{0,0,0}\small $S$}%
}}}}
\end{picture}%
}.
\end{gather*}
In general, the diagram representing $\tilde\alpha(\w)$ has $m$
edges $e_1,\ldots,e_m$ at the top (corresponding to the generators
$x_1,\dots,x_m$) and $n$ edges $e'_1,\ldots,e'_n$ at the bottom. For
each $j=1,\ldots,n$, the bottom edge $e'_j$ is locally attached to
$l'_j$ input edges, where $l'_j$ is the length of $w_j$.  In our
example, we have $(l'_1,l'_2,l'_3)=(3,4,4)$.  For each $i=1,\dots,m$,
the top edge $e_i$ is locally attached to as many output edges as the
number $l_i$ of occurrences of $x_i^{\pm 1}$ in $\w$.  In our
example, we have $(l_1,l_2)=(3,2)$.  Moreover, the diagram contains
$k$ ``caps'' labelled by $c$ encoding $k$ copies of $c: 0\to2$.
Each ``cap'' has two output ends.  The outputs of top edges, the
inputs of bottom edges and the outputs of ``caps'' are connected by
using the following rules:
\begin{itemize}
\item
If the $r$th symbol in $w_j$ ($1\le j\le n$, $1\le r\le l'_j$) is
$x_i^\epsilon$ ($\epsilon=\pm1$), then the $r$th input at $e'_j$ is
connected by an arc to one of the outputs of $e_i$.  If $\epsilon=-1$
here, then a label $S$ is added to the arc to encode the antipode
$S: 1 \to 1$.
\item
If the $r$th symbol in $w_j$ ($1\le j\le n$, $1\le r\le l'_j$) is
$c'_p$ (resp.\ $c''_p$), with $1\le p\le k$, then the $r$th input
at $e'_j$ is connected by an arc to the left (resp.\ right)
output of the $p$th ``cap''.
\end{itemize}
We interpret the diagram thus obtained  as a morphism in $\bfP$ in
the usual way.  At the top we have the tensor product of $m$
multi-output comultiplications, and at the bottom we have the
tensor product of $n$ multi-input multiplications.

These rules yield a well-defined {morphism $\tilde\alpha(\w):m\to n$ in $\bfP$}.  
Indeed, the only possible ambiguities
are the ordering of the outputs at each top edge, the positions
of the ``caps'' between the top and bottom, and the choices of the
connecting arcs.  Independence of $\tilde\alpha(\w)$ from those
choices follows from the cocommutativity of $\sfP$ and the general
properties of symmetric monoidal categories.  Thus we obtain a linear map
$$
\tilde\alpha: \tW(m,n)    \longrightarrow  \bfP(m,n)
$$
defined by $\w \mapsto \alpha(\w)$ on generators.

\begin{lemma}
  \label{r5}
	The map $\tilde\alpha$ induces a  linear map $\alpha: W(m,n) \to \bfP(m,n)$.
\end{lemma}

\begin{proof}
  It suffices to check that each of the relations defining the vector space  $W(m,n)$ as a quotient of  $\tW(m,n)$ in Section \ref{sec:space-wm-n}
 is mapped to $0$ in $\bfP$.  Indeed,
  \begin{itemize}
  \item the ``chord orientation'' relation is mapped to $0$ because of the symmetry axiom~\eqref{e:sym},
  \item the ``cancellation'' relation  is mapped to $0$ because of the antipode relation~\eqref{h4},
  \item the ``bead slide'' relation  is mapped to $0$ because of \eqref{e:inv2},
  \item the ``4T'' relation is mapped to $0$ because of \eqref{e:4T}.
  \end{itemize}
\end{proof}

\subsection{Surjectivity of $\alpha$}   \label{sec:surjectivity-alpha}

For $n \ge 0$, let $\fS_n$ denote the symmetric group of order~$n$.
Define a homomorphism
\begin{gather}
  \label{e10}
  \fS_n\longrightarrow \bfP(n,n),\quad \sigma\longmapsto P_\sigma
\end{gather}
by $P_{(i,i+1)} = \id_{i-1}\ot P_{1,1}\ot \id_{n-i-1}$ for  $i\in\{1,\dots, n-1\}$. Set
\begin{gather*}
  \mu^{[q_1,\ldots,q_n]}  =  \mu^{[q_1]}\ott\mu^{[q_n]},\quad
  \Delta^{[p_1,\ldots,p_m]}   =  \Delta^{[p_1]}\ott\Delta^{[p_m]},
\end{gather*}
for $q_1,\ldots,q_n,p_1,\ldots,p_m \ge 0$.

\begin{lemma}
  \label{r13}
  Let $m,n \ge0$.  Every homogeneous element of $\bfP(m,n)$ of
  degree $k$ is a linear combination of morphisms of the form
  \begin{gather}
    \label{e42}
    \begin{split}
   &\mu ^{[q_1,\ldots,q_n]}P_\sigma(S^{e_1}\otimes \cdots\otimes S^{e_s}\ot\id_{2k})
    (\id_s\ot  c^{\ot k})\Delta^{[p_1,\dots,p_m]}
    \\
    = \ &\mu ^{[q_1,\ldots,q_n]}P_\sigma
    \big((S^{e_1}\otimes \cdots\otimes S^{e_s})\Delta^{[p_1,\dots,p_m]}\ot c^{\ot k}\big),
    \end{split}
  \end{gather}
  where $s,p_1,\ldots,p_m,q_1,\ldots,q_n\ge 0$ with
  $s=p_1+\cdots+p_m=q_1+\cdots+q_n-2k$, $e_1,\ldots,e_s\in\{0,1\}$ and
  $\sigma \in\fS_{s+2k}$.
\end{lemma}

\begin{proof}
  We adapt the proof of \cite[Lemma 2]{Habiro3}.  The main difference
  here is that our category is a linear category, and we have
  an extra morphism~$c$.

  Let $\bfP ^0$ (resp.\ $\bfP ^+$, $\bfP ^-$, $\bfP^c$) denote the
  linear monoidal subcategory of $\bfP $ generated by the object
  $1$ and the set of morphisms $\{P_{1,1},S\}$
  (resp.\ $\{\mu ,\eta \}$, $\{\Delta ,\epsilon \}$, $\{c\}$).  We also use
  the symbol $\bfP^*$ (with $*=0,+,-,c$) to denote the set
  $\coprod_{m,n \ge 0}\bfP^*(m,n)$.

  We will consider compositions of such spaces.  For instance,
  $\bfP^+\bfP^0$ denotes the subset of $\bfP$ consisting of all
  well-defined linear combinations of compositions $f^+f^0$ of
  composable pairs of morphisms $f^+\in\bfP^+$ and $f^0\in\bfP^0$.

  First, we will prove
  \begin{gather}
    \label{e30}
    \bfP=\bfP^+\, \bfP^0\, \bfP^c\, \bfP^-.
  \end{gather}
  For $i\ge0$, let $\bfP_i$ denote the degree $i$ part of $\bfP$.
   If $i>0$, then $\bfP_i$ is the product $\bfP_1\cdots\bfP_1$ of $i$ copies of $\bfP_1$.
  Set $\bfP^c_i=\bfP^c\cap\bfP_i$.
  Note that $\bfP_0$ is the linear  symmetric monoidal subcategory of $\bfP$ generated by $\mu,\eta,\Delta,\epsilon,S$.
  Thus, the proof of \cite[Lemma 2]{Habiro3} gives
  \begin{gather}
    \label{e32}
    \bfP^-\bfP^+\subset\bfP^+\bfP^0\bfP^-,\quad
    \bfP^0\bfP^+\subset\bfP^+\bfP^0,\quad
    \bfP^-\bfP^0\subset\bfP^0\bfP^-,\\
    \label{e37}
    \bfP_0=\bfP^+\bfP^0\bfP^-.
  \end{gather}
  For $\bfP^c$, we have
  \begin{gather}
    \label{e31}
    \bfP^c\bfP^+\subset\bfP^+\bfP^c,\quad
    \bfP^c\bfP^0\subset\bfP^0\bfP^c,\\
    \label{e39}
    \bfP^-\bfP^c\subset\bfP^+\bfP^0\bfP^c\bfP^-.
  \end{gather}
  Here \eqref{e31} easily follows.   To prove \eqref{e39}, we use
  \begin{gather*}
    \bfP^-\bfP^c_1\subset\bfP^+\bfP^0\bfP^c_1\bfP^-,
  \end{gather*}
  which we can check using \eqref{e:add} and \eqref{e13'}--\eqref{e13''}.
  Then, proceeding by induction on $i\geq 1$ and using
  \eqref{e32}--\eqref{e31}, we obtain  $ \bfP^-\bfP^c_i
  \subset\bfP^+\bfP^0\bfP^c_i\bfP^-$.  This implies~\eqref{e39}.

 Using the inclusions obtained so far, we can check that $\bfP^+\bfP^0\bfP^c\bfP^-$ is closed under  composition, i.e.,
\begin{gather*}
  (\bfP^+\bfP^0\bfP^c\bfP^-)\, (\bfP^+\bfP^0\bfP^c\bfP^-)\subset\bfP^+\bfP^0\bfP^c\bfP^-.
\end{gather*}
Since $\bfP^+\bfP^0\bfP^c\bfP^-$ contains the identity morphisms, it is a
  linear subcategory of~$\bfP$.  Since $\bfP^+\bfP^0\bfP^c\bfP^-$
  contains $\bfP^+$, $\bfP^0$, $\bfP^c$ and $\bfP^-$, we obtain  \eqref{e30}.

  Let $k\ge0$.
  Homogeneous elements of $\bfP^0\bfP^c$ of degree $k$ are linear
  combinations of morphisms of the form
  $$
  \big(S^{e_1}\otimes \cdots\otimes S^{e_{s+2k}}\big)P_\sigma (\id_s\ot  c^{\ot k}),
  $$
  where $s \ge0 $, $e_1,\ldots,e_{s+2k}\in\{0,1\}$ and  $\sigma \in\fS_{s+2k}$.
  Thus, by \eqref{e30}, every homogeneous element of
  $\bfP(m,n)$ of degree $k$ is a linear combination of morphisms of
  the form
  \begin{gather}
    \label{e20}
     \mu ^{[q_1,\ldots,q_n]}P_\sigma(S^{e_1}\otimes \cdots\otimes S^{e_{s+2k}})
    (\id_s\ot  c^{\ot k})\Delta^{[p_1,\dots,p_m]},
  \end{gather}
  where  $p_1,\ldots,p_m,q_1,\ldots,q_n \ge0$ are such that   $p_1+\cdots+p_m=q_1+\cdots+q_n-2k$,
  $s := p_1+\cdots +p_m$, $e_1,\ldots,e_{s+2k}\in\{0,1\}$ and  $\sigma \in\fS_{s+2k}$.
    Since  \eqref{e13'''} gives
  \begin{gather*}
    (S^{e_1}\otimes \cdots\otimes S^{e_{s+2k}})(\id_s\ot c^{\ot k})
    = \pm  (S^{e_1}\otimes \cdots\otimes
    S^{e_s}\ot \id_{2k})(\id_s\ot c^{\ot k}),
  \end{gather*}
  a morphism of the form  \eqref{e20} is, up to sign, also of the form \eqref{e42}.
\end{proof}

\begin{lemma}
  \label{r12}
  The linear map  $\alpha: W(m,n)\to\bfP(m,n)$ is surjective.
\end{lemma}

\begin{proof}
  Let {$f:m\to n$ in $\bfP$ be as in \eqref{e42}.}
  Define a tensor
  word $\w: m\to n$ by
  \begin{gather*}
    \w=
    u_1\cdots u_{q_1}\ot
    u_{q_1+1}\cdots u_{q_1+q_2}\ott
    u_{q_1+\cdots+q_{n-1}+1}\cdots u_{q_1+\cdots+q_{n-1}+q_n},
  \end{gather*}
  where $u_j :=v_{\sigma^{-1}(j)}$ with
  \begin{gather*}
    v_j:=
    \begin{cases}
      x_{a(j)}^{(-1)^{e_j}}&
      (j=1,\dots, s),
      \\
      c'_{(j-s+1)/2}&
      (j=s+1,s+3,\dots,s+2k-1),
      \\
      c''_{(j-s)/2}&
      (j=s+2,s+4,\dots,s+2k)
    \end{cases}
  \end{gather*}
  for $j=1,\dots,q_1+\dots+q_n$.
  Here we define the map $a:\{1,\dots,s\}\to\{1,\dots,m\}$ by
  \begin{gather*}
    a(j)  =  \max  \big\{a\in\{1,\dots,m\}\;|\;j\le p_1+\dots+p_a  \big\}.
  \end{gather*}
  Then one can check  $\alpha([\w])=f$.  Hence, by Lemma \ref{r13}, $\alpha$ is surjective.
\end{proof}

\subsection{Proof of Theorem \ref{AB-pres}}    \label{sec:proof-theorem-refab}

Let $m,n \ge0$.   Consider the diagram
\begin{gather}
  \label{e11}
  \vcenter{\xymatrix{
    \bfP(m,n)
    \ar[rr]^F
    &&
    \AB(m,n){.}
    \\
    &
    W(m,n)
    \ar[ru]^{\tau}_{\cong }
    \ar@{->>}[lu]_{\alpha}
    &
   }}
\end{gather}
By Lemma \ref{r7}, $\tau$ is an isomorphism and, by Lemma
\ref{r12}, $\alpha$ is surjective.  Thus, to prove that  $F$ is an
isomorphism it suffices to prove that the diagram \eqref{e11} commutes.

We have factorization of morphisms in $\AB$ similar to Lemma
\ref{r13} for $\bfP$.

\begin{lemma}
  \label{r17}
   Every homogeneous element of $\AB(m,n)$ of degree $k$
 is a linear combination of morphisms of the form
  \begin{gather}
    \label{e44}
    \begin{split}
    f&=\mu ^{[q_1,\ldots,q_n]}P_\sigma(S^{e_1}\otimes \cdots\otimes S^{e_s}\ot\id_{2k})
    (\id_s\ot  c^{\ot k})\Delta^{[p_1,\dots,p_m]},
    \end{split}
  \end{gather}
  where $s,p_1,\ldots,p_m,q_1,\ldots,q_n\ge 0$ with
  $s=p_1+\cdots+p_m=q_1+\cdots+q_n-2k$, $e_1,\ldots,e_s\in\{0,1\}$ and
  $\sigma \in\fS_{s+2k}$.
\end{lemma}

\begin{proof}
  This follows from the surjectivity of $\tau: W(m,n)\to  \AB(m,n)$.
\end{proof}

Now one can easily see that, for {every} $f\in\AB(m,n)$ of degree $k$ decomposed as in \eqref{e44},
we have $F\alpha\tau^{-1}(f)=f$.  Hence the diagram \eqref{e11} commutes.  This completes the proof of  Theorem \ref{AB-pres}.

%
%
\section{A ribbon quasi-Hopf algebra in $\widehat{\AB}$} \label{sec:quasi-Hopf}

In this section, we construct a ribbon quasi-Hopf algebra in
$\widehat{\AB}$ for each choice of a Drinfeld associator.

\subsection{Ribbon quasi-Hopf algebras}    \label{sec:ribbon-quasi-hopf-1}

We recall the notions of quasi-triangular and ribbon quasi-Hopf
algebras in symmetric monoidal categories.  See \cite{Kassel} for an
introduction to quasi-triangular quasi-Hopf algebras, and see
\cite{AC,Sommerhauser} for their ribbon versions.

Let $\calC$ be a (possibly linear) symmetric strict monoidal
category with monoidal unit $I$ and symmetry $P_{X,Y}: X\ot Y \to
Y\ot X$.  Let $(H,\mu,\eta)$ be an algebra in $\calC$.  We defined the
convolution product $*$ on $\C(I,H^{\ot n})$ in \eqref{e12}.  For
$X\in\Ob(\C)$, we can extend $*$ to
\begin{gather*}
  *: \C(I,H^{\ot n})\times\C(X,H^{\ot n})\longrightarrow \C(X,H^{\ot n}),\quad
 g*f:=\mu_n(g\ot f), \\
  *: \C(X,H^{\ot n})\times\C(I,H^{\ot n})\longrightarrow  \C(X,H^{\ot n}),\quad
   f*h :=\mu_n(f\ot h).
\end{gather*}
Thus, the convolution monoid $\C(I,H^{\ot n})$ acts on $\C(X,H^{\ot n})$ from both left and right.
These actions commute, i.e., $(g*f)*h=g*(f*h)$.

A \emph{quasi-bialgebra} $H$ in $\calC$ is an algebra
$(H,\mu,\eta)$ equipped with morphisms of algebras
\begin{gather*}
  \Delta: H\longrightarrow H\ot H,\quad \epsilon: H\longrightarrow  I,
\end{gather*}
and a convolution-invertible morphism
\begin{gather*}
  \varphi: I\longrightarrow H^{\ot3}
\end{gather*}
such that
\begin{gather}
  (\epsilon \otimes \id ) \Delta =  \id = (\id \otimes \epsilon  ) \Delta, \\
  \label{e45}
   (\id\ot\Delta)\Delta  =    \varphi * (\Delta\ot\id)\Delta * \varphi^{-1},\\
    \label{e47}
	  (\id\ot\varphi)*  (\id\ot\Delta\ot\id)\varphi *(\varphi\ot\id)
       =
       (\id\ot\id\ot\Delta)\varphi *  (\Delta\ot\id\ot\id)\varphi ,\\
    \label{e48}
    (\id\ot\epsilon\ot\id)\varphi=\eta\ot\eta.
\end{gather}
 A quasi-bialgebra $H$ is  \emph{cocommutative} if
$P_{H,H} \Delta=\Delta$, and \emph{special} if we~have
\begin{gather}
  \label{e29}
  \mu(\id\ot x)=\mu(x \ot\id)\quad \text{for every $x: I \to H$}.
\end{gather}

A \emph{quasi-Hopf algebra} is a quasi-bialgebra $H$
equipped with an algebra anti-automorphism
\begin{gather*}
  S : H\longrightarrow H
\end{gather*}
and convolution-invertible morphisms
\begin{gather*}
  \alpha,\beta: I\longrightarrow H
\end{gather*}
such that
\begin{gather}
  \label{e50}
  \mu^{[3]}(S\ot\alpha\ot\id)\Delta=\alpha\epsilon,\quad
  \mu^{[3]}(\id\ot\beta\ot S)\Delta=\beta\epsilon,\\
  \label{e52}
  \mu^{[5]}(\id\ot\beta\ot S\ot\alpha\ot\id)\varphi=\eta,\quad
  \mu^{[5]}(S\ot\alpha\ot\id\ot\beta\ot S)\varphi^{-1}=\eta.
\end{gather}

A quasi-Hopf algebra $H$ is \emph{quasi-triangular} if it
is equipped with a convolution-invertible morphism
\begin{gather*}
  R: I\longrightarrow H\ot H
\end{gather*}
such that
\begin{gather}
  \label{e54}
  R * \Delta * R^{-1}=  P_{H,H}\Delta, \\
  \label{e55}
  (\Delta\ot\id)R=\varphi_{321}* R_{13} * \varphi^{-1}_{132} * R_{23} * \varphi_{123},\\
  \label{e56}
  (\id\ot\Delta)R=\varphi_{231}^{-1}* R_{13} * \varphi_{213} * R_{12} * \varphi_{123}^{-1},
\end{gather}
where we set
\begin{gather*}
  R_{12}=R\ot\eta,\quad
  R_{13}=(\id\ot\eta\ot\id)R,\quad
  R_{23}=\eta\ot R
\end{gather*}
and $\varphi^{\pm1}_{ijk}=P_{\binom{1,2,3}{i,j,k}}\varphi^{\pm1}$.
Here {$P_\sigma:H^{\ot3}\to H^{\ot3}$} for $\sigma\in\fS_3$ is the
permutation morphism defined similarly to \eqref{e10}.
A quasi-triangular quasi-Hopf algebra $H$ is
\emph{triangular} if $R_{21}=R$, where we set $R_{21}=P_{H,H}R:I\to H\ot H$.

{We can view every} cocommutative Hopf algebra $(H,\mu,\eta,\Delta,\epsilon,S)$ 
as a quasi-triangular quasi-Hopf algebra by setting $\varphi=
\eta^{\otimes 3}$, $\alpha = \beta= \eta$ and $R=\eta^{\otimes  2}$.

A quasi-triangular quasi-Hopf algebra $H$  is
\emph{ribbon} if it admits a convolution-invertible morphism
$$
\bfr:I \longrightarrow H
$$
such that
\begin{eqnarray}
\label{e200} S \bfr &=& \bfr, \\
\label{e201} \Delta \bfr &=&  R_{21} \ast R \ast (\bfr \otimes \bfr).
\end{eqnarray}

\subsection{Kohno--Drinfeld Lie algebras and associators} \label{sec:drinfeld-associators}

We recall the definition of Drinfeld associators.

For $n\ge 0$, the \emph{Kohno--Drinfeld Lie algebra}
$\mathfrak{t}_n$ is the Lie algebra over $\K$ generated by $t_{ij}$ ($i,j \in\{1,\dots, n\}$, $i\neq j$) with relations
\begin{gather*}
  t_{ij}=t_{ji},\quad [t_{ij}, t_{ik}+t_{jk}]=0\quad
  (\text{$i,j,k$ distinct}), \quad
  [t_{ij},t_{kl}]=0\quad \text{($i,j,k,l$ distinct)}.
\end{gather*}

\renewcommand{\nstrands}{\downarrow^{\otimes n}}

We regard the universal enveloping algebra $U(\mathfrak{t}_n)$
of $\mathfrak{t}_n$ as a subalgebra of the algebra $\A(\nstrands)\subset \A(+^{\otimes n},+^{\otimes n})$
of Jacobi diagrams on $\nstrands:={\downarrow\cdots\downarrow}$,
via the injective algebra homomorphism
\begin{equation} \label{eq:hcd}
  U(\mathfrak{t}_n)\longrightarrow \A(\nstrands)
\end{equation}
that maps each $t_{ij}$ to the chord diagram with a chord connecting the $i$-th and $j$-th strings.
(See \cite[Cor\. 4.4]{BN3} or \cite[Rem\. 16.2]{HM} for the injectivity of  \eqref{eq:hcd}.)

Let $\KXY$ denote the algebra of formal power series in two
non-commuting generators.  As usual, $\KXY$ is a
complete Hopf algebra, with $X$ and $Y$ primitive.  A \emph{Drinfeld
associator} is a group-like element $\varphi(X,Y)\in\KXY$ such that
\begin{gather}
  \label{e21}
  \varphi(t_{12},t_{23}+t_{24})\,
  \varphi(t_{13}+t_{23},t_{34})
  =
  \varphi(t_{23},t_{34})\,
  \varphi(t_{12} +t_{13},t_{24}+t_{34})\,
  \varphi(t_{12},t_{23}),
  \\
  \label{e23}
  \exp\big(\frac{t_{13}+t_{23}}2\big)
  =
  \varphi(t_{13},t_{12})\,
  \exp\big(\frac{t_{13}}2\big)\,
  \varphi(t_{13},t_{23})^{-1}
  \exp\big(\frac{t_{23}}2\big)\,
  \varphi(t_{12},t_{23}),
  \\
  \label{e24}
  \exp\big(\frac{t_{12}+t_{13}}2\big)
  =
  \varphi(t_{23},t_{13})^{-1}
  \exp\big(\frac{t_{13}}2\big)\,
  \varphi(t_{12},t_{13})\,
  \exp\big(\frac{t_{12}}2\big)\,
  \varphi(t_{12},t_{23})^{-1}.
\end{gather}

\begin{remark} \label{rem:associator}
A Drinfeld associator $\varphi(X,Y)$
gives rise to an associator  $\Phi:=  \varphi(t_{12},t_{23})^{-1} \in\A(\downarrow \downarrow \downarrow)$
 in the sense of Section \ref{sec:usual_Kontsevich}.
\end{remark}

\subsection{A ribbon quasi-Hopf algebra in $\hA$}   \label{sec:ribbon-quasi-hopf}

We consider $\AB(0,n)$ ($n\ge0$) as {an} algebra with the
  convolution product $*$.  We have an algebra isomorphism
$$
\iota: \A(\nstrands) \longrightarrow  \AB(0,n), \quad
\centre{\labellist
\scriptsize\hair 2pt
 \pinlabel{$\cdots$}  at 84 192
 \pinlabel{$\cdots$}  at 83 30
  \pinlabel{$D$}  at 81 110
\endlabellist
\centering
\includegraphics[scale=0.15]{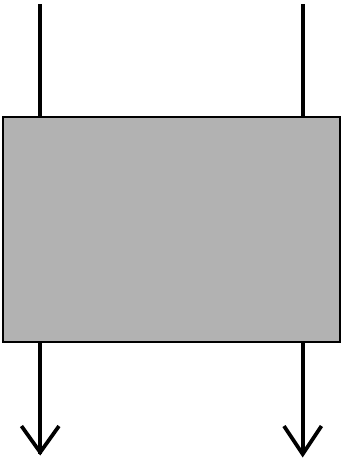}}
\longmapsto
\centre{\labellist
\scriptsize\hair 2pt
 \pinlabel{$\cdots$}  at 79 101
 \pinlabel{$D$}  at 76 179
 \pinlabel{$\cdots$}  at 80 274
 \pinlabel{$\cdots$}  at 310 179
\endlabellist
\centering
\includegraphics[scale=0.15]{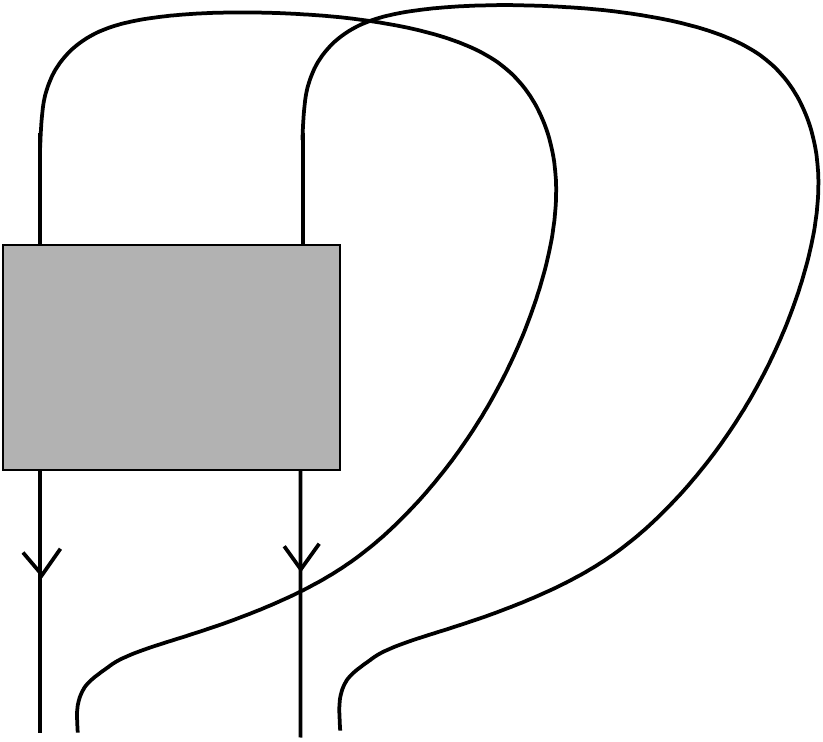}}.
$$
For  $1\le i<j\le n$, set
$$
c_{ij} :=  \iota(t_{ij})  = \big(\eta^{\ot (i-1)}\ot\id\ot\eta^{\ot (
  i-j-1)}\ot\id\ot\eta^{\ot (n-j)}\big)\, c \;\; :0\lto n.
$$

Clearly, the cocommutative Hopf algebra structure $(1,\mu,\eta,\Delta,
\epsilon, S)$ in $\AB$ given in Proposition \ref{r14} induces a
cocommutative Hopf algebra structure in the degree-completion
$\widehat{\AB}$ of $\AB$.

Let $\varphi(X,Y)\in\KXY$ be a Drinfeld associator.  {Define morphisms in $\hA$:}
\begin{eqnarray}
\label{eq:phi*} \varphi & = &  \iota\big(\varphi(t_{12},t_{23}) \big)
  = \varphi(c_{12},c_{23}) \ :0\lto3, \\
\label{eq:R*} R & = & \iota\big(\exp \big( \frac{1}{2} t_{12}\big)
\big) = \exp_\ast( c/2)  \ :0\lto2, \\
\label{eq:r*}  \bfr &=& \iota\big(  \exp\big(\frac{1}{2}
\figtotext{10}{20}{t_11} \big) \big)=  \exp_\ast(  \mu c/2 ) \ :0\lto1.
\end{eqnarray}
Set $\Phi = \varphi(t_{12},t_{23})^{-1} \in\A(\downarrow
\downarrow \downarrow)$, and define {$\nu :0\to 1$}
by
\begin{equation} \label{eq:nu_def}
\nu  =  \iota \left(
\begin{array}{c}
\labellist
\scriptsize\hair 2pt
 \pinlabel{$S_2(\Phi^{-1})$}   at 147 227
\endlabellist
\centering
\includegraphics[scale=0.11]{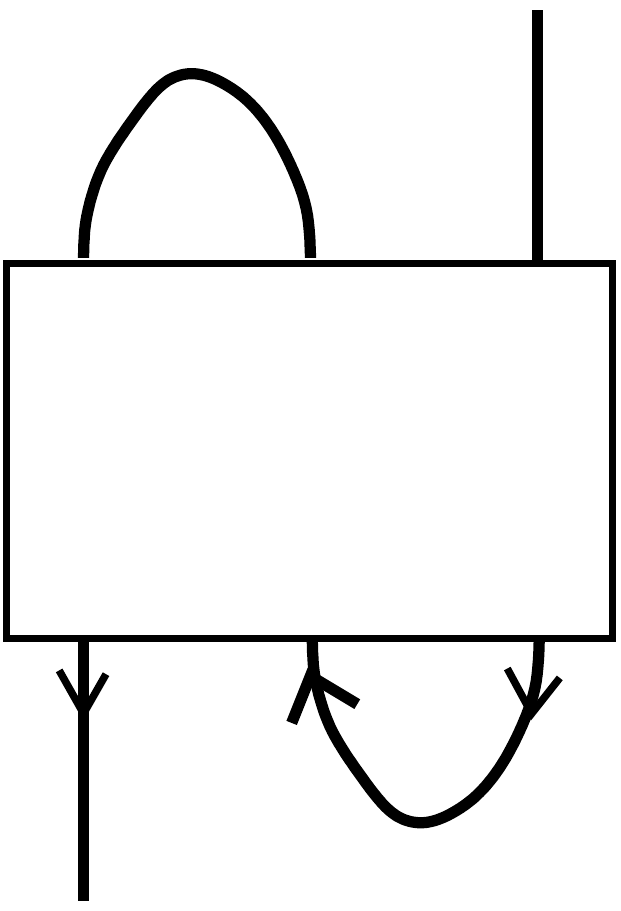} \end{array}\right)^{-1}
\quad \hbox{or, equivalently, by} \quad
\nu  =  \iota \left(
\begin{array}{c}
\labellist
\scriptsize\hair 2pt
 \pinlabel{$S_2(\Phi)$}   at 147 227
\endlabellist
\centering
\includegraphics[scale=0.11]{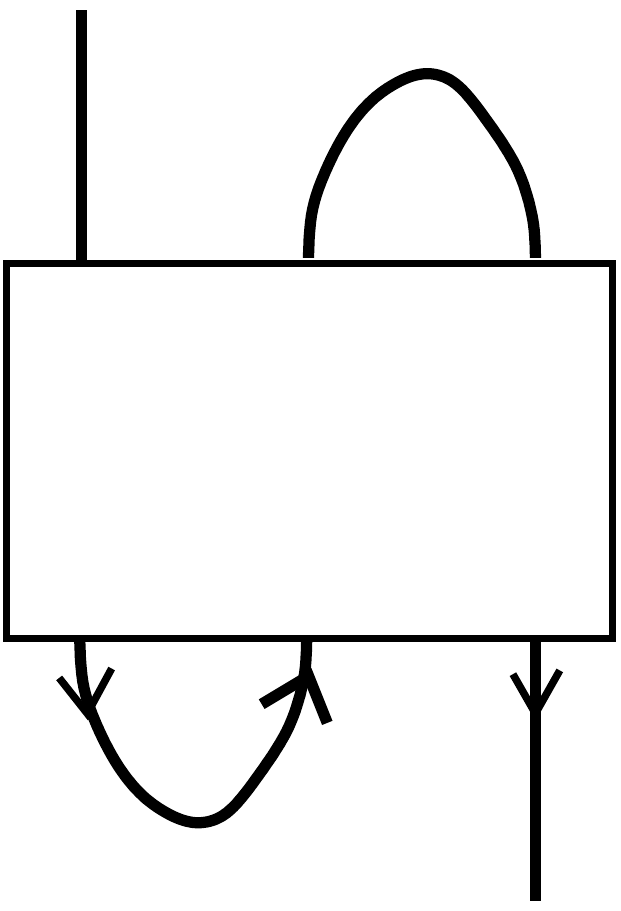} \end{array}\right)^{-1}
\end{equation}
where $S_2: \A(\downarrow \downarrow \downarrow) \to \A(\downarrow
\uparrow \downarrow)$ is the diagrammatic ``orientation-reversal
operation'' applied to the second string.  (The equivalence of
those two definitions of $\nu$ follows from
$\varphi(t_{23},t_{12}) = \varphi(t_{12},t_{23})^{-1}$, which is a
consequence of \eqref{e23}--\eqref{e24}; see \cite[Proposition~3.7]{BN1}.)

Let {$\be:0\to1$ in $\hA$} be convolution-invertible (equivalently, $\ep\be\neq0$),
and let $\alpha=\nu*\be^{-1}$. {We denote}
\begin{gather*}
  H_{\varphi,\be}=(1,\mu,\eta,\Delta,\epsilon,\varphi,S,\alpha,\beta,R,\bfr).
\end{gather*}
{and, for $\be =\eta$, we set $H_\varphi = H_{\varphi,\eta}$.}

\begin{theorem}  \label{prop:r3}
  For each Drinfeld associator $\varphi=\varphi(X,Y)$ and each
  convolution-invertible {$\be:0\to1$}, $H_{\varphi,\be}$ is a
  (triangular, cocommutative,  special) ribbon quasi-Hopf algebra in
  $\widehat{\AB}$.
\end{theorem}

\begin{proof}
To prove that $H_{\varphi,\be}$ is a ribbon quasi-Hopf algebra,
  it suffices to check \eqref{e45}--\eqref{e201} since
  $(1,\mu,\eta,\Delta,\epsilon,S)$ is a Hopf algebra.  First,
  note that
\begin{equation} \label{eq:general_fact}
\Delta^{[n]}\ast x = x \ast \Delta^{[n]}\quad \text{for $n\ge0$ and  $x:0\to n$}.
\end{equation}

We obtain \eqref{e45} and \eqref{e50} from \eqref{eq:general_fact}.
We obtain \eqref{e47} from the pentagon equation \eqref{e21}.
We obtain \eqref{e48} from $\varphi(0,0)=1$, which holds
since $\varphi{(X,Y)}$ is group-like.
We obtain \eqref{e52} from \eqref{eq:general_fact}, the well-known
  identity $S( \iota^{-1}(\nu) )={\iota^{-1}(\nu)} \in\A(\downarrow)$    and
$$
\mu^{[3]}(\id \otimes S \otimes \id)\varphi = \nu^{-1} = \mu^{[3]}(\id \otimes S \otimes \id) \varphi^{-1}_{321},
$$
which follows from \eqref{eq:nu_def}.
We obtain \eqref{e54} from \eqref{eq:general_fact} and cocommutativity of~$\Delta$.
We obtain \eqref{e55} and \eqref{e56} from the hexagon equations
\eqref{e23} and \eqref{e24}.
We obtain \eqref{e200} as follows:
$$
S \bfr = S  \exp_\ast(  \mu c/2 ) = \exp_\ast S (  \mu c/2 ) = \exp_\ast (  \mu (S \otimes S) P_{1,1} c/2 ) = \bfr.
$$

To obtain \eqref{e201}, let us apply to it the algebra isomorphism $\iota^{-1}$.
We have
\begin{gather*}
  \begin{split}
&\iota^{-1}(\Delta \bfr) \;= \;C_{++} (\iota^{-1}( \bfr) )\; =\; \exp C_{++} \left(\frac{1}{2} \figtotext{10}{20}{t_11} \right) 
=\exp \left( \frac12  \figtotext{25}{25}{t_1122} \right) \\
=&\exp \left(
 \frac12 \figtotext{15}{25}{t_110}
+\frac12 \figtotext{15}{25}{t_011}
+        \figtotext{12}{25}{t_12}\right)
=
\exp\left(          \!\figtotext{12}{25}{t_12}\!\right)
\exp\left( \frac12  \!\figtotext{15}{25}{t_110}\!\right)
\exp\left( \frac12  \!\figtotext{15}{25}{t_011}\!\right),
  \end{split}
\end{gather*}
where the second equality follows from Lemma
\ref{rem:functoriality_cabling_bis}, and the last equality follows
since $\figtotext{11}{16}{t_110}$, $\figtotext{11}{16}{t_011}$, $\figtotext{10}{16}{t_12}$ mutually commute.
Since $H_{\varphi,\be}$ is triangular, we have
\begin{eqnarray*}
\iota^{-1}\big( R_{21} \ast R \ast (\bfr \otimes \bfr) \big) &=& \big(\iota^{-1}(R)\big)^2 \, \big(\iota^{-1}(\bfr) \otimes \iota^{-1}(\bfr) \big) \\
&=& \exp\left(\figtotext{8}{16}{t_12}\right)\  \left( \exp\left(\frac{1}{2} \figtotext{8}{16}{t_11} \right)  \otimes \exp\left(\frac{1}{2} \figtotext{8}{16}{t_11} \right) \right).
\end{eqnarray*}
Hence we have \eqref{e201}.
\end{proof}

The universal property of $\AB$ (Theorem \ref{AB-pres})
implies the following generalization of Theorem \ref{prop:r3}.
Let $\C$ be a linear symmetric strict monoidal category equipped
with a filtration $\C= \calF^0 \supset \calF^1 \supset \calF^2 \supset
\cdots$.  Let $\widehat{\C}^\calF=\plim_k \C/\calF^k$ be the
completion of $\C$ with respect to $\calF$, and let
$$
j: \C  \longrightarrow \widehat{\C}^\calF
$$
be the canonical functor.
(See Section \ref{sec:terminology} for a brief review of filtrations and completions.)

\begin{cor}
  \label{r2}
 Let  $(H,c)$ be a Casimir Hopf algebra in $\C$ and assume that   $c\in\calF^1(H^{\ot0},H^{\ot2})$.
  Then there is a  unique {continuous} linear symmetric monoidal functor
  \begin{gather*}
    F_{(H,c)}:\widehat{\AB}\longrightarrow \widehat{\C}^\calF
    \end{gather*}
   that maps the Casimir Hopf algebra in $\widehat{\AB}$ to  $j(H,c)$.
  Therefore, $ F_{(H,c)} $ maps the ribbon quasi-Hopf algebra   in
  $\widehat{\AB}$ to a  ribbon quasi-Hopf algebra in $\widehat{\C}^\calF$.
\end{cor}

\begin{remark} \label{deformation}
  We can consider the quasi-triangular quasi-Hopf algebra
  $$
  {H_\varphi = (1,\mu,\eta,\Delta,\epsilon,\varphi,S,\nu,\eta,R)}
  $$
  as a deformation of the cocommutative Hopf
  algebra $H_0:=(1,\mu,\eta,\Delta,\epsilon,S)$ {in the following way}.
   Let $s\in\K$.  An \emph{$s$-associator} is a group-like element
  $\varphi(X,Y)\in\KXY$ satisfying
  the pentagon relation \eqref{e21}
  and the following two hexagon relations:
  \begin{gather}
    \label{e23-s}
    \exp\big(\frac{s(t_{13}+t_{23})}2\big)
    =
    \varphi(t_{13},t_{12})\,
    \exp\big(\frac{st_{13}}2\big)\,
    \varphi(t_{13},t_{23})^{-1}
    \exp\big(\frac{st_{23}}2\big)\,
    \varphi(t_{12},t_{23}),
    \\
    \label{e24-s}
    \exp\big(\frac{s(t_{12}+t_{13})}2\big)
    =
    \varphi(t_{23},t_{13})^{-1}
    \exp\big(\frac{st_{13}}2\big)\,
    \varphi(t_{12},t_{13})\,
    \exp\big(\frac{st_{12}}2\big)\,
    \varphi(t_{12},t_{23})^{-1}.
  \end{gather}
  Note that a $1$-associator is a Drinfeld associator in the sense of
  Section \ref{sec:drinfeld-associators}, and that $0$-associators
  constitute the so-called \emph{Grothendieck--Teichm\"uller group}.
  (In fact, Furusho \cite{Furusho} proved that if $\varphi(X,Y)$
  satisfies \eqref{e21}, then it  satisfies \eqref{e23-s}
  and \eqref{e24-s} for some $s$ in the algebraic closure of $\K$.)
  Given an $s$-associator $\varphi_s(X,Y)$, define {$\varphi_s:0\to3$
  and $\nu_s:0\to1$} by \eqref{eq:phi*} and \eqref{eq:nu_def},
  respectively, and define {$R_s:0\to2$}  by \eqref{eq:R*}
  with $c$ replaced with $sc$.  Then, by a proof completely parallel to that of
  Theorem \ref{prop:r3}, it follows that
  $$
  H_s := (1,\mu,\et,\De,\ep,\varphi_s,S,\nu_s,\et,R_s)
  $$
  is a quasi-triangular quasi-Hopf algebra  in $\hA$.
  Assume now that $\varphi$ is a Drinfeld associator.
  Then  $\varphi_s(X,Y):= {\varphi(sX,sY)}$
 is an $s$-associator for {every} {$s\in\mathbb K$},
  so that {$\{H_s\}_{s\in\mathbb K}$} is a one-parameter family of quasi-triangular quasi-Hopf algebras.
{We have $H_1=H_\varphi$ and $H_0$ is a cocommutative Hopf algebra.}
\end{remark}

%
%
\section{Weight systems} \label{sec:weight_systems}

We illustrate the results of the previous two sections by considering
weight systems, which transform Jacobi diagrams into linear maps.

\subsection{Casimir Lie algebras and weight systems} \label{sec:casimir-lie-algebras}

A \emph{Casimir Lie algebra} is a Lie algebra $\g$ (over $\K$)
equipped with an {ad-}invariant, symmetric $2$-tensor $c\in\g\ot \g$.  Then
the universal enveloping algebra $U(\g)$ of $\g$ together with
$c\in\g^{\ot2}\subset U(\g)^{\ot2}$ is a Casimir Hopf algebra in the
category of vector spaces.

Consequently, $(U(\g),c)$ is also a Casimir Hopf algebra in ${\mathbf{M}_{\g}} := \UgMod$, 
the linear symmetric strict monoidal category of
$U(\g)$-modules.  The universal property of~$\AB$ gives a unique
linear symmetric monoidal functor
$$
W_{(\g,c)}: \AB \longrightarrow \mathbf{M}_{\g}
$$
mapping the Casimir Hopf algebra $(H,c)$ in $\AB$ to the Casimir
Hopf algebra $(U(\g),c)$.
We call $W_{(\g,c)}$ the \emph{weight system} of the Casimir Lie algebra $(\g,c)$.

\begin{example}
  (1) A \emph{quadratic Lie algebra} is a pair $(\g,\kappa)$ of a
  finite-dimensional Lie algebra $\g$ and a non-degenerate symmetric
  {ad-}invariant bilinear form $\kappa : \g \times \g \to \K$.
  Let $c_\kappa \in\g \otimes \g$ be the $2$-tensor
  corresponding to $\kappa$ via
  $$
 \Hom(\g \otimes \g, \K)
 \cong \Hom(\g , \K) \otimes \Hom(\g , \K) \cong \g \otimes \g,
  $$
 with {the second isomorphism}  induced by~$\kappa$.  Then $(\g,c_\kappa)$ is a Casimir Lie
 algebra, and hence $(U(\g),c_\kappa)$ is a Casimir Hopf algebra in
 the category of vector spaces.

(2) The \emph{Cartan trivector} $T_\kappa \in\g^{\ot3}$ of
$(\g,\kappa)$ is the skew-symmetric, {ad-}invariant $3$-tensor
corresponding to the trilinear form
$$
\g \otimes \g \otimes \g \longrightarrow \K, \quad (x,y,z) \longmapsto \kappa([x,y],z).
$$
{We have}
$$
T_\kappa  =  W_{(\g,{c_\kappa})}\Big( \figtotext{60}{40}{Y}\Big) \  \in U(\g)^{\otimes 3}.
$$
\end{example}

\subsection{Continuous weight systems} \label{sec:cont-weight-syst}

Let $\K[[h]]$ be the formal power series algebra.  For a
  vector space $V$, let $V[[h]]$ denote the $h$-adic completion
  of $V\ot\K[[h]]$.

We fix a Casimir Lie algebra $(\g,c)$.
Let $\mathbf{M}_{\g}[[h]]$ be  the $\K[[h]]$-linear symmetric strict monoidal category such that
 $\Ob(\mathbf{M}_{\g}[[h]])=\Ob(\mathbf{M}_{\g})$ and
\begin{gather*}
  \mathbf{M}_{\g}[[h]](V,W ) = \mathbf{M}_{\g}( V,W )[[h]]
\end{gather*}
for $V,W\in{\Ob(\mathbf{M}_{\g}[[h]])}$.  The
composition in the category $\mathbf{M}_{\g}[[h]]$ and its symmetric strict monoidal
structure are inherited from $\mathbf{M}_{\g}$ in the obvious way.

Since $(\Ug,c)$ is a Casimir Hopf algebra in $\mathbf{M}_{\g}$, so is $(\Ug,hc)$ in $\mathbf{M}_{\g}[[h]]$.  By the
universal property of~$\AB$, there is a unique linear symmetric
monoidal functor
\begin{gather*}
  W_{(\g,hc)}:\AB \longrightarrow \mathbf{M}_{\g}[[h]]
\end{gather*}
such that $W_{(\g,hc)}(m)=\Ug^{\ot m}$ for $m \ge 0$ and which maps the
Casimir Hopf algebra $(H,c)$ in $\AB$ to the Casimir Hopf algebra
$(\Ug,hc)$ in $\mathbf{M}_{\g}[[h]]$.
By continuity, the functor $W_{(\g,hc)}$ above extends uniquely as
\begin{gather*}
  \hW_{(\g,hc)}:\hA \longrightarrow \mathbf{M}_{\g}[[h]].
\end{gather*}

\begin{remark}
  \label{rem1}
  (1)  Let $\g$ be a Lie algebra. One could also work within  $\UghMod$, the category of
  $\Ugh$-modules, instead of $\mathbf{M}_{\g}[[h]]$.
  In fact, there is a canonical linear functor
  \begin{gather*}
    i:\mathbf{M}_{\g}[[h]] \longrightarrow \UghMod
  \end{gather*}
  which maps each $U(\g)$-module $V$ to $V[[h]]$ and maps each $f:V\to
   W$ in $\mathbf{M}_{\g}[[h]]$, i.e., $f\in\mathbf{M}_{\g}(V,W)[[h]]$, to the map
   $i(f):V[[h]]\to W[[h]]$ induced by $f$.  Since the functor $i$ is
   fully faithful, we may regard $\mathbf{M}_{\g}[[h]]$ as a subcategory of
   $\UghMod$.

   (2) Let $(\g,c)$ be a Casimir Lie algebra.  Then, by Corollary \ref{r2}, the composition
   \begin{gather*}
     i\circ \hW_{(\g,hc)}: \widehat{\AB} \longrightarrow \UghMod
   \end{gather*}
   maps the ribbon quasi-Hopf algebra in $\hA$ (see Theorem
    \ref{prop:r3}) to a ribbon quasi-Hopf algebra structure on
    $U(\g)[[h]]$.  This structure is known from Drinfeld's {work \cite{Drinfeld2}.}  See
    \cite[Theorem XIX.4.2]{Kassel} for the description of the
    underlying quasi-triangular quasi-bialgebra structure.
\end{remark}

%
%
\section{Construction of the functor $Z$} \label{sec:extended_Kontsevich}

In this section, we construct  a functor
\begin{gather}
  Z:\Bq  \longrightarrow \hA,
\end{gather}
from the non-strictification $\Bq$ of $\B$ to the degree-completion
$\hA$ of $\AB$.

\subsection{The category $\Bq$ of bottom $q$-tangles in handlebodies} \label{sec:category-bq-bottom}

 Define the category $\Bq$ of bottom $q$-tangles in handlebodies
to be the non-strictification (see Section~\ref{sec:tangles}) of
the strict monoidal category $\B$.  Here we identify $\Ob(\B)=\NZ$
with the free monoid $\Mon(\bu)$ on an element $\bu$.
Hence we have $\Ob(\Bq)=\Mag(\bu)$, the free unital magma on $\bu$.

\begin{example}
For  $w\in\Mag(\bullet)$, we can regard  $\Bq(\varnothing,w)$ as a subset of
$\T_q\big(\varnothing,w(+-)\big)$, where $ w(+-)\in\Mag(\pm) $
is obtained from $w$ by substituting $\bu=(+-)$.
\end{example}

\subsection{The extended Kontsevich integral $Z$}  \label{sec:extend-konts-integr}

The rest of this section is devoted to the proof of following result.

\begin{theorem} \label{th:extended_Kontsevich}
 There is a functor $Z: \Bq \to \hA$ such that
\begin{itemize}
\item[(i)] for $w\in\Mag(\bu)$, we have $Z(w)=|w|$,
\item[(ii)] if $T\in\Bq(\varnothing,w) \subset
\T_q\big(\varnothing,w(+-)\big)$, $w\in\Mag(\bullet)$, then the
value of $Z$ on $T$ is the usual Kontsevich integral $Z(T)$, as
defined in Section \ref{sec:usual_Kontsevich},
\item[(iii)] $Z$ is tensor-preserving, i.e., we have
  \begin{gather*}
    Z(T \otimes T') = Z(T) \otimes Z(T')
  \end{gather*}
  for  morphisms $T$ and $T'$ in $\Bq$.
\end{itemize}
\end{theorem}

 The functor $Z:\Bq\to \hA$ is \emph{not} a monoidal functor.  By
replacing the target monoidal category $\hA$ with an appropriate
``parenthesized'' version $\hAq^\varphi$, we can make $Z$ into a
braided monoidal functor $\Zqphi:\Bq\to\hAq^\varphi$; see Section \ref{sec:A_q}.

\subsection{Notations}  \label{sec:notations}

Let $\calC$ be a monoidal category with {(left)} duals. The dual of $x\in\Ob(\calC)$ is denoted by $x^*$.
For $x\in\Ob(\calC)$, set
\begin{gather*}
  \dbl (x) = x\otimes x^* \in\Ob(\calC).
\end{gather*}
We extend this definition to finite sequences of objects of $\calC$ as follows.

First, assume that the monoidal category $\calC$ is strict.  For   $\ul{x}=(x_1,\ldots ,x_k) \in\Ob(\calC)^k$ ($k\ge1$), set
\begin{gather*}
  \dbl (\ul{x}) = \dbl (x_1,\ldots ,x_k) :=\dbl (x_1)\otimes \dots \otimes \dbl (x_k) \in\Ob(\calC ).
\end{gather*}
Now, assume that the monoidal category $\calC$ is non-strict.
For   $w \in\Mag(\bullet)$ of length $k \ge1$
and   $\ul{y}=(y_1,\ldots ,y_k)\in\Ob(\calC )^k$,
let $w (\ul{y}) =  w (y_1,\ldots ,y_k)\in\Ob(\calC )$ be the object obtained from $w $
by replacing its $k$ consecutive letters by $y_1,\ldots ,y_k$  in this order.
(For instance, if $k=3$ and $w =((\bu\bu)\bu)$, then
$w (y_1,y_2,y_3)=(y_1\otimes y_2)\otimes y_3$.)   Set
\begin{gather*}
  \dbl ^w (\ul{x})=\dbl ^w (x_1,\ldots ,x_k) := w (\dbl (x_1),\ldots ,\dbl (x_k))
\end{gather*}
for $\ul{x}=(x_1,\ldots ,x_k)\in\Ob(\calC )^k$.

{Moreover}, we will need the following notation when the monoidal
category $\calC$ is non-strict.  Let $w \in\Mag(\bullet)$ of length
$k$, let $\underline{x}=(x_1,\dots,x_k), \underline{y}=(y_1,\dots,y_k)
\in\Ob(\calC )^k$ and let
{$f_1:x_1\to y_1, \dots, f_k:x_k\to y_k$ be morphisms in $\calC$.}
Then
$$
w(f_1,\ldots,f_k) :w(\underline{x})\lto w(\underline{y})
$$
denotes the ``$w$-parenthesized'' tensor product of $f_1,\dots,f_k$.

\subsection{Construction of $Z$} \label{sec:construction_Z}

Let $m,n\ge0$.  We decompose the handlebody $V_m$ as
\begin{equation}  \label{eq:handlebody}
V_m = \underbrace{([-1,1]^2 \times [0,7/8] )}_{\hbox{\small a ``lower'' copy of the cube}}
 \cup\hspace{10pt} \underbrace{ \big( ([-1,1]^2 \times [7/8,1] ) \cup (\hbox{$m$ $1$-handles})\big)}_{\hbox{\small an ``upper'' copy of $V_m$}}.
\end{equation}
{For every} {$T:m\to n$ in $\B$} {there is} 
an $n$-component  bottom tangle in $V_m$ such that
\begin{itemize}
\item it intersects transversally the square $[-1,1]^2 \times \{7/8\}$
in finitely many points uniformly distributed along the line $[-1,1] \times \{0\} \times \{7/8\}$,
\item its intersection with the ``upper'' part of \eqref{eq:handlebody} consists of finitely many parallel copies of the cores of the $1$-handles.
\end{itemize}
Then $T$ is determined by the intersection of this representative
tangle with the ``lower'' part of \eqref{eq:handlebody}, which defines
a tangle
$$
U :\dbl(v_1,\dots,v_m) \lto (+-)^n\quad \text{in $\T$}
$$
for some $v_1,\dots,v_m \in\Mon(\pm)$.  We call $U$ a \emph{cube presentation} of  $T$.

\begin{example}
  \label{r22}
Here is  a $3$-component bottom tangle $T$ in $V_2$, together with a cube presentation $U$  where  $v_1=v_2=(++)$:
\begin{gather}
  \label{e63}
\centre{
\labellist
\small\hair 2pt
 \pinlabel{$\leadsto$}  at 110 22
 \pinlabel{$T=$} [r] at -2 22
 \pinlabel{$=U$} [l] at 219 22
\pinlabel{\scriptsize $++$} at 151 46
\pinlabel{\scriptsize $--$} at 168 46
\pinlabel{\scriptsize $++$} at 187 46
\pinlabel{\scriptsize $--$} at 203 46
\pinlabel{\scriptsize $+-$} [t] at 196 2
\pinlabel{\scriptsize $+-$} [t] at 160 2
\pinlabel{\scriptsize $+-$} [t] at 175 2
\endlabellist
\centering
\includegraphics[scale=1.0]{cube_presentation}
}
\end{gather}

\end{example}

If $T$ is  upgraded to {a morphism $T:v\to w$ in $\Bq$ with $|v|=m$, $|w|=n$,}
and if
$v_1,\dots,v_m$ are upgraded to  $v_1,\dots,v_m \in\Mag(\pm)$, then we call
$$
U:\dbl^v(v_1,\dots,v_m)\lto w(+-)\quad \text{in $\Tq$}
$$
a \emph{cube presentation} of the bottom $q$-tangle $T$.

\begin{definition}
\label{r25}
Let $T:v\to w$ in $\Bq$ with $m=|v|$, $n=|w|$, and   let
$$U:\dbl^v(v_1,\dots,v_m)\lto w(+-)\quad \text{in $\Tq$}$$  be a cube presentation of $T$.
The \emph{extended Kontsevich integral} of $T$ is the {morphism}
$$
Z(T): m\lto n\quad \text{in $\hA$}
$$
with  square presentation
$$
Z(U) \circ (a_{v_1} \otimes \id_{v_1^*} \otimes \cdots \otimes a_{v_m} \otimes \id_{v_m^*})
:\dbl(v_1,\dots,v_m)\lto(+-)^n\quad \text{in $\AT$}.
$$
Thus, diagrammatically, we have\\
\begin{gather}
  \label{e7}
\quad \ \ \labellist
\scriptsize\hair 2pt
 \pinlabel{$U$}  at 140 50
 \pinlabel{$Z(U)$}  at 503 50
 \pinlabel{$Z\Bigg($} [r] at 0 61
 \pinlabel{$\Bigg)$} [l] at 289 60
 \pinlabel{$=$}  at 324 59
 \pinlabel{$\underbrace{\hphantom{aaaaaaaaaaaaaaaaaa}}_w$} [t] at 144 4
 \pinlabel{$\overbrace{\hphantom{aaaaaaaaaaaaaaaaaaaaaaaaaaaa}}^{v}$} [b] at 146 150
 \pinlabel{$\cdots$}  at 143 8
 \pinlabel{$\cdots$}  at 143 133
 \pinlabel{\tiny $\stackrel{v_1}{\cdots}$} at 40 98
 \pinlabel{\tiny $\stackrel{v_m}{\cdots}$}  at 193 98
 \pinlabel{$a_{v_1}$}  at 398 99
 \pinlabel{$a_{v_m}$}  at 550 100
 \pinlabel{$\cdots$}  at 507 130
 \pinlabel{$\cdots$}  at 507 10
 \pinlabel{$\underbrace{\hphantom{aaaaaaaaaaaaaaaaaa}}_n$} [t] at 505 4
 \pinlabel{$\overbrace{\hphantom{aaaaaaaaaaaaaaaaaaaaaaaaaaaa}}^{m}$} [b] at 505 150
 \pinlabel{.} at 655 59
\endlabellist
\centering
\includegraphics[scale=0.53]{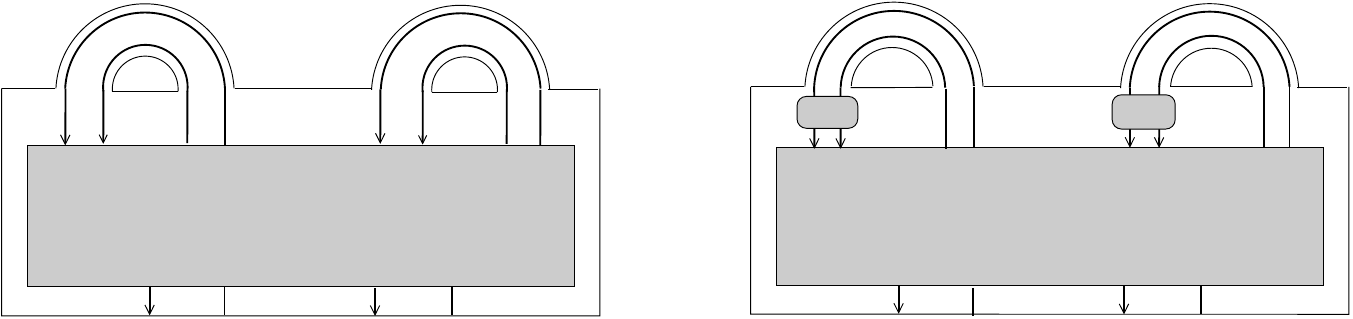}
\end{gather}
\end{definition}

\vspace{0.4cm}
The next lemma shows that the extended Kontsevich integral  is well defined.

\begin{lemma}
Let {$T:v\to w$ in $\Bq$},  $m=|v|$, $n=|w|$. For {all} cube presentations
$$
U:\dbl^v(v_1,\dots, v_m)\lto w(+-), \quad
U':\dbl^v(v'_1,\dots, v'_m)\lto w(+-)\quad \text{in $\Tq$}
$$
of $T$, the {morphisms}
$$
Z(U)\! \circ\! (a_{v_1}\! \otimes \id_{v_1^*} \otimes \cdots\otimes  a_{v_m}\! \otimes \id_{v_m^*}),
\ Z(U')\! \circ\! (a_{v'_1}\! \otimes \id_{(v'_1)^*} \otimes \cdots \otimes  a_{v'_m}\! \otimes \id_{(v'_m)^*})
$$
are square presentations of the same {morphism $m\to n$ in $\hA$.}
\end{lemma}

\begin{proof}
For {$V:x\to y$ in $\Tq$}, let {$r(V): y^*\to x^*$}
denote the $\pi$-rotation of $V$ around the $\vec y$ axis of $\R^3$.

We can realize an isotopy of bottom tangles with cube presentations as
a sequence of isotopies of cube presentations and ``sliding subtangles
through the handles''. (Similar arguments appear in \cite{Lieberum}.)
Thus, without loss of generality, we can assume that there are
{morphisms $T_0:v\big(v_1(v'_1)^*,\ldots,v_m(v'_m)^*\big)\lto w(+-)$
and $T_1:v'_1\to v_1,\; \dots,\; T_m:v'_m\to v_m$ in $\Tq$}
such that $$
\begin{cases}
U = T_0  \circ v\big(\id_{v_1} \otimes r(T_1),\dots,\id_{v_m} \otimes r(T_m)  \big),\\
U' =  T_0  \circ v\big(T_1 \otimes \id_{(v'_1)^*},\dots, T_m \otimes \id_{(v'_m)^*}\big).
\end{cases}
$$
It follows that
$$
\begin{cases}
Z(U) \circ A= Z(T_0)  \circ \big(  a_{v_1} \otimes Z(r(T_1)) \otimes  \cdots \otimes   a_{v_m} \otimes Z( r(T_m) ) \big), \\
Z(U') \circ A' =  Z(T_0)  \circ \big(Z(T_1) a_{v'_1}\otimes \id_{(v'_1)^*} \otimes \cdots \otimes  Z(T_m)a_{v'_m} \otimes \id_{(v'_m)^*} \big),
\end{cases}
$$
where  $A:=(a_{v_1}\! \otimes \id_{ v_1^*} \otimes \cdots\otimes  a_{v_m}\! \otimes \id_{ v_m^* })$,
and $A'$ is defined similarly from the words $v'_1,\dots, v'_m$.
Thus, it suffices to prove that
$$
\labellist
\scriptsize\hair 2pt
 \pinlabel{ $\cdots$}  at 146 92
 \pinlabel{ $a_{v_1}$}  at 40 36
 \pinlabel{ $Z(r(T_1))$}  at 100 35
 \pinlabel{$a_{v_m}$}  at 193 37
 \pinlabel{ $Z(r(T_m))$}  at 252 36
 \pinlabel{$\stackrel{v_1}{\cdots}$}  at 40 55
 \pinlabel{$\stackrel{v_1}{\cdots}$}  at 39 17
 \pinlabel{$\stackrel{v_m}{\cdots}$}  at 192 17
 \pinlabel{$\stackrel{v_m}{\cdots}$}  at 191 59
\endlabellist
\centering
\includegraphics[scale=0.88]{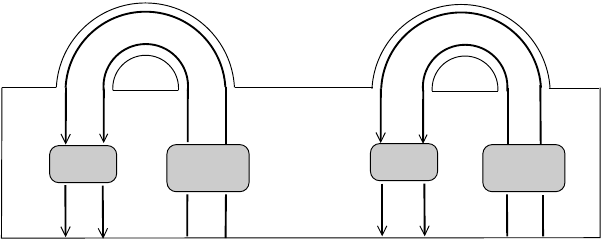}
$$
is equal to
$$
\labellist
\scriptsize\hair 2pt
 \pinlabel{$\cdots$}  at 148 108
 \pinlabel{$a_{v'_1}$}  at 40 62
 \pinlabel{$Z(T_1)$}  at 40 28
 \pinlabel{$a_{v'_m}$}  at 192 62
 \pinlabel{$Z(T_m)$}  at 191 28
 \pinlabel{$\stackrel{v'_1}{\cdots}$}  at 40 79
 \pinlabel{$\stackrel{v'_1}{\cdots}$}  at 41 47
 \pinlabel{$\stackrel{v_1}{\cdots}$}  at 42 11
 \pinlabel{$\stackrel{v'_m}{\cdots}$}  at 192 79
 \pinlabel{$\stackrel{v'_m}{\cdots}$}  at 192 47
 \pinlabel{$\stackrel{v_m}{\cdots}$}  at 193 11
\endlabellist
\centering
\includegraphics[scale=0.9]{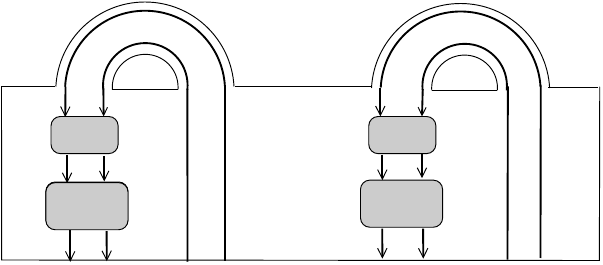}
$$
in the space of $F(x_1,\dots,x_m)$-colored Jacobi diagrams on the appropriate oriented $1$-manifold.
For this, it suffices to show that
$$
\labellist
\small\hair 2pt
 \pinlabel{$rZr(T_i)$}  at 20 61
 \pinlabel{$a_{v_i}$}  at 20 27
 \pinlabel{$a_{v'_i}$}  at 91 62
 \pinlabel{$Z(T_i)$}  at 91 28
 \pinlabel{$=$}  at 55 45
  \pinlabel{ \normalsize $: {v'_i} \lto v_i$ in $\AT$,} [l] at 115 45
 \pinlabel{$\stackrel{v'_i}{\cdots}$}  at 20 80
 \pinlabel{$\stackrel{v_i}{\cdots}$}  at 20 43
 \pinlabel{$\stackrel{v_i}{\cdots}$}  at 20 11
 \pinlabel{$\stackrel{v_i}{\cdots}$}  at 91 11
 \pinlabel{$\stackrel{v'_i}{\cdots}$}  at 92 47
 \pinlabel{$\stackrel{v'_i}{\cdots}$}  at 92 80
\endlabellist
\centering
\includegraphics[scale=1.1]{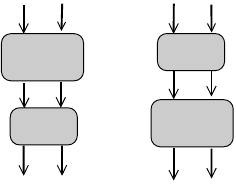}
$$
{which}  follows by applying the usual Kontsevich integral to the following identity of $q$-tangles:
$$
\!\!\!\!\!\!\!\!\!\!\!\!\!\!\!\!\!\!\!\!\!\!\!\!
\labellist
\small\hair 2pt
 \pinlabel{$r(T_i)$}  at 70 24
 \pinlabel{$=$}  at 116 25
 \pinlabel{$T_i$} at 163 26
 \pinlabel{\normalsize$:\varnothing\lto v_i(v'_i)^*$ in $\Tq$.}  [l] at 238 26
  \pinlabel{$\cdots$}  at 13 9
 \pinlabel{$\cdots$}  at 71 9
 \pinlabel{$\cdots$}  at 165 9
 \pinlabel{$\cdots$}  at 225 9
\endlabellist
\centering
\includegraphics[scale=0.8]{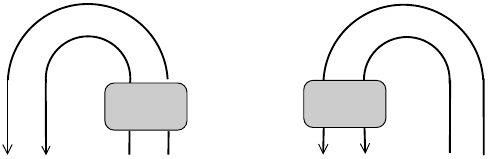}
$$
\end{proof}

Obviously, we have  (ii) in Theorem
\ref{th:extended_Kontsevich}.  We have  (iii) since the usual
Kontsevich integral itself is tensor-preserving.  Therefore, it
remains to prove that $Z$ is functorial.

\subsection{Functoriality of $Z$}   \label{sec:functoriality-z}

To prove that $Z$ is a functor, we need a recurrence formula on the
cabling  anomalies  {$a_w:w\to w$ in $\AT$.}

\begin{lemma} \label{lem:induction_cabling}
For {each} $w\in\Mag(\pm)$ of length $n$ and {each} map $$\varpi:{ \pi_0(\downarrow
\stackrel{w}{\cdots} \downarrow)} =\{1,\dots,n\} \longrightarrow \Mag(\pm),$$ we
have
$$
a_{C_\varpi(w)} = \big( r^{[w_1]}(a_{\varpi(1)})\otimes \cdots \otimes r^{[w_n]}(a_{\varpi(n)}) \big) \circ C_\varpi(a_w)
\ \in\A\big(\downarrow \stackrel{C_\varpi(w)}{\cdots} \downarrow\big),
$$
where {$r^{[+]}=\id$ and $r^{[-]}=r$ with $r$ being the $\pi$-rotation.}
\end{lemma}

\begin{proof}
Setting $x = C_\varpi(w) \in\Mag(\pm)$, we have
$$
\centre{ \labellist
\small \hair 2pt
 \pinlabel{\;$\stackrel{x^*}{\cdots}$}   at 363 65
 \pinlabel{\;$\stackrel{x}{\cdots}$}  at 104 60
 \pinlabel{\;$a_x$}   at 107 184
\endlabellist
\centering
\includegraphics[scale=0.12]{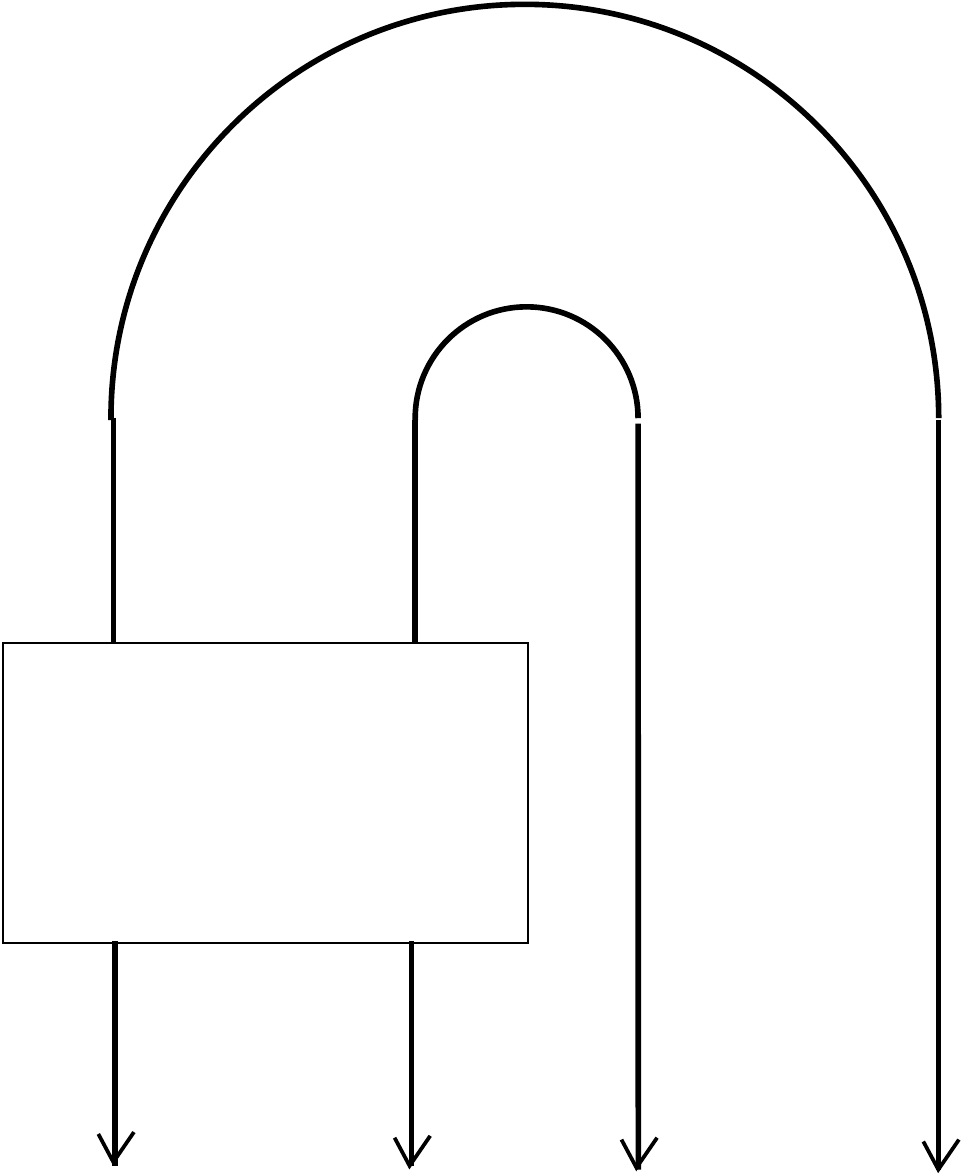}}
= ZC_x\Big(\!\!\!\begin{array}{c}   \figtotext{20}{10}{capleft} \\ {}^{(+-)}  \end{array}\!\!\!\Big)
= Z C_\varpi C_w\Big(\!\!\!\begin{array}{c}   \figtotext{20}{10}{capleft} \\ {}^{(+-)}  \end{array}\!\!\!\Big).
$$
By Lemma \ref{lem:cabling}, we deduce that 
\begin{eqnarray*}
 \centre{ \labellist
\scriptsize \hair 2pt
 \pinlabel{\;$\stackrel{x^*}{\cdots}$}   at 363 65
 \pinlabel{\;$\stackrel{x}{\cdots}$}  at 104 60
 \pinlabel{\;$a_x$}   at 107 184
\endlabellist
\centering
\includegraphics[scale=0.12]{anomaly2}} 
 & = &  c(w w^*,{\varpi_t}) \circ C_\varpi  Z  C_w\Big(\!\!\!\begin{array}{c}   \figtotext{20}{10}{capleft} \\ {}^{(+-)}  \end{array}\!\!\!\Big) \\
 &=& c(w w^*,{\varpi_t}) \circ C_\varpi \Bigg(
\centre{ \labellist
\scriptsize \hair 2pt
 \pinlabel{\;$\stackrel{w^*}{\cdots}$}   at 363 65
 \pinlabel{\;$\stackrel{w}{\cdots}$}  at 104 60
 \pinlabel{\;$a_w$}   at 107 184
\endlabellist
\centering
\includegraphics[scale=0.12]{anomaly2}} \Bigg)\\
&=& c(w w^*,{\varpi_t}) \circ
\centre{ \labellist
\scriptsize \hair 2pt
 \pinlabel{\;$\stackrel{x^*}{\cdots}$}   at 363 65
 \pinlabel{\;$\stackrel{x}{\cdots}$}  at 104 60
 \pinlabel{\;\;$C_\varpi\!(a_w\!)$}   at 107 184
\endlabellist
\centering
\includegraphics[scale=0.12]{anomaly2}} .
\end{eqnarray*}
The series of diagrams $c(w w^*,{\varpi_t})$ is obtained from $\id_w \otimes \id_{w^*}$
by replacing the $i$-th string of $\id_w$ by $a_{\varpi(i)}$ if $w_i=+$  or by {$\id_{\varpi(i)^*}$} if $w_i=-$
and, next, by replacing the $i$-th string of $\id_{w^*}$ by $a_{\varpi({n-i+1})}$ if $w_{n-i+1}=-$ or by 
{$\id_{\varpi(n-i+1)^*}$} if $w_{n-i+1}=+$.
Thus, using the STU relation, we obtain
\begin{eqnarray*}
\centre{ \labellist
\small \hair 2pt
 \pinlabel{\;$\stackrel{x^*}{\cdots}$}   at 363 65
 \pinlabel{\;$\stackrel{x}{\cdots}$}  at 104 60
 \pinlabel{$a_{x}$}   at 107 184
\endlabellist
\centering
\includegraphics[scale=0.12]{anomaly2}}
&=&  \big( r^{[w_1]}(a_{\varpi(1)})\otimes \cdots \otimes r^{[w_n]}(a_{\varpi(n)}) \otimes \id_{x^*} \big) \circ
\centre{ \labellist
\scriptsize \hair 2pt
 \pinlabel{\;$\stackrel{x^*}{\cdots}$}   at 363 65
 \pinlabel{\;$\stackrel{x}{\cdots}$}  at 104 60
 \pinlabel{\;\;$C_\varpi\!(a_w\!)$}   at 107 184
\endlabellist
\centering
\includegraphics[scale=0.12]{anomaly2}}
\end{eqnarray*}
and {we have the conclusion}.
\end{proof}

Let $v\overset{T}{\lto}w\overset{T'}{\lto}x$ in $\Bq$ with $|v|=m$, $|w|=n$, $|x|=p$, and let
$$
U:\dbl^v(v_1,\dots, v_m)\lto w(+-),
\quad
U':\dbl^w(w_1,\dots, w_n)\lto x(+-)\quad \text{in $\Tq$}
$$
be cube presentations of $T$ and $T'$, respectively. Then
$$
U' \circ C_\varpi(U) :\dbl^v(v'_1,\dots,v'_m)\lto x(+-)\quad
  \text{in $\Tq$}
$$
is a cube presentation of {$T' \circ T :v\to x$},
where $\varpi: \pi_0(U) \to \Mag(\pm)$ is an appropriate map
and $v'_1,\dots, v'_m\in\Mag(\pm)$ are such that $C_{{\varpi_s}}(\dbl^v(v_1,\dots,v_m)) = \dbl^v(v'_1,\dots,v'_m)$.
Therefore, $Z(T' \circ T)$ has the following square presentation:
\begin{eqnarray*}
&& Z\big(U' \circ C_\varpi(U)\big) \circ (a_{v'_1}\! \otimes \id_{(v'_1)^*} \otimes \cdots\otimes  a_{v'_m}\! \otimes \id_{(v'_m)^*}\big) \\
&=& Z(U') \circ Z\big(C_\varpi(U)\big)   \circ \big(a_{v'_1}\! \otimes \id_{(v'_1)^*} \otimes \cdots\otimes  a_{v'_m}\! \otimes \id_{(v'_m)^*}\big) \\
&=& Z(U') \circ c\big(w(+-),{\varpi_t}\big) \circ C_{ \varpi} Z(U) \\
&& \circ \, c\big(\dbl^v(v_1,\dots, v_m),{\varpi_s}\big)^{-1} \circ \big(a_{v'_1}\! \otimes \id_{(v'_1)^*} \otimes \cdots\otimes  a_{v'_m}\! \otimes \id_{(v'_m)^*}\big) \\
&=& Z(U') \circ \big(a_{w_1}\! \otimes \id_{ w_1 ^*} \otimes \cdots\otimes  { a_{w_n}\! \otimes \id_{w_n^*}}
  \big) \circ C_{ \varpi} Z(U) \\
&& \circ\, c\big(\dbl^v(v_1,\dots, v_m),{\varpi_s}\big)^{-1} \circ \big(a_{v'_1}\! \otimes \id_{(v'_1)^*} \otimes \cdots\otimes  a_{v'_m}\! \otimes \id_{(v'_m)^*}\big).
\end{eqnarray*}
Here the second identity is given by Lemma \ref{lem:cabling}.
By $m$ applications of Lemma~\ref{lem:induction_cabling}  and using
the STU relation, we obtain the following square presentation of $Z(T' \circ T)$:
\begin{eqnarray*}
&& Z(U') \circ (a_{w_1}\! \otimes \id_{ w_1^* } \otimes \cdots\otimes { a_{w_n}\! \otimes \id_{w_n^*}} ) \circ C_{ \varpi} Z(U) \\
&& \circ\, C_{{\varpi_s}}\big(a_{v_1}\! \otimes \id_{ v_1^* } \otimes \cdots\otimes  a_{v_m}\! \otimes \id_{ v_m^* }\big)\\
&=& Z(U') \circ (a_{w_1}\! \otimes \id_{ w_1^* } \otimes \cdots\otimes { a_{w_n}\! \otimes \id_{w_n^*}}  )  \\
&& \circ\, C_\varpi\big(Z(U) \circ (a_{v_1}\! \otimes \id_{ v_1^* } \otimes \cdots\otimes  a_{v_m}\! \otimes \id_{ v_m^* })\big).
\end{eqnarray*}
By Lemma \ref{ex:restricted}, we have $Z(T' \circ T)= Z(T') \circ Z(T)$.

\subsection{Proof of Theorem \ref{r46}}  \label{sec:proof-theorem-refr46-1}

Consider a morphism $T:v\to w$ in $\Bq$ with a decomposition into  
$q$-tangles $T_0,T_1,\dots,T_m$ as {shown} in \eqref{e100}, {where $m:= \vert v\vert$, $n:= \vert w \vert$ and }
\begin{gather*}
  T_0:v(u_1u_1',\dots,u_mu_m')\lto w(+-),\quad T_i:\vn\lto u_iu_i'\quad
  (i=1,\dots,m)\quad \text{in $\Tq$}.
\end{gather*}
{To deduce Theorem \ref{r46} from Theorem \ref{th:extended_Kontsevich}, it suffices to show that the functor
$ Z:\Bq\to\hA$ resulting from the latter  satisfies \eqref{e25} with $Z^\B:=Z$.}
Let us write
\begin{gather*}
  T = \big[T_0;\; T_1,\dots,T_m \big] = \big[T_0;\; (T_i)_{i=1,\dots,m}\big]
\end{gather*}
and extend this notation $[-;\;-,\dots,-]$ to  other compatible
sequences of $q$-tangles $T'_0,T'_1,\dots,T'_m$.  
{{We use the} same kind of notation for Jacobi diagrams.}
In these notations, what we have to prove is the following:
\begin{gather}
  \label{e62}
  Z(T)=\big[Z(T_0);\;Z(T_1),\dots,Z(T_m)\big].
\end{gather}

For each $i=1,\dots,m$, let $\ti T_i:u_i^*\to u'_i$ be the
unique morphism {in $\Tq$} such that
\begin{gather}
  \label{e65}
  T_i = (\id_{u_i}\ot\ti T_i)\circ \CAP{u_i}.
\end{gather}
Then, we have
\begin{eqnarray*}
   &&[Z(T_0);\;Z(T_1),\dots,Z(T_m)]\\
    &\by{e65} & \left[Z(T_0);\;\big(Z\big((\id_{u_i}\ot\ti T_i)\circ    \CAP{u_i}\big)\big)_{i=1,\dots,m}\right]\\
    &=&\left[Z(T_0);\;\big(Z(\id_{u_i}\ot\ti T_i)\circ Z(\CAP{u_i})\big)_{i=1,\dots,m}\right]\\
   & =&
    \left[Z(T_0)\circ
      \bigotimes_{i=1}^m Z(\id_{u_i}\ot \ti T_i)
      ;\;\left(Z(\CAP{u_i})\right)_{i=1,\dots,m}\right]
    \\
   & =&
    \left[Z\left(T_0\circ
      \bigotimes_{i=1}^m(\id_{u_i}\ot \ti T_i)\right)
      ;\;\left(Z(\CAP{u_i})\right)_{i=1,\dots,m}\right].
\end{eqnarray*}
By Definition \ref{r25} and \eqref{e61}, the right hand side is equal to
\begin{eqnarray*}
    &&Z\left(\left[T_0\circ
      \bigotimes_{i=1}^m\left(\id_{u_i}\ot \ti T_i\right)
      ;\;\left(\CAP{u_i}\right)_{i=1,\dots,m}\right]\right)\\
    &=&Z\left(\left[T_0
      ;\;\big((\id_{u_i}\ot \ti T_i)\circ\CAP{u_i}\big)_{i=1,\dots,m}\right]\right)\\
    & \by{e65} &Z\left(\big[T_0      ;\;\left(T_i\right)_{i=1,\dots,m}\big]\right) \ = \ Z(T).
\end{eqnarray*}
Hence we have \eqref{e62},  {which} completes the proof of Theorem \ref{r46}.

\subsection{Group-like property of $Z$}  \label{sec:group-like-property}

Recall from Section \ref{sec:coalgebras} that the category $\hA$ is
enriched over cocommutative coalgebras, and that there is a monoidal
subcategory $\hAg$ of~$\hA$, the group-like part of $\hA$.

\begin{proposition} \label{prop:group-like}
The extended Kontsevich integral $Z$ takes group-like values,
  i.e., {for} $T:v\to w$ in $\Bq$,  we have
\begin{gather*}
  Z(T)\in\hAg(|v|,|w|).
\end{gather*}
Thus we have a (tensor-preserving) functor $Z:\Bq \to \hAg$.
\end{proposition}

\begin{proof}
Since the usual Kontsevich integral takes group-like values, this
follows from Lemma \ref{ex:coalgebra_structures} and the
definition of $Z(T)$ using a cube presentation of~$T$.
\end{proof}

\subsection{$\bfF$-grading on $Z$}   \label{sec:bff-grading-z}

We recall from Section
\ref{sec:two_gradings} that the linear category $\AB$ is
graded over the opposite of the category $\bfF$ of finitely generated
free groups, and that this grading corresponds to homotopy {classes} of Jacobi diagrams in handlebodies. 
Similarly, we define the   \emph{homotopy {class}} of {$T:m\to n$ in $\B$}
to be the group homomorphism $h(T): \free{n}\to \free{m}$ induced by
$i_T:V_n \to V_m$ on fundamental groups.

\begin{proposition} \label{prop:homotopy_types}
The extended Kontsevich integral $Z$ preserves the homotopy {class}: if
{$T:v\to w$ in $\Bq$,} then we have
$$
Z(T) \in\hA(|v|,|w|)_{h(T)}.
$$
\end{proposition}

\begin{proof}
This follows from the definition of $Z(T)$ using a cube presentation of~$T$.
\end{proof}

%
%
\section{The braided monoidal functor $\Zqphi$ and computation of $Z$}   \label{sec:construction_Zq}

In this section, we assume that the associator $\Phi \in\A(\downarrow\downarrow\downarrow) $ used in the construction of
$Z:\Bq \to \hA$ arises from a Drinfeld associator $\varphi(X,Y)\in\K\langle\!\langle X,Y \rangle\!\rangle$ as explained in Remark
\ref{rem:associator}.  We compute $Z$ on a generating set of $\Bq$
and construct a braided monoidal functor $\Zqphi:\Bq \to \hA_q^\varphi$, which is a variant of $Z$ with values in a
deformation of the non-strictification of $\hA$.

\subsection{Generators of $\Bq$}  \label{sec:generators_B}

As announced in \cite[\S 14.5]{Habiro} and will be proved in
\cite{Habiro2}, the strict monoidal category $\B$ is generated by the morphisms
\begin{gather}
  \label{e64}
\centre{\small
\begin{array}{c|ccc|c}
\!\!\!\!\psi  := \figtotext{29}{35}{psi_B} &
\mu := \figtotext{29}{35}{mu_B} &
\Delta := \figtotext{29}{35}{Delta_B} &
S := \figtotext{29}{35}{S_B} &
r_+:= \figtotext{25}{25}{v+_B} \\
\!\!\!\!\psi^{-1}  := \figtotext{29}{35}{psi-_B} &
\eta := \figtotext{29}{35}{eta_B} &
\epsilon := \figtotext{29}{35}{epsilon_B} &
S^{-1} := \figtotext{29}{35}{S-_B} &
r_- := \figtotext{25}{25}{v-_B}
\end{array}}
\end{gather}

The monoidal category $\B$ has a unique braiding
$$\psi_{p,q}: p +  q \longrightarrow q + p,\quad p,q\ge0$$
such that $\psi_{1,1}=\psi$.  The object $1$ is a Hopf algebra in the
braided category $\B$, with multiplication~$\mu$, unit $\eta$,
comultiplication $\Delta$, counit $\epsilon$ and invertible
antipode~$S$.  The canonical functor $\B\to\Cob$ (see Section
\ref{sec:sLCob}) maps this Hopf algebra to the Hopf algebra in
$\Cob$ given by Crane \& Yetter \cite{CY} and Kerler \cite{Kerler},
and it maps the morphisms $r_\pm $ to the ``ribbon elements'' in the
sense of \cite{Kerler_towards}.

\begin{example}    \label{ex:ad_c}
We can use the Hopf algebra structure of $1$ in $\B$  to define some additional morphisms.
 The \emph{adjoint action} is the morphism
$$
\ad := \mu^{[3]} (\id_1 \otimes \psi)  (\id_1 \otimes S \otimes  \id_1 ) (\Delta \otimes \id_1)
=  \figtotext{60}{40}{ad_B}   : 2 \lto 1.
$$
Using the ribbon elements $r_\pm$ and following \cite{Kerler_towards},
we define
$$
\cc := \big(\mu  (r_-\otimes \id_1) \otimes \mu  ( \id_1 \otimes r_- )
\big)   \Delta  r_+=   \figtotext{60}{40}{c+_B}
:0\lto2.
$$
\end{example}

The above generating set for the strict monoidal category $\B$
induces a  generating set for the non-strict monoidal category $\Bq$:
\begin{gather*}
\psi^{\pm1}:\bu \bu\lto \bu\bu,\quad
\mu:\bu\bu\lto \bu,\quad
\eta,r_{\pm}:\varnothing\lto \bu,\\
\Delta :\bu\lto  \bu \bu,\quad
\epsilon :\bu\lto \varnothing,\quad
S^{\pm1}:\bu\lto  \bu.
\end{gather*}
The associativity  isomorphisms of $\Bq$ are denoted by
{$\alpha_{u,v,w} :(uv)w\to u(vw)$.}

\subsection{Values of $Z$ on the generators}  \label{sec:values-z-generators}

We compute the values of $Z$ on the generators of the monoidal
category $\Bq$ given in the previous subsection.  Our formulas will
be expressed only in terms of the chosen Drinfeld associator
$\varphi(X,Y)$, and they will involve the structural morphisms of the
Casimir Hopf algebra $(H,c)=(1,\eta,\mu,\epsilon,\Delta,S,c)$
in~$\AB$. (See Proposition \ref{r14}.)

As in Section \ref{sec:ribbon-quasi-hopf}, we equip $\hA(0,m)$
  ($m\ge0$) with its convolution product $\ast$ and
 consider the {following} morphisms {in $\hA$:}
\begin{eqnarray*}
  \varphi &= &\varphi(c_{12},c_{23}) \  :0\lto 3, \\
  R & = & \exp_*(c/2)  \   :0\lto 2, \\
  \bfr &=& \exp_*(\mu c /2) \  :0\lto 1.
\end{eqnarray*}
Set
\begin{gather*}
\nu = \big(\mu^{[3]} (\id_1 \otimes S \otimes \id_1) \varphi\big)^{-1}
:0\lto1\quad \text{in $\hA$},
\end{gather*}
where $(\ )^{-1}$ denotes convolution-inverse.  Note that  $\nu$
corresponds to the element \eqref{eq:nu_def} of $\A(\downarrow)$
through the isomorphism $\iota$ of Section
\ref{sec:ribbon-quasi-hopf}.

In what follows, we use the usual graphical calculus for morphisms
in $\hA$, where morphisms run downwards.  The antipode $S$, the
iterated multiplication $\mu^{[n]}$ and the iterated comultiplication
$\Delta^{[n]}$ ($n\geq 0$) are depicted by
$$
{\labellist
\small\hair 2pt
 \pinlabel{$,$}  at 52 23
 \pinlabel{$\cdots$}  at 107 42
 \pinlabel{and} at 167 23
 \pinlabel{$\cdots$} at 223 5
\endlabellist
\centering
\includegraphics[scale=0.8]{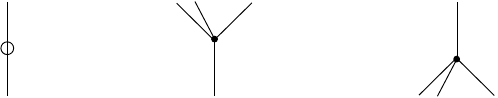}}
$$
respectively. For instance, the  adjoint action of the Hopf algebra $H$
$$
\ad = \mu^{[3]}   (\id_1 \otimes   P_{1,1})  (\id_1 \otimes S \otimes
\id_1 ) (\Delta \otimes \id_1) \  :2   \lto1
$$
is depicted by
$$
\centre{\labellist
\small\hair 2pt
 \pinlabel{$:=$}  at 42 24
\endlabellist
\centering
\includegraphics[scale=1.0]{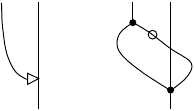}}.
$$

\begin{proposition} \label{prop:Z_generators}
We have
\begin{equation}   \label{eq:Z(r+)}
Z(\psi^{\pm1}) =
\centre{\labellist
\scriptsize\hair 2pt
 \pinlabel{$R^{\pm1}$}  at 20 117
\endlabellist
\centering
\includegraphics[scale=.52]{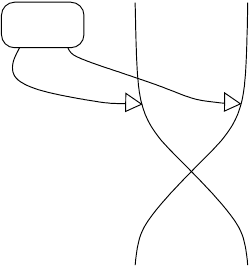}},
\quad Z(r_\pm) =
\centre{\labellist
\scriptsize \hair 2pt
 \pinlabel{$ \bfr^{\mp 1}$}  at 22 64
\endlabellist
\centering
\includegraphics[scale=0.4]{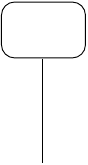}} ,
\quad Z(\eta) = \centre{\includegraphics[scale=.52]{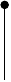}} ,
\quad Z(\epsilon) = \centre{\includegraphics[scale=.52]{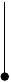}} ,
\end{equation}
\begin{equation}   \label{eq:Z(alpha)}
Z(\alpha_{u,v,w}^{\pm 1}) =
\centre{\labellist\scriptsize\hair 2pt
 \pinlabel{$\varphi^{\pm 1}$}   at 35 109
 \pinlabel{$\cdots$}  at 123 115
 \pinlabel{$\cdots$}  at 211 115
 \pinlabel{$\cdots$} at 303 115
 \pinlabel{$\cdots$}  at 304 7
 \pinlabel{$\cdots$}  at 212 7
 \pinlabel{$\cdots$}  at 122 8
 \pinlabel{$\underbrace{\hphantom{aaa}}_{\vert u \vert}$} [t] at 119 0
 \pinlabel{$\underbrace{\hphantom{aaa}}_{\vert v \vert}$}  [t] at 209 1
 \pinlabel{$\underbrace{\hphantom{aaa}}_{\vert w \vert}$}  [t] at 301 1
\endlabellist
\centering
\includegraphics[scale=0.35]{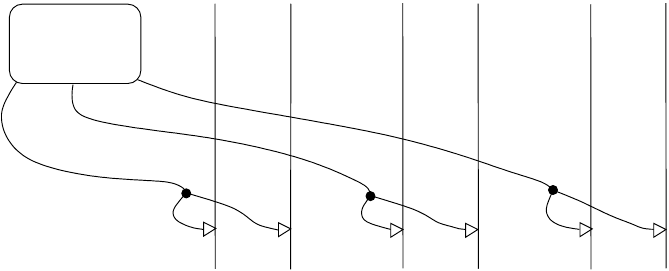}}
\quad \hbox{for  $u,v,w \in\Mag(\bullet)$,}
\end{equation}
\vspace{0.5cm}
\begin{equation}   \label{eq:Z(S)}
Z(S^{\pm 1}) =
\centre{
\labellist
\small\hair 2pt
 \pinlabel{$\bfr^{\pm 1/2}$}  at 17 55
\endlabellist
\centering
\includegraphics[scale=0.8]{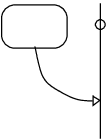}
},
\end{equation}
\begin{equation}  \label{eq:Z(mu)}
Z(\mu) = \centre{\labellist
\scriptsize \hair 2pt
 \pinlabel{$\nu$}  at 107 35
 \pinlabel{$\varphi^{-1}$}  at 157 97
 \pinlabel{$\varphi$}  at 20 97
\endlabellist
\centering
\includegraphics[scale=0.7]{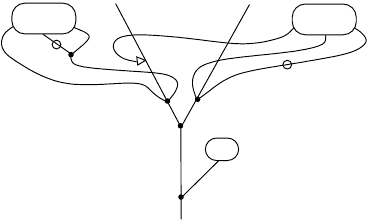}},
\end{equation}
\begin{equation}  \label{eq:Z(Delta)}
Z(\Delta) = \centre{\labellist
\scriptsize \hair 2pt
 \pinlabel{$\varphi$} at 13 200
 \pinlabel{$\varphi^{-1}$}  at 60 200
 \pinlabel{$\varphi$}  at 213 162
 \pinlabel{$\varphi^{-1}$}  at 211 108
 \pinlabel{$R$}  at 49 90
 \pinlabel{$\varphi^{-1}$} at 50 54
\endlabellist
\centering
\includegraphics[scale=0.85]{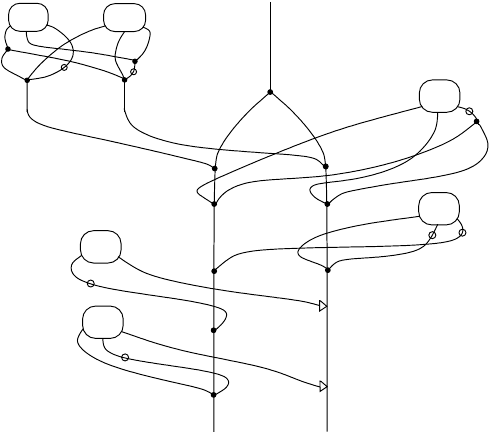}}.
\end{equation}
\end{proposition}

\begin{proof}
First, we briefly explain how to compute $Z(\psi)$,
$Z(r_+)$, $Z(\eta)$, $Z(\epsilon)$ and $Z(\alpha_{u,v,w})$.
One can compute $Z(\psi^{-1})$, $Z(r_-)$ and $Z(\alpha_{u,v,w}^{-1})$ similarly.
We only indicate  the decompositions into $q$-tangles of some cube presentations
leading to \eqref{eq:Z(r+)} and \eqref{eq:Z(alpha)}:
$$
\psi = \centre{\includegraphics[scale=.5]{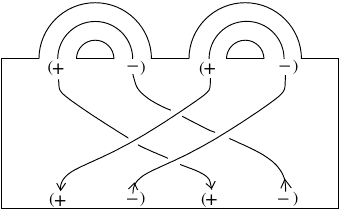}}, \;
r_+= \centre{\includegraphics[scale=.18]{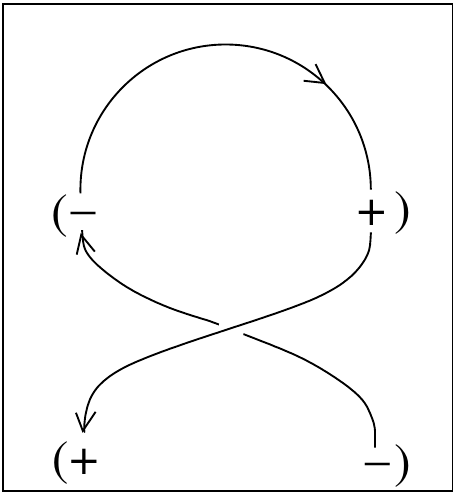}}, \;
\eta =  \centre{\includegraphics[scale=.5]{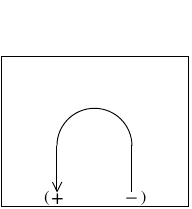}}, \;
\epsilon =  \centre{\includegraphics[scale=.5]{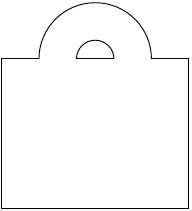}},
$$
\vspace{-0.1cm}
$$
\alpha_{u,v,w}=\centre{\labellist
\scriptsize\hair 2pt
 \pinlabel{$\cdots$}  at 71 33
 \pinlabel{$\cdots$}  at 75 82
 \pinlabel{$\cdots$}  at 200 82
 \pinlabel{$\cdots$}  at 215 38
 \pinlabel{$\cdots$}  at 363 81
 \pinlabel{$\cdots$}  at 360 36
  \pinlabel{$\underbrace{\hphantom{aaaaaaaaaaaaa}}_{u(+-)}$} [t] at 73 4
 \pinlabel{$\underbrace{\hphantom{aaaaaaaaaaaaa}}_{v(+-)}$}  [t]  at 237 4
 \pinlabel{$\underbrace{\hphantom{aaaaaaaaaaaaa}}_{w(+-)}$} [t] at 361 4
\endlabellist
\centering
\includegraphics[scale=0.6]{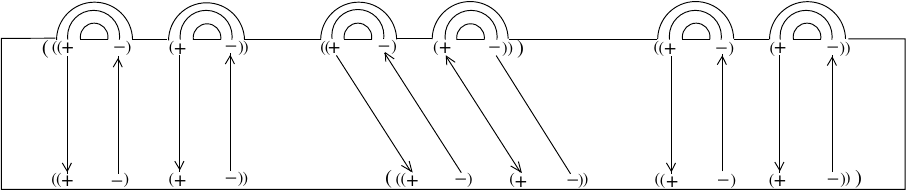}}
$$
\\\medskip
{\noindent}
We leave the details to the interested  reader.

Now we compute $Z(S)$. Since
$$
S = \figtotext{70}{70}{S_dec},
$$
we have
\begin{eqnarray*}
Z(S) \ = \ \centre{
\labellist
\scriptsize\hair 2pt
 \pinlabel{$\exp\!\Big(\!-\!\frac{1}{2}$} [r] at 44 32
 \pinlabel{$\Big)$}  [l] at 82 32
 \pinlabel{$\exp\!\Big($} [r] at 76 58
 \pinlabel{$\frac{1}{2} \Big)$} [l]  at 82 58
\endlabellist
\centering
\includegraphics[scale=0.95]{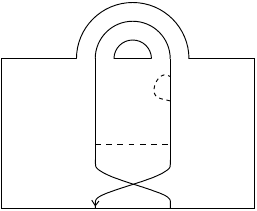}
} &= &\centre{
\labellist
\scriptsize\hair 2pt
 \pinlabel{$\exp\!\Big(\!-\!\frac{1}{2}$} [r] at 44 32
 \pinlabel{$\Big)$}  [l] at 82 32
 \pinlabel{$\exp\!\Big($} [r] at 76 58
 \pinlabel{$\frac{1}{4} \Big)$} [l]  at 82 58
 \pinlabel{$\exp\!\Big($} [r] at 40 58
 \pinlabel{$\frac{1}{4} \Big)$} [l]  at 44 58
\endlabellist
\centering
\includegraphics[scale=0.95]{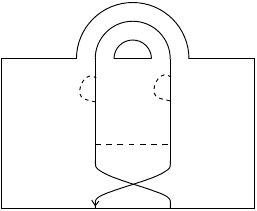}
} \\
&=& \centre{\labellist
\small\hair 2pt
 \pinlabel{$\Bigg)$} [l] at 107 43
 \pinlabel{$\exp\Bigg( \frac{1}{4}$} [r] at 43 43
\endlabellist
\centering
\includegraphics[scale=0.95]{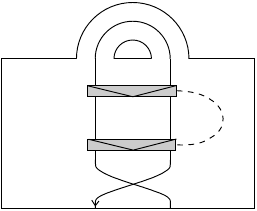}}
\end{eqnarray*}
which implies \eqref{eq:Z(S)}.  The computation of $Z(S^{-1})$ is similar.

One can easily derive \eqref{eq:Z(mu)}
from the following decomposition into $q$-tangles of a cube
presentation of $\mu$:
$$
\mu = \figtotext{100}{80}{mu_d}
$$

Finally, let us consider \eqref{eq:Z(Delta)}.
We need to compute  $a:=\iota(a_{(++)})\in\hA(0,2)$, where $a_{(++)}
\in\A(\downarrow \downarrow)$ is the cabling anomaly.
Since we have
$$
\centre{\labellist
\scriptsize\hair 2pt
 \pinlabel{$a_{(++)}$}  at 15 29
\endlabellist
\centering
\includegraphics[scale=0.8]{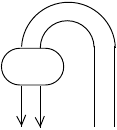}
}
= Z\Bigg(\!\!\!\figtotext{50}{50}{a_d_2}\!\!\!\Bigg) = Z\Bigg(\!\!\!\figtotext{50}{50}{a_d_3}\!\!\!\Bigg),
$$
 we obtain
\begin{equation}  \label{eq:a=}
a=\centre{\labellist
\scriptsize\hair 2pt
 \pinlabel{$\varphi$}  at 15 53
 \pinlabel{$\varphi^{-1}$}  at 68 54
\endlabellist
\centering
\includegraphics[scale=0.9]{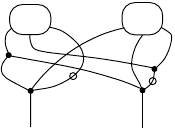}}.
\end{equation}
Then \eqref{eq:Z(Delta)} follows from \eqref{eq:a=} and the following
decomposition into $q$-tangles of a cube presentation of $\Delta$:
$$
\Delta = \figtotext{100}{100}{Delta_d}
$$
\end{proof}

\subsection{The braided monoidal functor $\Zqphi:\Bq\to\hA^\varphi_q$}  \label{sec:A_q}

Using the above computations of $Z$ on the braiding and
associativity isomorphisms of $\Bq$, we define a non-strict
{braided} monoidal category
$\hA_q^\varphi$ as follows.

Let $\hA_q$ denote the non-strictification of the linear strict
monoidal category~$\hA$, see Section \ref{sec:tangles}.  (The
non-strictification defined there extends to linear strict
monoidal categories in the obvious way.)  We identify $\Ob(\hA)=\NZ$
with $\Mon(\bu)$: consequently, $\Ob(\hA_q)=\Mag(\bu)$.  The symmetry
in $\hA$ gives one in~$\hA_q$:
\begin{gather*}
  P_{v,w} := P_{|v|,|w|}\in\hA_q(v w,w v)={\hA}(|v|+|w|,|w|+|v|)\quad \text{for $v,w\in\Mag(\bu)$.}
\end{gather*}
Thus $\hA_q$ is a linear symmetric non-strict monoidal category.

Using the Drinfeld associator $\varphi =\varphi(X,Y)$, we deform $\hA_q$ {into} a linear
  braided non-strict monoidal category ${\hA_q}^\varphi$ as  follows.
  The underlying category, the tensor product {functor}  and the
  monoidal unit of $\hA_q^\varphi$ are the same as {those of} $\hA_q$.  The
  tensor product for $\hA_q^\varphi$ is strictly unital, and the
  left and right unitality  isomorphisms in $\hA_q^\varphi$ are
  the identities.  Define the associativity  isomorphism
  $\alpha_{u,v,w} : {(u v) w \to u (v w)}$  by
\begin{equation} \label{eq:Alpha}
\alpha_{u,v,w} :=
\centre{\labellist\scriptsize\hair 2pt
 \pinlabel{$\varphi$}   at 35 109
 \pinlabel{$\cdots$}  at 123 115
 \pinlabel{$\cdots$}  at 211 115
 \pinlabel{$\cdots$} at 303 115
 \pinlabel{$\cdots$}  at 304 7
 \pinlabel{$\cdots$}  at 212 7
 \pinlabel{$\cdots$}  at 122 8
 \pinlabel{$\underbrace{\hphantom{aaa}}_{\vert u \vert}$} [t] at 119 5
 \pinlabel{$\underbrace{\hphantom{aaa}}_{\vert v \vert}$}  [t] at 209 5
 \pinlabel{$\underbrace{\hphantom{aaa}}_{\vert w \vert}$}  [t] at 301 5
\endlabellist
\centering
\includegraphics[scale=0.35]{Z_alpha}},
\end{equation}
\vskip 5mm

\noindent
and define the braiding $\psi_{v,w}: v w \to w v$ by
\begin{equation}  \label{eq:Psi}
\psi_{v,w} := \centre{\labellist
\scriptsize\hair 2pt
 \pinlabel{$R$}  at 33 237
 \pinlabel{$\cdots$}  at 123 237
 \pinlabel{$\cdots$}  at 212 237
 \pinlabel{$\cdots$} at 122 9
 \pinlabel{$\cdots$}  at 214 9
 \pinlabel{$\underbrace{\hphantom{aaa}}_{\vert w \vert}$} [t] at 122 8
 \pinlabel{$\underbrace{\hphantom{aaa}}_{\vert v \vert}$} [t] at 210 8
\endlabellist
\centering
\includegraphics[scale=0.35]{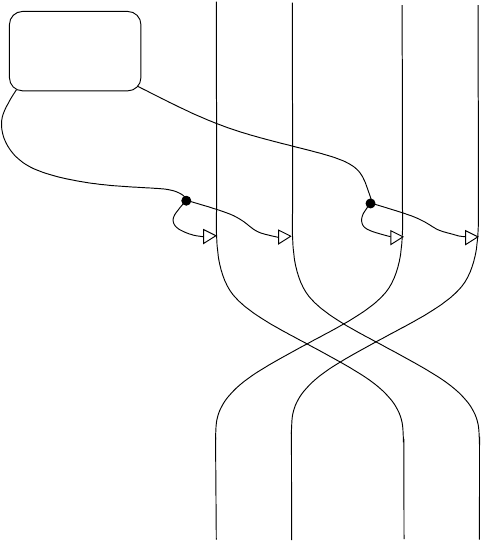}}.
\end{equation}
\vspace{0.5cm}

The tensor-preserving functor $Z:\Bq\to\hA$ is upgraded to a
braided monoidal functor as follows.

\begin{theorem}
  \label{r16}
  With the above description, the category $\hA_q^\varphi$ is
  braided monoidal and there is a (unique) braided monoidal   functor
  \begin{gather}
    \Zqphi:\Bq \longrightarrow \hA_q^\varphi
  \end{gather}
  which is the identity on objects, such that
  \begin{gather*}
    \Zqphi(f)=Z(f)\in\hA_q^\varphi(w,w')=\hA(|w|,|w'|)
  \end{gather*}
  for {morphisms} $f:w\to w'$ {in $\Bq$}.
\end{theorem}

\begin{proof}
 We can check that $\hA_q^\varphi$ is a braided monoidal category
using the properties of a Drinfeld associator (see Section
\ref{sec:drinfeld-associators}).  Alternatively, using the
universality of $Z$ proved in the next section, this follows since
$\Bq$ itself is a braided monoidal category and we have
$Z(\psi_{v,w}) =\psi_{v,w}$, $Z(\alpha_{u,v,w})=\alpha_{u,v,w}$; see
\eqref{eq:Z(r+)} and \eqref{eq:Z(alpha)}.

Clearly, $\Zqphi$ is a well-defined functor.  Since $Z$ is
tensor-preserving, so is $\Zqphi$.  By \eqref{eq:Z(r+)} and
\eqref{eq:Z(alpha)}, $\Zqphi$ preserves the braidings and the
associativity  isomorphisms.  Both $\Bq$ and $\hA_q^\varphi$
have the identity left and right unitality  isomorphisms.
Hence we have the assertion.
\end{proof}

\begin{remark}
  \label{r41}
  The braided monoidal structure of $\hAqp$ descends to a braided
  monoidal structure on the category $\hA$, with the tensor product {functor}
  defined in Section \ref{sec:symmetry_A}, as follows. 
   For  $m,n,p\ge0$, the associativity isomorphism $\alpha_{m,n,p}: (m+n) +p\to
  m+(n+p)$ and the braiding $\psi_{n,p}:n+p\to p+n$ are defined to be
  the right hand sides of \eqref{eq:Alpha} and \eqref{eq:Psi},
  respectively, where we set $|u|=m$, $|v|=n$ and $|w|=p$.  (Note that
  $\alpha_{m,n,p}$ and $\psi_{n,p}$ depend only on the choices of $m,n,p$
  and not on $u,v,w$.)  Let $\hA^\vp$ denote the linear braided {non-strict}
  monoidal category thus obtained, with the identity left and right
  unitality isomorphisms.  There is a
  fully faithful, linear braided monoidal functor
  \begin{gather*}
    \pi:\hAqp\lto\hA^\vp
  \end{gather*}
  which maps each object $w\in\Mag(\bu)$ to its length, and maps the
  morphisms identically. 
  {Clearly},    $\pi$ is an equivalence of linear braided monoidal categories. 
\end{remark}

\subsection{Transmutation of quasi-triangular quasi-Hopf algebras}\label{sec:transm-quasi-triang}

In Section \ref{sec:transm-funct-z}, we will interpret the formulas for $Z$
in  Proposition \ref{prop:Z_generators} in terms of
transmutation.  For a quasi-triangular Hopf algebra $H$, Majid
introduced a Hopf algebra $\ul{H}$ in the braided monoidal category
$\HMod$ of $H$-modules, called the \emph{transmutation} of
$H$ \cite{Majid:algebras-and-Hopf-algebras,Majid}.  Here we consider transmutation of quasi-triangular
quasi-Hopf algebras introduced by Klim \cite{Klim}.

Let $H=(H,\eta,\mu,\ep,\Delta,\varphi,S,\alpha,\beta,R)$ be a
quasi-triangular quasi-Hopf algebra in a symmetric strict
monoidal category $\C$ with monoidal unit $I$.
Following \cite[Theorem 3.1]{Klim}, define morphisms
\begin{gather*}
  \ul{\eta}:I\to H,\quad
  \ul{\ep}:H\to I,\quad
  \ul{\mu}:H\ot H\to H,\quad
  \ul{\Delta}:H\to H\ot H,\quad
  \ul{S}:H\to H
\end{gather*}
in $\C$ by
\begin{gather}
  \label{e33}
  \ul{\eta}=\be,\quad \ul{\ep}=\ep,\\
  \label{e15}
  \ul{\mu}(b\ot b')=q^1(x^1\trr b)S(q^2)x^2b'S(x^3),\\
  \label{e16}
  \ul{\Delta}(b)=x^1X^1b_{(1)}g^1S(x^2R^2y^3{X^3}_{(2)})\ot
  x^3R^1\trr y^1X^2b_{(2)}g^2S(y^2{X^3}_{(1)}),\\
  \label{e18}
  \ul{S}(b)=X^1R^2x^2\be S(q^1(X^2R^1x^1\trr b)S(q^2)X^3x^3),
\end{gather}
where the adjoint action $\ad: H \otimes H \to H$ is denoted by
$l\otimes r \mapsto l \trr r$, we use Sweedler's notation $\Delta(z) =
z_{(1)} \otimes z_{(2)}$ for   $z\in H$, and we set
\begin{gather*}
  q=q^1\ot q^2 = X^1\ot S^{-1}(\alpha X^3)X^2,\\
  \varphi=X^1\ot X^2\ot X^3,\quad
  \varphi^{-1}=x^1\ot x^2\ot x^3=y^1\ot y^2\ot y^3,\\
  g=g^1\ot g^2 = (\Delta(S(x^1)\alpha x^2))\delta(S\ot S)(\Delta^{\op}(x^3)),\\
  \delta=\delta^1\ot\delta^2 = B^1\be S(B^4)\ot B^2\be S(B^3),\\
  B^1\ot B^2\ot B^3\ot B^4 =
  (\De\ot\id\ot\id)(\varphi)\, (\varphi^{-1}\ot 1),\\
  R=  R^1\ot R^2.
\end{gather*}
Here we use the notations for a quasi-triangular quasi-Hopf algebra
over a field, but the meaning of the above formulas in the category
$\C$ should be clear.  Let $\ul{H} = (H,\ad)$ denote the object $H$
with the adjoint action.  Klim proved that
$\ul{H}=(\ul{H},\ul{\eta},\ul{\mu},\ul{\ep},\ul{\Delta},\ul{S})$ is a
Hopf algebra in the braided monoidal category $\HMod$.  (Recall that
the monoidal category $\HMod$ is not strict in general, although we
assumed that $\C$ is strict monoidal.)

By straightforward computation, we can rewrite \eqref{e15} and \eqref{e16} as follows.

\begin{lemma}
  \label{r21}
We have
\begin{gather}
  \label{e35}
  \ul\mu = \mu\ga_2\ga_1,\\
  \label{e34}
  \ul\De=\theta_5\theta_4\theta_3\theta_2\theta_1\De,
\end{gather}
where  we define $\ga_1,\ga_2,\theta_1,\ldots,\theta_5:H^{\ot2}\to H^{\ot2}$  by
\begin{gather*}
    \ga_1(b\ot b') = (x^1\trr b)\ot x^2 b'S(x^3),\quad
    \ga_2(b\ot b') =X^1 b S(X^2)\alpha X^3 \ot b',\\
  \theta_1(b\ot b')=b g^1\ot b' g^2,\quad
  \theta_2(b\ot b')= X^1 b S({X^3}_{(2)})\ot X^2 b'S({X^3}_{(1)}),\\
  \theta_3(b\ot b')= b S(y^3) \ot y^1 b' S(y^2),\quad
  \theta_4(b\ot b')= b S(R^2)\ot (R^1\trr b'),\\
  \theta_5(b\ot b')= x^1 b S(x^2) \ot (x^3\trr b').
\end{gather*}
\end{lemma}

\subsection{Transmutation and the functor $Z$}  \label{sec:transm-funct-z}

Consider {now} the quasi-triangular quasi-Hopf algebra {in $\hA$}
\begin{gather}
  \label{e49}
H := {H_{\varphi} }=(1,\eta,\mu,\ep,\De,\varphi,S,\alpha,\be,R)
\end{gather}
 given by Theorem \ref{prop:r3} with
$\be=\eta$ (and, hence, $\alpha=\nu$).  Let
$\ul{H}=(\ul{H},\ul\eta,\ul\mu,\ul\ep,\ul\Delta,\ul{S})$ be the
transmutation of $H$, which is a Hopf algebra in the braided
non-strict monoidal category $\HMod$ of $H$-modules in $\hA$.

Let $H^{\Bq}=(\bu,\eta,\mu,\ep,\De,S)$ denote the Hopf algebra in
$\Bq$ defined in Section \ref{sec:generators_B}.
It follows from Theorem \ref{r16} that
\begin{gather*}
  \Zqphi(H^{\Bq})=\big(\bu,\Zqphi(\eta),\Zqphi(\mu),\Zqphi(\ep), \Zqphi(\De),\Zqphi(S)\big)
\end{gather*}
is a  Hopf algebra in the braided non-strict monoidal category $\hA_q^\varphi$.

Next, we define a {fully faithful linear} functor
\begin{gather*}
  F:\hA_q^\varphi \longrightarrow\HMod
\end{gather*}
by $F(w)=w(\ul{H})$ for $w\in\Mag(\bu)$ and
\begin{gather*}
  F(f)=f\quad \text{for $f\in\hA_q^\varphi(v,w)=\hA(|v|,|w|)$\quad {with} $v,w\in\Mag(\bu)$}.
\end{gather*}
Then, by \eqref{eq:Alpha} and \eqref{eq:Psi}, $F$ is a braided
monoidal functor.  Hence $F(\Zqphi(H^{\Bq}))$ is a Hopf algebra in
$\HMod$.

\begin{theorem}
  \label{r19}
  The two braided Hopf algebras $F(\Zqphi(H^{\Bq}))$ and $\ul{H}$
  coincide.
\end{theorem}

\begin{proof}
  Since the antipode of a braided bialgebra is unique, it suffices to
  prove
  \begin{gather*}
    Z(\eta) = \underline{\eta},\quad
    Z(\epsilon) = \underline{\epsilon},\quad
    Z(\mu) = \underline{\mu},\quad
    Z(\Delta) = \underline{\Delta}.
  \end{gather*}
  It is easy to check the first two identities.
  We can check  $Z(\mu)=\ul\mu$ by using
  \eqref{eq:Z(mu)}, \eqref{e35} and
  $$
  \mu\gamma_2(b\ot b')
    =X^1 b S(X^2)\nu X^3 b'
    \by{eq:general_fact}  X^1 b S(X^2) X^3 b'\nu,
  $$
    where, as before, $b$ and $b'$ are formal variables denoting virtual
    elements in the Hopf algebra.

  Let us now prove  $Z(\De)=\ul\De$.   We have
  \begin{eqnarray*}
      g  &=&   \De(S(x^1)\nu x^2)\, \delta\, \De(S(x^3)) \\
     & \by{eq:general_fact} & \De(S(x^1)\nu x^2)\, \De(S(x^3))\, \delta\\
      & =& \De(S(x^1)\nu x^2 S(x^3))\, \delta \\
     & \by{e52}& \delta \  = \ {X^1}_{(1)}x^1 S(X^3)\ot{X^1}_{(2)}x^2 S(x^3)S(X^2) \  \by{eq:a=} \  a.
  \end{eqnarray*}
  Hence,
  $\theta_1\Delta(b)=\De(b)g=\De(b)a\by{eq:general_fact}a\De(b)$.
  Therefore,
  \begin{gather*}
    \ul\De(b)  \by{e34} \theta_5\cdots\theta_1\De(b)=\theta_5\cdots\theta_2(a\De(b))\by{eq:Z(Delta)}Z(\De)(b)
  \end{gather*}
  where we use $P_{1,1}R =R$ in the last identity.  Hence $\ul\De=Z(\De)$.
\end{proof}

\subsection{Computations of $Z$ up to degree $2$} \label{sec:computations-z-up}

Here we give the values of $Z$ for the generators of $\Bq$ up to degree $2$.

\begin{proposition}
  \label{r29}
We have $Z(\eta) = \figtotext{30}{25}{eta}$, $Z(\epsilon) = \figtotext{30}{25}{epsilon}$
and the following identities hold true up to degree $2$:
\nc\fffa{\figtotext{70}{45}}
\nc\fffb{\figtotext{60}{45}}
\nc\fffc{\figtotext{33}{45}}
\nc\fffd{\figtotext{40}{30}}
\begin{eqnarray*}
Z(\mu)  & = & \fffa{mu_d2} + \frac{1}{24} \fffa{mu2_d2} +  \frac{1}{48} \fffa{mu3_d2} \\
&& -  \frac{1}{48} \fffa{mu4_d2} -  \frac{1}{48} \fffa{mu5_d2}, \\
Z(\Delta) &=&  \fffb{Delta_d2} - \frac{1}{2} \fffb{Delta2_d2} + \frac{1}{8} \fffb{Delta3_d2} \\
&&+  \frac{1}{48} \fffb{Delta8_d2}  - \frac{1}{12} \fffb{Delta4_d2} + \frac{1}{24} \fffb{Delta5_d2} \\
&&+ \frac{1}{24} \fffb{Delta6_d2} + \frac{1}{24} \fffb{Delta7_d2},\\
Z(S^{\pm 1})&=&  \fffc{S_d2} \pm \frac{1}{2}  \fffc{S2_d2} \mp \frac{1}{2}  \fffc{S3_d2}\\
&& +  \frac{1}{8}  \fffc{S4_d2} -  \frac{1}{4}  \fffc{S5_d2} +  \frac{1}{8}  \fffc{S6_d2} ,\\
Z({r_\pm}) &=& \fffd{r_d2} {\mp} \frac{1}{2} \fffd{r2_d2}  + \frac{1}{8} \fffd{r3_d2},\\
Z(\psi^{\pm 1}) &=& \fffa{psi_d2} \pm \frac{1}{2} \fffa{psi2_d2}  + \frac{1}{8} \fffa{psi3_d2}, \\[0.4cm]
Z(\alpha_{u,v,w}^{\pm 1}) &=&
\centre{\labellist\scriptsize\hair 2pt
 \pinlabel {$\overbrace{\hphantom{aaaaaa}}^{\vert u \vert}$} [b] at 77 95
 \pinlabel {$\overbrace{\hphantom{aaaaaa}}^{\vert v \vert}$} [b] a at 204 94
 \pinlabel {$\overbrace{\hphantom{aaaaaa}}^{\vert w\vert }$} [b] a [b] at 333 95
\endlabellist
\centering
\includegraphics[scale=0.3]{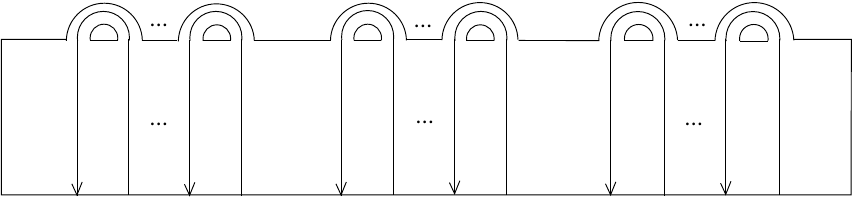}}
\mp
\frac{1}{24}  \centre{\labellist\scriptsize\hair 2pt
 \pinlabel {$\overbrace{\hphantom{aaaaaa}}^{\vert u \vert}$} [b] at 77 95
 \pinlabel {$\overbrace{\hphantom{aaaaaa}}^{\vert v \vert}$} [b] a at 204 94
 \pinlabel {$\overbrace{\hphantom{aaaaaa}}^{\vert w\vert }$} [b] a [b] at 333 95
\endlabellist
\centering
\includegraphics[scale=0.3]{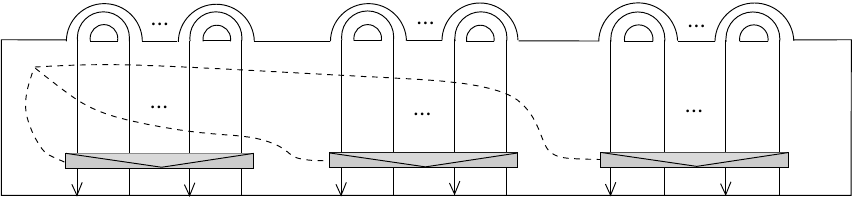}}.
\end{eqnarray*}
\end{proposition}

\begin{proof}
One can check these formulas by direct computations using
Proposition \ref{prop:Z_generators} and the well-known identity
\begin{gather}
  \label{e28}
\varphi(X,Y) = 1+\frac{1}{24} [X,Y]+\hbox{(terms of degree\;$>2$)},
\end{gather}
which follows from \eqref{e23} and \eqref{e24}.  We leave the details to the interested reader.
\end{proof}

\begin{remark}
  \label{r9}
  The quasi-triangular quasi-Hopf algebra $H=H_{\varphi}$ in $\hA$
  given in \eqref{e49} has the following structure up to degree $2$.  The
  morphisms $\et,\mu,\ep,\De,S$ are as depicted in \eqref{e51} and
  concentrated in degree $0$.  Combining \eqref{eq:phi*} to
  \eqref{e28}, using \eqref{eq:R*} and using \eqref{eq:nu_def},
  respectively, we obtain the following identities up to degree $2$:
  \begin{gather*}
  \varphi^{\pm1}=
  \figtotext{60}{40}{Y0}
  \mp\frac1{24}\figtotext{60}{40}{Y},\\
  R^{\pm1}=\figtotext{60}{40}{r0}
  \pm\frac12\figtotext{60}{40}{r}
  +\frac18\figtotext{60}{40}{rr}
  ,\\
  \nu =
  \figtotext{60}{35}{r_d2}
    +\frac1{48} \figtotext{60}{35}{nu3}.
\end{gather*}
  We can also deduce Proposition \ref{r29} from these identities by  using {Theorem} \ref{r19}.
\end{remark}

%
%
\section{Universality of the extended Kontsevich integral}   \label{sec:universality}

In this section, we show that the extended Kontsevich integral
$\Zqphi:\Bq\to\hAqp$ (given by Theorem \ref{r16}) induces
an isomorphism $\Zqphi:\wh{\KBq}\to\hAqp$ of linear braided monoidal
categories, where $\wh{\KBq}$ is the completion of the
linearization $\KBq$ of $\Bq$ with respect to the Vassiliev--Goussarov
filtration.  This implies the universality of $Z$ among
Vassiliev--Goussarov invariants of bottom tangles in handlebodies.

\subsection{Ideals in monoidal categories}  \label{sec:terminology}

Let $\calC$ be a linear (possibly non-strict) monoidal category.
{We partly borrow from \cite[\S 3.3]{KT} the following terminology}.
An \emph{ideal} $\calI$ of $\calC$ consists of a family of
linear subspaces $\calI(v,w) \subset \calC(v,w)$ for all
$v,w\in\Ob(\calC)$ such that $f\otimes g,f\circ g\in\calI$ for morphisms $f,g\in\calC$ with either $f\in\calI$ or
$g\in\calI$.  For instance, the ideal \emph{generated} by a set $S$ of
morphisms of $\calC$ is the smallest ideal of $\calC$ containing
$S$.  Every ideal $\calI$ of $\calC$ defines a congruence relation
in $\calC$, and the quotient category~$\calC/\calI$ is a linear
monoidal category.

A \emph{filtration} $\calF$ in $\calC$ is a decreasing sequence
$\calC= \calF^0 \supset \calF^1 \supset \calF^2 \supset \cdots$ of
ideals of $\calC$ such that $\calF^k \circ \calF^l \subset \calF^{k+l}$ for
$k,l\geq 0$.  Then $\calF^k \otimes \calF^l \subset \calF^{k+l}$ follows.
The \emph{completion of $\calC$ with respect to $\mathcal{F}$}
$$
\widehat{\calC}^{\calF} := \varprojlim_k\,\calC / \calF^k
$$ inherits a structure of filtered linear monoidal
category from $\calC$.  Let $\widehat{\calF}$ denote the
filtration of $\widehat{\calC}^\calF$ induced by $\calF$.
Let $\Gr^{\calF}\!\calC$ denote the graded
linear monoidal category associated to $\calF$: we have
  $\Ob(\Gr^{\calF}\!\calC)=\Ob(\calC)$ and
$$
\Gr^{\calF}\!\calC(v,w)  = \bigoplus_{k \ge 0} \calF^k(v,w)/ \calF^{k+1}(v,w).
$$

The product $\calJ\calI$ of two ideals $\calI,\calJ\subset\C$ is the
ideal of $\C$ generated by $gf$ for all composable pairs of
$g\in\calJ$, $f\in\calI$.  For an ideal $\calI\subset\C$, the \emph{$\calI$-adic filtration}
$\calC=\calI^0\supset\calI^1\supset\calI^2\supset\cdots$ of $\C$
is defined inductively by $\calI^0=\C$ and
$\calI^{k+1}=\calI\calI^k$ for~$k\ge0$.  We write
$\widehat{\calC}=\widehat{\calC}^\calI$ and
$\Gr\calC=\Gr^{\calI}\calC$ if the ideal $\calI$ is clear from the
context.  Note that $\calI^k$ contains {all morphisms} of $\calC$
{that can be} obtained by taking compositions and tensor products of a finite number
of morphisms in $\C$ containing at least $k$ elements of $\calI$.

We define the tensor power $w^{\ot k}$ of an object $w$ in $\C$
inductively by $w^{\ot0}=I$, the monoidal unit, and $w^{\ot
(k+1)}=w^{\ot k}\ot w$.  The tensor power $f^{\ot k}:v^{\ot k} \to
w^{\ot k}$ of a morphism $f:v \to w$  is defined similarly.
(If $\C$ is a strict monoidal category, {then} these tensor powers coincide
with those we have already used.)

\subsection{The Vassiliev--Goussarov filtration} \label{sec:vass-gouss-filtr}

We now generalize the Vassiliev--Goussarov filtration for
links/tangles in a ball (see e.g.\ \cite{BN1}) to bottom tangles
in handlebodies.
In the definition of the Vassiliev--Goussarov filtration, 
one usually uses only crossing-change moves to form alternating sums of tangles that generate the filtration.  
We here also use framing-change moves since we work with framed tangles.

Let $\KBq$ denote the linearization of the category $\Bq$.
A \emph{plot} $P$ of a diagram $D$ of a bottom $q$-tangle
{$T:v\to w$} is a disk in which
$D$ appears as either a crossing or a positive curl:
\begin{gather*}
  \vhi{ms1}\ ,\quad \vhi{ms3}\ .
\end{gather*}
We get a new bottom $q$-tangle {$T_P:v\to w$} from $T$ by
  the following move at~$P$:
\begin{gather}
  \label{eq:elementary_moves}
  T=\vhi{ms1}\ \longmapsto\ T_P=\vhi{ms2}\ ,
  \quad\quad
  T=\vhi{ms3}\ \longmapsto\ T_P=\vhi{ms4}\ .
\end{gather}
More generally, if $P$ is a finite set of pairwise disjoint plots
of $D$, then we obtain a new bottom $q$-tangle $T_P$ from $T$ by
applying the above move in each plot of $P$.

For $k\ge0$, let $\V^k(v,w)$ denote the linear subspace of $\KBq(v,w)$ spanned by
$$
[T;P] := \sum_{S \subset P} (-1)^{\vert S \vert}\, T_S,
$$
where $T\in\Bq(v,w)$ and $P$ is a set of $k$ pairwise
disjoint plots of {an arbitrary} diagram of $T$.  The spaces
$\V^k(v,w)$ give the \emph{Vassiliev--Goussarov filtration} of $\KBq$:
$$
\KBq = \V^0 \supset \V^1 \supset \V^2 \supset \cdots .
$$
As we will see, $\V$ is a filtration of $\KBq$ in the sense of Section \ref{sec:terminology}.

Recall from Section \ref{sec:generators_B} the morphisms
{$\eta,r_+,r_-:\vn\to\bu$ and $\cc:\vn\to\bu\bu$ in $\Bq$}.  Let
$\calJ$ be the ideal of $\KBq$ generated by
\begin{gather*}
r_+-\eta   \in\K \Bq(\varnothing,\bullet).
\end{gather*}
We have
\begin{equation} \label{eq:c}
\cc -\eta^{\otimes 2} \in\calJ(\varnothing,\bullet \bullet).
\end{equation}
Indeed, since $r_-= \mu (r_- \otimes \eta)  \Jeq \mu( r_-\otimes r_+) =\eta$, we have
$$
\cc
=\big( \mu (r_- \ot \id )  \ot \mu ( \id  \ot r_-)\big) \De r_+
\Jeq \big( \mu (\et \ot \id )  \ot \mu ( \id  \ot \et)\big) \De \et
=\eta^{\ot2}.
$$
Now we give a categorical description of the Vassiliev--Goussarov filtration.

\begin{proposition} \label{prop:V=J}
The Vassiliev--Goussarov filtration coincides with the $\calJ$-adic
filtration; i.e., we have  $\V^k = \calJ^k$ for  $k\geq 0$.
\end{proposition}

\begin{proof}
We first prove that $\V$ is a filtration.
It is easy to see that each $\V^k$ is an ideal.
To prove $\V^{k'} \circ \V^k \subset \V^{k+k'}$,
{consider morphisms $w\overset{T}{\lto}w'\overset{T'}{\lto}w''$ in $\Bq$,}
 and let $P$ (resp.\ $P'$)
consist of $k$ (resp.\ $k'$) pairwise disjoint plots of a diagram
of $T$ (resp.\ $T'$).  We will prove $[T';P'] \circ
[T;P]\in\V^{k+k'}$.  We can assume that the diagrams of $T$ and $T'$
arise from diagrams of some cube presentations $U$ and $U'$ of $T$ and
$T'$, respectively. Then
\begin{eqnarray*}
[T';P'] \circ [T;P] &=& \sum_{S\subset P} \sum_{S'\subset P'} (-1)^{\vert S \vert +\vert S' \vert}\, T'_{S'} \circ T_S
\end{eqnarray*}
has a cube presentation
\begin{eqnarray*}
 \sum_{S\subset P} \sum_{S'\subset P'} (-1)^{\vert S \vert +\vert S' \vert}\, U'_{S'} \circ C_\varpi(U_S)
 &=& [U';P'] \circ C_\varpi([U,P])
\end{eqnarray*}
for some map $\varpi: \pi_0(U) \to \Mag( \pm)$.  We can decompose the
cabling of each of the moves in \eqref{eq:elementary_moves} into a
finite sequence of moves in \eqref{eq:elementary_moves}.
Therefore we have
$$
C_\varpi([U,P]) = \sum_i [V_i,P_i],
$$
where each $V_i$ is a $q$-tangle differing from $C_\varpi(U)$
only in the plots of $P$, and each $P_i$ consists of $k$ smaller plots in a diagram of $V_i$. It follows that
$$
\sum_i [U';P'] \circ [V_i,P_i] = \sum_i [U' \circ V_i;P' \cup P_i]
$$
is a cube presentation of $[T';P'] \circ [T;P]$.  Hence {it}  belongs to $\V^{k+k'}$.

Now we prove $\calJ^k = \V^k$ for  $k\geq 1$.
We have $\calJ^k\subset \V^k$ since $\V$ is a filtration and we have $r_+-\eta\in\V^1$.
To prove $\V^k \subset \calJ^k$, consider an element
$[T;P] \in\V^k(v,w),$ where $T \in\Bq(v,w)$ and $P$ is a set of $k$ pairwise
disjoint plots of a diagram $D$ of $T$.  Assume that $P$ has $k'$
plots containing crossings and $k'':=k-k'$ plots containing
positive curls.  We can realize the  moves in \eqref{eq:elementary_moves} by the moves
$\eta^{\ot2}\mapsto\cc$ and $\eta\mapsto r_+$.  Thus, by
moving the plots of $P$ towards the upper right corner of $D$
using planar isotopy and Reidemeister moves, we obtain
$$
[T;P] = \pm T' \circ \big(\id_v \otimes ((\cc-\eta^{\otimes 2})^{\ot k'}
\otimes (r_+-\eta)^{\ot k''})\big),
$$
where $T'$ is a morphism in $\Bq$.
By \eqref{eq:c}, it follows that $[T;P]\in\calJ^k$.
\end{proof}

\begin{remark} \label{rem:strict_V_J}
  (1) One can define the filtrations $\V$ and $\calJ$ in  $\KB$ as well.
  Proposition \ref{prop:V=J} is valid in this setting, too.

  (2) A result similar to Proposition \ref{prop:V=J} is given in
  \cite{Habiro}.  The braided monoidal category $\mathsf{B}$ defined
  there is a subcategory of the category $\mathcal{T}$ of tangles, and
  there is a braided monoidal functor $\mathsf{B}\to\B$.  The result
  \cite[Theorem 9.19]{Habiro} essentially states that the
  Vassiliev--Goussarov filtration of the linearization $\Z\mathsf{B}$
  of $\mathsf{B}$ coincides with the $\calI_{\mathsf{B}}$-adic
  filtration, where $\calI_{\mathsf{B}}$ is the ideal in
  $\Z\mathsf{B}$ generated by the morphism corresponding to
  $\eta^{\ot2}-\cc$.

  (3) In Sections \ref{sec:terminology} and
  \ref{sec:vass-gouss-filtr}, one can work over a commutative,
  unital ring.  In particular, Proposition \ref{prop:V=J} holds for
  $\Z\Bq$ and $\Z\B$ as well.
\end{remark}

Let $\widehat{\KBq}:=\widehat{\KBq}^{\V}
=\widehat{\KBq}^{\calJ} $, the completion of $\KBq$ with respect to
the Vassiliev--Goussarov filtration or the $\J$-adic filtration.
 Let $\wh{\KB}=\wh{\KB}^{\V}=\wh{\KB}^{\J}$ be the completion
of $\KB$ similarly defined.  Then $\wh{\KBq}$ is naturally identified
with the non-strictification $(\wh{\KB})_q$ of $\wh{\KB}$.

Let $\Gr\KBq$ denote the graded linear braided monoidal category
associated to the filtration $\V=\J$ of $\KBq$.  The braidings
$\psi_{v,w}$ {in $\Gr\KBq$}
are actually symmetries, i.e., we
have $\psi_{w,v}\psi_{v,w}=\id_{v,w}$ in $\Gr\KBq$.  Indeed, since
$\psi_{w,v}\psi_{v,w}$ and $\id_{v,w}$ are related by finitely many
crossing changes, we have
$$\psi_{w,v}\psi_{v,w}-\id_{v \otimes w}\in\V^1.$$
Thus, $\Gr\KBq$ is a graded
linear symmetric (non-strict) monoidal category.  Similarly,
$\Gr\KB$ is a graded linear symmetric strict monoidal category.

\subsection{The degree filtration of $\Aq$}  \label{sec:degree-filtration}

Let $\Aq$ denote the non-strictification of the linear strict
monoidal category $\AB$.  The grading of $\AB$ induces that of $\Aq$.
Thus $\Aq$ has a \emph{degree filtration}
$$
\AB_q = \calD^0 \supset  \calD^1 \supset  \calD^2 \supset \cdots
$$
defined by
$$
\calD^k(v,w) = \bigoplus_{i\ge k} \AB_i(\vert v\vert,\vert w \vert ) \subset \AB_q(v,w)
\quad \quad \hbox{for } k\ge0,\ v,w\in\Mag(\bu).
$$

Now we give a categorical description of the degree filtration $\calD$. 
{Each} of the generators of the monoidal category $\AB$
provided by Theorem \ref{AB-pres},
say $f\in\AB(m,n)$, {have a lift in}  $f\in\Aq(\bu^{\ot m},\bu^{\ot n})$.
Let $\calI$ be the ideal of $\AB_q$
generated by the Casimir element $r\in\Aq(\vn,\bu)$.
Then $\I$ is also generated by the Casimir $2$-tensor
$c\in\Aq(\vn,\bu\bu)$ since
$$
r=\frac12\mu c \quad \hbox{and} \quad c=\De r-r\ot\et-\et\ot r.
$$

\begin{proposition} \label{prop:D=I}
The degree filtration coincides with the $\calI$-adic filtration;
i.e., we have $\calD^k = \calI^k$ for $k\geq 0$.
\end{proposition}

\begin{proof}
We have $\calI^k \subset \calD^k$ for $k\geq 0$ since $\calD$ is a
filtration and we have  $r\in\calD^1$.

To prove $ \calD^k \subset \calI^k$, consider a restricted chord
diagram $D \in\calD^k(v,w)$.  By moving the $k$ chords towards the
top-right corner of a projection diagram of $D$, we obtain $D =
D'\circ\big(\id_v \otimes\, c^{\ot k}\big)\in\I^k$, where $D'$ is a morphism in $\Aq$.
\end{proof}

\begin{remark}
We can define the filtrations $\calD$ and $\calI$ in the linear strict monoidal category $\AB$ as well.
Proposition \ref{prop:D=I} is valid in this setting, too.
\end{remark}

We can naturally identify $\wAq^{\calD}=\wAq^{\I}$,
namely the degree-completion or the $\I$-adic completion of $\AB_q$, with the
non-strictification $\hA_q$ of $\hA$ defined in Section~\ref{sec:A_q}.  It should not be confused with the monoidal
category $\hAqp$, which is a deformation of $\hAq$ whose associativity
 isomorphisms involve a Drinfeld associator $\varphi$.
However, $\hAq$ and $\hAqp$ have naturally identified underlying categories and
tensor product {functors}.
Thus we may regard the ideals $\hD^k$
of $\hAq=\wAqD$ as ideals of~$\hAqp$.

Consider the graded linear braided monoidal category $\Gr\hAqp$
associated to the filtration $\hD$ on $\hAq^\varphi$.  We have
the following.

\begin{proposition}
  \label{r24}
  The category $\Gr\hAqp$ is symmetric monoidal, and is isomorphic to
  $\Aq$ as a graded linear symmetric monoidal category.  (Thus,
  the structure of $\Gr\hAqp$ does not depend on the choice of
  $\varphi$.)
\end{proposition}

\begin{proof}
  The braiding $\psi_{v,w}$ in $\hAqp$ defined in \eqref{eq:Psi}
  becomes symmetric in $\Gr\hAqp$, i.e.,
  $\psi_{w,v}\psi_{v,w}=\id_{v\ot w}$ in $\Gr\hAqp$, since
  $\psi_{w,v}\psi_{v,w}-\id_{v\ot w}\in\hD^1$. Thus, $\Gr\hAqp$ is symmetric monoidal.

  The associativity  isomorphism $\alpha_{u,v,w}:(u\ot v)\ot w\to
  u\ot(v\ot w)$ in $\hAqp$ defined in \eqref{eq:Alpha} is congruent
  modulo $\hD^1$ to the associativity  isomorphism
  \begin{gather*}
    \alpha_{u,v,w}= \id
    \in\widehat{\AB}(|u|+|v|+|w|,|u|+|v|+|w|)= \hAq((u\ot v)\ot w,u\ot(v\ot w))
  \end{gather*}
  in $\hAq$.  Similarly, the braiding $\psi_{u,v}:u\ot v\to v\ot u$
    in $\hAqp$ defined in \eqref{eq:Psi} is congruent modulo $\hD^1$ to the symmetry
    $P_{u,v}:u\ot v\to v\ot u$ in $\hAq$.
 Hence $\Gr\hAqp$ is isomorphic to $\Gr\hAq=\Aq$ as a linear
  symmetric monoidal category.
\end{proof}

In the following we identify $\Gr\hAqp$ with $\Aq$.

\subsection{Universality of $\Zqphi$}  \label{sec:universality-zq}

We first check that $\Zqphi$ is filtration-preserving.

\begin{proposition}
  \label{r18}
  The functor $\Zqphi:\KBq\to\hAqp$ induced by $\Zqphi:\Bq\to\hAqp$
  preserves filtrations; i.e., $\Zqphi(\V^k)\subset\hD^k$ for
  $k\ge0$.  Hence $\Zqphi$ induces a filtered linear
  braided monoidal functor
  \begin{gather}
    \Zqphi:\wh{\KBq}\longrightarrow\hAqp.
  \end{gather}
\end{proposition}

\begin{proof}
  Since $r_+-\et$ generates the ideal $\J=\V^1\subset\KBq$, since we have
  \begin{equation}
    \label{eq:degree_1}  \Zqphi(r_+-\eta)
    \by{eq:Z(r+)} \bfr^{-1} - \eta = -r +(\deg \geq 2) \ \in\widehat{\calD}^1,
  \end{equation}
  and since $\Zqphi$ is a monoidal functor, it follows that
  $\Zqphi(\V^1)\subset\hD^1$.  Hence,
  $$
  \Zqphi(\V^k)\subset \Zqphi(\V^1)^k\subset(\hD^1)^k\subset\hD^k.
  $$
\end{proof}

By Proposition \ref{r18}, $\Zqphi:\hKBq\to\hAqp$ induces a graded
linear braided monoidal functor
\begin{gather*}
  \Gr\Zqphi:\Gr\KBq\longrightarrow\Gr\hAqp=\Aq.
\end{gather*}
We already know that both $\Gr\KBq$ and $\Gr\hAqp=\Aq$ are symmetric
monoidal.  Thus, $\Gr\Zqphi$ is a graded linear \emph{symmetric}
monoidal functor.  Recall that $\Gr\KBq$ and $\Aq$ are the
non-strictifications of $\Gr\KB$ and $\AB$, respectively.  It is
easy to see that the functor $\Gr\Zqphi$ is the non-strictification of a
unique graded linear symmetric monoidal functor
\begin{gather*}
  \bZ:\Gr\KB\longrightarrow \AB.
\end{gather*}
More concretely, we can define $\bZ$ by
$$
\bZ(t) := \big(\hbox{degree $k$ part of $\Zqphi(t_q)$}\big)
$$
for $t\in\V^k(m,n)$, $m,n,k\ge0$, where $t_q=t\in\V^k(\bu^{\ot m},\bu^{\ot n})\subset
  \KBq(\bu^{\ot m},\bu^{\ot n})$.

\begin{theorem}
  \label{r20}
  The functor $\Gr\Zqphi:\Gr\KBq\to\Aq$ is an isomorphism of graded
  linear symmetric (non-strict) monoidal categories.  The functor
  $\bZ:\Gr\KB\to\AB$ is an isomorphism of graded linear symmetric
  strict monoidal categories.
\end{theorem}

\begin{proof}
  It suffices to prove the latter assertion, since the former
  corresponds to the latter by non-strictification.

Let $H^\B=(1,\mu,\eta,\De,\ep,S)$ be the Hopf algebra in $\B$ defined
in Section \ref{sec:generators_B}.  It induces a Hopf algebra
$H^{\Gr\KB}=(1,\mu,\eta,\De,\ep,S)$ in $\Gr\KB$,  concentrated
in the degree $0$ part $\Gr^0\KB=\V^0/\V^1$.

Let us prove that $H^{\Gr\KB}$ has a Casimir Hopf algebra structure.
Since $\Delta$ and $\psi_{1,1}\Delta$ in $\B$ are related by some
crossing changes, we have $\Delta-\psi_{1,1}\De\in\V^1(1,2)$, i.e.,
$\De=\psi_{1,1}\De$ in $\Gr\KB$.  Thus $H^{\Gr\KB}$ is cocommutative.
Furthermore,
$$
\widetilde{\cc} \, := \cc-\eta^{\otimes 2}\ \in\V^1(0,2)
$$
gives a Casimir $2$-tensor for $H^{\Gr\KB}$.  Indeed, the identities in $\B$
\begin{eqnarray*}
  (\Delta \otimes \id_1)\,  \cc &=& (\id_2 \otimes \mu)\,   (\id_1 \otimes \cc \otimes \id_1 )\, \cc,\\
 \psi\,  \cc &=& (\ad \otimes \id_1)\,  (r_+\otimes \cc), \\
 \cc\,  \epsilon &=& (\ad\ot\ad)\, (\id_1 \ot \psi\ot \id_1)\, (\Delta\ot \cc)
\end{eqnarray*}
imply
\begin{gather*}
 (\Delta \otimes \id_1)\,  \widetilde{\cc}  - (\id_1 \otimes \eta \otimes \id_1)\,  \widetilde{\cc} - \eta \otimes \widetilde{\cc}
 \ = \ (\id_2 \otimes \mu)\,   (\id_1 \otimes \widetilde{\cc}  \otimes \id_1)\, \widetilde{\cc}\ \in\V^2(0,3),\\
  \psi\, \widetilde{\cc}   - \widetilde{\cc}\ =\ \psi\, \cc  - \cc \ = \ (\ad \otimes \id_1)\,  \big((r_+-\eta)\otimes \widetilde{\cc} \big) \ \in\V^2(0,2),\\
 \widetilde{\cc}\,  \epsilon\ =\ (\ad\ot\ad)\, (\id_1 \ot \psi\ot \id_1)\,  (\Delta\ot \widetilde{\cc}\, ),
\end{gather*}
respectively.
By Theorem \ref{AB-pres},  there is a unique symmetric monoidal functor
$$
\GG: \AB \longrightarrow \Gr\KB
$$
which maps the Casimir Hopf algebra
$(H^{\AB},c)$ in $\AB$ to the Casimir Hopf algebra
$(H^{\Gr\KB},-\wt{\cc})$ in $\Gr\KB$.

We prove that $\GG$ is full.  By Proposition \ref{prop:V=J}, $\Gr\KB$
is generated by its degree~$0$ part $\V^0/\V^1$ and its degree~$1$
part $\V^1/\V^2$.  We have $\V^0/\V^1=\GG(\AB_0)$ since $\V^0/\V^1$
(resp.\ $\AB_0$) is generated by the Hopf algebra $H^{\Gr\KB}$
(resp.\ $H^{\AB}$).  Thus it suffices to prove $\V^1/\V^2=
\calJ^1/\calJ^2\subset\GG(\AB_1)$.  As an ideal of $\KB/\J^2$, $\J$ is
generated by $\wtr_+:= r_+-\eta \in\calJ^1/\calJ^2$.  Hence it
suffices to check $\wtr_+\in\GG(\AB_1)$.  Since $\mu (r_+\otimes r_+)
= \mu \cc \in\B(0,1)$, we have
$$
\calJ^2\ni
\mu\,  ( \widetilde{r}_+\otimes  \widetilde{r}_+) = \mu\,  \cc - 2r_++\eta =\mu\, \wt{\cc} - 2\wtr_+.
$$
Therefore,
\begin{equation}    \label{eq:sigma(r)}
\GG(r) = \GG\big(  \frac{1}{2} \mu\,   c\big) =
  \frac{1}{2} \GG(\mu)\,   \GG(c) = - \frac{1}{2} \mu\,  \widetilde{\cc} = - \widetilde{r}_+.
\end{equation}

The linear symmetric monoidal functor $\bZ\GG:\AB\to\AB$
preserves $H^{\AB}$.  Moreover, \eqref{eq:degree_1} and
\eqref{eq:sigma(r)} {imply} $\bZ\GG(r)=r$.  Thus $\bZ\GG$ is the
identity on the generators of $\AB$; hence $\bZ\GG=\id_{\AB}$.
Therefore $\bZ$ is an isomorphism.
\end{proof}

We conclude this section with a stronger version of Theorem \ref{r20}
and two remarks about it.

\begin{theorem} \label{th:universality}
  The functor $\Zqphi:\hKBq\to\hAqp$ is an isomorphism of filtered
  linear braided (non-strict) monoidal categories.
\end{theorem}

\begin{proof}
  Theorem \ref{r20} and an induction on $k$ shows that
  $\Zqphi:\widehat{\KBq}/\widehat{\calV}^{k+1}\to\hAqp/\widehat{\calD}^{k+1}$
  is an isomorphism for all $k\geq 0$.
  Hence $\Zqphi:\hKBq\to\hAqp$ is an isomorphism.
\end{proof}

\begin{remark}
The map $G: \AB(m,n) \to (\Gr\KB)(m,n)$, $m,n\ge0$, defined in the
proof of Theorem \ref{r20} can also be constructed as a direct sum of
maps
$$
G_d: \AB(m,n)_d \longrightarrow (\Gr\KB)(m,n)_d
$$ indexed by $d\in\bfF(n,m)$, using calculus of claspers instead of
the presentation of~$\AB$.  More precisely, the map $G_d$ is defined
by \emph{fixing} an $n$-component bottom tangle $\gamma_d$ in $V_m$ of
homotopy {class} $d$, and by realizing every $(m,n)$-Jacobi diagram $D$ of
homotopy {class} $d$ as a ``simple strict graph clasper'' $C_D$ on
$\gamma_d$ in the sense of \cite{Habiro_claspers}.  Then $G_d(D)$ is
defined as the alternating sum of clasper surgeries on the connected
components of $C_D$. See \cite[\S 8.2]{Habiro_claspers} for the
special case $m=0$.
\end{remark}

\begin{remark}
  Theorem \ref{th:universality} implies the universal property of $Z$ {and $\Zqphi$}
  among $\K$-valued Vassiliev--Goussarov invariants of bottom
  tangles in handlebodies.  For links in handlebodies (and, more
  generally, for links in thickened surfaces), similar results have
  been obtained in \cite{AMR,Lieberum}.
\end{remark}

%
%
\section{Relationship with the LMO functor}  \label{sec:LMO_functor}

In this section, we explain how the extended Kontsevich integral
relates to the LMO functor introduced in \cite{CHM}.

\subsection{Review of the LMO functor}  \label{sec:review-lmo-functor}

The LMO functor as defined in \cite{CHM}
\begin{gather*}
  \tiZ:\LCob_q\lto\tsA
\end{gather*}
is a functor from the category $\LCob_q$ of Lagrangian $q$-cobordisms
to (the degree-completion of) the category $\tsA$ of ``top-substantial
Jacobi diagrams''.  Here we consider the restriction of $\tiZ$ to the
category $\sLCob_q\cong\Bq$ of \emph{special Lagrangian
$q$-cobordisms}, which is the non-strictification of the strict
monoidal category $\sLCob\cong\B$ recalled in Section \ref{sec:sLCob}.  By
\cite[Corollary 5.4]{CHM}, it turns out that $\tiZ$ on $\sLCob$ takes
values in a subcategory $\sA$ of $\tsA$, which we call the category of
\emph{special (top-substantial) Jacobi diagrams}.  Thus, we here consider the
restricted version of the LMO functor:
\begin{gather*}
    \tiZ:\sLCob_q\lto\sA.
\end{gather*}

The category $\sA$ is defined as follows.  Set $\Ob(\sA)=\NZ$.
Given a finite set $U$, {a} \emph{$U$-labeled Jacobi diagram} is a
unitrivalent graph {with oriented trivalent vertices,}
{with} each univalent vertex  labeled by an element of $U$.
We identify two $U$-labeled Jacobi diagrams if there is a
homeomorphism from one to the other preserving the vertex-orientations
and the labelings.  Let $\A(U)$ denote the vector space generated
by $U$-labeled Jacobi diagrams modulo the AS and IHX
relations~\eqref{eq:AS_IHX}.  For $m,n\ge0$, let $\sA(m,n)$ be
the subspace of
$$
\A\big(\{1^+,\dots,m^+\} \cup \{1^-,\dots,n^-\}\big)
$$ spanned by {\emph{special}} Jacobi diagrams, which are those
diagrams with no connected component without labels in
$\{1^-,\dots,n^-\}$.  {(Recall that a \emph{top-substantial} Jacobi
diagram in {\cite{CHM} allows} such connected components that are not struts.  Thus, we have $\sA(m,n)\subset\tsA(m,n)$, where
$\tsA(m,n)$ is the space of top-substantial Jacobi diagrams.)}  The
composition $D'\circ D$ of two {special} Jacobi diagrams {$m\overset{D}{\lto}n\overset{D'}{\lto}p$ in $\sA$}
 is the sum of all possible ways of gluing some $i^-$-vertices of $D$ with some
$i^+$-vertices of $D'$ for all $i\in\{1,\dots,n\}$.  Define the
identity morphisms in ${\sA}$ by
$$
\id_m = \exp_\sqcup \left( \sum_{i=1}^m
 \! \! \begin{array}{c}
\phantom{.}\\[-10pt]
{\labellist \small \hair 2pt
\pinlabel {\scriptsize $i^-$} [l] at 37 18
\pinlabel {\scriptsize $i^+$} [l] at 37 170
\endlabellist
\includegraphics[scale=0.1]{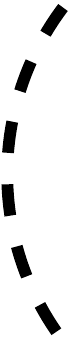}}
\end{array} \ \ \right) :m\lto m,
$$ where $\sqcup$ denotes the disjoint union of Jacobi diagrams.

The category $\sA$ has a strict monoidal structure such that $m\otimes
m'=m+m'$ for $m,m'\ge0$, and the tensor product $D
\otimes D'$ of two {special} Jacobi diagrams $D$ and $D'$ is the
disjoint union $D \sqcup D'$ with the appropriate re-numbering of the
colors of~$D'$.  The category $\sA$ is  graded, where
the \emph{degree} of a {special} Jacobi diagram {is} half the
total number of vertices.  The degree-completion of $\sA$ is also denoted by~$\sA$.

The \emph{LMO functor} is a functor $\widetilde{Z}: \sLCob_q \to \sA$
with the following properties:
\begin{itemize}
\item[(i)] We have $\tZ(w)=|w|$ for $w\in\Mag(\bu)$.
\item[(ii)] Let $T\in\Bq(\varnothing,w) \subset
  \T_q\big(\varnothing,w(+-)\big)$ with $w\in\Mag(\bullet)$, $|w|=n$
  and let $E_T\in\sLCob_q(\varnothing,w)$ be the cobordism
  corresponding to $T$.  Then we have $\tZ(E_T)=\chi^{-1}Z(T)$, where $Z(T)$ is the usual
  Kontsevich integral (as defined in Section \ref{sec:usual_Kontsevich})
  and
$$
\chi: \sA(0,n) \stackrel{\cong }{\longrightarrow} \A(X_n)
$$
is the diagrammatic analog of the PBW isomorphism (see \cite{BN2}).
Here, recall  $X_n=\!\figtotext{10}{8}{capleft}\!\!_{1} \cdots\!\!
\figtotext{10}{8}{capleft}\!\!_n$.
\item[(iii)] For {morphisms $T$ and $T'$ in $\LCob_q$,}
  we have $\widetilde{Z}(T \otimes T') =
\widetilde{Z}(T) \otimes \widetilde{Z}(T')$.
\end{itemize}

\subsection{From the extended Kontsevich integral to the LMO functor}  \label{sec:Z_to_LMO}

For $m,n\ge0$, define a linear map $\kappa: \AB(m,n) \to \sA(m,n)$ by
$$
\kappa\left(\!\! \centre{\labellist
\scriptsize\hair 2pt
 \pinlabel {$1$} at 236 20
 \pinlabel {$n$}  at 491 20
 \pinlabel {$1$} [bl] at 261 300
 \pinlabel {$m$} [bl] at 621 300
 \pinlabel {$\cdots$}  at 364 267
  \pinlabel {$\cdots$}  at 364 15
\endlabellist
\centering
\includegraphics[scale=0.16]{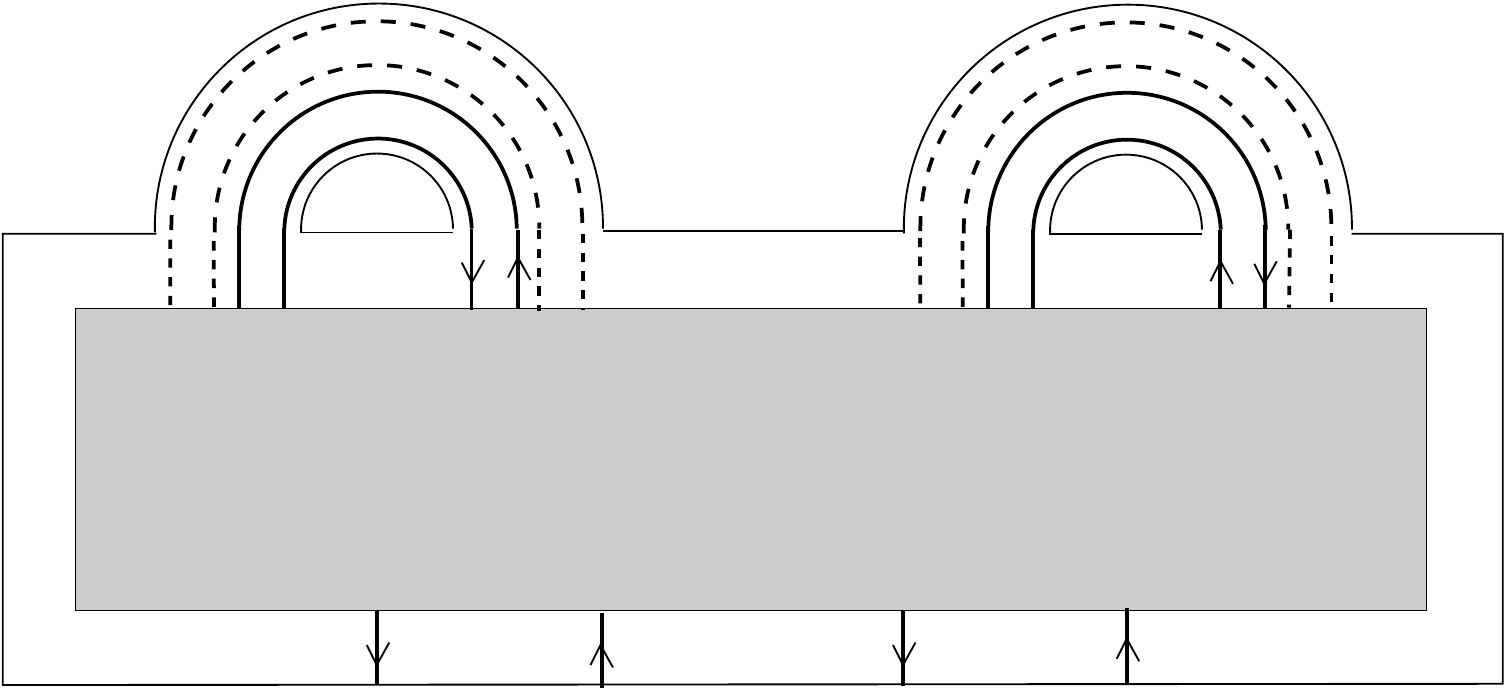}}\!\! \right)
:= \chi^{-1}\left(\quad \ \ \centre{\labellist
\scriptsize\hair 2pt
 \pinlabel {$1$}  at 201 16
 \pinlabel {$n$}  at 455 16
 \pinlabel {$\cdots$}  at 329 16
 \pinlabel {$\cdots$}  at 330 274
 \pinlabel {\tiny $\exp(1^+$} [r] at 16 220
 \pinlabel {\tiny $)$} [l] at 105 221
 \pinlabel {\tiny $\exp(m^+$} [r] at 376 221
 \pinlabel {\tiny $)$} [l] at 465 221
\endlabellist
\centering
\includegraphics[scale=0.25]{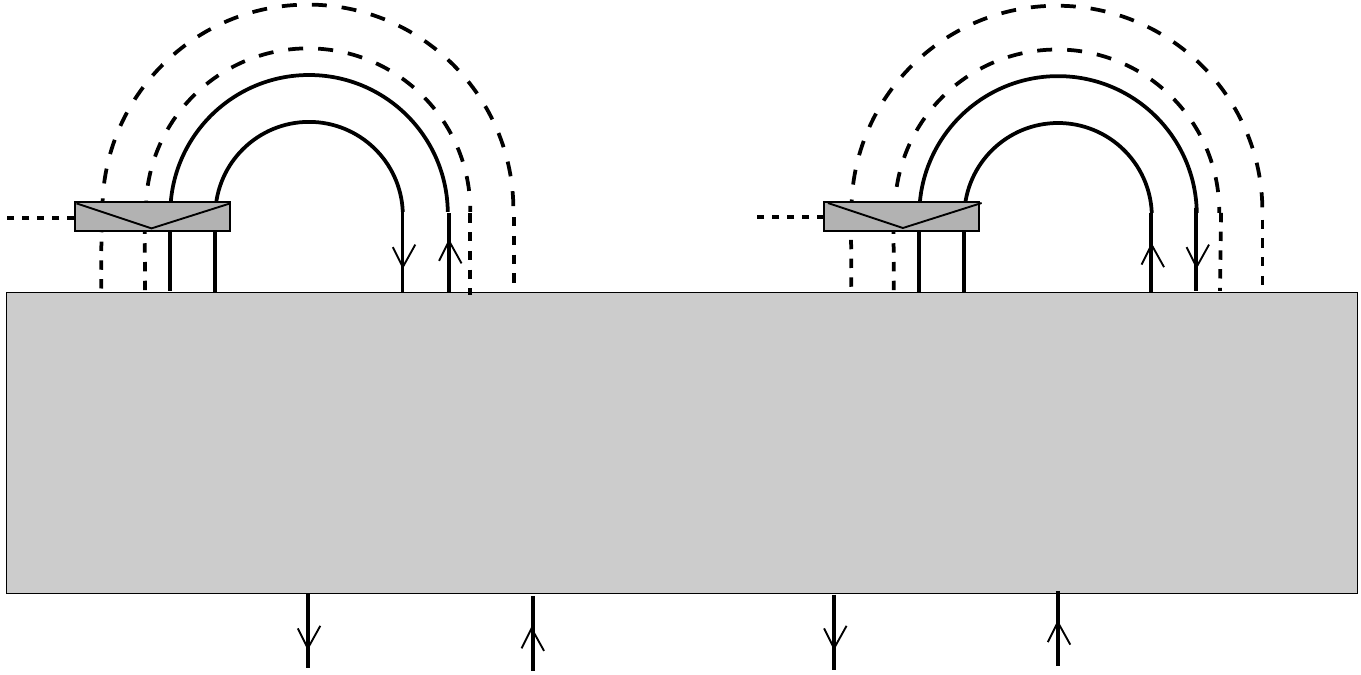}}\!\!\!\!\right),
$$ where an $F(x_1,\dots,x_m)$-colored Jacobi diagram on $X_n$, is
presented by a projection diagram in the square with handles.  By the
IHX and STU relations, $\kappa$ is well-defined. 
{Note that $\kappa$ is an {analog} of the ``hair map'' considered by Garoufalidis, Kricker and Rozansky in \cite{GK,GR}.}

\begin{proposition}
The maps $\kappa: \AB(m,n) \to \sA(m,n)$ {for $m,n\ge0$} define a
 monoidal functor $\kappa: \AB \to \sA$, which induces a monoidal
 functor $\kappa:\hA\to\sA$ by continuity.
\end{proposition}

\begin{proof}
Consider two restricted Jacobi diagrams $D$ and $D'$
with square presentations $S$ and $S'$, respectively:
$$
D' = \centre{\labellist
\tiny\hair 2pt
 \pinlabel {$1$} at 236 30
 \pinlabel {$p$}  at 491 29
  \pinlabel {$S'$}  at 360 130
 \pinlabel {$1$} [bl] at 261 300
 \pinlabel {$n$} [bl] at 621 300
 \pinlabel {$\cdots$}  at 364 267
  \pinlabel {$\cdots$}  at 364 15
\endlabellist
\centering
\includegraphics[scale=0.13]{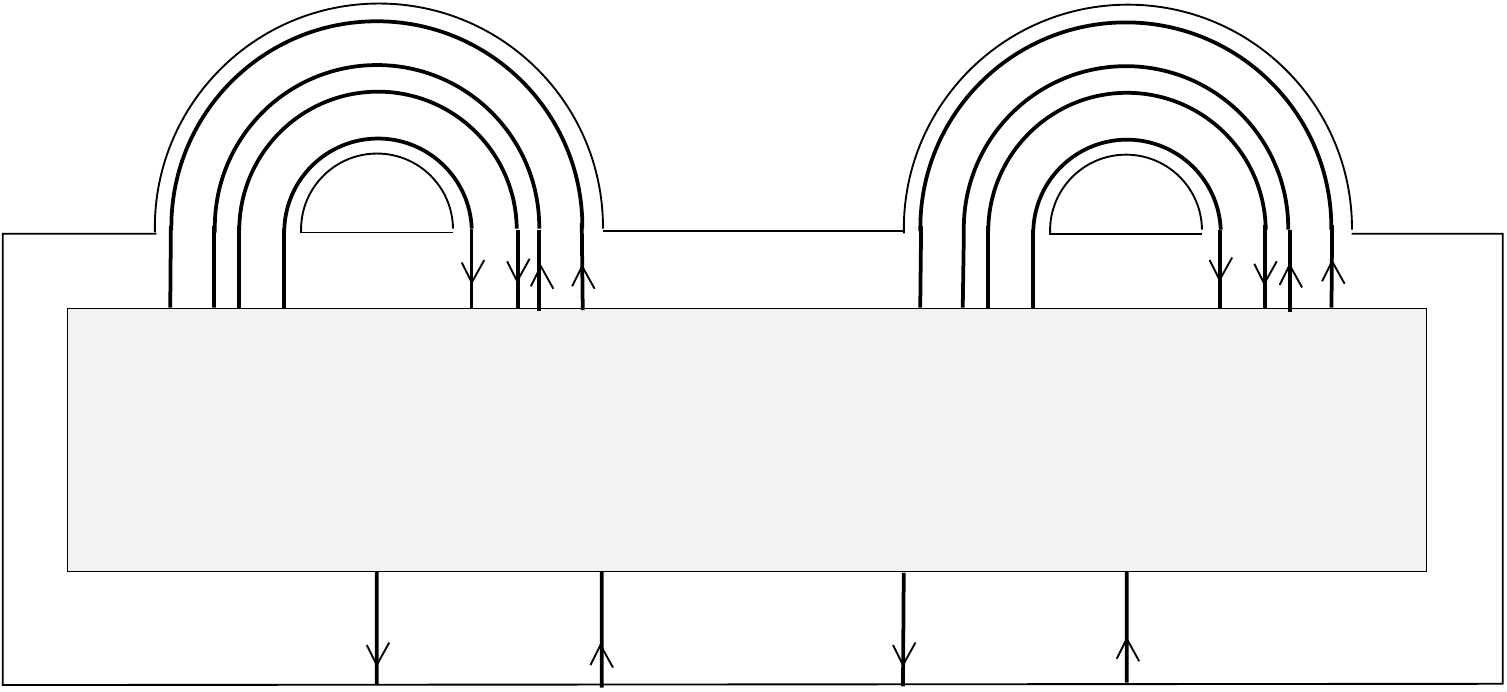}} \in\AB(n,p),
\ D = \centre{\labellist
\tiny\hair 2pt
 \pinlabel {$1$} at 236 30
 \pinlabel {$n$}  at 491 29
 \pinlabel {$1$} [bl] at 261 300
   \pinlabel {$S$}  at 360 130
 \pinlabel {$m$} [bl] at 621 300
 \pinlabel {$\cdots$}  at 364 267
  \pinlabel {$\cdots$}  at 364 15
\endlabellist
\centering
\includegraphics[scale=0.13]{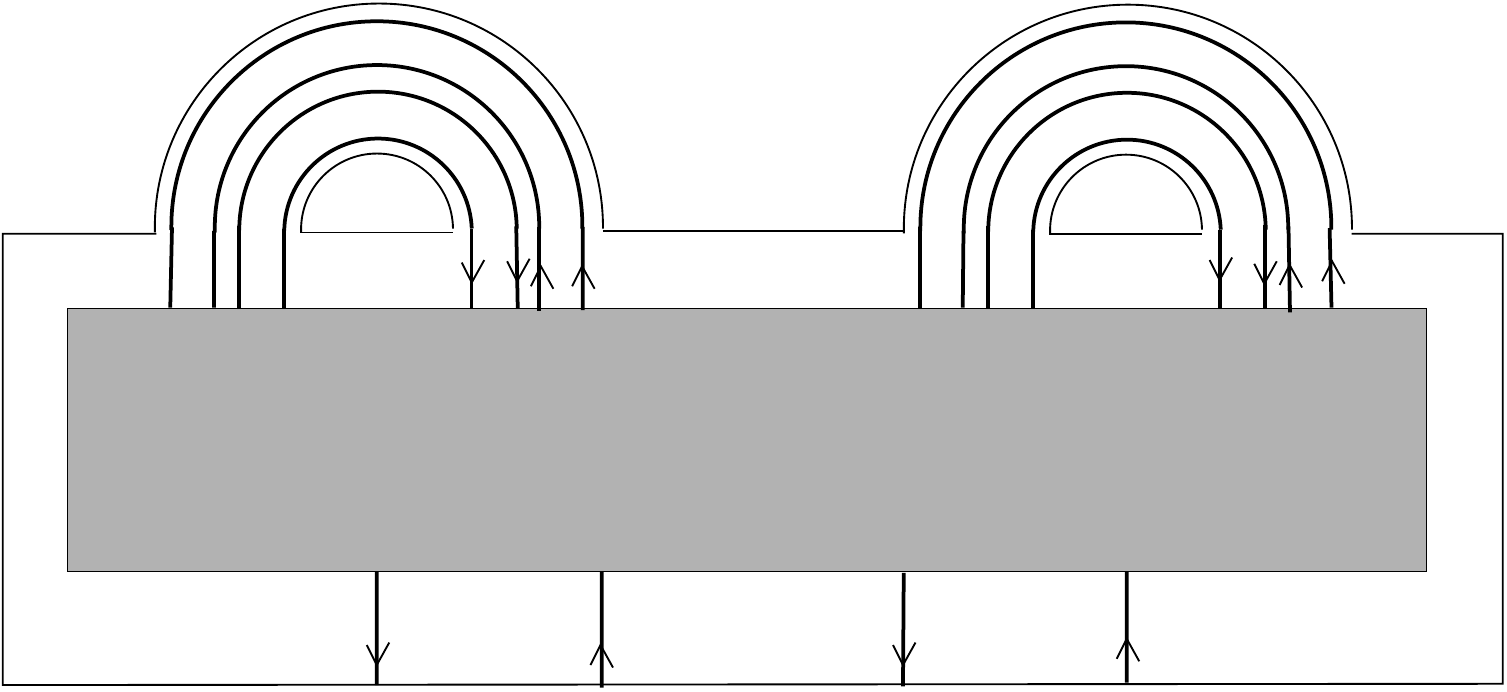}} \in\AB(m,n).
$$
In what follows, we express exponentials with square
brackets and, given two Jacobi diagrams $E$ and $E'$ labeled by the
finite sets $\{1^-,\dots,n^-\}$ and $\{1^+,\dots,n^+\}$, respectively,
let $\langle E', E \rangle$ denote the sum of all possible ways
of gluing some $i^-$-vertices of $E$ with some $i^+$-vertices of $E'$
for all $i\in\{1,\dots,n\}$.  Then $\kappa(D')\circ \kappa(D)$ is equal~to
\begin{eqnarray*}
&& \kappa \Big( \centre{\labellist
\tiny\hair 2pt
 \pinlabel {$1$} at 236 30
 \pinlabel {$p$}  at 491 29
 \pinlabel {$1$} [bl] at 261 300
 \pinlabel {$n$} [bl] at 621 300
 \pinlabel {$\cdots$}  at 364 267
  \pinlabel {$\cdots$}  at 364 15
   \pinlabel {$S'$}  at 360 130
\endlabellist
\centering
\includegraphics[scale=0.14]{diag2}}\Big)
\circ \kappa \Big(\centre{\labellist
\tiny\hair 2pt
 \pinlabel {$1$} at 236 30
 \pinlabel {$n$}  at 491 29
 \pinlabel {$1$} [bl] at 261 300
 \pinlabel {$m$} [bl] at 621 300
 \pinlabel {$\cdots$}  at 364 267
  \pinlabel {$\cdots$}  at 364 15
   \pinlabel {$S$}  at 360 130
\endlabellist
\centering
\includegraphics[scale=0.14]{diag1}}  \Big)\\
&=&
\Big\langle\ \chi^{-1}\Big( \quad \centre{\labellist
\tiny \hair 2pt
 \pinlabel {$1$}  at 201 16
 \pinlabel {$p$}  at 455 16
 \pinlabel {$\cdots$}  at 329 16
 \pinlabel {$\cdots$}  at 330 274
 \pinlabel {$[1^+$} [r] at 20 220
 \pinlabel {$]$} [l] at 105 221
 \pinlabel {$[n^+$} [r] at 380 221
 \pinlabel { $]$} [l] at 465 221
  \pinlabel {$S'$}  at 340 130
\endlabellist
\centering
\includegraphics[scale=0.15]{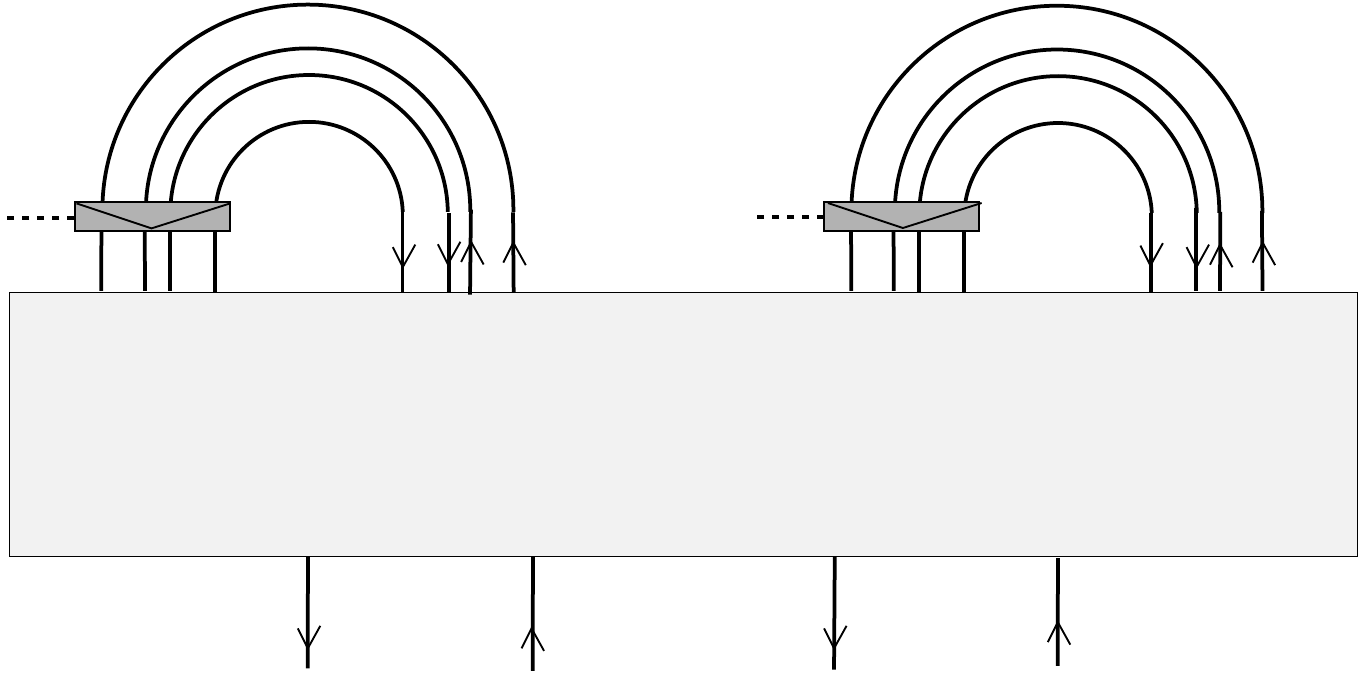}}\Big) \ , \ \chi^{-1}\Big( \quad \centre{\labellist
\tiny \hair 2pt
 \pinlabel {$1$}  at 201 16
 \pinlabel {$n$}  at 455 16
 \pinlabel {$\cdots$}  at 329 16
 \pinlabel {$\cdots$}  at 330 274
 \pinlabel {$[1^+$} [r] at 20 220
 \pinlabel {$]$} [l] at 105 221
 \pinlabel {$[m^+$} [r] at 380 221
 \pinlabel { $]$} [l] at 465 221
  \pinlabel {$S$}  at 340 130
\endlabellist
\centering
\includegraphics[scale=0.15]{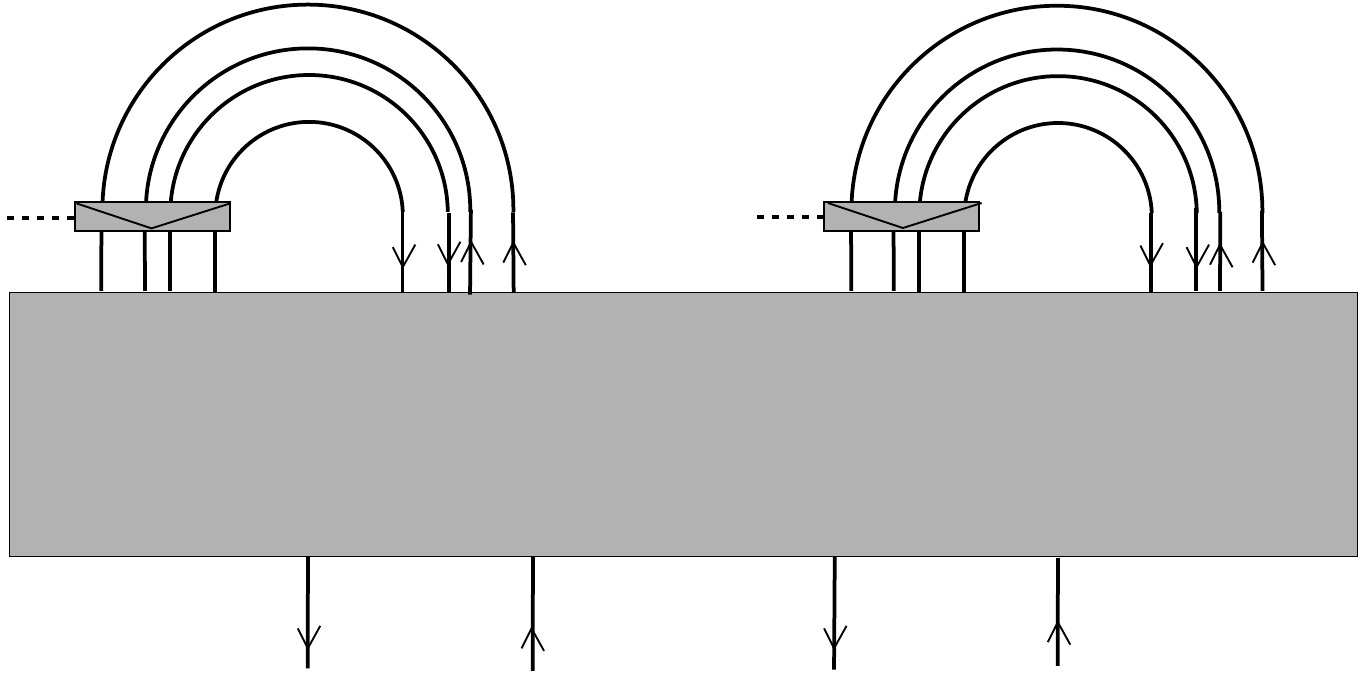}}\Big) \ \Big\rangle\\
&=&
\chi^{-1}\Bigg(
\Big\langle\  \quad \centre{\labellist
\tiny \hair 2pt
 \pinlabel {$1$}  at 201 16
 \pinlabel {$p$}  at 455 16
 \pinlabel {$\cdots$}  at 329 16
 \pinlabel {$\cdots$}  at 330 274
 \pinlabel {$[1^+$} [r] at 20 220
 \pinlabel {$]$} [l] at 105 221
 \pinlabel {$[n^+$} [r] at 380 221
 \pinlabel { $]$} [l] at 465 221
  \pinlabel {$S'$}  at 340 130
\endlabellist
\centering
\includegraphics[scale=0.15]{diag3}} \ , \ \chi^{-1}\Big( \quad \centre{\labellist
\tiny \hair 2pt
 \pinlabel {$1$}  at 201 16
 \pinlabel {$n$}  at 455 16
 \pinlabel {$\cdots$}  at 329 16
 \pinlabel {$\cdots$}  at 330 274
 \pinlabel {$[1^+$} [r] at 20 220
 \pinlabel {$]$} [l] at 105 221
 \pinlabel {$[m^+$} [r] at 380 221
 \pinlabel { $]$} [l] at 465 221
  \pinlabel {$S$}  at 340 130
\endlabellist
\centering
\includegraphics[scale=0.15]{diag4}}\Big) \ \Big\rangle \Bigg)\\
&=& \chi^{-1} \Bigg( \centre{\labellist
\tiny\hair 2pt
 \pinlabel {$\cdots$}  at 327 200
 \pinlabel {$S'$}  at 325 120
 \pinlabel {$\cdots$}  at 325 27
 \pinlabel {$1$}  at 197 27
 \pinlabel {$p$}  at 448 25
 \pinlabel {$1$}  at 143 202
 \pinlabel {$n$}  at 499 203
\endlabellist
\centering
\includegraphics[scale=0.15]{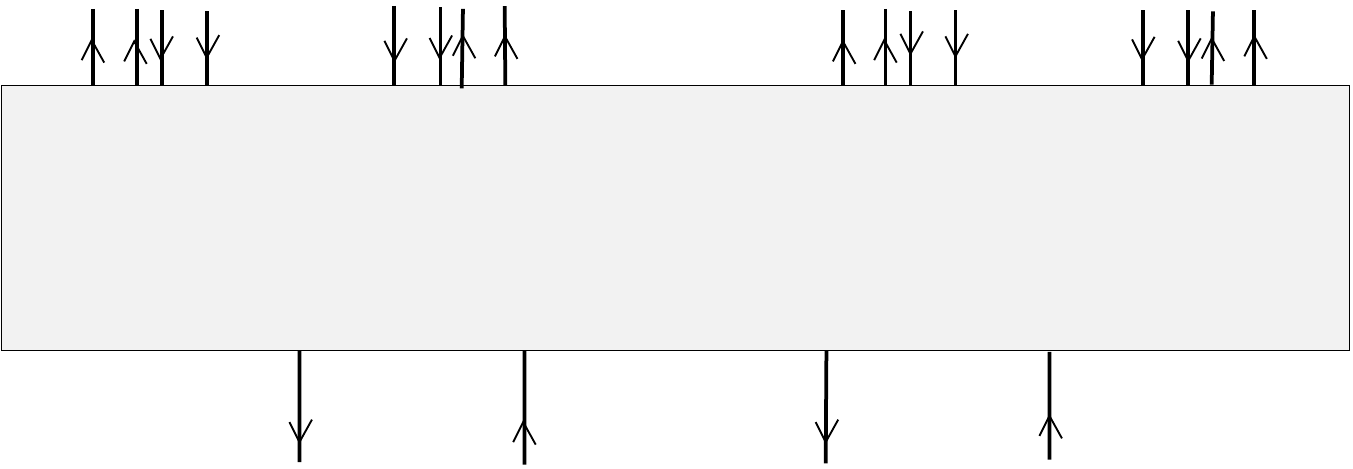}} \circ
C_\varpi\Big(\quad \centre{\labellist
\tiny \hair 2pt
 \pinlabel {$1$}  at 201 16
 \pinlabel {$n$}  at 455 16
 \pinlabel {$\cdots$}  at 329 16
 \pinlabel {$\cdots$}  at 330 274
 \pinlabel {$[1^+$} [r] at 20 220
 \pinlabel {$]$} [l] at 105 221
 \pinlabel {$[m^+$} [r] at 380 221
 \pinlabel { $]$} [l] at 465 221
  \pinlabel {$S$}  at 340 130
\endlabellist
\centering
\includegraphics[scale=0.15]{diag4}} \Big) \Bigg) \\
&=& \chi^{-1} \Big( \centre{\labellist
\tiny\hair 2pt
 \pinlabel {$\cdots$}  at 327 200
 \pinlabel {$S'$}  at 325 120
 \pinlabel {$\cdots$}  at 325 27
 \pinlabel {$1$}  at 197 27
 \pinlabel {$p$}  at 448 25
 \pinlabel {$1$}  at 143 202
 \pinlabel {$n$}  at 499 203
\endlabellist
\centering
\includegraphics[scale=0.15]{diag5}} \circ \quad
\centre{\labellist
\tiny \hair 2pt
 \pinlabel {$1$}  at 180 25
 \pinlabel {$n$}  at 475 25
 \pinlabel {$\cdots$}  at 329 16
 \pinlabel {$\cdots$}  at 330 274
 \pinlabel {$[1^+$} [r] at 20 220
 \pinlabel {$]$} [l] at 105 221
 \pinlabel {$[m^+$} [r] at 380 221
 \pinlabel { $]$} [l] at 465 221
  \pinlabel {$C_\varpi(S)$}  at 340 130
\endlabellist
\centering
\includegraphics[scale=0.15]{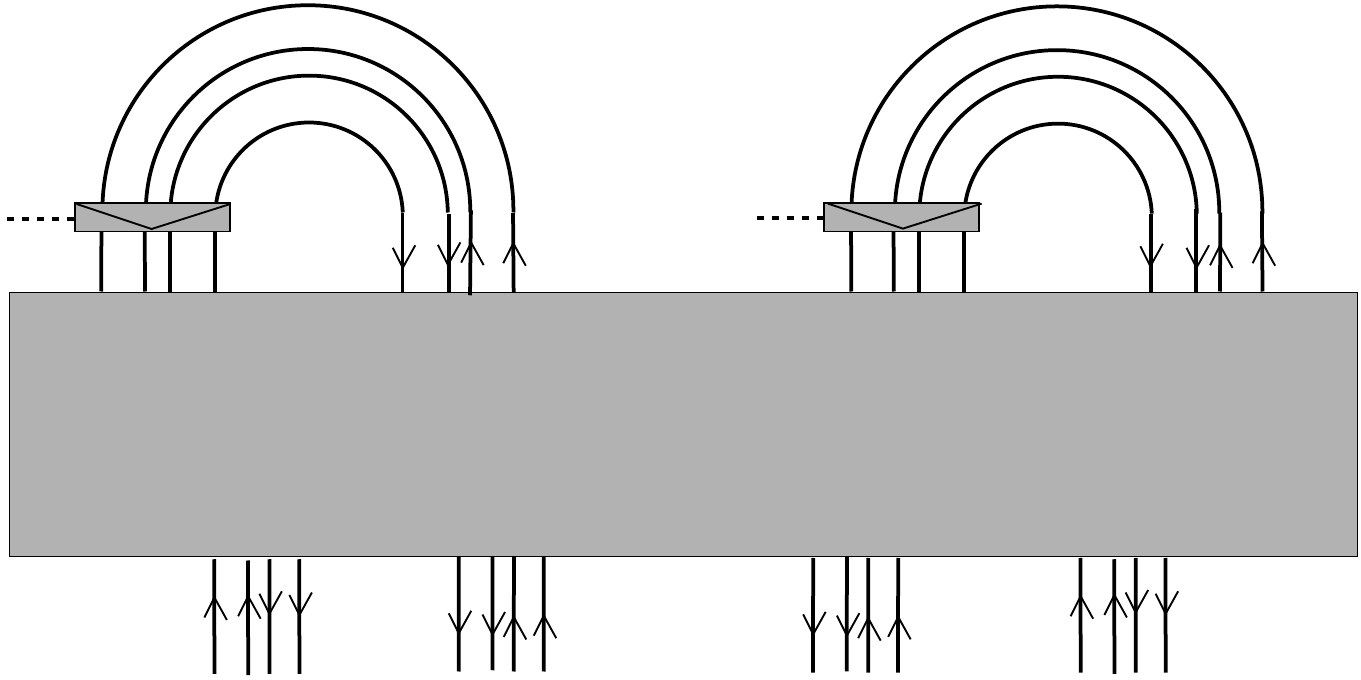}} \Big) \\
&=& \chi^{-1} \Big(
\quad \centre{\labellist
\tiny \hair 2pt
 \pinlabel {$1$}  at 201 16
 \pinlabel {$p$}  at 455 16
 \pinlabel {$\cdots$}  at 329 16
 \pinlabel {$\cdots$}  at 330 274
 \pinlabel {$[1^+$} [r] at 20 220
 \pinlabel {$]$} [l] at 105 221
 \pinlabel {$[m^+$} [r] at 380 221
 \pinlabel { $]$} [l] at 465 221
  \pinlabel {$S' \circ C_\varpi(S)$}  at 340 130
\endlabellist
\centering
\includegraphics[scale=0.15]{diag4}}  \Big)  \ = \ \kappa \Big(\centre{\labellist
\tiny\hair 2pt
 \pinlabel {$1$} at 236 30
 \pinlabel {$p$}  at 491 29
 \pinlabel {$1$} [bl] at 261 300
 \pinlabel {$m$} [bl] at 621 300
 \pinlabel {$\cdots$}  at 364 267
  \pinlabel {$\cdots$}  at 364 15
   \pinlabel {$S' \circ C_\varpi(S)$}  at 360 130
\endlabellist
\centering
\includegraphics[scale=0.13]{diag1}}  \Big),
\end{eqnarray*}
where the last four $\circ$ denote compositions in  $\AT$,
and $\varpi: \pi_0(S) \to \Mon(\pm)$ is an appropriate map.  We
deduce from Example~\ref{ex:restricted} that $\kappa(D')\circ
\kappa(D) = \kappa(D' \circ D)$.
We can easily check $\kappa(\id_m)=\id_m$ for $m\ge0$.
Thus we obtain a functor~$\kappa :\AB\to\sA$, which is obviously
monoidal.
\end{proof}

\begin{theorem} \label{th:Kontsevich_to_LMO}
The following square of functors commutes:
\begin{gather}
  \label{e38}
\xymatrix{
\Bq \ar[d]_-{\cong } \ar[r]^-Z & \hA \ar[d]^-\kappa \\
\sLCob_q \ar[r]_-{\widetilde{Z}} & \sA
}
\end{gather}
\end{theorem}

\begin{proof}
Let $T:v\to w$ in $\Bq$, $|v|=m$, $|w|=n$, and let
  $U:\dbl^v(v_1,\dots, v_m)\to w(+-)$ in $\Tq$ be {a cube presentation of $T$.}
Then $\kappa(Z(T))$ is equal~to
$$
\quad \ \ \labellist
\scriptsize\hair 2pt
 \pinlabel {$U$}  at 140 50
 \pinlabel {$Z(U)$}  at 503 50
 \pinlabel {$\kappa Z\Bigg($} [r] at 0 61
 \pinlabel {$\Bigg)$} [l] at 289 60
 \pinlabel { $=  \kappa\Bigg($}  at 340 59
 \pinlabel {$\cdots$}  at 143 8
 \pinlabel {$\cdots$}  at 143 140
 \pinlabel { ${\cdots}$} at 40 98
 \pinlabel {${\cdots}$}  at 193 98
 \pinlabel {$a_{v_1}$}  at 398 99
 \pinlabel {$a_{v_m}$}  at 550 100
 \pinlabel {$\cdots$}  at 507 140
 \pinlabel {$\cdots$}  at 507 10
 \pinlabel {$\Bigg)$} at 655 59
\endlabellist
\centering
\includegraphics[scale=0.53]{def_Z_B}
$$
$$
\qquad\qquad\qquad\qquad\qquad\qquad\qquad\qquad \qquad \quad
\labellist
\scriptsize\hair 2pt
 \pinlabel {$\cdots$}  at 148 9
 \pinlabel {$a_{v_1}$}  at 43 99
 \pinlabel {$a_{v_m}$}  at 195 100
 \pinlabel {$Z(U)$}  at 146 48
 \pinlabel {\tiny $\exp(1^+$} [r] at 8 118
 \pinlabel {\tiny $)$} [l] at 55 115
 \pinlabel {\tiny $\exp(m^+$} [r] at 163 118
 \pinlabel {\tiny $)$} [l] at 205 116
 \pinlabel {$\cdots$}  at 145 135
  \pinlabel {$\chi^{-1}\Bigg($} [r] at 10 53
    \pinlabel \small {$=$} [r] at -30 53
 \pinlabel {$\Bigg)$} [l] at 279 50
\endlabellist
\centering
\includegraphics[scale=0.55]{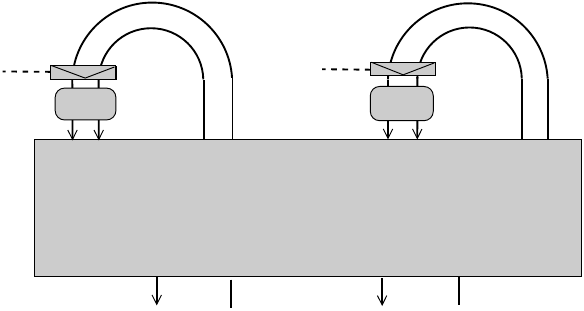}
$$
and the result directly follows from \cite[Lemma 5.5]{CHM}.
\end{proof}

Theorem \ref{th:Kontsevich_to_LMO} shows that the extended Kontsevich
integral $Z$ dominates the LMO functor~$\widetilde{Z}$.  However, the
converse might not hold since, as we will see in the next subsection,
the functor $\kappa$ is not faithful.
Some {other} remarks about the functor $\kappa$ follow.

\begin{remark}  
  \label{r38}
  Theorem \ref{r32}  {in  {Section \ref{sec:intro}}
  } is stated in a way slightly different from Theorem  \ref{th:Kontsevich_to_LMO}.  
  {In fact, the latter differs from the former {simply because} 
    we have restricted the source of $\tiZ$ to $\sLCob_q \subset \LCob_q $
  and its target to $\sA \subset \tsA$.}
\end{remark}

\begin{remark}
  \label{r40}
  Several interesting structures in $\AB$ {(and $\hA$)} are mapped by $\kappa$ into
  the categories $\sA$ and $\tsA$.  {For instance}, the symmetry $P_{m,n}:m+n\to n+m$
  in \eqref{e6} is mapped to  {a symmetry}
    \begin{gather*}
    \nc\ST{\begin{array}{c}\includegraphics[scale=0.08]{one-chord}\end{array}}
     P_{m,n}:= \exp_{\sqcup} \left(\sum_{i=1}^m \ST_{\!\!\!(n+i)^-}^{\!\! i^+} + \sum_{j=1}^n \ST_{\!\! j^-}^{\!\!\! (m+j)^+}\right)
     :m+n\lto n+m
  \end{gather*}
   for the strict monoidal {category  $\sA$ (resp$.$ $\tsA$)}.
  Similarly, the braided monoidal structure of
  $\hA_q^\varphi$ is mapped by $\kappa$ into {a braided monoidal
  structure on the non-strictification of   $\sA$ (resp$.$ $\tsA$)}.  The
  Casimir Hopf algebra in $\AB$ {(given by Proposition~\ref{r14})}
  and the ribbon quasi-Hopf algebra in   $\hA$  {(given by Theorem \ref{prop:r3})} 
  are mapped by $\kappa$ into such structures in $\sA${,} and hence in $\tsA$.
\end{remark}

\begin{remark}
  \label{r39}
  Recall from Section \ref{sec:coalgebras} that the categories $\AB$
  and $\hA$ are enriched over the category $\CC$ of cocommutative
  coalgebras.  It is not difficult to verify that the categories
  $\tsA$ and $\sA$ are enriched over $\CC$, with the coalgebra
  structure on the morphism spaces described in \cite{CHM}, where connected
  Jacobi diagrams are primitive as usual.  Then one can check that the
  {``hair functor''} $\kappa:\hA\to\sA$ is a {$\CC$}-functor, i.e.,
  the maps $\kappa:\hA(m,n)\to\sA(m,n)$ are coalgebra maps.  By applying
  the ``group-like part functor'' $\grp: \CC \to\mathbf{Set}$, we
  obtain a group-like version of $\kappa$:
  \begin{gather*}
    \kappa^\grp:\hA^\grp\lto \sA^\grp.
  \end{gather*}
\end{remark}

\subsection{Non-faithfulness of $\kappa$}   \label{sec:non-faithf-kappa}

Using Vogel's results \cite{Vogel}, Patureau-Mirand {has}  proved that
{the ``hair map'' in \cite{GK,GR}} is not injective \cite[Theorem
4]{Patureau-Mirand}.  The next proposition is proved by adapting his
arguments to our situation.

\begin{theorem} \label{prop:non-injectivity}
If $m,n\geq 1$, then  $\kappa: \AB(m,n) \to \sA(m,n)$ is not injective, 
{and, {therefore}, neither is $\kappa: \hA(m,n) \to \sA(m,n)$.} 
\end{theorem}

\begin{proof}
 Let $G(n)$ be the subspace of $\A( \{1,\dots,n\})$ spanned by
connected Jacobi diagrams with exactly $n$ univalent vertices labeled
from $1$ to $n$. There is a natural action of the symmetric group
$\mathfrak{S}_n$ on $G(n)$, and we consider the subspace $\Lambda$ of
$G(3)$ consisting of those $x\in G(3)$ such that $\sigma \cdot x =
\hbox{sgn}(\sigma) x$ for all $\sigma \in\mathfrak{S}_3$.
According to Vogel \cite{Vogel}, the space $\Lambda$ admits a structure of
commutative algebra with non-trivial zero divisors.
Based on these results, Patureau-Mirand \cite[Corollary 2]{Patureau-Mirand} proved the existence of an element $r
\in\Lambda \setminus \{0\}$ of degree $17$ such that
\begin{eqnarray}
\label{eq:r1} \centre{\labellist
\scriptsize\hair 2pt
 \pinlabel {$r$}  at 18 76
\endlabellist
\centering
\includegraphics[scale=0.25]{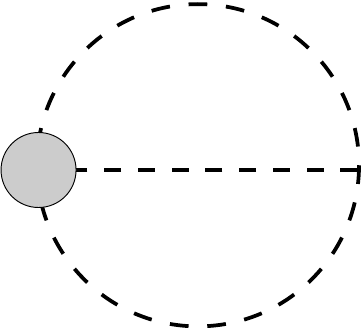}}
 & \neq  & 0 \in\A(\varnothing), \\
 \label{eq:r2} \centre{\labellist
\scriptsize\hair 2pt
 \pinlabel {$r$}  at 18 82
 \pinlabel {$1$} [t] at 18 3
 \pinlabel {$2$} [t] at 170 3
 \pinlabel {$3$} [l] at 231 47
\endlabellist
\centering
\includegraphics[scale=0.25]{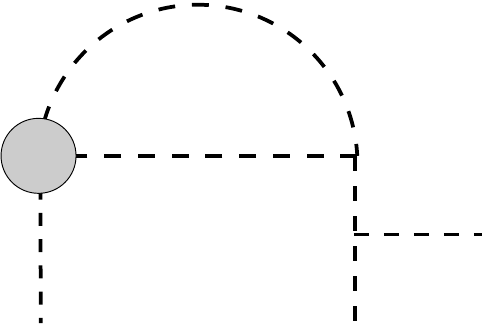}}
  & =  & 0 \in G(3).
\end{eqnarray}
Then we define
$$
u =
\centre{\labellist
\scriptsize\hair 2pt
 \pinlabel {$1$} [l] at 252 29
 \pinlabel {$x_1$} [l] at 216 133
 \pinlabel {$r$}  at 57 175
\endlabellist
\centering
\includegraphics[scale=0.25]{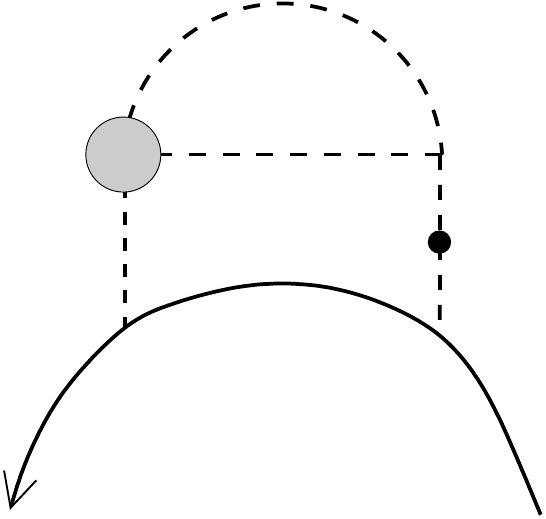}}
- \centre{\labellist
\scriptsize\hair 2pt
 \pinlabel {$1$} [l] at 252 29
 \pinlabel {$r$}  at 57 175
\endlabellist
\centering
\includegraphics[scale=0.25]{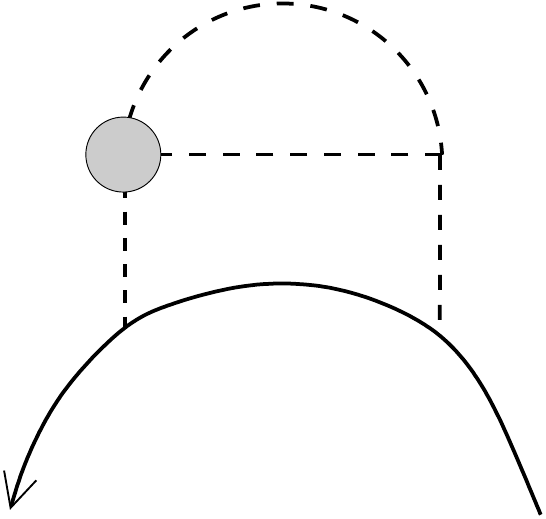}}  \in\A\big(X_1,F(x_1)\big) = \AB(1,1).
$$
By \eqref{eq:r2}, we have $\chi(\kappa(u)) =0
\in\A\big(X_1, \{1^+\}\big)$, and hence $\kappa(u)=0$.
More generally, if $m,n\geq 1$, then $\kappa: \AB(m,n) \to
\sA(m,n)$ vanishes on $u\otimes\eta^{{\ot (n-1)}}\epsilon^{{\ot (m-1)}}$.
Thus, to prove that it is not injective, it suffices to check $u\neq 0$.

Recall the projection $p:  \A\big(X_1,F(x_1)\big) \to  \A\big(X_1\big)$
introduced in the proof of Lemma~\ref{lem:usual_to_colored}. We have
\begin{eqnarray*}
\chi^{-1}\big(p(u)\big) &=& - \chi^{-1}\Bigg( \centre{\labellist
\scriptsize\hair 2pt
 \pinlabel {$1$} [l] at 252 29
 \pinlabel {$r$}  at 57 175
\endlabellist
\centering
\includegraphics[scale=0.25]{u_0}} \Bigg)
\ = \ - \centre{\labellist
\scriptsize \hair 2pt
 \pinlabel {$r$}  at 18 88
 \pinlabel {$1$} [t] at 19 3
 \pinlabel {$1$} [t] at 171 2
\endlabellist
\centering
\includegraphics[scale=0.25]{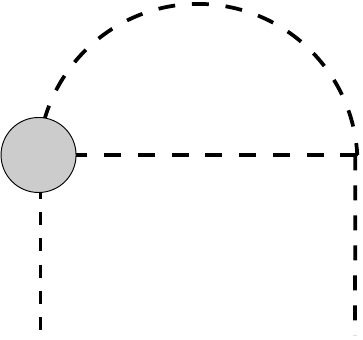}} \ \in\A(\{1\}),
\end{eqnarray*}
since $G(1)=0$ \cite[Proposition 4.3]{Vogel}.
By \eqref{eq:r1}, the right hand side is not zero, and hence  we have $u\neq 0$.
\end{proof}

\subsection{Jacobi diagrams colored by a cocommutative Hopf algebra}  \label{sec:Hopf_colored}

In order to give a Hopf-algebraic description of the kernel of $\kappa$
in the next subsection, we need to generalize some constructions of Section
\ref{sec:colored_Jd}.

Let $X$ be a compact oriented $1$-manifold,
and let $H$ be a cocommutative Hopf algebra with comultiplication $\Delta: H
\to H \otimes H$, counit $\epsilon: H \to \K$ and (involutive) antipode $S: H \to H$.

Recall from Section \ref{sec:colored_Jd} the notion of chord diagrams colored by a set.  Let
$\mathcal{D}^{\operatorname{ch}}(X,H)$ be the vector space
generated by $H$-colored chord diagrams on $X$, modulo the following
local relations:
\begin{equation} \label{eq:moves_ccd_bis}
\centre{\labellist
\scriptsize \hair 2pt
 \pinlabel {$x$} [t] at 26 219
 \pinlabel {$y$} [t] at 68 219
 \pinlabel {$xy$} [t] at 196 219
 \pinlabel {$1$} [t] at 414 220
 \pinlabel {$\leftrightarrow$}   at 126 226
  \pinlabel {$,$}   at 452 330
 \pinlabel {$,$}   at 307 226
 \pinlabel {$,$} [l] at 614 226
 \pinlabel {$\leftrightarrow$}   at 127 118
 \pinlabel {$,$}   at 311 118
 \pinlabel {$\leftrightarrow$}   at 128 7
 \pinlabel {$,$}   at 312 13
 \pinlabel {$,$} [l] at 614 117
 \pinlabel {$\leftrightarrow$}   at 487 225
 \pinlabel {$x$} [t] at 26 113
 \pinlabel {$y$} [t] at 69 113
 \pinlabel {$xy$} [t] at 192 112
 \pinlabel {$1$} [t] at 411 112
 \pinlabel {$\leftrightarrow$}   at 486 118
 \pinlabel {$x$} [t] at 49 5
 \pinlabel {$S(x)$} [t] at 192 5
  \pinlabel {$.$} [l] at 614 10
   \pinlabel {$\leftrightarrow \ k$}  at 140 334
 \pinlabel {$k  x +ly$} [t] at 51 332
 \pinlabel {$x$} [t] at 219 330
 \pinlabel {$y$} [t] at 379 331
 \pinlabel {$+\ l$}  at 303 335
  \pinlabel {$\leftrightarrow \ k$}  at 140 444
 \pinlabel {$k  x  +ly$} [t] at 51 442
 \pinlabel {$x$} [t] at 219 440
 \pinlabel {$y$} [t] at 379 441
 \pinlabel {$+\ l$}  at 303 445
  \pinlabel {$,$}   at 452 440
 \pinlabel {$x$} [b] at 379 16
 \pinlabel { ${\displaystyle \quad \leftrightarrow  \sum_{(x)}}$}  at 520 0
 \pinlabel {$x'$} [t] at 645 11
 \pinlabel {$x''$} [b] at 643 43
 \pinlabel {$,$}  at 695 9
\endlabellist
\centering
\includegraphics[scale=0.32]{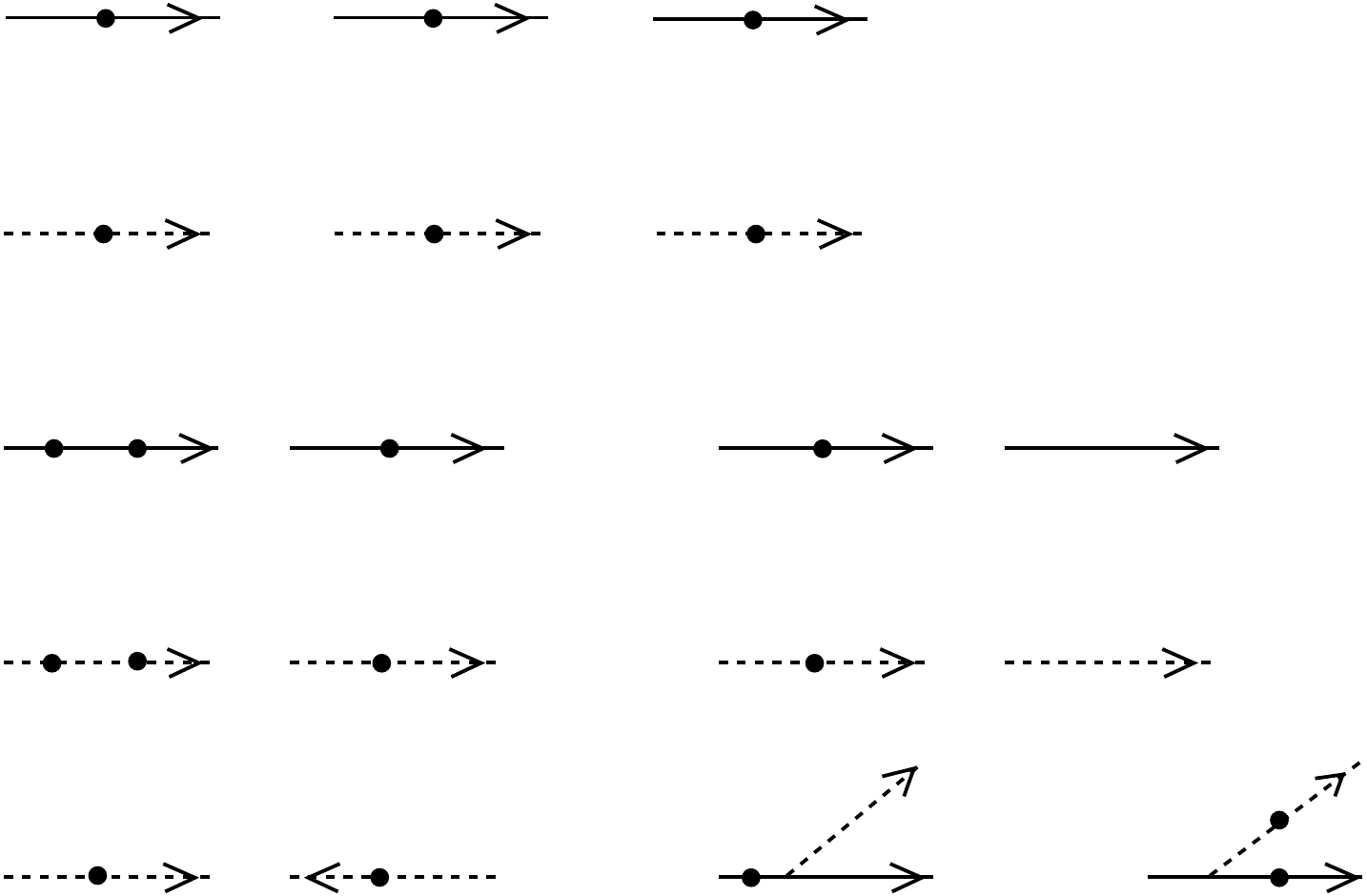}}
\end{equation}\\[0.2cm]
for all $x,y \in H$ and $k,l\in\K$, where $\Delta(x)=\sum_{(x)}
x'\otimes x''$ is written using Sweedler's notation.  Let
$\mathcal{R}^{\operatorname{ch}}(X,H)$ be the subspace of
$\mathcal{D}^{\operatorname{ch}}(X,H)$ generated by the 4T relations \eqref{eq:4T}, and
set
$$
\A^{\operatorname{ch}}(X,H) = \mathcal{D}^{\operatorname{ch}}(X,H)/\mathcal{R}^{\operatorname{ch}}(X,H).
$$ We still let $\A^{\operatorname{ch}}(X,H)$ denote the
degree-completion of this space, where the \emph{degree} of an
$H$-colored chord diagram on $X$ is the number of chords.

More generally, let  $\mathcal{D}^{\operatorname{Jac}}(X,H)$ be the vector space generated
by  $H$-colored Jacobi diagrams on $X$, modulo the local relations
\begin{equation} \label{eq:moves_cJd_bis}
\qquad \qquad
\labellist
\small \hair 2pt
 \pinlabel {${\displaystyle \leftrightarrow  \ \sum_{(x)}}$}   at 200 45
 \pinlabel {\small $x$} [b] at 49 62
 \pinlabel {\small $x'$} [b] at 349 84
 \pinlabel {\small $x''$} [t] at 345 32
 \pinlabel {   $\forall x\in H$,} [r] at -40 56
  \pinlabel {\quad and \eqref{eq:moves_ccd_bis}.} [l] at 399 58
\endlabellist
\centering
\includegraphics[scale=0.3]{coloring_Jd}
\end{equation}
Let $\mathcal{R}^{\operatorname{Jac}}(X,H)$ be the subspace of $\mathcal{D}^{\operatorname{Jac}}(X,H)$
generated by the STU relations~\eqref{eq:STU}, and set
$$
\A^{\operatorname{Jac}}(X,H) = \mathcal{D}^{\operatorname{Jac}}(X,H)/\mathcal{R}^{\operatorname{Jac}}(X,H).
$$ We still let $\A^{\operatorname{Jac}}(X,H)$ denote the
degree-completion of this space, where the \emph{degree} of an
$H$-colored Jacobi diagram on $X$ is half the total number of
vertices.

\begin{example}
Assume that $H=\K[\pi]$ is the group Hopf algebra of a group $\pi$.
Then $\A^{\operatorname{ch}}(X,H)$ and $\A^{\operatorname{Jac}}(X,H)$
are canonically isomorphic to the spaces
$\A^{\operatorname{ch}}(X,\pi)$ and $\A^{\operatorname{Jac}}(X,\pi)$,
respectively,  introduced in Section \ref{sec:colored_Jd}.
\end{example}

Let $I:= \ker(\epsilon:H \to \K)$ be the augmentation ideal of
$H$ and, for $k\geq 0$, let $F_k\mathcal{D}^{\operatorname{ch}}(X,H)$
be the subspace of $\mathcal{D}^{\operatorname{ch}}(X,H)$ spanned by
$H$-colored chord diagrams on~$X$ with (at least) $k$ beads
colored by elements of $I$.
Let $F_k\A^{\operatorname{ch}}(X,H)$ denote the image of
$F_k\mathcal{D}^{\operatorname{ch}}(X,H)$ in
$\A^{\operatorname{ch}}(X,H)$.  Thus we obtain a filtration
\begin{eqnarray*}
 \mathcal{A}^{\operatorname{ch}}(X,H) &= & F_0 \mathcal{A}^{\operatorname{ch}}(X,H)
\supset F_1 \mathcal{A}^{\operatorname{ch}}(X,H) \supset F_2 \mathcal{A}^{\operatorname{ch}}(X,H) \supset \cdots.
\end{eqnarray*}
The \emph{$I$-adic completion}
$$
\widehat{\mathcal{A}}^{\operatorname{ch}}(X,H) := \plim_k \frac{\mathcal{A}^{\operatorname{ch}}(X,H)}{ F_k \mathcal{A}^{\operatorname{ch}}(X,H)}
$$
of $\mathcal{A}^{\operatorname{ch}}(X,H)$ inherits a filtration from
$\mathcal{A}^{\operatorname{ch}}(X,H)$.  Let $\alpha:
\mathcal{A}^{\operatorname{ch}}(X,H) \to
\widehat{\mathcal{A}}^{\operatorname{ch}}(X,H)$ be the canonical map.
Applying the same definitions to Jacobi diagrams yields the space
$\widehat{\mathcal{A}}^{\operatorname{Jac}}(X,H)$.  According to the
next theorem, we can identify the filtered spaces
${\A}^{\operatorname{ch}}(X,H)$ and ${\A}^{\operatorname{Jac}}(X,H)$
(resp.\ $\widehat{\A}^{\operatorname{ch}}(X,H)$ and
$\widehat{\A}^{\operatorname{Jac}}(X,H)$) and simply denote them by
${\A}(X,H)$ (resp.\ $\widehat{\A}(X,H)$).

\begin{theorem}
  \label{r28}
The canonical map
$$
\phi: \A^{\operatorname{ch}}(X,H) \longrightarrow \A^{\operatorname{Jac}}(X,H)
$$ is an isomorphism of filtered space{s}.
 Furthermore, the AS and IHX
relations~\eqref{eq:AS_IHX} hold in $\A^{\operatorname{Jac}}(X,H)$.
\end{theorem}

\begin{proof}
Clearly, $\phi$ is a filtration-preserving linear map, i.e.,
 $\phi(F_k\A^{\operatorname{ch}}(X,H))$ is contained in
  $F_k\A^{\operatorname{Jac}}(X,H)$ for $k\ge0$.  We can check
  $$
  \phi(F_k\A^{\operatorname{ch}}(X,H))= F_k\A^{\operatorname{Jac}}(X,H)
  $$
  by using the STU relation, the
  identity
\begin{eqnarray*}
\centre{\labellist
\scriptsize\hair 2pt
 \pinlabel {${}^{x}\bullet$}  [r] at 146 48
\endlabellist
\centering
\includegraphics[scale=0.3]{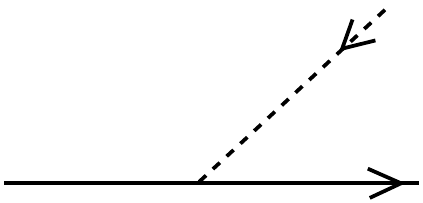}}
\ =\ \sum_{(x)}\centre{\labellist
\scriptsize\hair 2pt
 \pinlabel {$\stackrel{S(x')}{\bullet}$} [b] at 48 -3
 \pinlabel {}  at 140 51
 \pinlabel {$\stackrel{x''}{\bullet}$} [b] at 145 -3
\endlabellist
\centering
\includegraphics[scale=0.3]{down}}
\end{eqnarray*}
in $\A^{\operatorname{Jac}}(X,H)$ for $x\in H$,
and the inclusion $\Delta(I)\subset I \otimes \K +\K\otimes I$.

The proofs of the injectivity of $\phi$ and the AS
and IHX relations given in Theorem~\ref{th:chord_Jac} for a group
algebra $H=\K[\pi]$ work for a general~$H$.
\end{proof}

Every homomorphism $f:H \to H'$ of cocommutative Hopf algebras induces
a linear map $f_\ast: \A(X,H) \to \A(X,H')$ by applying $f$ to
all beads of an $H$-colored Jacobi diagram on $X$.  Thus we obtain a
 functor $\A(X,-)$ from  cocommutative Hopf algebras to vector spaces,
 which admits a ``continuous'' version as follows.

Let $\widehat{H} := \plim_k H/I^k$ be the $I$-adic completion of $H$,
which is a  cocommutative complete Hopf algebra.
The canonical map $H \to \widehat{H}$ will be omitted from our notations, although it may not be injective.
We can express every $\hat x\in\widehat{H}$ as
$$
\hat x=\sum_{k=0}^\infty x(k) \quad \hbox{where}\ x(k) \in I^k.
$$
For a set $S$, we write an $S$-colored Jacobi diagram $D$ on $X$ as
$$
D=D(s_1,\dots,s_r),
$$
where $s_1,\dots, s_r$ are the colors of the beads numbered  from $1$ to $r$, and $D(-,\dots,-)$ stands for
the corresponding Jacobi diagram on $X$ with ``uncolored'' beads.
Thus every $\widehat{H}$-colored Jacobi diagram on $X$
$$
D=D(\hat x_1,\dots,\hat x_r) \quad \hbox{where} \quad
\hat x_1=\sum_{k_1=0}^\infty x_1(k_1) , \dots,
\hat x_r =\sum_{k_r=0}^\infty x_r(k_r)
$$
defines  an element
$$
D(\hat x_1,\dots,\hat x_r) :=  \sum_{k_1,\dots,k_r=0}^\infty \alpha\big( D\big(x_1(k_1),\dots,x_r(k_r)\big)\big)
  \ \in\widehat{\A}(X,H).
$$
We can easily verify the following lemma.

\begin{lemma} \label{lem:change_colors}
Every homomorphism $f: \widehat{H} \to \widehat{H'}$ of complete Hopf
algebras, between the $I$-adic completions of cocommutative Hopf
algebras $H$ and $H'$, induces a unique filtration-preserving linear
map $f_* : \widehat{\mathcal{A}}(X,H)\to \widehat{\mathcal{A}}(X,H')$
such~that
\begin{equation}   \label{eq:f_*}
f_*\alpha \big(D(x_1,\dots,x_k)\big) = D\big( f( x_1),\dots,f(x_k)\big)
\end{equation}
for every $H$-colored Jacobi diagram $D(x_1,\dots,x_k)$ on $X$.
Moreover, we have $(f' f)_* = f'_* f_*$ for {all} such homomorphisms
$\widehat{H} \overset{f}\to \widehat{H'} \overset{f'}\to \widehat{H''}$.
\end{lemma}

\subsection{A Hopf-algebraic description of the kernel of $\kappa$} \label{sec:re-kappa}

In this subsection, we fix $m,n\geq 1$ and set
$$
\free{m} = F(x_1,\dots,x_m),\quad \quad
X_n =  \!\figtotext{10}{10}{capleft}\!\!_{1} \cdots\!\! \figtotext{10}{10}{capleft}\!\!_{n}.
$$  Recall that the degree-completion of $\A(X_n ,\K[\free{m}])$ is denoted
by the same notation $\A(X_n ,\K[\free{m}])$, and the $I$-adic completion
of $\A(X_n ,\K[\free{m}])$ is denoted by $\widehat{\A}(X_n ,\K[\free{m}])$.
We have seen in Section \ref{sec:Hopf_colored} that the canonical
homomorphism $\K[\free{m}] \to \widehat{\K[\free{m}]}$,
where $\widehat{\K[\free{m}]}$ is the $I$-adic completion of $\K[\free{m}]$,
has a diagrammatic counterpart $\alpha:\A(X_n ,\K[\free{m}]) \to \widehat{\A}(X_n ,\K[\free{m}])$.
{Theorem} \ref{prop:non-injectivity}
and Proposition \ref{alpha_kappa} below  imply that $\alpha$ is not injective,
in contrast with the well-known injectivity of $\K[\free{m}]\to\widehat{\K[\free{m}]}$.

\begin{proposition} \label{alpha_kappa}
There is a canonical isomorphism $\sA(m,n) \cong  \widehat{\A}( X_n, \K[\free{m}] )$
which makes the following diagram commute:
$$
\xymatrix{
\hA(m,n) \ar@{=}[d] \ar[r]^\kappa&
\sA(m,n)\ar[d]_{}^{\cong }\\
\A(X_n, \K[\free{m}]) \ar[r]^\alpha& \widehat{\A}(X_n,\K[\free{m}])
}
$$
In particular, the kernel of $\kappa$ coincides with the kernel of $\alpha$.
\end{proposition}

\begin{proof}
Set $U_m = \{1^+,\dots,m^+\}$.  We can merge the notions of ``$U_m$-labeled Jacobi
diagram'' and ``Jacobi diagram on $ X_n $''
into the notion of ``$U_m$-labeled Jacobi diagram on $X_n$'';
the degree of such a diagram is half the total number of vertices.
(Here we assume that each connected component of a $U_m$-labeled Jacobi
diagram on $X_n$ has at least one univalent vertex on $X_n$.)  Let
$\A(X_n,U_m)$  be  the (degree-completion of the)
vector space generated by $U_m$-labeled Jacobi diagrams on $X_n$ modulo the STU relation.

Let  $\chi:  \sA(m,n) \to \A(X_n,U_m)$ be the diagrammatic analog of the PBW isomorphism.
We will prove that the maps $\alpha$ and $\kappa$ fit into the following commutative diagram:
\begin{equation}   \label{eq:alpha_kappa}
\vcenter{\xymatrix{
\hA(m,n) \ar@{=}[d] \ar[r]^\kappa& \sA(m,n)\ar[r]_-{\cong }^\chi &  \A(X_n,U_m) \\
\A(X_n, \K[\free{m}]) \ar[r]^\alpha& \widehat{\A}(X_n, \K[\free{m}] )
\ar[r]^-{f_*}_-{\cong } &   \widehat{\mathcal{A}}(X_n,T({V_m}))
\ar[u]_-{\widehat{\lambda}}
}}
\end{equation}
Here $T(V_m)$ is the tensor algebra over the vector space $V_m$
with basis $\{v_1,\dots,v_m\}$, equipped
with the usual Hopf algebra structure.  Using Lemma
\ref{lem:change_colors}, the isomorphism $f_*$ is induced by the
complete Hopf algebra isomorphism $f: \widehat{\K[\free{m}]} \to
\widehat{T(V_m)}$ that maps each $x_i$ to $\exp(v_i)$.
The map $\widehat{\lambda}$  is induced by
a filtration-preserving linear map $\lambda: \mathcal{A}(X_n,T(V_m)) \to \A(X_n,U_m)$ defined below.

We can transform every $T(V_m)$-colored Jacobi diagram $D$ on
$X_n$  into a $U_m$-labeled Jacobi diagram on
$X_n$ by applying the following transformations to beads:\\[0.2cm]
$$
\forall i_1,\dots,i_r \in\{1,\dots,m\}, \quad
\centre{\labellist
\small\hair 2pt
 \pinlabel {$\leadsto$} at 294 152
 \pinlabel {$\leadsto$}  at 296 9
 \pinlabel {$\stackrel{\hbox{$v_{i_1} v_{i_2} \cdots v_{i_r}$}}{\bullet}$}  at 99 162
 \pinlabel {$\stackrel{\hbox{$v_{i_1} v_{i_2} \cdots v_{i_r}$}}{\bullet}$}  at 103 18
 \pinlabel {$\cdots$}  at 469 173
 \pinlabel {$\cdots$}  at 469 27
 \pinlabel {$i_1^+$} [b] at 397 189
 \pinlabel {$i_2^+$} [b] at 432 190
 \pinlabel {$i_r^+$} [b] at 505 190
 \pinlabel {$i_1^+$} [b] at 397 45
 \pinlabel {$i_2^+$} [b] at 434 45
 \pinlabel {$i_r^+$} [b] at 507 46
\endlabellist
\centering
\includegraphics[scale=0.3]{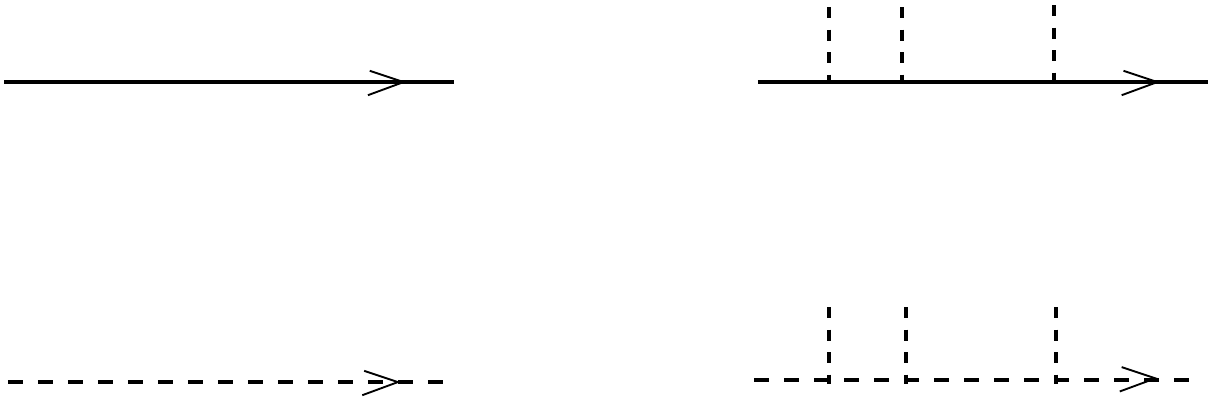}}
$$
By the STU (and AS, IHX) relations in $\A(X_n,U_m)$, the above procedure defines a linear map
$\mathcal{D}^{\operatorname{Jac}}(X_n,T(V_m)) \to \A(X_n,U_m)$,
which induces a linear map
$$
\lambda: \mathcal{A}(X_n,T(V_m)) \longrightarrow  \A(X_n,U_m).
$$
Obviously, $\lambda$ is filtration-preserving.  One
easily checks from the definitions that the resulting map $\hat
\lambda$ makes the diagram \eqref{eq:alpha_kappa} commute.

To prove the proposition, it suffices to show  that $\widehat{\lambda}$ is an isomorphism.
For this, we will  construct an inverse to $\widehat{\lambda}$.
Let $D$ be a $U_m$-labeled Jacobi diagram on $X_n$.
We can  decompose $D$ uniquely as
$$
D = D_0 \cup D'_1 \cup \cdots \cup D'_r,
$$ where $D_0$ is a Jacobi diagram on $X_n$,
we have $r$ distinguished points $\ast_1,\dots,\ast_r$ in the
  interiors of the edges of $X_n \cup D_0$,
each $D'_i$ is {a} $(U_m\cup \{\ast_i\})$-labeled tree-shaped Jacobi diagram
with exactly one univalent vertex labeled by $\ast_i$, and we have
$D'_i \cap D'_j=\varnothing$ for {all} $i\neq j$ and $D'_i \cap D_0 =\{ \ast_i\}$ for all $i$.
Note that each tree $D'_i$, rooted at
$\ast_i$, defines a Lie word with letters in $U_m$: hence, using the
correspondence $i^+\leftrightarrow v_i$ between $U_m$ and the basis
of $V_m$, each $D'_i$ defines a primitive element $d'_i\in T(V_m)$. Choose an orientation on each edge of $D_0$.
For each $i\in\{1,\dots, r\}$, set $\varepsilon_i = 0$ if $\ast_i$ belongs to $X_n$
or if $\ast_i$ belongs to an edge
of $D_0$ and the tree $D'_i$ is above this edge when its orientation goes from
left to right; set $\varepsilon_i =  1$ otherwise.  Let $S$ denote the
antipode of $T(V_m)$.  By considering the $T(V_m)$-colored Jacobi
diagram on $X_n$ that is obtained from $D_0$ by putting a bead colored
by $S^{\varepsilon_i}(d'_i)$ at $\ast_i$ for all $i\in\{1,\dots, r\}$,
we obtain an element
$$
\rho(D)  \in\mathcal{D}^{\operatorname{Jac}}(X_n,T(V_m))
$$ which does not depend on the above choice of edge-orientations.  It
is easy to verify that an STU relation in $\A(X_n,U_m)$ is mapped by
$\rho$ either to $0$ or to an STU relation in 
$\mathcal{D}^{\operatorname{Jac}}(X_n,T(V_m))$.  Hence we obtain a linear map
$$\rho:\A(X_n,U_m)\longrightarrow\A(X_n,T(V_m)),$$ with $\A(X_n,U_m)$ before
  completion, inducing a linear
map
\begin{gather*}
  \widehat{\rho}:\A(X_n,U_m)\longrightarrow \widehat{\A}(X_n,T(V_m)),
\end{gather*}
with $\A(X_n,U_m)$ after degree-completion.
{Obviously, we have $\widehat{\rho} \circ \widehat{\lambda} =\id $.
Using} the STU (and AS, IHX) relations in $\A(X_n,U_m)$, it is 
easy to check  $\widehat{\lambda} \circ \widehat{\rho} =\id$.
\end{proof}

\section{Perspectives}\label{sec:perspectives}

We plan to consider several developments of the functor $Z=Z^\B:\B_q
\to \hA$ in forthcoming works.  {For simplicity, the degree-completion
$\hA$ of $\AB$ will now be denoted as $\AB$.}
{Also, we will ignore parenthesization of objects in non-strict monoidal
  categories and write $\mathcal B$ for $\mathcal B_q$, for instance.}

\subsection{Incorporation of tangles}  \label{sec:incorp-tangl}

One can naturally construct a {braided strict monoidal} category $\B\T$
which contains both the categories $\B$ and $\T$ as braided
monoidal subcategories.  The objects of $\B\T$ are words in the
letters $+,-,\bullet$ and morphisms are bottom tangles in handlebodies
mixed with additional tangles.  Similarly, there is  $\LCob\T$  
containing both $\LCob$ and $\T$ as braided monoidal subcategories.
We can  extend the functors $Z: \B \to \AB$ and $\widetilde{Z}:
\LCob \to \tsA$ to $\B\T$ and $\LCob\T$, respectively, so that
the commutative square \eqref{e17} extends to
\begin{gather}
  \vcenter{\xymatrix{
      \B\T \ar[r]^{{Z}}\ar[d]_{{E}}&
      \AB^{\!\dagger}
      \ar[d]^{{\kappa}}\\
      \LCob\T \ar[r]^{{\tiZ}}&     \tsA^\dagger {.}
  }}
\end{gather}
Here {$\AB^{\!\dagger}$} is a linear symmetric monoidal category extending both
$\AB$ and {the linear version of $\mathcal{A}$ mentioned in Remark~\ref{r30}}, 
and, similarly, {$\tsA^\dagger$} 
extends both $\tsA$ and {this linear version of $\mathcal{A}$}.
We {remark} that the {extension of~$\tiZ$ to $\LCob\T$} also {generalizes} Nozaki's
extension of the LMO functor to Lagrangian cobordisms of punctured surfaces \cite{Nozaki}.

As a symmetric monoidal linear category, {$\AB^{\!\dagger}$} is free on a triple $(H,V,V^*)$ {consisting}
of a Casimir Hopf algebra $H$, a left $H$-module $V$ and its dual $V^*$.
The functor $Z$ induces an isomorphism of graded linear symmetric monoidal categories
 between the associated graded {of}  
 the Vassiliev--Goussarov filtration for $\B\T$ and {$\AB^{\!\dagger}$}{.}

{Let $m\geq 0$ be an integer and recall that $S\subset \partial V_m$ is the bottom square.
Consider the compact oriented surface $\Sigma_{m,1}:= {\partial V_m\setminus \operatorname{int}(S)}$ 
of genus $m$ with one boundary component.}
{The morphisms in $\B\T$ 
whose underlying bottom tangle in a handlebody is $\id_m\in \B(m, m)$ 
can be regarded as tangles in the thickened surface $\Sigma_{m,1} \times I$.}
{In particular},
by specializing the above functor $Z:\B\T \to {\AB^{\!\dagger}}$ {to that kind of morphisms,}  
we obtain 
\begin{itemize}
\item  expansions of the free group $\pi_1(\Sigma_{m,1})$, 
{which refine the  symplectic expansions
derived from the LMO functor  \cite{Massuyeau},}
\item representations of pure braid groups on $\Sigma_{m,1}$, and,
 more generally, representations of monoids of string-links in  $\Sigma_{m,1} \times I$.
\end{itemize}
We plan to study elsewhere these new representations.

\subsection{Handlebody groups and twist groups}  \label{sec:handl-groups-twist}

Fix an integer $m\ge0$. 
The automorphism group of the object $m$ in $\mathcal{H}\cong\B^\op$ is naturally identified with the \emph{handlebody group}
\begin{gather*}
  \mathcal{H}_{m,1}:= \Homeo(V_m,S)/\!\cong,
\end{gather*}
which is the group of isotopy classes rel $S$ of self-homeomorphisms of $V_m$ that restrict to~$\id_S$.  
Hence the functor $Z: \B\to \AB$  restricts  to a monoid homomorphism
\begin{equation} \label{Z_H}
  Z:\mathcal{H}_{m,1} \longrightarrow  A_m := \AB(m,m)^\op.
\end{equation}
It is well known that the  group $ \mathcal{H}_{m,1}$ naturally embeds
into the {\emph{mapping class group}} 
$$
\mathcal{M}_{m,1} :=  \Homeo(\Sigma_{m,1}, \partial \Sigma_{m,1})/ {\cong}
$$
of the surface $\Sigma_{m,1}= \partial V_m\setminus \operatorname{int}(S)$.
Since the LMO functor $\wt Z$ is injective\footnote{This follows easily
  from the injectivity of $\wt Z$ on the Torelli group \cite[Corollary  8.22]{CHM} 
  since the strut part of $\wt Z$ encodes the action {of the Lagrangian subgroup of $\mathcal{M}_{m,1}$ on $H_1(\Sigma_{m,1};\Z)$}.} 
  on the Lagrangian subgroup of	$\mathcal{M}_{m,1}$
  (i.e., the automorphism group of the object $m$ in  $\LCob$),
Theorem \ref{r32} implies that $Z: \mathcal{H}_{m,1} \to A_m$ is injective.
We plan to use this  homomorphism to  study  the {algebraic} structure of $\mathcal{H}_{m,1}$
and the inclusion of this group in the monoid  $\mathcal{H}(m,m)$.

In particular, we are interested in the \emph{twist group} $\T_{m,1}$ 
which is the kernel of the natural homomorphism $\mathcal{H}_{m,1}\to\Aut(F_m)$.  
Here $F_m:=\pi_1(V_m,S)$ is the fundamental group of $V_m$ based at the contractible subspace $S$.
Note that $\T_{m,1}$ is  the kernel of the
degree $0$ part of $Z:\mathcal{H}_{m,1} \to A_m$, since the latter gives the homotopy {class} of bottom tangles in handlebodies.
It is known that, as a subgroup of $\mathcal{M}_{m,1}$, the group $\T_{m,1}$  is generated by
{Dehn} twists along  boundaries of  properly embedded disks in $V_m \setminus S$ \cite{Luft}.  

The pair (handlebody group, twist group) can be regarded
as an analogue of  the pair (mapping class group, Torelli group).
We recall some of the features of the Johnson--Morita theory, which  consists in studying the group $\mathcal{M}_{m,1}$
via its action on the lower central series of the fundamental group $\pi$ of $\Sigma_{m,1}$
(see \cite{Morita} for a survey):
\begin{enumerate}
\item the \emph{Johnson filtration}
$J_0\mathcal{M}_{m,1}\supset J_1\mathcal{M}_{m,1} \supset \cdots \supset J_k\mathcal{M}_{m,1} \supset \cdots$
consists of the kernels of the actions of $\mathcal{M}_{m,1}$ on the successive nilpotent quotients of $\pi$
(so that $J_0\mathcal{M}_{m,1}={\mathcal{M}_{m,1}}$   and $J_1\mathcal{M}_{m,1}$ is the Torelli group);
\item for every  $k\geq 1$, the \emph{$k$-th Johnson homomorphism} $\tau_k$ maps  $J_k\mathcal{M}_{m,1}$ to an abelian group and encodes the action of $J_k\mathcal{M}_{m,1}$ on the $k$-th nilpotent quotient of $\pi$;
\item for every $k\geq 1$, the $k$-th Johnson homomorphism has a diagrammatic description and then corresponds 
to the leading term of the ``tree reduction'' of the LMO functor $\tiZ$ on $J_k\mathcal{M}_{m,1}$ \cite{CHM,HM1};
\item more generally, the action of $J_1\mathcal{M}_{m,1}$ on (the Malcev completion of) $\pi$ is encoded
in the full  ``tree reduction'' of the LMO functor $\tiZ$ by means of a  ``symplectic expansion'' \cite{Massuyeau}.
\end{enumerate}

There is an analogue of the Johnson--Morita theory  for the pair (handlebody group, twist group). 
This has been introduced in \cite[\S 10.1]{HM3} as an instance of a ``general theory'' of Johnson homomorphisms,
and will be studied with further details in a  forthcoming work.
In this approach, the group $\mathcal{H}_{m,1}$ is studied via its action on the lower central series
of the kernel of the homomorphism $\pi\to F_m$ induced by the inclusion $\Sigma_{m,1} \hookrightarrow V_m$. 
Then the analogue of (1) is a filtration of $\mathcal{H}_{m,1}$ whose first term is $\mathcal{T}_{m,1}$,
and the analogue of (2) consists of two sequences of homomorphisms $(\tau_k^0)_k$ and $(\tau_k^1)_k$ 
which happen to be equivalent one to the other.
There are also analogues of (3) and (4), which involve the ``tree reduction'' of $Z: \mathcal{H}_{m,1} \to A_m$
and the refinement of the ``symplectic expansion''  mentioned in Section \ref{sec:incorp-tangl}.

We expect the homomorphism $Z: \mathcal{H}_{m,1} \to A_m$ to be a powerful tool 
to study the associated graded of the lower central series of~$\T_{m,1}$
in relation with the associated graded of the Vassiliev--Goussarov filtration 
that has been identified in Section \ref{sec:universality-zq}.

\subsection{Extension of $Z$ to boundary Lagrangian cobordisms}

The reader may wonder whether one can extend the functor
$Z:{\sLCob} \to\AB$  
on {$\sLCob \cong \B$} 
to the category $\LCob$ of
Lagrangian cobordisms,
 with the target category still involving some homotopy classes of Jacobi
diagrams in handlebodies.  This does not hold,
but one can extend $Z$
to a functor ${}^b\!Z:\bLCob\to\bA$ which fits into the following
commutative diagram of monoidal categories and monoidal functors:
\begin{gather*}
  \xymatrix{
    \sLCob\ar[r]^{Z}\ar[d]_{}&
    \AB\ar[d]^{}\\
    \bLCob\ar[r]^{{}^b\!Z}\ar[d]_{}&
    \bA\ar[d]^{\kappa'}\\
    \LCob\ar[r]_{\tiZ}&
    \tsA.
  }
\end{gather*}
The category $\bLCob$ of {\em boundary Lagrangian cobordisms}, defined
below, is a braided monoidal subcategory of $\LCob$ which contains
$\sLCob$ as a braided monoidal subcategory.  The vertical arrows on
the left are inclusion functors.  Like $\sLCob$ and $\LCob$, the
objects of $\bLCob$ are {non-negative}
integers.  The morphisms from $m$
to $n$ in $\bLCob$ are {cobordisms}
$C=(C,c):m\to n$, in the sense of
Section~\ref{sec:sLCob}, such that the composite 
${\underline{C}}:=  V_n\circ C:m\to 0$
is a homology handlebody where the $m$ meridian
curves in $\partial {\underline{C}} \cong \partial V_m$ bound mutually disjoint,
connected, oriented surfaces $S_1,\dots,S_m$.  This notion may be
thought of as a cobordism version of {boundary links}.
{Note that $\bLCob(0,0)=\LCob(0,0)$ is essentially the monoid of homology $3$-spheres
(whereas $\sLCob(0,0)$ is trivial)}.

The target category $\bA$ of ${}^b\!Z$ is  {much} larger than $\AB$: there, Jacobi
diagrams in handlebodies may involve connected components with no
univalent vertex. {Note that $\bA(0,0)$ is the target of the LMO invariant of homology $3$-spheres
(whereas $\AB(0,0)$ is {$1$-dimensional}).}
The category $\bA$ includes $\AB$ as a symmetric
monoidal linear subcategory, {and}  the functor $\kappa':\bA\to\tsA$ is a
natural extension of the hair map $\kappa:\AB\to\tsA$.

We plan to construct the functor ${{}^b\!}Z$ as follows.  Every boundary
Lagrangian {cobordism}
$C:m\to n$ is obtained from a special Lagrangian cobordism ${C'}:m\to n$
by surgery along a framed link $L$ in {$C'$} such that
\begin{itemize}
\item each component of $L$ is null-homotopic in $V_n\circ {C'}=V_m$,
\item the linking matrix of $L$ is diagonal with diagonal entries
  $\pm1$, where linking numbers and framings of components of $L$ are
  defined in $V_n\circ {C'}=V_m$.
\end{itemize}
{The Kontsevich integral $Z(C'\cup L)\in \AB^{\!\dagger}(m,n)$ is as outlined in Section \ref{sec:incorp-tangl}. 
Then} the invariant ${{}^b\!}Z(C)$ is obtained from $Z(C'\cup L)$  
by applying an equivariant version of the Aarhus
integral developed by Garoufalidis and Kricker~\cite{GK} to each
component of {the surgery link} $L$.

{
We hope that ${}^b\!Z$ will be useful to study the LMO invariant of homology $3$-spheres
in relation with their fundamental groups. 
Indeed, it seems difficult to conduct such a study using the LMO functor $\tiZ$ instead of ${}^b\!Z$.
}

\def\cprime{$'$}
\providecommand{\bysame}{\leavevmode\hbox to3em{\hrulefill}\thinspace}
\providecommand{\MR}{\relax\ifhmode\unskip\space\fi MR }
\providecommand{\MRhref}[2]{  \href{http://www.ams.org/mathscinet-getitem?mr=#1}{#2} }
\providecommand{\href}[2]{#2}

\end{document}